\documentclass[11pt]{amsart}

\usepackage{amssymb,amsfonts, amsmath, amsthm}
\usepackage[all,arc]{xy}
\usepackage{enumerate}
\usepackage{mathtools}
\usepackage[mathscr]{euscript}
\usepackage[T1]{fontenc}
\usepackage[utf8]{inputenc}
\usepackage{array}
\usepackage{tikz}
\usepackage{mathpazo}
\usepackage{tikz-cd}
\usepackage{textcomp}
\usepackage[pagebackref, colorlinks=true, linkcolor=blue, citecolor = red]{hyperref}
 \renewcommand*{\backref}[1]{}
 \renewcommand*{\backrefalt}[4]{({%
     \ifcase #1 Not cited.%
           \or On p.~#2%
           \else On pp.~#2%
     \fi%
     })}
\usepackage[nameinlink,capitalise,noabbrev]{cleveref}
\crefname{subsection}{Subsection}{Subsection}

\usepackage{url}
\usetikzlibrary{patterns}
\usepackage{geometry}
\usepackage{todonotes}

\DeclareMathAlphabet{\mathbbe}{U}{bbold}{m}{n}

\def\DDelta{{\mathbbe{\Delta}}}
\newcommand{\DD}{\DDelta}

\newcommand{\C}{\mathscr{C}}
\newcommand{\cC}{\mathcal{C}}
\newcommand{\D}{\mathscr{D}}
\newcommand{\cD}{\mathcal{D}}

\newcommand{\cF}{\mathcal{F}}
\newcommand{\E}{\mathscr{E}}
\renewcommand{\H}{\mathscr{H}}

\newcommand{\bI}{\mathbb{I}}
\newcommand{\sL}{\mathscr{L}}

\newcommand{\cN}{\mathcal{N}}
\newcommand{\s}{\mathscr{S}}

\newcommand{\sOall}{\mathscr{O}^{\mathrm{(all)}}}
\renewcommand{\P}{\mathscr{P}}
\newcommand{\sP}{s\P}
\newcommand{\R}{\mathscr{R}}
\newcommand{\T}{\mathscr{T}}

\newcommand{\V}{\mathscr{V}}
\newcommand{\W}{\mathscr{W}}
\newcommand{\cW}{\mathcal{W}}
\renewcommand{\ss}{s\mathscr{S}}
\newcommand{\sss}{ss\mathscr{S}}

\newcommand{\cM}{\mathcal{M}}
\newcommand{\Hom}{\mathrm{Hom}}

\newcommand{\Map}{\mathrm{Map}}
\newcommand{\QCat}{\mathscr{Q}\mathscr{C}\mathrm{at}}

\newcommand{\Nat}{\mathrm{Nat}}
\newcommand{\Fun}{\mathrm{Fun}}
\newcommand{\map}{\mathrm{map}}
\newcommand{\uHom}{\underline{\Hom}}
\newcommand{\uMap}{\underline{\Map}}
\newcommand{\uFun}{\underline{\Fun}}
\newcommand{\uNat}{\underline{\Nat}}
\newcommand{\comma}{,}
\newcommand{\set}{\mathscr{S}\mathrm{et}}
\newcommand{\cat}{\C\mathrm{at}}
\newcommand{\id}{\mathrm{id}}

\newcommand{\Diag}{\mathrm{Diag}}

\newcommand{\Ho}{\mathrm{Ho}}
\newcommand{\ev}{\mathrm{ev}}

\renewcommand{\exp}[2]{\mathrm{exp}(#1,#2)}
\newcommand{\rotatebot}{\rotatebox{270}{$\bot$}}

\newcommand{\sset}{s\set}

\newcommand{\opGroth}{\mathrm{op}\mathscr{G}\mathrm{roth}}

\newcommand{\Fib}{\mathscr{F}\mathrm{ib}}

\newcommand{\Ch}{\mathscr{C}\mathrm{h}}
\newcommand{\coDisc}{\mathrm{co}\mathcal{D}\mathrm{isc}}
\newcommand{\qcatFun}{\mathcal{F}\mathrm{un}}

\newcommand{\Val}{\mathscr{V}\mathrm{al}}
\newcommand{\VEmb}{\mathscr{V}\mathscr{E}\mathrm{mb}}
\newcommand{\fDiag}{\mathscr{D}\mathrm{iag}}
\newcommand{\phiDiag}{\mathrm{Diag}_1}

\newcommand{\Obj}{\mathrm{Obj}}
\newcommand{\Tw}{\mathrm{Tw}}

\newcommand{\LFib}{\mathcal{L}\Fib}
\newcommand{\ds}{\displaystyle}
\newcommand{\incl}{\mathrm{incl}}
\newcommand{\Seg}{\mathscr{S}\mathrm{eg}}
\newcommand{\CSS}{\mathscr{C}\mathscr{S}\mathscr{S}}
\newcommand{\segcat}{\Seg\cat}
\newcommand{\SegLFib}{\Seg\LFib}
\newcommand{\CSLFib}{\mathscr{C}\mathscr{S}\LFib}
\newcommand{\Cart}{\mathscr{C}\mathrm{art}}
\newcommand{\coCart}{\mathrm{co}\Cart}
\newcommand{\SegcoCart}{\Seg\coCart}
\newcommand{\St}{\mathrm{St}}
\newcommand{\Un}{\mathrm{Un}}
\newcommand{\Str}{\mathscr{S}\mathrm{tr}}
\newcommand{\Bet}{\mathscr{B}\mathrm{et}}

\newcommand{\sint}{s\hspace{-0.04in}\int}

\newcommand{\bT}{\mathbb{T}}
\newcommand{\sH}{s\mathscr{H}}

\newcommand{\bH}{\mathbb{H}}

\newcommand{\sbT}{s\bT}

\newcommand{\sbI}{s\bI}

\newcommand{\reb}{]}
\newcommand{\leb}{[}
\newcommand{\ordered}[1]{< #1 >}
\newcommand{\Yon}{\mathrm{Yon}}

\newcommand{\Und}{\mathscr{U}\mathrm{nd}}
\newcommand{\Inc}{\mathscr{I}\mathrm{nc}}
\newcommand{\Disc}{\mathscr{D}\mathrm{isc}}

\newcommand{\vecknedge}[2]{[#1,...,#2]}
\newcommand{\vecknedgetwo}[3]{[#1,...,#2,#3]}

\newcommand{\edge}[2]{[#1,#2]}

\newcommand{\rbot}{\rotatebox{270}{$\bot$}}

\newcommand{\adjun}[4]{
\begin{tikzcd}[row sep=0.5in, column sep=0.5in]
 #1  \arrow[r, shift left=1.8, "#3"] \pgfmatrixnextcell
 #2 \arrow[l, shift left=1.6, "#4", "\bot"'] 
\end{tikzcd}
}

\newcommand{\comsq}[8]{
  \begin{tikzcd}[row sep=0.5in, column sep=0.5in]
    #1 \arrow[r, "#5"] \arrow[d, "#6"']
    \pgfmatrixnextcell #2 \arrow[d, "#7"] \\
    #3 \arrow[r, "#8"]
    \pgfmatrixnextcell #4
  \end{tikzcd}
}

\newcommand{\pbsq}[8]{
  \begin{tikzcd}[row sep=0.5in, column sep=0.5in]
    #1 \arrow[r, "#5"] \arrow[d, "#6"'] \arrow[dr, phantom, "\ulcorner", very near start]
    \pgfmatrixnextcell #2 \arrow[d, "#7"] \\
    #3 \arrow[r, "#8"']
    \pgfmatrixnextcell #4
  \end{tikzcd}
}

\newtheorem{theone}[equation]{Theorem}
\newtheorem{lemone}[equation]{Lemma}
\newtheorem{propone}[equation]{Proposition}
\newtheorem{corone}[equation]{Corollary}

\theoremstyle{definition}
\newtheorem{defone}[equation]{Definition}
\newtheorem{exone}[equation]{Example}

\theoremstyle{remark}
\newtheorem{remone}[equation]{Remark}
\newtheorem{notone}[equation]{Notation}

\newtheoremstyle{TheoremNum}
{}{}              
{\itshape}                      
{}                              
{\bfseries}                     
{.}                             
{ }                             
{\thmname{#1}\thmnote{ \bfseries #3}}
\theoremstyle{TheoremNum}
\newtheorem{thmn}{Theorem}

\numberwithin{equation}{section}

\title{Yoneda Lemma for \texorpdfstring{$\cD$}{D}-Simplicial Spaces}
\author{Nima Rasekh}
\address{{\'E}cole Polytechnique F{\'e}d{\'e}rale de Lausanne, SV BMI UPHESS, Station 8, CH-1015 Lausanne, Switzerland}
\email{nima.rasekh@epfl.ch}

\date{August 2021}

\begin{document}

\begin{abstract}
For a small category $\cD$ we define fibrations of simplicial presheaves on the category $\cD\times\DD$, which we call {\it localized $\cD$-left fibration}. We show these fibrations can be seen as fibrant objects in a model structure, the {\it localized $\cD$-covariant model structure}, that is Quillen equivalent to a category of functors valued in simplicial presheaves on $\D$, where the Quillen equivalence is given via a generalization of the {\it Grothendieck construction}. We use our understanding of this construction to give a detailed characterization of fibrations and weak equivalences in this model structure and in particular obtain a Yoneda lemma. 

We apply this general framework to study {\it  Cartesian fibrations} of $(\infty,n)$-categories, for models of $(\infty,n)$-categories that arise via simplicial presheaves, such as {\it $n$-fold complete Segal spaces}. This, in particular, results in the {\it Yoneda lemma} and {\it Grothendieck construction} for Cartesian fibrations of $(\infty,n)$-categories. 
\end{abstract}

\maketitle
 \addtocontents{toc}{\protect\setcounter{tocdepth}{1}}

\tableofcontents

 \section{Introduction}\label{Sec Introduction}

\subsection{\texorpdfstring{$(\infty,1)$}{(oo,1)}-Categories and Their Fibrations}
The theory of {\it $(\infty,1)$-categories} (also known as the theory of {\it $\infty$-categories} or simply {\it homotopy theories}) has proven very effective in using categorical techniques to study homotopical phenomena and we can see its effect in all branches of homotopy theory from {\it chromatic homotopy theory} \cite{mathewstojanoska2016tmf} to {\it algebraic K-theory} \cite{gepnergrothnikolaus2015infiniteloopspacemachine} to {\it derived geometry} \cite{lurie2018sag} to even {\it homotopy type theory} \cite{gepnerkock2017univalence}.

Unfortunately, such theoretical benefits do come with a price: Any notion of category that enables us to study objects that arise in homotopy theory needs to be suitably flexible and weak. In particular, in the $(\infty,1)$-categorical world any classical categorical notion such as associativity and functoriality needs to be defined in a suitably weak and ``up to homotopy" manner to have the desired applications. 

As a result it is often quite challenging (and sometimes even impossible) to define a functor of $(\infty,1)$-categories directly as each homotopy used to specify a functor necessitates a higher homotopy creating a never-ending cycle of ``homotopies between homotopies". Fortunately, classical category theory already provides us with an alternative method to study many important functors that arise in category theory. In order to understand it we need to review the role of {\it fibrations} in category theory.

Grothendieck, and more generally Bourbaki, introduce what is now known as {\it discrete Grothendieck opfibrations} which are functors of categories that satisfy certain lifting conditions and then show that the now called {\it Grothendieck construction} \cite{grothendieck2003etalegroup} gives rise to an equivalence
$$\ds \int_\C: \Fun(\C, \set) \xrightarrow{ \ \simeq \ } \opGroth_{/\C},$$
between functors valued in in $\set$ and discrete Grothendieck opfibrations. 

This approach to functors has been generalized to the $(\infty,1)$-categorical setting in various steps. The first were taken by Joyal using the most prominent model of $(\infty,1)$-categories (in the sense of \cite{toen2005unicity}), {\it quasi-categories} \cite{boardmanvogt1973qcats}. He introduced the notion of {\it left fibrations} of quasi-categories and studied their relations to functors valued in spaces \cite{joyal2008notes,joyal2008theory}. 

A next big step was taken by Lurie \cite{lurie2009htt}, who again used quasi-categories to define {\it coCartesian fibrations} and prove an equivalence between coCartesian fibrations and functors valued in quasi-categories, by constructing a Quillen equivalence 
\begin{center}
	\adjun{(\sset^+_{/S})^{Cart}}{\Fun(\mathfrak{C}[S]^{op},(\sset^+)^{Cart})^{proj}}{\St^+_S}{\Un^+_S}
\end{center}
of the relevant model structures \cite[Theorem 3.2.0.1]{lurie2009htt}, which in particular restricts to an equivalence between left fibrations and functors valued in spaces \cite[Theorem 2.2.1.2]{lurie2009htt}. Using Cartesian fibrations and the straightening construction he then studies many important categorical constructions, such as {\it adjunctions} \cite[Section 5.2]{lurie2009htt}, {\it $\infty$-toposes} \cite[Chapter 6]{lurie2009htt}, {\it symmetric monoidal structures} and {\it $\infty$-operads} \cite[Chapter 2]{lurie2017ha}. As a very elegant example, for a given quasi-category $S$ and object $x$, there is no direct way to construct a {\it corepresentable functor} $\Map_S(x,-)$, given that composition is only defined up to homotopy. However, the corresponding right fibration, by \cite[Theorem 2.2.1.2]{lurie2009htt}, is the under-quasi-category $S_{x/} \to S$, which can be constructed very easily \cite[Proposition 1.2.9.3]{lurie2009htt}.

In the years since the original result by Lurie appeared, there have been many generalizations and alternative perspectives on the $(\infty,1)$-categorical Grothendieck construction. Heuts and Moerdijk present an alternative construction to the Grothendieck construction for left fibrations to the one given by Lurie also in the context of quasi-categories, which has the benefit that it is very explicit and easy to follow over strict categories \cite{heutsmoerdijk2015leftfibrationi,heutsmoerdijk2016leftfibrationii}. Another construction of the straightening construction for left fibrations, also in the context of quasi-categories can be found in the work of Stevenson \cite{stevenson2017covariant}. There is also an analysis of left fibrations via the {\it universal left fibration} due to Cisinski \cite{cisinski2019highercategories}, using ideas from {\it Cisinski model categories} \cite{cisinski2006cisinskimodelstructure}, whose results have been expanded by Nguyen to also apply to the coCartesian model structure \cite{nguyen2019covariant}.

Parallel to the development of left and coCartesian fibrations for quasi-categories, there has also been a development of fibrations for {\it complete Segal spaces} \cite{rezk2001css}, which is another model of $(\infty,1)$-categories. This in particular includes work by deBrito \cite{debrito2018leftfibration} and Kazhdan and Varshavsky \cite{kazhdanvarshvsky2014yoneda}, as well as \cite{rasekh2017left}. 

Last but not least, there has also been a model-independent study of coCartesian fibrations. The work by Mazel-Gee \cite{mazelgee2019cartfib,mazelgee2019grothendieck}, while relying on quasi-categories, focuses on universal properties rather than specific model structure to study Cartesian fibrations and the Grothendieck construction. Similarly, the work by Ayala and Francis \cite{ayalafrancis2020fibrations}, although technically in the context of quasi-categories, relies primarily on the universal properties of {\it exponentiable fibrations}. A more precise model-independent approach is due to Riehl and Verity, who introduced the model-independent notion of an {\it $\infty$-cosmos} \cite{riehlverity2018elements} and use it to study fibrations and the Yoneda lemma \cite{riehlverity2017inftycosmos}, which can then be applied to the models where fibrations had been studied already, such as quasi-categories and complete Segal spaces, but also models where it had not been studied, such as {\it Segal categories} \cite{hirschowitzsimpson1998segalcat,bergner2007threemodels} and {\it marked simplicial sets} \cite{verity2008complicial,lurie2009htt} (see \cite{bergner2010survey,bergner2018book} for a review of most of these models). Finally, left fibrations \cite{riehlshulman2017rezktypes} and coCartesian fibrations \cite{buchholtzweinberger2021cartesianfibhott} have also been studied in the context of {\it homotopy type theory} \cite{hottbook2013}, a new logical foundations to mathematics more amenable to homotopical thinking. 

To summarize, there are now many different methods to effectively translate between functors and fibrations and these methods have been used to great effect in the study of homotopy theory.

\subsection{First Steps into the World of \texorpdfstring{$(\infty,n)$}{(oo,n)}-Categories}
Coming back to classical categories, we often expand our definitions to be able to better understand certain categories and their properties. The common example is the category of small categories, which we often consider as a {\it $2$-category} \cite{kellyross1974twocat} enabling us to include natural transformations as part of the data and hence study {\it equivalences of categories} rather than just isomorphisms. A similar line of thinking is used when studying the $2$-category of {\it elementary toposes} or {\it Grothendieck toposes} \cite{maclanemoerdijk1994topos,johnstone2002elephanti,johnstone2002elephantsii}. 

Applying that same mindset to $(\infty,1)$-categories, we hence want a proper theory of {\it $(\infty,2)$-categories}, which is the appropriate framework to study $(\infty,1)$-categories. In fact, thinking inductively, we rather want a proper theory of {\it $(\infty,n)$-categories}. As a result, there is now a long list of models of $(\infty,n)$-categories, such as {\it $n$-fold complete Segal spaces} \cite{barwick2005nfoldsegalspaces}, {\it $\Theta_n$-spaces} \cite{rezk2010thetanspaces}, {\it $\Theta_n$-sets} \cite{ara2014highersegal}, {\it $n$-complicial sets} \cite{verity2008complicial} and {\it $n$-comical sets} \cite{dkls2020cubical,campionkapulkinmaehara2020comical} (see \cite{bergner2011modelsinftyn} for a review of many of these models). Moreover, $(\infty,n)$-categories have found interesting applications in {\it derived algebraic geometry} \cite{gaitsgoryrozenblyum2017dagI,gaitsgoryrozenblyum2017dagII} (via the models of {\it scaled simplicial sets} \cite{gagnaharpazlanari2020graytensor} and comical sets) and {\it topological field theories} \cite{lurie2009cobordism,calaquescheimbauer2019cobordism} (via the model of $n$-fold complete Segal spaces).  

However, our knowledge of $(\infty,n)$-categories is far more limited than $(\infty,1)$-categories. First of all, for a general $n$ while we do know that some models coincide, such as $n$-fold complete Segal spaces, $\Theta_n$-spaces and $\Theta_n$-sets \cite{ara2014highersegal,bergnerrezk2013comparisoni,bergnerrezk2020comparisonii}, or complicial sets and comical sets \cite{dohertykapulkinmaehara2021comical}, and even have a axiomatic characterization of $(\infty,n)$-categories \cite{barwickschommerpries2011unicity}, the general equivalence of all models is known only for $n=2$ \cite{gagnaharpazlanari2019twocat,lurie2009goodwillie}. Even in that case most equivalences are given via very complicated zigzags of Quillen equivalences of model categories and constructing direct equivalences is an ongoing effort \cite{bergnerrovelliozornova2021comparisonthetatwo}.

 Moreover, even the most foundational categorical concepts of $(\infty,2)$-categories, such as {\it limits}, are currently being developed using marked (or scaled) simplicial sets \cite{gagnaharpazlanari2020inftytwolimits,gagnaharpazlanari2021bilimits,garcia2020markedcolimits} or {\it double $(\infty,1)$-categories} \cite{moser2020double,clingmanmoser2020bipres,clingmanmoser2020bilim,mosersarazolaverdugo2020twocatmodel}. The same challenges persist when looking at fibrations of $(\infty,n)$-categories. There are now interesting results for fibrations of $(\infty,2)$-categories using the model of scaled simplicial sets \cite{lurie2009goodwillie,harpaznuitenprasma2019twisted,gagnaharpazlanari2021inftytwocartfib,garciastern2020theorema,garciastern2020twistedarrow,garciastern2021twocat}, but no other models or higher dimensions. 

 One option would be to forgo any focus on specific models and approach $(\infty,n)$-categories from the perspective of enriched quasi-categories \cite{gepnerhaugseng2015enrichedinftycat,hinich2020enrichedyoneda}. Using this approach has resulted in enriched forms of the Yoneda lemma \cite{hinich2020enrichedyoneda,berman2020enrichedpresheaves}, which then can be applied in particular to $(\infty,n)$-categories, although it still lacks a theory of fibrations. However, even if a theory Cartesian fibrations of enriched quasi-categories is developed, the enriched quasi-categorical approach does have certain drawbacks. First of all, developing a proper theory of enriched $\infty$-categories is incredibly challenging and remains complicated to understand. Moreover, many applications of $(\infty,n)$-categories (such as topological field theories) do in fact use concrete models (such as $n$-fold complete Segal spaces) \cite{calaquescheimbauer2019cobordism} and we expect that the further study of the theory of $(\infty,n)$-categories, such as (lax) (co)limits and adjunctions does require study them using concrete models rather than enriched quasi-categories. 

The goal of this work is to take an important step by studying fibrations for most models of $(\infty,n)$-categories that arise as simplicial presheaves. This includes $n$-fold complete Segal spaces, but also complete Segal spaces enriched over $\Theta_n$-spaces. In line with previous important results regarding fibrations, we in particular want to study the Yoneda lemma for $(\infty,n)$-coCartesian fibrations and a Grothendieck construction for $(\infty,n)$-coCartesian fibrations, generalizing the straightening construction from $(\infty,1)$-categories to the $(\infty,n)$-categorical setting. 

\subsection{Generalizing to \texorpdfstring{$\cD$}{D}-Fibrations}
Before we proceed to the results of this paper it is instructive to review the most general method for constructing of models of $(\infty,n)$-categories based on simplicial presheaves (\cite[Corollary 7.2]{bergnerrezk2020comparisonii}, generalizing \cite{rezk2010thetanspaces,barwick2005nfoldsegalspaces}), which always proceeds in three steps:
\begin{enumerate}
	\item Start with the {\it Kan model structure} on simplicial sets.
	\item Choose the appropriate diagram category $\Theta_k \times \DD^{n-k}$ and give the category of simplicial presheaves the {\it injective model structure}.
	\item Impose various relations: the {\it Segal condition} for composition, the {\it completeness condition} and, if necessary, the {\it discreteness conditions}. Obtain the final model structure via {\it Bousfield localization}.
\end{enumerate}
We want to generalize these three steps from an absolute setting to a fibrational setting. 

For the first step we need a fibrational notion of spaces. This is precisely obtained by studying left fibrations of simplicial spaces \cite{rasekh2017left}, which is also reviewed in \cref{subsec:lfib}.

Our next step is to generalize left fibrations to diagram categories. Instead of restricting ourselves to specific diagrams, we take a more general approach: {\it $\cD$-simplicial spaces}. Let $\cD$ be a small category. Then a $\cD$-simplicial space is a simplicial presheaf $\cD^{op} \times\DD^{op}\to \s$. A $\cD$-left fibration is then simply an injective fibration of $\cD$-simplicial spaces $Y \to X$, such that for every object $d$ in $\cD$, the map of simplicial spaces $Y[d] \to X[d]$ is a left fibration as studied in \cite{rasekh2017left} (\cref{def:nleft}).

This simple looking definition has far reaching implications. Concretely, it can be realized as fibrant objects in a model structure on an over-category (\cref{the:ncov model}), which has many desirable properties (such as {\it invariance under level-wise complete Segal space equivalences} (\cref{the:dcov css inv}) or the {\it recognition principle for $\cD$-covariant equivalences} (\cref{the:ncov equiv nlfib})). It also comes with its own {\it Yoneda lemma for $\cD$-left fibrations} (\cref{the:levelwise Yoneda}). Most importantly, we present a {\it Grothendieck construction} for $\cD$-left fibrations over an appropriate base (\cref{the:simp grothendieck}), which proves $\cD$-left fibration over a nerve $N_\cD\C$ is equivalent to $\sP(\cD)$-enriched functors $\C \to \sP(\cD)$. 

In the last step we need to localize $\cD$-left fibrations to obtain fibrations which have the appropriately localized values: the $S$-localized $\cD$-left fibration (\cref{def:cdleft s}). We will show that most results regarding $\cD$-left fibrations will carry over to its localized version. In particular it again comes with a model structure (\cref{the:cdcov s}), with similar properties, such as {\it invariance} (\cref{the:cdcov s invariant s equiv}) and {\it recognition of weak equivalences} (\cref{the:recognition principle s}). Crucially, we get a {\it Grothendieck construction} for $S$-localized $\cD$-left fibrations, which gives us an equivalence between $S$-localized $\cD$-left fibrations over $N_\cD\C$ and $\sP(\cD)$-enriched functors $\C\to \sP(\cD)$ valued in $S$-local objects (\cref{the:grothendieck simp s}).

\subsection{Coming back to Fibration of \texorpdfstring{$(\infty,n)$}{(oo,n)}-Categories}
Having developed such a general theory of $S$-localized fibrations for $\cD$-spaces we can construct fibrations of $(\infty,n)$-categories, which is the content of \cref{sec:infn fib}. The results will apply to various fibrations (\cref{not:left fib}) and various bases (\ref{eq:inftyn categories}), however, here we will focus on the case of $(\infty,n)$-coCartesian fibrations over $(\infty,n+1)$-categories, as it is expected that this is the most relevant example. Recall the category of simplicial presheaves on $\Theta_k \times \DD^{n-k}$ carries a model structure where the fibrant objects are $(\infty,n)$-categories (\cite[Corollary 7.2]{bergnerrezk2020comparisonii}, also reviewed in \cref{subsec:infn cat}).

\begin{thmn}
	Let $\C$ be $\Theta_k \times \DD^{n-k+1}$-space that is an $(\infty,n+1)$-category, then there is a simplicial left proper combinatorial model structure on the over-category $\sP(\Theta_k \times \DD^{n-k+1})_{/\C}$, called the $(\infty,n)$-coCartesian model structure, that has the following properties.
	\begin{itemize}
		\item {\bf Enrichment:} The $(\infty,n)$-coCartesian model structure is enriched over the model structure for $(\infty,n)$-categories if $k=n$ (\cref{the:cocart model}).
		\item {\bf Fiber-wise Criterion for Fibrations:} A map $\D \to \C$ is an $(\infty,n)$-coCartesian fibration (meaning it is fibrant in the $(\infty,n)$-coCartesian model structure) if it is an $(\infty,n)$-left fibration (\cref{not:left fib}) and fiber-wise an $(\infty,n)$-category (\cref{cor:fibrancies nfib}).
		\item {\bf Yoneda lemma:} For an object $c$ in $\C$, there is a representable $(\infty,n)$-coCartesian fibration $\C_{c/} \to \C$ (\cref{the:yoneda}) and for every $(\infty,n)$-coCartesian fibration $\sL \to \C$ we have an equivalence of $(\infty,n)$-categories 
		$$\Fib_cL \simeq \uMap_{/\C}(\C_{c/},L)$$
		for every object $c$ in $\C$ (\cref{cor:yoneda}).
		\item {\bf Fiber-wise Dwyer-Kan Equivalences:} A map $\sL\to \sL'$ of Segal $(\infty,n)$-coCartesian fibrations over $\C$ (\cref{not:left fib}) is an $(\infty,n)$-coCartesian equivalence if and only if for every object $c$, the map of fibers $\Fib_c\sL\to \Fib_c\sL'$ is a Dwyer-Kan equivalence of Segal $(\infty,n)$-categories (\cref{cor:fib dk}). 
		\item {\bf Recognition Principle for coCartesian Equivalences:} $F: \D \to \E$ over $\C$ is an $(\infty,n)$-coCartesian equivalence if and only if 
		$$(\D \times_\C \C_{/c})^{grpd} \to (\E \times_\C \C_{/c})^{grpd}$$ 
		is an $(\infty,n)$-category equivalence for all objects $c$ in $\C$ (\cref{cor:recognition principle inftyn}). Here $(-)^{grpd}$ is the groupoidification $(\infty,n)$-category (\cref{not:diag grpd}). 
		\item {\bf Twisted Arrows:} Let $\C$ be a $(\infty,n+1)$-category. There exists an $(\infty,n)$-coCartesian fibration $\Tw(\C) \to \C^{op} \times \C$, with fiber over $(x,y)$ the mapping $(\infty,n)$-category $\map_\C(x,y)$ (\cref{cor:twisted cocart}).
		\item {\bf Invariance of coCartesian Fibrations:} A Dwyer-Kan equivalence $F: \C \to \D$ of Segal $(\infty,n+1)$-categories (\cref{rem:dk infn}) induces a Quillen equivalence of $(\infty,n)$-coCartesian model structures (\cref{cor:invariance property}).
		\item {\bf Exponentiability:} If $n=k$ and $p:\R \to \C$ is an $(\infty,n)$-right or $(\infty,n)$-left fibration, then the adjunction
		\begin{center}
			\adjun{(\sP(\cD\times\DD)_{/\C})^{\cat_{(\infty,n)}}}{(\sP(\cD\times\DD)_{/\R})^{\cat_{(\infty,n)}}}{p^*}{p_*}
		\end{center}
		is a Quillen adjunction (\cref{cor:pullback rfib s nonfib base nfib}), were both sides have the $(\infty,n)$-category model structure.
		\item {\bf Grothendieck Construction:} For an $\sP(\cD)$-enriched category $\C$, there are Quillen equivalences (\cref{the:grothendieck infn})

			\strut  \hspace{-1in}	
			\begin{tikzcd}[row sep=0.3in, column sep=0.9in]
				\Fun(\C,\sP(\cD)^{inj_S})^{proj} \arrow[r, shift left = 1.8, "\sint_{\cD/\C}"] & 
				(\sP(\cD\times\DD)_{/N_\cD\C})^{(cov_{(\infty,n)})_S} \arrow[l, shift left=1.8, "\sH_{\cD/\C}", "\bot"'] \arrow[r, shift left=1.8, "\sbT_{\cD/\C}"] & \Fun(\C,\sP(\cD)^{inj_S})^{proj} \arrow[l, shift left=1.8, "\sbI_{\cD/\C}", "\bot"'] 
			\end{tikzcd}, 
		
	which, if $n=k$, gives us an equivalence of $(\infty,n+1)$-categories (\cref{cor:common grothendieck})
	 \begin{center}
		\begin{tikzcd}[row sep=0.5in, column sep=0.9in]
			\Fun_{\cat_{(\infty,n+1)}}(\C^{op},\cat_{(\infty,n)}) \arrow[r, shift left = 1.8] & 
			\Cart_{(\infty,n)/\C} \arrow[l, shift left=1.8, "\simeq"'] 
		\end{tikzcd}.
	\end{center} 
	Under this equivalence the twisted arrow construction (\cref{cor:twisted cocart}) corresponds to the mapping functor $\uMap_\C(-,-): \C^{op} \times \C \to \cat_{(\infty,n)}$ (\cref{ex:twisted infn}).
	\item {\bf Universality and Univalence:} There exists a {\it universal $(\infty,n)$-coCartesian fibration} (\cref{cor:lfib classifier ninf}) $\pi_*: (\cat_{(\infty,n)})_{[0]/} \to \cat_{(\infty,n)}$, which is {\it univalent} (\cref{cor:universal cocart univalent}).
	\end{itemize}
\end{thmn}

\subsection{Applications of \texorpdfstring{$(\infty,n)$}{(oo,n)}-Cartesian Fibrations} \label{subsec:applications}
What are the applications of a theory of fibrations of $(\infty,n)$-categories? Given the prominent role Cartesian fibrations play in $(\infty,1)$-category theory, we would expect a similar role for $(\infty,n)$-Cartesian fibrations. However, we can also point to several concrete applications that could result from this work:

\begin{enumerate}
	\item {\bf Topological Field Theories:} {\it Topological field theories} (TFT's) are functors out of the $(\infty,n)$-category of {\it bordisms} and lie at the intersection of {\it physics} \cite{witten1988tqft}, {\it higher category theory} \cite{baezdolan1995tqft} and {\it differential geometry} \cite{lurie2009cobordism,stolzteichner2011fieldtheory} and have hence attracted the attention from various kinds of mathematicians. Unfortunately, given their higher categorical nature constructing field theories has proven challenging, as they are obtained by constructing $(\infty,n)$-functors out of the $n$-fold Segal spaces of bordism \cite{calaquescheimbauer2019cobordism} and constructing functors often involves dealing with infinite coherence conditions. Hence many different ideas, such as {\it $2$-Segal sets} \cite{turaev2010tft,dyckerhoffkapranov2019higherset}, {\it derived algebraic geometry} \cite{calaquehaugsengscheimbauer2019aksz}, {\it stable homotopy theory} \cite{freedhopkins2021invtft} and {\it homological constructions} \cite{mmst2020tqft} have been used to construct and study specific field theories.
	
	As we outlined above, in the world of $(\infty,1)$-categories Cartesian fibrations have been used very effectively to construct many important functors \cite{lurie2009htt,lurie2017ha}. Combining these observations we would hope to construct topological field theories by constructing appropriate fibrations over the $n$-fold Segal space of bordisms.
	
	\item {\bf $(\infty,n)$-Univalence:} The concept of {\it univalence} first arose in {\it homotopy type theory} \cite{hottbook2013} and made its way into higher category theory in order to prove the existence of models of homotopy type theories in the higher categorical setting \cite{kapulkinlumsdaine2012kanunivalent,gepnerkock2017univalence}. However, we can also study univalence independent of its relation to homotopy type theory via {\it representability of Cartesian fibrations} \cite{rasekh2021univalence,rasekh2017cartesian}. 
	
	Given the effectiveness of Cartesian fibrations in the study of univalence, our next expected step is to use $(\infty,n)$-Cartesian fibrations to study $(\infty,n)$-univalence. This would also involve understanding {\it representable $(\infty,n)$-Cartesian fibrations}. While $(\infty,n)$-univalence could be defined and studied completely in terms of $(\infty,n)$-categories it should still have connections to homotopy type theory. In particular, the $(\infty,2)$-categorical univalence is expected to have connections to the {\it directed univalence axiom} \cite{riehl2018jmm}.
	\item {\bf $(\infty,n)$-Topos Theory:} The development of {\it Grothendieck $(\infty,1)$-toposes} (called $\infty$-topos in \cite{lurie2009htt}) was strongly guided by a decent understanding of {\it Grothendieck $1$-toposes}. In particular, we know that a locally presentable $1$-category is a Grothendieck $1$-topos if it satisfies certain conditions known as the {\it Giraud axioms} \cite{maclanemoerdijk1994topos} (also called {\it weak descent} in \cite{rezk2010toposes}). This has been lifted directly to the $\infty$-categorical setting by Lurie \cite[Section 6]{lurie2009htt} and Rezk \cite{rezk2010toposes}. Unfortunately the $n$-categorical situation is far less clear. It is not completely understood which conditions we should impose on a locally presentable $n$-category in order to get a working notion of a Grothendieck $n$-topos and even the $2$-categorical case is an active field of research \cite{weber2007twotopos}. Accordingly, there is also no analogous definition for a Grothendieck $(\infty,n)$-topos. 
	
	On the other hand, we can now characterize $(\infty,1)$-toposes via the existence of univalent morphisms \cite{gepnerkock2017univalence,rasekh2018elementarytopos,stenzel2020comprehension}. Hence, assuming the development of univalent morphisms in the $(\infty,n)$-setting (as outlined in the previous item), we could reasonably expect to define $(\infty,n)$-toposes in a completely analogous manner, without having to worry about the $(\infty,n)$-categorical analogues of presentability and Giraud's axioms.
\end{enumerate}
 
 \subsection{Outline}
 The first section (\cref{sec:review}) provides a review and sets relevant notation. In particular, we quickly review enriched model categories via Joyal-Tierney calculus (\cref{subsec:joyal tierney calculus}) and enriched categories (\cref{subsec:enriched}), however, spend a lot of time and develop a lot of notation for categories enriched over presheaf categories in \cref{subsec:enriched psh} as it is a key concept in this paper. We then focus on two classes of model structures on presheaf categories, level-wise model structures in \cref{subsec:level model} and enriched model structures in \cref{subsec:enriched css}. Finally, as we want to study fibrations, we also review discrete Grothendieck opfibrations in \cref{subsec:grothendieck fib} and left fibrations in \cref{subsec:lfib}.
 
  We then start our study of $\cD$-left fibrations in \cref{sec:left fib}. In the first subsection, \cref{subsec:cdlfib}, we focus on various alternative characterizations and elementary properties. We then move on to prove that $\cD$-left fibrations come with a model structure and study its properties in \cref{subsec:cdcov}. Finally, we also review the Yoneda lemma for $\cD$-left fibrations in \cref{subsec:yoneda}.
  
  In the next section, \cref{sec:grothendieck}, we move on to a key result regarding $\cD$-left fibrations: the Grothendieck constructions. We proceed in three steps:
  \begin{enumerate}
  	\item First a strict Grothendieck construction relating diagrams of Grothendieck opfibrations to $\P(\cD)$-enriched functors $F:\C \to \P(\cD)$ (\cref{subsec:strict groth}).
  	\item Then a homotopical Grothendieck construction relating $\cD$-left fibrations over nerves to $\sP(\cD)$-enriched functors $F: \C \to \sP(\cD)$ via model structures (\cref{subsec:enriched groth}).
  	\item Finally, a generalization to $\cD$-left fibrations over very general $\cD$-simplicial spaces (\cref{subsec:grothendieck general}).
  \end{enumerate}
  
  \cref{sec:localized} localizes most results in \cref{sec:left fib} and \cref{sec:grothendieck}. In particular, \cref{subsec:localized model} constructs model structure for localized $\cD$-left fibrations, also giving some first properties. \cref{subsec:localized groth} localizes the Grothendieck construction given in \cref{subsec:enriched groth}. Finally, \cref{subsec:sloc features} uses the Grothendieck construction to study more refined features of the localized $\cD$-covariant model structure.
  
  In \cref{sec:infn fib} we finally apply all results in \cref{sec:localized} to the particular case of $(\infty,n)$-fibrations (\cref{subsec:infn fib}), after a quick review of simplicial presheaf models of $(\infty,n)$-categories in \cref{subsec:infn cat}. In \cref{subsec:examples}, we focus on some examples and counter-examples. 
  
  Finally, in \cref{sec:next} we discuss several possible next steps based on the work in this paper.
  
 \subsection{Background}
 We will assume familiarity with model category theory as introduced in \cite{hovey1999modelcategories,dwyersspalinski1995modelcat} and in particular Bousfield localizations \cite{hirschhorn2003modelcategories}. We will also extensively use results regarding left fibrations of simplicial spaces as introduced in \cite{rasekh2017left} either directly or by translating important concepts and there is only a quick review in \cref{subsec:lfib}. We will also use ideas from enriched category theory as given in \cite{kelly2005enriched}, however most relevant concepts have been reviewed in \cref{subsec:enriched}. For \cref{subsec:grothendieck general} we rely on results regarding Segal categories by Bergner \cite{bergner2007threemodels} as well enriched quasi-categories by Gepner and Haugseng \cite{gepnerhaugseng2015enrichedinftycat}. Finally, \cref{sec:infn fib} fundamentally relies on the results regarding $(\infty,n)$-categories due to Rezk \cite{rezk2010thetanspaces} and Bergner and Rezk \cite{bergnerrezk2013comparisoni,bergnerrezk2020comparisonii}.

 \subsection{Notation} \label{subsec:notation} 
 Throughout we mostly work with categories enriched in the presheaf category $\P(\cD)= \Fun(\cD^{op},\set)$ and simplicial presheaf category $\sP(\cD)= \Fun(\cD^{op} \times \DD^{op},\set)$ and so use various notation regarding enrichments.
 \begin{itemize}
  \item If $\C$ is enriched over $\P(\cD)$ we denote the {\it set} of maps between them by $\Hom_\C(X,Y)$ and the enriched morphisms by $\uHom_\C(X,Y)$. For a given object $X$ and maps 
  $g: Y \to X, Z \to X$, we will denote the set of maps $\Hom_{\C_{/X}}(Y,Z)$ by $\Hom_{/X}(Y,Z)$ and similarly use $\uHom_{/X}(Y,Z)$.
  \item There is one exception to the previous rule. If $\C$ is a category of functors, then we denote the 
  {\it set of natural transformations} from $F$ to $G$ by $\Nat(F,G)$, following conventional notation, and the enriched version by $\uNat(F,G)$.
  \item If $\C$ is enriched over $\sP(\cD)$, we denote the {\it mapping simplicial set} by $\Map_\C(X,Y)$ and the enriched mapping space by $\uMap_\C(X,Y)$. Similar to the last one we will, instead of $\Map_{\C_{/X}}(Y,Z)$, use $\Map_{/X}(Y,Z)$ and $\uMap_{/X}(Y,Z)$.
  \item If $\C$ is Cartesian closed, we denote the internal mapping object by $Y^X$.
 \end{itemize}
 For a functor between small categories $F: \C \to \D$ and bicomplete category $\E$ we use the following notation for the induced 
 functors at the level of presheaf categories:
 \begin{center}
  \begin{tikzcd}[row sep=0.5in, column sep=0.5in]
   \Fun(\C,\E) \arrow[r, "F_!", "\bot"', shift left=0.7, bend left =30] \arrow[r, "F_*"', "\bot", shift right=0.7, bend right=30]& \Fun(\D,\E) \arrow[l, "F^*"' description] 
  \end{tikzcd}
 \end{center}
 Here $F^*$ is defined by precomposition, $F_!$ is the left Kan extension and $F^*$ the right Kan extension.
 \par 
 Similarly, for a given morphism $f:c \to d$ in a category $\C$ with small limits and colimits we denote the adjunctions 
 \begin{center}
  \begin{tikzcd}[row sep=0.5in, column sep=0.5in]
   \C_{/c} \arrow[r, "f_!", "\bot"', shift left=1.4, bend left=30] \arrow[r, "f_*"', "\bot", shift right=1.4, bend right = 30]& \C_{/d} \arrow[l, "f^*" description] 
  \end{tikzcd}
 \end{center}
 where $f_!$ is the post-composition functor, $f^*$ the pullback functor and $f_*$ is the right adjoint to $f^*$.
 
 Finally, let $\C$ be a category with final object $1$. Then we use notation $\{y\}: X \to Y$ 
 for the unique map that factors through the map $1 \to Y$ that picks out the element $y$ in $Y$.
 
\subsection{Acknowledgments}
 I want to thank Lyne Moser for many interesting conversations about the Grothendieck construction. I would also like to thank the Max-Planck-Institut f{\"u}r Mathematik for its hospitality and financial support.

\section{A Deep Dive into Enrichments and Some Fibrations} \label{sec:review}
 In this section we review several relevant concepts that we use throughout.

  \subsection{Joyal-Tierney Calculus and Enriched Model Categories} \label{subsec:joyal tierney calculus}
 We make extensive use of Cartesian and enriched model categories and so quickly review the relevant notation and concepts motivated by language introduced by Joyal and Tierney \cite[Section 7]{joyaltierney2007qcatvssegal}.
 
 \begin{notone}
 	For this subsection let $\C$ be a locally Cartesian closed bicomplete category.
 \end{notone}
 
 \begin{defone} \label{def:pushout product}
 	Let $f: A \to B$ and $g: C \to D$ be two maps in $\C$. We define the pushout product 
 	as the universal map out of the pushout
 	$$f \square g: A \times D \coprod_{A \times C} B \times C \to B \times D$$
 	induced by the commutative square
 	\begin{center}
 		\begin{tikzcd}[row sep=0.3in, column sep=0.3in]
 			A \times C \arrow[r, "\id_A \times g"] \arrow[d, "f \times \id_C"] & A \times D \arrow[d] \arrow[ddr, bend left = 20, , "f \times id_D"] & \\
 			B \times C \arrow[r] \arrow[drr, bend right = 20, "\id_B \times g"'] & \ds A \times D \coprod_{A \times C} B \times C \arrow[dr, "f \square g" description] & \\
 			& & B \times D 
 		\end{tikzcd}
 		.
 	\end{center}
 	Moreover, for two sets of maps $\mathcal{A}$ and $\mathcal{B}$ we use the notation 
 	$$\mathcal{A} \square \mathcal{B} = \{f \square g : f \in \mathcal{A}, g \in \mathcal{B} \}.$$
 \end{defone}
 
 \begin{defone} \label{def:pullback exponential} 
 	For two maps $f: A \to B$ and $p: Y \to X$ we define the pullback exponential 
 	$$\exp{f}{p}: Y^B \to Y^A \underset{X^A}{\times} X^B$$ 
 	induced by the commutative square
 	\begin{center}
 		\begin{tikzcd}[row sep=0.3in, column sep=0.3in]
 			Y^B \arrow[dr, "\exp{f}{p}" description] \arrow[drr, "Y^f", bend left =20] \arrow[ddr, "p^f"', bend right = 20] & & \\
 			& Y^A \underset{X^A}{\times} X^B \arrow[r] \arrow[d] & Y^A \arrow[d, "p^A"] \\
 			& X^B \arrow[r, "X^f"] & X^A
 		\end{tikzcd}
 		.
 	\end{center}
 	Moreover, for two sets of maps $\mathcal{A}$ and $\mathcal{X}$ we use the notation 
 	$$\exp{\mathcal{A}}{\mathcal{X}} = \{\exp{f}{p} : f \in \mathcal{A}, p \in \mathcal{X} \}$$
 \end{defone}
 
 These two functors give us an adjunction of arrow categories: 
 \begin{center}
 	\adjun{\Fun (\leb 1 \reb , \C)}{\Fun (\leb 1 \reb , \C)}{- \square f }{\exp{f}{-}}
 	.
 \end{center}
 
 The key result about these two constructions is that they can help us better understand lifting conditions.
 
 \begin{notone}
 	Let $\mathcal{L}$ and $\mathcal{R}$ be two sets of morphisms in $\C$. 
 	If $\mathcal{L}$ has the left lifting property with respect to $\mathcal{R}$ then we use the notation
 	$\mathcal{L} \pitchfork  \mathcal{R}$.
 \end{notone}
 
 \begin{propone} \label{prop:joyal tierney lifting}
 	(\cite[Proposition 7.6]{joyaltierney2007qcatvssegal})
 	Let $\mathcal{A}$, $\mathcal{B}$ and $\mathcal{X}$ be three sets of morphisms in $\C$. Then:
 	$$\mathcal{A} \square \mathcal{B} \pitchfork \mathcal{X} \Leftrightarrow
 	\mathcal{A} \pitchfork \exp{\mathcal{B}}{\mathcal{X}} \Leftrightarrow
 	\mathcal{B} \pitchfork \exp{\mathcal{A}}{\mathcal{X}}.$$
 \end{propone}

We will assume familiarity with standard results of model category theory as given in \cite{hovey1999modelcategories,hirschhorn2003modelcategories,dwyersspalinski1995modelcat}. We will in particular assume familiarity with the theory of Bousfield localization of left proper simplicial combinatorial model categories. However, we will use pushout products and pullback exponentials to quickly review Cartesian and enriched model categories. 
 
 \begin{defone} \label{def:cartesian model cat}
 	Let $\C^\cM$ be a model category. We say the model structure $\cM$ is {\it Cartesian} if it satisfies one of the following equivalent conditions
 	\begin{itemize}
 		\item For two cofibrations $i,j$ $i \square j$ is a cofibration, which is trivial if either is trivial.
 		\item For a cofibration $i$ and fibration $p$, $\exp{i}{p}$ is a fibration, which is trivial if either is trivial. 
 		\item For two cofibrations $i,j$ $i \square j$ is a cofibration. Moreover for each object $A$, $(\emptyset \to A) \square i$ is a trivial cofibration if $i$ is a trivial cofibration.
 		\item For two cofibrations $i,j$ $i \square j$ is a cofibrations. Moreover for each object $A$ in $\C$ the functor $A \times -: \C^\cM \to \C^\cM$ is left Quillen.
 	\end{itemize} 
 \end{defone} 
 
 Using Cartesian model structures we can define enriched model structures. 
 
 \begin{defone}\label{def:enriched model cat}
 	Let $\C^{\cM}$ be a category $\C$ with a Cartesian model structure $\cM$. Let $\D^{\cN}$ be enriched and (co)tensored over $\C$, where the tensor is given by the functor $-\otimes -: \C \times \D \to \C$. Then we say $\D^{\cN}$ is enriched over $\C^{\cM}$ if it satisfies one of the following equivalent conditions. 
 	\begin{itemize}
 		\item For an $\cM$-cofibration $i$ and an $\cN$-cofibration $j$, $i \square j$ is an $\cN$-cofibration, which is trivial if either is trivial.
 		\item For an $\cM$-cofibration $i$ and an $\cN$-fibration $p$, $\exp{i}{p}$ is an $\cN$-fibration, which is trivial if either is trivial. 
 		\item For an $\cM$-cofibration $i$ and an $\cN$-cofibration $j$, $i \square j$ is an $\cN$-cofibration. For each object $A$ in $\C$ and $\cN$-trivial cofibration $j$, $(\emptyset \to A) \square j$ is an $\cN$-trivial cofibration. For each $B$ in $\D$ and an $\cM$-trivial cofibration $i$, $i \square (\emptyset \to B)$ is a $\cN$-trivial cofibration.
 		\item For two cofibrations $i,j$ $i \square j$ is a cofibrations. Moreover for each object $A$ in $\C$ the functor $A \otimes -: \D^\cN \to \D^\cN$ is left Quillen and for each object $B$ in $\D$ the functor $- \otimes B: \C^\cM \to \D^\cN$ is left Quillen.
 	\end{itemize} 
 \end{defone}

 \begin{remone}
 	Notice the fact that the first two conditions are equivalent follows from \cref{prop:joyal tierney lifting} and the fact that the first and third coincide follows from \cite[Proposition 7.27]{joyaltierney2007qcatvssegal}. The equivalence of the third and fourth condition is just the definition of a left Quillen functor.
 \end{remone}
 
 \subsection{Review of Enriched Categories}\label{subsec:enriched}
 In this short section we review some basic aspects of enriched category theory. The main source on enriched category theory is \cite{kelly2005enriched}, but \cite[Chapter 3]{riehl2014categoricalhomotopytheory} is also very readable introduction to this subject. Let $\V$ be a category with finite products. Then $(\V,\times,1_\V)$ is a symmetric monoidal category, where the symmetric monoidal structure comes from the product and the unit is simply the final object. 
 
 \begin{defone}
 	A {$\V$-enriched category} $\C$ is a collection of objects $\Obj_\C$ and for two objects $X,Y \in \Obj_\C$ an object $\C_\V(X,Y)$ in $\V$ along with composition maps 
 	$$\C_\V(X,Y) \times \C_\V(Y,Z) \to \C_\V(X,Z)$$
 	and identity elements $1 \to \C_\V(X,X)$, such that composition is unital and associative.
\end{defone}
 
 There are alternative ways to characterize the data of an enriched category.
  
 \begin{defone} \label{def:tensored}
 	Let $\C$ be a $\V$-enriched category. We say $\C$ is {\it tensored over $\V$} if for all objects $X$ in $\C$, the functor $\C_\V(X,-):\C \to \V$ has a left adjoint $- \otimes_\V X: \V\to \C$. 
 \end{defone}
 
 In general an enriched category is not necessarily tensored, however, there are important exceptions.
 
 \begin{propone} \label{prop:enrich vs tensor}
 	Let $\C$ be $\V$-enriched category so that the underlying category is locally presentable. Then $\C$ is tensored over $\V$ and the tensor $-\otimes_\V -: \C \times \V \to \C$ uniquely determines the enrichment.
 \end{propone}

 We will make use of the functor category of enriched categories and so will review it here as well.
 
\begin{defone}
  Let $\C$ and $\D$ be two $\V$-enriched categories. A {\it $\V$-enriched functor} of $\V$-enriched categories $F:\C \to \D$ is a map of objects $F_{\Obj}: \C_{\Obj} \to \D_{\Obj}$ and for any two objects $X,Y$ in $\C$ a morphism in $\V$
  $$F(X,Y): \C_\V(X,Y) \to \D_\V(FX,FY)$$
  such that $F$ respects composition and the identity. 
\end{defone}

\begin{notone} \label{not:catv}
  We will denote the category with objects $\V$-enriched categories and morphisms $\V$-enriched functors by $\cat_\V$.
 	
\end{notone}

 Having functors we can define natural transformations.
 
 \begin{defone}
 	Let $F,G: \C\to \D$ be two $\V$-enriched functors of $\V$-enriched categories. A {\it $\V$-enriched natural transformation} of $\V$-enriched functors is a collection of morphisms $\alpha_X: 1 \to \D(FX,GX)$ such that for all objects $X,Y$ the following diagram commutes
 	\begin{center}
 		\begin{tikzcd}
 			\C_\V(X,Y) \arrow[r] \arrow[d]  & \D_\V(FX,FY) \arrow[d] \\
 			\D_\V(GX,GY) \arrow[r] & \D_\V(FX,GY)
 		\end{tikzcd}.
 	\end{center}
  We denote the {\it set} of $\V$-enriched natural transformations by $\Nat(F,G)$.
 \end{defone}

 \begin{defone} \label{def:funcd}
 	Let $\C,\D$ be two $\V$-enriched categories. Define the functor category $\Fun(\C,\D)$ as the category (not enriched) with objects $\V$-enriched functors $F: \C \to \D$ and morphisms $\V$-enriched natural transformations $\Nat(F,G)$.
 \end{defone}
  
  We can use tensors to add an enrichment to $\Fun(\C,\D)$.
  
  \begin{defone} \label{def:ufuncd}
   Let $\C,\D$ be two $\V$-enriched categories, such that $\D$ is tensored and cotensored over $\V$. Then the category $\Fun(\C,\D)$ tensored and cotensored over $\V$ defined point-wise as $(F \otimes_\V X)(C) = F(C) \otimes_\V X$. In particular, if $\C$ is small and the underlying category $\D$ is locally presentable, then $\Fun(\C,\D)$ is $\V$-enriched. We denote the $\V$-enriched category by $\uFun(\C,\D)$ and the $\V$-enriched natural transformations $\uNat(F,G)$.
  \end{defone}

  For a proof of the point-wise tensor and cotensor see \cite[Lemma 5.2]{moser2019enrichedproj}. One key aspect of enrichments are change of enrichment functors. Let $(\W,\times,1_\W)$ be another category with finite products and let $\Fun^{\times}(\V,\W)$ denote the category of product preserving functors from $\V$ to $\W$. Then we have the following result.
  
  \begin{lemone} \label{lemma:ch}
  	There is a {\it change-of-enrichment functor} 
  	$$\Ch: \Fun^{\times}(\V,\W) \to \Fun(\cat_\V,\cat_\W)$$
  	defined as follows. For a given functor $F: \V \to \W$ define $\Ch(F):\cat_\V \to \cat_\W$ that takes a $\V$-category $\C$ to the $\W$-category $\Ch(F)(\C)$ which has the same objects and $\Ch(F)(\C)_\W(X,Y) = F(\C_\V(X,Y))$.
  \end{lemone}
  For a proof see \cite[Proposition 6.3, Theorem 6.4]{eilenbergkelly1966closedcat}. Let us give the key example of an enriched category.
 
 \begin{defone} \label{def:ncat}
 	Let $\set$ be the category of small sets and set functions. Then a $\set$-enriched category is precisely a category, which we will also call strict $1$-category. By induction a {\it strict $n$-category} is defined as a $(n-1)$-enriched category. We denote the category of strict $n$-categories and enriched functors by $\cat_n$. For two strict $n$-categories we denote the $(n-1)$-category of functors by $\Fun(\C,\D)$.
 \end{defone}
 
 We use Cartesian closure to construct many interesting enriched categories, as explained in \cite[Definition 3.3.6]{riehl2014categoricalhomotopytheory}.
 
 \begin{lemone}\label{lemma:cartesian closed enriched}
 	Let $\V$ be a Cartesian closed category. Then $\V$ is enriched over itself with $\C(X,Y) = Y^X$. 
 \end{lemone}

 We shall use this lemma to give further examples of enriched categories in the next section.
 
 \begin{remone}
	The enrichment we consider here throughout is a {\it strict enrichment}. For example, $\cat$-enriched categories are strict $2$-categories and not bicategories (where associativity and unitality holds only up to natural isomorphism). There are a few occasions (\cref{cor:equiv qcat}, \cref{cor:equiv qcat s}, \cref{cor:qcat equiv ncat}) where we consider weakly enriched quasi-categories in sense of \cite{gepnerhaugseng2015enrichedinftycat}. 
 \end{remone}
 
\subsection{Categories Enriched over Categories of Presheaves} \label{subsec:enriched psh}
In this section want to review categories enriched over categories of set-valued presheaves. First we review the desired diagram categories.
Let $\DDelta$ be the category of finite linearly ordered sets with linearly ordered functions. We denote the objects by $[0],[1],[2], \cdots$. 
\begin{notone}\label{not:brackets}
	Every morphism $[n] \to [m]$ is uniquely characterized by an increasing sequence of $0 \leq k_1 \leq ... \leq k_n \leq m$. We denote this morphism by $\ordered{k_1,...,k_n}:[n] \to [m]$. In particular every morphism $[0] \to [n]$ is of the form $\ordered{i}$, where $0 \leq i \leq n$.
\end{notone} 
For more details regarding the category $\DDelta$ see \cite{goerssjardine1999simplicialhomotopytheory} or the initial section of \cite{rezk2001css} or \cite{rasekh2017left}.

Let $\cD$ be a small category. Let $\P(\cD) = \Fun(\cD^{op},\set)$, which we call the category of $\cD$-sets. Moreover, let $\P(\cD \times \DDelta) = \Fun(\cD^{op} \times \DDelta^{op},\set)$, which we call the category of $\cD$-simplicial sets.

We denote the {\it representable $\cD$-simplicial sets} by $F[d,n]$, where $d$ is an object in $\cD$ and an $[n]$ is an object in $\DDelta$. We can use the representable functors to deduce that the category of $\cD$-simplicial sets is {\it Cartesian closed}, where for two $\cD$-simplicial sets  $X,Y$ the Cartesian closure $Y^X$ is uniquely characterized by the natural bijection 
$$\Hom_{\P(\cD\times\DD)}(F[d,n],Y^X) \cong \Hom_{\P(\cD\times\DD)}(F[d,n] \times X,Y)$$
Using the Cartesian closure and \cref{lemma:cartesian closed enriched} we get the following result.

\begin{lemone} \label{lemma:snset enriched}
	The category of $\cD$-simplicial sets is enriched over itself. For two $\cD$-simplicial sets $X,Y$ we denote the enriched morphisms by $\uHom_{\P(\cD\times\DD)}(X,Y)$ and note it is defined as 
	$$\uHom_{\P(\cD\times\DD)}(X,Y)[d,n] = \Hom_{\P(\cD\times\DD)}(F[d,n] \times X,Y)$$ 
\end{lemone}

Following \cref{not:catv} we use  $\cat_{\P(\cD)}$ ($\cat_{\P(\cD \times \DDelta)}$) to denote the category of $\P(\cD)$-enriched ($\P(\cD\times\DDelta)$-enriched) categories. Let us see the key example of a $\P(\cD)$-enriched category.

\begin{exone} \label{ex:psh}
	Let $\C$ be a small $\P(\cD)$-enriched category. Then, by \cref{def:ufuncd}, the functor category $\Fun(\C,\P(\cD))$ is $\P(\cD)$-enriched with tensor given by $(F \times X)(c) = F(c) \times X$, where $F: \C \to \P(\cD)$ is a functor and $X$ a $\cD$-simplicial set. In particular, for two functors $F,G: \C \to \P(\cD)$, the $\cD$-simplicial set $\uNat(F,G)$ is defined as 
	$$\uNat(F,G)[d] = \Nat(F \times \{F[d]\},G),$$
	where $\{F[d]\}$ denotes the constant functor with value $F[d]$. We will denote $\sP(\cD)$-enriched functor category by $\uFun(\C,\P(\cD))$ following \cref{def:ufuncd}.
\end{exone} 

Given two enriched categories $\C,\E$ we want an easy criterion that can help us determine when a functor $G:\C \to \E$ is in fact a $\P(\cD)$-enriched functor.
\begin{lemone} \label{lemma:pcd enriched functor}
	Let $\C,\E$ be two $\P(\cD)$-enriched categories such that they are tensored over $\P(\cD)$ (\cref{def:tensored}). Moreover, let $G: \C\to \E$ be a functor of underlying categories such that $G(X\otimes F[d]) = G(X) \otimes F[d]$. Then $G$ is a $\P(\cD)$-enriched functor. 
\end{lemone} 

\begin{proof}
	Notice, for two objects $X,Y$, $\uHom_\C(X,Y)[d]= \Hom_\C(X \otimes F[d],Y)$ and so we can define $G: \uHom(X,Y) \to \uHom(GX,GY)$ level-wise as
	$$G[d]:\uHom_\C(X,Y)[d] = \Hom_\C(X \otimes F[d],Y) \to \Hom_\C(GX \otimes F[d],GY)= \uHom_\E(GX,GY)[d],$$
	where we are using the fact that there is a functor $\Hom(- \otimes -,-):\cD^{op} \times \C^{op} \times \C \to \set$.
\end{proof}

Some enriched functors need to be defined using alternative methods.

\begin{lemone} \label{lemma:prod enriched}
	The product functor  $- \times - : \P(\cD) \times \P(\cD) \to \P(\cD)$ is $\P(\cD)$-enriched.
\end{lemone}

\begin{proof}
 Fix four objects $X,X',Y,Y'$ and let $\ev_{X,X'}: X \times \uHom(X,X') \to X'$ and $\ev_{Y,Y'}:Y \times \uHom(Y,Y') \to Y'$ be the evaluation functors. We define 
 $$\uHom_{\P(\cD)}(X,X') \times \uHom_{\P(\cD)}(Y,Y') \to \uHom_{\P(\cD)}(X \times Y, X' \times Y')$$
 as the adjoint of the product map $\ev_{X,X'} \times \ev_{Y,Y'}$.
\end{proof}

We have a similar result for adjunctions.

\begin{lemone} \label{lemma:enriched adjunction via tensors}
	Let $\C,\E$ be two $\P(\cD)$-enriched categories such that they are tensored over $\P(\cD)$ (\cref{def:tensored}). Moreover, let  $L: \C \to \E$ and $R:\E \to \C$ two $\P(\cD)$-enriched functors such that $L\dashv R$ is an adjunction of underlying categories. If $L(X\otimes F[d]) = LX \otimes F[d]$, then $L \dashv R$ is an enriched adjunction.
\end{lemone}

\begin{proof}
 For every object $d$ in $\cD$ we have a natural isomorphism 
 $$\uHom_\E(LX,Y)[d] = \Hom_\E(LX \otimes F[d],Y) = \Hom_\E(L(X\otimes F[d]),Y) \cong \Hom_\C(X \otimes F[d],RY) = \uHom_\C(X,RY)[d].$$
\end{proof}

Notice, there is an enriched Yoneda lemma  for $\P(\cD)$-enriched categories \cite[2.30]{kelly2005enriched}.

\begin{lemone}\label{lemma:enriched yoneda}
	Let $\C$ be a small $\P(\cD)$-enriched category. Then for any object $c$ in $\C$ and enriched functor $F: \C \to \P(\cD)$ there an isomorphism of $\P(\cD)$-sets $F(c) \cong \uNat(\uHom(c,-),F)$. In particular for any two objects $c,c'$ there is an isomorphism of $\P(\cD)$-sets 
	$\uHom(c',c) \cong \uNat_{\Fun(\C,\P(\cD))}(\uHom(c,-),\uHom(c',-))$.
\end{lemone}

The enriched Yoneda lemma implies that we have an {\it enriched Yoneda embedding} as stated in \cite[2.34]{kelly2005enriched}.

\begin{lemone}\label{eq:enriched yon}
	Let $\C$ be a small $\P(\cD)$-enriched category. There is a fully faithful $\sP(\cD)$-enriched functor 
	$$\Yon: \C \to \uFun(\C^{op},\P(\cD)),$$
	which takes an object $c$ to the representable functor $\uHom_\C(-,c)$.
\end{lemone}

\begin{remone} \label{rem:enriched mapping}
The enriched Yoneda lemma in particular implies that for a given small $\P(\cD)$-enriched category $\C$ we have an {\it enriched mapping object} functor 
$$ \uHom(-,-): \C^{op} \times \C \to \P(\cD)$$
\end{remone} 

This is \cite[Theorem 3.73]{kelly2005enriched}.

\begin{theone} \label{the:enriched colimit}
	A $\P(\cD)$-enriched category $\C$ is $\P(\cD)$-cocomplete if the underlying category has all small colimits and it is tensored over $\P(\cD)$. Moreover, a $\P(\cD)$-enriched functor $F: \C \to \D$ between $\P(\cD)$-cocomplete categories preserves $\P(\cD)$-colimits if the underlying functor preserves colimits and preserves the tensor.
\end{theone}

We can use the enriched Yoneda lemma to study enriched Kan extensions \cite[Theorem 4.51]{kelly2005enriched}.

\begin{theone}\label{the:enriched Kan ext}
 Let $\C$ be a small $\P(\cD)$-enriched category and $\D$ a $\P(\cD)$-enriched category with $\P(\cD)$-colimits. Then there is an isomorphism between $\P(\cD)$-enriched functors $\C \to \D$ and colimit preserving $\P(\cD)$-enriched functors $\uFun(\C^{op},\P(\cD)) \to \D$, where the equivalence is given via left Kan extension and restriction along the Yoneda embedding. 
\end{theone}

We also want a method to construct enriched functors in $\uFun(\C^{op},\P(\cD))$. For that we have the following result \cite[2.20]{kelly2005enriched}.

\begin{theone}\label{the:enriched prod hom adj}
	Let $\C$ be a small $\P(\cD)$-enriched category and $\D$ a $\P(\cD)$-enriched category. There is an isomorphism of (set-enriched) functor categories 
	$$\Fun(\C^{op} \times \D,\sP(\cD)) \cong \Fun(\D,\uFun(\C^{op},\sP(\cD)))$$
\end{theone}
We can generalize the {\it nerve} to a {\it $\cD$-nerve}.

\begin{defone}\label{def:nervesnset}
	Let 
	$$N_{\cD}: \cat_{\P(\cD)} \to \P(\cD \times \DDelta) $$
	be the functor which takes an $\P(\cD)$-enriched category $\C$ to the $\D$-simplicial set
	$$N_{\cD}\C[k] = \coprod_{X_0,..., X_k}\uHom(X_0,X_1) \times ... \times \uHom(X_{k-1},X_k)$$
	where the boundary maps are degeneracy maps are given by composition and inserting identity maps.
\end{defone} 

The nerve construction respects opposites. Let
\begin{equation} \label{eq:op}
	(-)^{op}: \P(\cD \times \DDelta) \to \P(\cD \times \DDelta)
\end{equation}
be the functor that precomposes a functor $F:\cD^{op} \times \DD \to \set$ with $\id_{\cD^{op}} \times \sigma^{op}$, where $\sigma: \DD \to \DD$ is unique automorphism. We now have the following simple result.

\begin{lemone}\label{lemma:op}
	The following diagram commutes
	\begin{center}
		\begin{tikzcd}
			\cat_{\P(\cD)} \arrow[r, "N_\cD"] \arrow[d, "(-)^{op}"] & \P(\cD\times\DD) \arrow[d, "(-)^{op}"] \\
			\cat_{\P(\cD)} \arrow[r, "N_\cD"] & \P(\cD\times\DD)
		\end{tikzcd}
	\end{center}
	where the left hand $(-)^{op}$ takes a category to the opposite category.
\end{lemone}

Notice the nerve behaves the way we would expect. First we recall the Segal condition. 

\begin{defone}
	A simplicial set $X: \DDelta^{op} \to \set$ satisfies the {\it strict Segal condition} if for all $n \geq 2$ the evident map 
	$$X_n \to X_1 \times_{X_0} ... \times_{X_0} X_1$$ 
	is an isomorphism of sets. 
\end{defone}

\begin{propone}\label{prop:nervesnset ff}
	The $\cD$-nerve is fully faithful with essential image given by $F: \cD^{op} \times \DDelta^{op} \to \set$ such that for all $d$ in $\cD$, the simplicial set $F[d,-]: \DDelta^{op} \to \sset$ satisfies the strict Segal condition and the functor $F[-,0]: \cD^{op} \to \set$ is a constant functor. 
\end{propone}

 We need several functors regarding $\cD$-simplicial sets that we will now introduce. 

\begin{defone} \label{def:und}
	Let $d$ be an element in $\cD$. Then this gives us a functor $\DDelta \to \cD \times \DDelta$ that gives us an adjunction on presheaf categories 
	\begin{center}
		\adjun{\sset}{\P(\cD)}{\Inc_d}{\Und_d}
	\end{center} 
	We call the right adjoint $\Und_d$ the {\it underlying simplicial set} and we have in particular, $\Inc_d(F[n]) = F[d,n]$.
\end{defone}

\begin{defone} \label{def:disc}
	We have a projection functor $\cD\times \DDelta \to \DDelta$  which gives us a functor $\Disc:\sset \to \P(\cD \times \DDelta)$, which we call the {\it discrete inclusion}.
\end{defone}

 \begin{remone} \label{rem:D s F}
 	We use the notation $D[n] = \Disc(F[n])$ and notice that $D[n] \times F[d,0] \cong F[d,n]$. In particular if $\cD$ has a terminal object $t$, then $D[n] \cong F[t,n]$.
\end{remone}

 \begin{defone} \label{def:val}
	Fix $k$ in $\DDelta$, then we have an inclusion map $\incl_k:\cD \to \cD \times \DDelta$ which induces an adjunction 
	\begin{center}
		\adjun{\P(\cD)}{\P(\cD\times\DDelta)}{(\incl_k)_!}{(\incl_k)^*}
	\end{center}
	We denote the right adjoint $(\incl_k)^*$ also by $\Val_k: \P(\cD\times\DDelta) \to \P(\cD)$ and call it the {\it $k$-th value functor}. If $k=0$ then we will simply denote it by $\Val:\P(\cD\times\DDelta) \to \P(\cD)$ and call it the {\it value}. Moreover, we denote the left adjoint $(\incl_0)_!$ by $\VEmb$ and call it the {\it value embedding}. 
\end{defone}

Up until now we have focused on $\cD$-simplicial sets, however, as our goal is to do homotopy theory we need analogous definitions for $\cD$-diagrams valued in spaces.

\begin{defone}\label{def:s}
	Let $\s$ be the category of simplicial sets which we call {\it spaces}. The representable spaces are denoted $\Delta[l]$.
\end{defone}

\begin{remone}
	Notice there is an apparent overlap, namely if we take $\P(\cD \times \DDelta)$ with $\cD$ the trivial category then it coincides with \cref{def:s}. However, the key is to consider $\s$ as a homotopical category with a model structure, which is why we call the objects in $\s$ spaces. In particular, it should be replaceable with any appropriately Quillen equivalent model structure, without changing the results, whereas $\sset= \P(\DD)$ is treated as a category without any model structure. 
\end{remone}

We can now repeat all the definitions we have given up until now replacing sets with spaces.

\begin{remone} \label{rem:ss analogue}
	We have following facts regarding simplicial spaces:
	\begin{enumerate}
		\item A $\cD$-space is a functor $X:\cD^{op} \to \s$ and the category of $\cD$-spaces is denoted by $\sP(\cD)$.
		\item A $\cD$-simplicial space is a functor $X:\cD^{op} \times \DDelta^{op} \to \s$ and the category of $\cD$-simplicial spaces is denoted by $\sP(\cD\times\DDelta)$.
		\item We denote the representable $\cD$-simplicial spaces by $F[d,k] \times \Delta[l]$.
		\item We can use representable simplicial spaces to observe that $\sP(\cD)$ is Cartesian closed with 
		$$\Hom_{\sP(\cD\times\DD)}(F[d,k] \times \Delta[l],Y^X)\cong \Hom_{\sP(\cD\times\DD)}(F[d,k] \times \Delta[l] \times X,Y)$$
		which in particular means $\sP(\cD)$ is enriched over itself (\cref{lemma:cartesian closed enriched}). We denote the enrichment by $\uMap_{\sP(\cD\times\DD)}(X,Y)$ and category of $\sP(\cD)$-enriched categories by $\cat_{\sP(\cD)}$. 
		\item Following \cref{ex:psh} for a a given small $\sP(\cD)$-enriched category $\C$, we can define the $\sP(\cD)$-enriched functor category $\uFun(\C,\sP(\cD))$, which is point-wise tensored and cotensored over $\sP(\cD)$, hence for two $\sP(\D)$-enriched functors $F,G$ we have 
		$$\uNat(F,G)[d,k]= \Nat(F \times \{F[d,k]\},G),$$
		where $\{F[d,k]\}:\C \to \sP(\cD)$ is the constant functor.
		\item We can define a functor 
		$$N_{\cD}: \cat_{\sP(\cD)} \to \sP(\cD \times\DDelta)$$
		similar to \cref{def:nervesnset} and notice similar to \cref{lemma:op} it commutes with the opposite category construction.
		\item Following \cref{def:und} for every object $d$ in $\cD$ we define an adjunction 
		\begin{center}
			\adjun{\sset}{\P(\cD)}{\Inc_d}{\Und_d}
		\end{center} 
		where the right adjoint $\Und_d$ is called  the {\it underlying simplicial space}.
		\item Similar to \cref{def:disc} we define the functor $\Disc: \ss \to \sP(\cD\times\DDelta)$. 
		\item Analogous to \cref{def:val} we define an adjunction 
		\begin{center}
			\adjun{\sP(\cD)}{\sP(\cD\times\DDelta)}{(\incl_k)_!}{(\incl_k)^*}
		\end{center}
		and use the notation $\Val_k: \sP(\cD\times\DDelta)\to \sP(\cD)$ and in particular $\Val= \Val_0:\sP(\cD\times\DDelta)\to \sP(\cD)$ for the right adjoint, which we call the {\it value}. Moreover, we denote the left adjoint by $\VEmb: \sP(\cD) \to \sP(\cD\times\DDelta)$, which we call the value embedding.
	\end{enumerate}
\end{remone}

In addition to the previous constructions we also have diagonal functors.

\begin{defone} \label{def:diag}
	Let $\Diag: \cD\times \DDelta \to \cD \times \DDelta\times \DDelta$ given by $\Diag(d,k) = (d,k,k)$. Then we get an adjunction
	\begin{center}
		\adjun{\sP(\cD\times\DDelta)}{\sP(\cD)}{\Diag^*}{\Diag_*}
	\end{center}
	To simplify notation we use $\fDiag = \Diag^*: \sP(\cD\times\DDelta) \to \sP(\cD)$ for the left adjoint.
\end{defone}

 We can also use $\cD$-simplicial spaces to construct model structures.

\begin{propone} \label{prop:injectivemod}
	There is a unique simplicial proper combinatorial Cartesian model structure on the category of $\cD$-spaces $\sP(\cD)$ ($\cD$-simplicial spaces $\sP(\cD\times\DDelta)$), called the {\it injective model structure} and denoted $\sP(\cD)^{inj}$ ($\sP(\cD\times\DDelta)^{inj}$), with the following specifications:
	\begin{itemize}
		\item[W] weak equivalences level-wise Kan equivalence of spaces
		\item[C] Cofibrations are level-wise injections.
		\item[F] Fibrations are morphisms that have the right lifting property with respect to trivial cofibrations. 
	\end{itemize}
\end{propone}

For a proof see \cite{heller1982injectivemodel,heller1983injectivemodelerratum}. We can use the injective model structure on $\cD$-spaces to define {\it Dwyer-Kan equivalences of $\sP(\cD)$-enriched categories}.

\begin{defone} \label{def:loc dk}
	Let $\cD$ be a small category with terminal object $t$. 
	Let $G: \C \to \D$ be an $\sP(\cD)$-enriched functor of $\sP(\cD)$-enriched categories. We say $G$ is a {\it Dwyer-Kan equivalence} if the following two conditions hold.
	\begin{itemize}
		\item For any two objects $c,c'$ in $\C$, the induced map of $\cD$-spaces 
		$$\uMap(c,c') \to \uMap(Fc,Fc')$$
		is an injective equivalence. 
		\item The homotopy category of the underlying functor of simplicially enriched categories $\Ho G[t]: \Ho\C[t] \to \Ho\D[t]$ is essentially surjective.
	\end{itemize}
\end{defone}

\subsection{The Level-Wise Model Structure on $\cD$-Simplicial Spaces} \label{subsec:level model}
In this subsection we introduce a method that allows us to generalize localization model structures from the Reedy model structure on simplicial spaces to a model structure on $\cD$-simplicial spaces. For that we need the following alternative perspective on $\cD$-simplicial spaces.

\begin{defone}
	Define the {\it level-wise Reedy} model structure on $\sP(\cD\times\DD)= \Fun(\cD^{op},\ss^{Ree})^{inj}$ as the injective model structure on the Reedy model structure on simplicial spaces. More concretely, we can describe the cofibrations and weak equivalences as follows:
	\begin{itemize}
		\item A map of $\cD$-simplicial spaces $f: X \to Y$ is a cofibration if for all objects $d$ in $\cD$, the map of simplicial spaces $\Und_df:\Und_dX \to \Und_dY$ is a cofibration in the Reedy model structure on simplicial spaces.
		\item A map of $\cD$-simplicial spaces $f: X \to Y$ is a weak equivalence if for all objects $d$ in $\cD$, the map of simplicial spaces $\Und_df:\Und_dX \to \Und_dY$ is a weak equivalence in the Reedy model structure on simplicial spaces.
	\end{itemize}
\end{defone}

 The characterization has the following immediate result.
 
\begin{lemone}
	The level-wise Reedy model structure coincides with the injective model structure on $\cD$-simplicial spaces. 
\end{lemone}

\begin{proof}
	Both model structures have the same weak equivalences (level-wise Kan equivalences) and cofibrations (level wise inclusions).
\end{proof}

We can use the level-wise Reedy model structure to easily construct interesting new model structures on $\sP(\cD\times\DDelta)$. 

\begin{propone} \label{prop:levelwise localized model structure}
	Let $S$ be a set of cofibrations of simplicial spaces. Then there exist a unique simplicial left proper combinatorial model structure on $\cD$-simplicial spaces, which we call the {\it level-wise $S$-localized model structure} and denote  $\sP(\cD\times\DDelta)^{lev_S}$, and has the following specifications.
	\begin{enumerate}
		\item Cofibrations are inclusions.
		\item An object $X$ is fibrant if it is injectively fibrant and for every $d$ in $\cD$ the simplicial space $\Und_dX$ is $S$-local. 
		\item A morphism $f: X \to Y$ is a weak equivalence if and only if for every $d$ in $\cD$ the map of simplicial spaces $\Und_df: \Und_dX \to \Und_dY$ is an $S$-local equivalence.
		\item A weak equivalence (fibration) between fibrant objects is an injective equivalence (injective fibration).
		\item If the $S$-localized Reedy model structure on simplicial spaces is Cartesian then this model structure is also Cartesian.
	\end{enumerate}
\end{propone}

\begin{proof}
	Let $\ss^{Ree_S}$ be the $S$-localized Reedy model structure on the category of simplicial spaces. We define the level-wise $S$-localized model structure as the injective model structure on $\Fun(\cD^{op},\ss^{Ree_S})$ and notice it satisfies the conditions given in the statement by definition of the injective model structure. 
\end{proof}

We now want to use this result to construct several important model structure on the category of $\cD$-simplicial spaces based on model structures on simplicial spaces. First of all we need two model structures that relate $\cD$-simplicial spaces and simplicial spaces, similar to the result in \cite[Subsection 2.5]{rasekh2017left}.

\begin{theone}\label{the:diagmodel}
	There is a unique, combinatorial, simplicial model structure on $\sP(\cD\times\DD)$, 
	called the {\it $\cD$-diagonal model structure} and denoted by $\sP(\cD\times\DD)^{\cD-diag}$, with the following specifications.
	\begin{itemize}
		\item[(W)] A map $f:X \to Y$ is a weak equivalence if the diagonal map 
		$$\fDiag(f): \fDiag(X) \to \fDiag(Y)$$ 
		is an injective equivalence of $\cD$-spaces.
		\item[(C)] A map $f:X \to Y$ is a cofibration if it is an inclusion.
		\item[(F)] A map $f:X \to Y$ is a fibration if it satisfies the right lifting condition for trivial cofibrations.
	\end{itemize}
	In particular, an object $W$ is fibrant if and only if it is injectively fibrant and the simplicial space $\Und_d(W)$ is {\it homotopically constant} for all $d$ in $\cD$, which is equivalent to the counit map $\VEmb\Val(W) \to W$ be an injective equivalence of $\cD$-simplicial spaces.
\end{theone}

\begin{proof}
   Let $\ss^{diag}$ be the diagonal model structure on simplicial spaces, which, by \cite[Theorem 2.11]{rasekh2017left}, is obtained by localizing with respect to the cofibrations $F[0] \to F[n]$ for all $n \geq 0$. Now applying \cref{prop:levelwise localized model structure} gives us the $\cD$-diagonal model structure. 
\end{proof}

\begin{theone} \label{the:valmodel}
	There is a unique, cofibrantly generated, simplicial model structure on $\sP(\cD\times\DD)$,
	called the {\it value model structure} and denoted by $\sP(\cD\times\DD)^{val}$, with the following specification.
	\begin{itemize}
		\item[(W)] A map $f: X \to Y$ is a weak equivalence if 
		$$\Val f : \Val X \to \Val Y$$ 
		is an injective equivalence of $\cD$-spaces.
		\item[(C)] A map $f: X \to Y$ is a cofibration if it is an inclusion.
		\item[(F)] A map $f: X \to Y$ is a fibration if it satisfies the right lifting condition for trivial cofibrations.
	\end{itemize}
\end{theone}
 
 \begin{proof}
  Let $\ss^{Kan}$ be the Kan model structure on simplicial spaces, which, by \cite[Theorem 2.12]{rasekh2017left}, is obtained by localizing with respect to the cofibrations $\partial F[n] \to F[n]$ for all $n > 0$. Now applying \cref{prop:levelwise localized model structure} gives us the value model structure. 
 \end{proof}

We can now relate these model structures to the injective model structure on $\cD$-spaces. 

\begin{theone} \label{the:diag and val equiv inj}
	The $\sP(\cD)$-enriched adjunctions 
	\begin{center}
		\begin{tikzcd}
			 \sP(\cD\times\DD)^{\cD-diag} \arrow[r, shift left=1.8, "\fDiag=\Diag^*", "\bot"'] & 
			 \sP(\cD)^{inj} \arrow[r, shift left=1.8, "\VEmb", "\bot"'] \arrow[l, shift left=1.8, "\Diag_*"] & 
			 \sP(\cD\times\DD)^{val} \arrow[l, shift left=1.8, "\Val"]
		\end{tikzcd} 	
	\end{center}
	are Quillen equivalences. 
\end{theone}

\begin{proof}
	By \cite[Theorem 2.13]{rasekh2017left} we have simplicial Quillen equivalences
	\begin{center}
		\begin{tikzcd}[row sep=0.5in, column sep=0.5in]
			\ss^{diag} \arrow[r, shift left =1, "\Diag^*"] & \s^{Kan} \arrow[l, shift left =1, "\Diag_*"] \arrow[r, "(\pi_1)^*", shift left =1]  
			& \ss^{Kan} \arrow[l, "(\pi_1)_*", shift left=1] 
		\end{tikzcd}
	\end{center}
	The result now follows from applying $\Fun(\cD^{op},-)$ to this adjunction and giving all sides the injective model structure combined with \cite[Remark A.2.8.6]{lurie2009htt}.
\end{proof}

We now use \cref{prop:levelwise localized model structure} to construct model structures on $\cD$-simplicial spaces based on model structures on simplicial spaces motivated by $(\infty,1)$-category theory.

 \begin{defone}
  A simplicial space $X$ is called a {\it Segal space} if it is Reedy fibrant and for all $n \geq 2$ the map 
  $$X_n \to X_1 \times_{X_0} ... \times_{X_0} X_1$$
  is a Kan equivalence of spaces. 	
 \end{defone}
  
  For more details on Segal spaces see the original source \cite[Section 5]{rezk2001css}.
  We can generalize this definition to $\cD$-simplicial spaces.
  
  \begin{defone} \label{def:levelsegal}
  	An $\cD$-simplicial space $X$ is a {\it level-wise Segal space} if it is injectively fibrant and for all $d$ in $\cD$ the simplicial space $\Und_dX$ is a Segal space.
  \end{defone}

 Segal spaces are the fibrant objects in a model structure on simplicial spaces (\cite[Theorem 7.1]{rezk2001css}), that is obtained by localizing the Reedy model structure on simplicial spaces with respect to the maps $G[n] \to F[n]$, for $n \geq 2$ \cite[Section 5]{rezk2001css}. We can use this observation to show level-wise Segal spaces come with a model structure. 
 
  \begin{propone}\label{prop:levelsegal}
 There exist a unique simplicial left proper combinatorial Cartesian model structure on $\sP(\cD\times\DD)$, which we call the {\it level-wise Segal model structure} and denote $\sP(\cD\times\DD)^{lev_{Seg}}$, which has the following specifications.
 	\begin{enumerate}
 		\item Cofibrations are inclusions.
 		\item An object $X$ is fibrant if it is injectively fibrant and for every $d$ in $\cD$ the simplicial space $\Und_dX$ is a Segal space. 
 		\item A morphism $f: X \to Y$ is a weak equivalence if and only if for every $d$ in $\cD$ the map of simplicial spaces $\Und_df: \Und_dX \to \Und_dY$ is a Segal equivalence. 
 		\item A weak equivalence (fibration) between fibrant objects is an injective equivalence (injective fibration).
 	\end{enumerate}
 \end{propone}

\begin{proof}
	This is a direct application of \cref{prop:levelwise localized model structure} with respect to the maps $G[n] \to F[n]$, where $n \geq 2$. Notice the Segal space model structure is in fact Cartesian \cite[Theorem 7.1]{rezk2001css}, which implies the level-wise Segal space model structure is Cartesian as well.
\end{proof}

 Segal spaces do not give us a model of $(\infty,1)$-categories, for that we need another condition, called the {\it completeness condition}. 
 There two ways to add a condition to a Segal space so that the resulting object is in fact an $(\infty,1)$-category and we will need both of them. The first one is the completeness condition. 
 
 \begin{defone}
 	A Segal space $X$ is called a {\it complete Segal space} if the square
 	\begin{center}
 	 \begin{tikzcd}
 	 	X_0 \arrow[r] \arrow[d, "\Delta"] & X_3 \arrow[d] \\
 	 	X_0 \times X_0 \arrow[r] & X_1 \times X_1
 	 \end{tikzcd} 
 	\end{center}
 	is a homotopy pullback square of spaces. 
 \end{defone}

 For more details on complete Segal spaces see \cite[Section 6]{rezk2001css}. Again complete Segal spaces can be generalized to $\cD$-simplicial spaces.
 
 \begin{defone}\label{def:levelcss}
 	An $\cD$-simplicial space $X$ is a {\it level-wise complete Segal space} if it is injectively fibrant and for all $d$ in $\cD$ the simplicial space $\Und_dX$ is a complete Segal space.
 \end{defone}

 Similar to Segal spaces, complete Segal spaces are the fibrant objects in a model structure on simplicial spaces (\cite[Theorem 7.2]{rezk2001css}), which is obtained by localizing with respect to the cofibrations $G[n] \to F[n]$, where $n \geq 2$, and $F[0] \to E[1]$ \cite[Section 5]{rezk2001css}. 
 Using this we can easily construct a model structure for level-wise complete Segal spaces. 
 
  \begin{propone}\label{prop:levelcss}
	There exist a unique simplicial left proper combinatorial Cartesian model structure on $\P(\cD\times\DD)$, which we call the {\it level-wise complete Segal model structure} and denote $\sP(\cD\times\DD)^{lev_{CSS}}$, which has the following specifications.
	\begin{enumerate}
		\item Cofibrations are inclusions.
		\item An object $X$ is fibrant if it is injectively fibrant and for every $d$ in $\cD$ the simplicial space $\Und_dX$ is a complete Segal space. 
		\item A morphism $f: X \to Y$ is a weak equivalence if and only if for every $d$ in $\cD$ the map of simplicial spaces $\Und_df: \Und_dX \to \Und_dY$ is a complete Segal equivalence. 
		\item A weak equivalence (fibration) between fibrant objects is an injective equivalence (injective fibration). 
	\end{enumerate}
\end{propone}

\begin{proof}
		This is a direct application of \cref{prop:levelwise localized model structure} with respect to the maps $G[n] \to F[n]$, where $n \geq 2$, and $F[0] \to E[1]$. Notice the complete Segal space model structure is in fact Cartesian \cite[Theorem 7.2]{rezk2001css}, which implies the level-wise complete Segal space model structure is Cartesian as well.
\end{proof}
 
 In \cite[Theorem 7.7]{rezk2001css} Rezk proves that a map of Segal spaces is a complete Segal space equivalence if and only if it is fully faithful and essentially surjective, also known as a {\it Dwyer-Kan equivalence} of Segal spaces (\cite[7.4]{rezk2001css}). Using \cref{prop:levelcss} we can immediately generalize this to a level-wise version.
 
 \begin{lemone} \label{lemma:dk levelwise Segal}
 	Let $f: X \to Y$ be a map of level-wise Segal spaces. Then $f: X \to Y$ is a level-wise complete Segal space equivalence if and only if for all objects $d$ in $\cD$ the map of simplicial spaces $\Und_df: \Und_dX \to \Und_dY$ is a Dwyer-Kan equivalence.
 \end{lemone}

 The second way is by adding a discreteness condition to Segal spaces. 
 
 \begin{defone}
 	A {\it Segal category} $W$ is a Segal space such that the space $W_0$ is discrete. 
 \end{defone}

Segal categories also come with a model structure \cite[Theorem 5.1]{bergner2007threemodels}, that we can generalize here to $\cD$-simplicial spaces

 \begin{propone}\label{prop:levelsegalcat}
	There exist a unique simplicial left proper combinatorial model structure on $\P(\cD\times\DD)$, which we call the {\it level-wise Segal category model structure} and denote $\sP(\cD\times\DD)^{lev_{SegCat}}$, which has the following specifications.
	\begin{enumerate}
		\item Cofibrations are inclusions.
		\item An object $X$ is fibrant if it is injectively fibrant and for every $d$ in $\cD$ the simplicial space $\Und_dX$ is a Segal category. 
		\item A morphism $f: X \to Y$ is a weak equivalence if and only if for every $d$ in $\cD$ the map of simplicial spaces $\Und_df: \Und_dX \to \Und_dY$ is a complete Segal equivalence. 
		\item A weak equivalence (fibration) between fibrant objects is an injective equivalence (injective fibration). 
	\end{enumerate}
\end{propone}

\begin{proof}
	This is a direct application of \cref{prop:levelwise localized model structure} to the generating trivial cofibrations in the model structure for Segal categories \cite[Theorem 5.1]{bergner2007threemodels}.
\end{proof}

There is a Quillen equivalence between the model structure for Segal categories and complete Segal space, which we can generalize to an equivalence of level-wise model structure.

\begin{propone} \label{prop:levcss vs legsegcat}
	Let
	\begin{center}
		\adjun{\sP(\cD\times\DD)^{lev_{SegCat}}}{\sP(\cD\times\DD)^{lev_{CSS}}}{I_\cD}{R_\cD}
	\end{center}
	be the adjunction $(I,R)$ given in \cite[Theorem 6.3]{bergner2007threemodels} level-wise. Then this adjunction is a Quillen equivalence.  
\end{propone} 

\begin{proof}
	This again follows from \cite[Remark A.2.8.6]{lurie2009htt} applied to the Quillen equivalence  $(I,R)$ (\cite[Theorem 6.3]{bergner2007threemodels}).
\end{proof}

Notice, $I_\cD$ is the identity on Segal categories and we will often omit it to simplify notation. The result has the following important implication.

\begin{corone}\label{cor:zigzag}
	Let $X$ be an arbitrary $\cD$-simplicial space. Then there exists a natural zigzag 
	$$X \to \hat{X} \leftarrow R_\cD\hat{X}$$
	where $\hat{X}$ is the fibrant replacement in the level-wise complete Segal space model structure and $R_\cD\hat{X}$ is the corresponding level-wise Segal category and both maps are level-wise complete Segal space equivalences. 
\end{corone}
 
\subsection{Enriched (complete) Segal Spaces and Categories} \label{subsec:enriched css}
 Level-wise (complete) Segal spaces are an important first step towards studying $\sP(\cD)$-enriched category theory via simplicial objects, however, we need to add another condition that gives us a well-defined notion of objects.
 
 \begin{defone}
 	A map of simplicial spaces $X \to Y$ is {\it essentially surjective} if the map $\hat{X} \to \hat{Y}$ is essentially surjective (meaning $\hat{X}_0 \to \hat{Y}_0$ is a surjection on path-components). Here $\hat{X}$ and $\hat{Y}$ are fibrant replacements in the complete Segal space model structure.
 \end{defone}
 
 Notice this definition does not depend on the choice of fibrant replacement as any two fibrant replacements are Reedy equivalent.
  
 \begin{defone} \label{def:weakly constant objects}
 	A $\cD$-simplicial space $X$ has {\it weakly constant objects} if for for all morphisms $f: d \to d'$ in $\cD$ the induced map of simplicial spaces $\Und_{d'}X \to \Und_dX$ is essentially surjective. 
 \end{defone}
 
 Notice this condition can be simplified.
 
  \begin{lemone}
 	Let $\cD$ have a terminal object $t$. Then a $\cD$-simplicial space $X$ has weakly constant objects if and only if the map $\Und_tX \to \Und_dX$ is essentially surjective for all objects $d$. Here the morphism is given by the unique morphism $d \to t$. 
 \end{lemone}
  
  \begin{defone} \label{def:enriched segal}
  	An {\it $\sP(\cD)$-enriched Segal space} $X$ is a level-wise Segal space with weakly constant objects. If $\cD$ has a terminal object $t$, we denote $X[t,0] = \Obj_X$ and call it the {\it objects of $X$}.
  \end{defone}
  
  Similar to Segal spaces this notion has appropriate categorical properties.
  
  \begin{defone} \label{def:map}
  	Let $\cD$ have a terminal object $t$.
  	Let $x,y$ be two objects in the $\sP(\cD)$-enriched Segal space $X$. Define the {\it mapping $\cD$-space} $\map_X(x,y)$ by the pull back of $\cD$-spaces 
  	\begin{center}
  		\begin{tikzcd}
  			\map_X(x,y) \arrow[r] \arrow[d] & \Val_1(X) \arrow[d, "(s \comma t)"] \arrow[d] \\ 
  			D\leb 0 \reb \arrow[r] & \Val(X) \times \Val(X) 
  		\end{tikzcd}
  	\end{center}
  \end{defone}
  
  We can use that to define a homotopy category, similar to Segal spaces. 
  
  \begin{defone}
  	Let $\cD$ have a terminal object $t$ and let $X$ be an $\sP(\cD)$-enriched Segal space. Define the $\P(\cD)$-enriched category $\Ho X$, called the {\it homotopy category} as follows. Objects are given by $(\Obj_X)_0$ and for two objects $x,y$ define $\Hom_{\Ho X}(x,y) = \pi_0(\map_X(x,y))$.
  \end{defone}

  Ideally we would like to say that if $X$ is an $\sP(\cD)$-enriched Segal space then $\Val(X)$ is homotopically constant. We want to show that this statement holds, up to level-wise Dwyer-Kan equivalence. In order to make this into a precise argument we need $\sP(\cD)$-enriched Segal categories.
  
  \begin{defone} \label{def:enriched segal cat}
 	An {\it $\sP(\cD)$-enriched Segal category} $X$ is a level-wise Segal category with weakly constant objects. If $\cD$ has a terminal object $t$, we denote $X[t,0] = \Obj_X$ and call it the {\it objects of $X$}.
 \end{defone} 

We now want to prove that every $\sP(\cD)$-enriched Segal space is level-wise Dwyer-Kan equivalent to an $\sP(\cD)$-enriched Segal category with constant objects. For that we need codiscrete simplicial objects.

Let $\C$ be a category with finite products. Then the restriction functor $(-)_0: \Fun(\DD^{op},\C) \to \C$ has a right adjoint $\coDisc: \C \to \Fun(\DD^{op},\C)$, which takes an object $X$ in $\C$ to the simplicial object $\coDisc(X)$ with $\coDisc(X)_n = X^{n+1}$ and with face and degeneracy maps given by projections and diagonals. We have the following simple result regarding $\coDisc$.

\begin{lemone} \label{lemma:technical stuff}
	The simplicially enriched adjunction 
	\begin{center}
		\adjun{\ss^{Seg}}{\s^{Kan}}{(-)_0}{\coDisc}
	\end{center}
	is a Quillen adjunction such that the right adjoint reflects weak equivalences. 
\end{lemone}

\begin{proof}
	The left adjoint preserves cofibrations and the right adjoint takes a Kan complex $K$ to a Segal space $\coDisc(K)$. Hence the adjunction is a Quillen adjunction by \cite[Corollary A.3.7.2]{lurie2009htt}. In order to prove the right adjoint reflects weak equivalences it suffices to prove that for every space $K$ the counit map is an equivalence. However, the counit is just the identity map and hence we are done.
\end{proof}

We now have the following important construction. Let $\Inc_t\Und_t X \to X$ be the counit of the adjunction (\cref{def:und}). We can factor this map into a cofibration followed by a trivial fibration 
\begin{equation} \label{eq:fib rep two}
 \Inc_t\Und_t X \to \bar{X} \to X	
\end{equation}
 in the injective model structure on $\cD$-simplicial spaces. Notice, by the triangle identity for adjunction $\Und_t\Inc_t\Und_tX \to \Und_tX$ is the identity map and so $\Und_t\bar{X}= \Und_t X$. 

\begin{defone}\label{def:bH}
 Let $\cD$ be a small category with a terminal object $t$.
 For a given level-wise Segal category $X$ define $\bH X$ as the pullback 
 \begin{equation} \label{eq:pb}
 	\begin{tikzcd}
 		\bH X \arrow[r] \arrow[d] & \coDisc X[t,0] \arrow[d, hookrightarrow] \\
 		\bar{X} \arrow[r] \arrow[d] & \coDisc \Val(\bar{X}) \\
 		X
 	\end{tikzcd},
 \end{equation}  
where $\bar{X}$ is defined via the factorization \ref{eq:fib rep two}.	
\end{defone}

\begin{remone}
	The definition of $\bH$ is a generalization of the construction of the functor $R$ by Bergner in \cite[Section 6]{bergner2007threemodels}, which is used to construct Segal categories out of complete Segal spaces. 
\end{remone}

By \cref{lemma:technical stuff} all three parts of the pullback diagram are level-wise Segal spaces and so the $\cD$-simplicial space $\bH X$ 
is also a level-wise Segal category with $\Val(RX) = X[t,0]$. We now have the following result.

\begin{lemone}\label{lemma:more technical lemma}
	Let $\cD$ have a terminal object $t$. Let $X$ be an $\sP(\cD)$-enriched Segal category. The map $\bH X \to X$ defined in \ref{eq:pb} is a level-wise complete Segal space equivalence if and only if $X$ has weakly constant objects.
\end{lemone}

\begin{proof}
	 We need to show $RX \to X$ is a level-wise complete Segal space equivalence. For an object $d$ in $\cD$, $\Und_dRX \to \Und_dX$ is by definition fully faithful and the inclusion of object $X[t,0] = \Und_dRX[0]  \to \Und_dX[0]$ is essentially surjective if and only if $X$ has weakly constant objects and so the result follows. 
\end{proof}

We can combine this result with \cref{cor:zigzag} to get the following corollary.

\begin{corone} \label{cor:zigzag better}
	Let $X$ be an arbitrary $\cD$-simplicial space. There exists a natural zigzag 
	$$X \to \hat{X} \leftarrow \tilde{X}$$
	which satisfies the following conditions:
	\begin{enumerate}
		\item $\hat{X}$ is a level-wise complete Segal space.
		\item $\tilde{X}$ is a level-wise Segal category and $\Val(\tilde{X})$ is discrete.
		\item The first arrow is a level-wise complete Segal space equivalence.
		\item The second map is a level-wise complete Segal space equivalence and if and only if $X$ has weakly constant objects. 
	\end{enumerate}
\end{corone}

\begin{proof}
	Let $\hat{X}$ be the fibrant replacement of $X$ in the level-wise complete Segal space model structure \cref{prop:levelcss}. Define $\tilde{X} = \bH R_\cD \hat{X}$, which is a level-wise Segal category with discrete objects by \cref{def:bH}. The first map is a level-wise complete Segal space equivalence by definition and the second map is a level-wise complete Segal space equivalence if and only if $X$ has weakly constant objects by \cref{lemma:more technical lemma}. 
\end{proof}

As a first application we can use this result to better understand equivalences of $\sP(\cD)$-enriched Segal spaces and adjust \cref{lemma:dk levelwise Segal} to $\sP(\cD)$-enriched Segal spaces. 

\begin{defone}
	Let $\cD$ have a terminal object $t$. A map of $\sP(\cD)$-enriched Segal spaces $f: X \to Y$ is an {\it $\sP(\cD)$-enriched Dwyer-Kan equivalence} if the following two conditions hold:
	\begin{itemize}
		\item {\bf Fully faithful:} For every two objects $x,y$, the map of $\cD$-spaces 
		$$\map_X(x,y) \to \map_Y(fx,fy)$$
		is an injective equivalence of $\cD$-spaces. 
		\item {\bf Essentially surjective:} The map of Segal spaces $\Und_tf \to \Und_tX \to \Und_tY$ is essentially surjective. 
	\end{itemize}
\end{defone}

This new notion relates appropriately with other notions of equivalences.

\begin{lemone} \label{lemma:dk equiv segal enriched}
	Let $\cD$ have a terminal object $t$ and let $f: X \to Y$ be a map of $\sP(\cD)$-enriched Segal spaces. Then the following are equivalent.
	\begin{enumerate}
		\item $f$ is a level-wise complete Segal space equivalence.
		\item $f$ is a level-wise Dwyer-Kan equivalence.
		\item $f$ is an $\sP(\cD)$-enriched Dwyer-Kan equivalence.
	\end{enumerate}
\end{lemone}

\begin{proof}
	The first two are equivalent by \cref{lemma:dk levelwise Segal}. Now, following \cref{cor:zigzag better}, $f$ is a level-wise Dwyer-Kan equivalence if and only if $\tilde{f}$ is a level-wise Dwyer-Kan equivalence. However, $\Val(\tilde{X})$ is discrete and so this coincides with being a $\sP(\cD)$-enriched Dwyer-Kan equivalence. 
\end{proof}

We can in particular apply this to nerves.

\begin{exone}\label{ex:enriched equiv is levelwise css equiv}
	Let $F:\C \to \D$ be an $\sP(\cD)$-enriched functor of $\sP(\cD)$-enriched categories. Then, by \cref{lemma:dk equiv segal enriched}, $F$ is a Dwyer-Kan equivalence of $\sP(\cD)$-enriched categories (\cref{def:loc dk}) if and only if $N_\cD F: N_\cD\C\to N_\cD\D$ is a level-wise (or, equivalently, $\sP(\cD)$-enriched) Dwyer-Kan equivalence.
\end{exone}

The notion of $\sP(\cD)$-enriched Segal space is deliberately chosen to be as general as possible to include many examples of $\cD$-simplicial spaces. The drawback is that it cannot be a fibrant object in a localization model structure on $\sP(\cD\times\DD)$. However, this will be relevant in later parts (\cref{sec:infn fib}).

\begin{defone}
	Let $\cD$ have a terminal object $t$. An {\it $\sP(\cD)$-enriched complete Segal space} $X$ is an $\sP(\cD)$-enriched Segal space, such that $\Und_tX$ is complete and $\Val(X)$ is homotopically constant.
\end{defone}

\begin{theone}\label{the:spd enriched css}
	Let $\cD$ be a small category with terminal object $t$.
	There exists a simplicial combinatorial left proper model structure on $\cD$-simplicial spaces, called the {\it $\sP(\cD)$-enriched complete Segal space model structure} and denoted by $\sP(\cD\times\DD)^{\cD-CSS}$, which has the following properties:
	\begin{itemize}
		\item The cofibrations are monomorphisms.
		\item An object $W$ is fibrant if it is an $\sP(\cD)$-enriched complete Segal space.
		\item A morphism $f:A \to B$ is a weak equivalence if and only if for every $\sP(\cD)$-enriched complete Segal space the map 
		$$f^*:\Map(B,W) \to \Map(A,W)$$
		is a Kan equivalence.
		\item Fibrations (weak equivalences) between fibrant objects are injective fibrations (weak equivalences).
		\item Let $W$ be a level-wise Segal space with homotopically constant value, then the fibrant replacement $W \to \hat{W}$ is a Dwyer-Kan equivalence (\cref{def:loc dk}).
	\end{itemize}
\end{theone}

\begin{proof}
	The model structure is a left Bousfield localization with respect to the following sets of morphisms:
	\begin{itemize}
		\item $\Disc(G[n]) \times F[d] \to F[d,n]$
		\item $D[0] \to \Disc(E[1])$
		\item $F[t,0] \to F[d,0]$
	\end{itemize}
	By definition the fibrant objects are the $\sP(\cD)$-enriched complete Segal spaces. The other properties follow formally from Bousfield localizations. 
	
	The last part follows from the same steps as the proof in \cite[Lemma 8.19]{bergnerrezk2020comparisonii}. Notice the proof there is given for a specific diagram category, rather than a general $\cD$, however, it only uses the fact that the diagram has a terminal object.
\end{proof}

\begin{remone}
	Notice, every $\sP(\cD)$-enriched Segal space is level-wise complete Segal space equivalent to an $\sP(\cD)$-enriched complete Segal space. Indeed, we can simply complete the underlying $\sP(\cD)$-enriched Segal category $\tilde{X}$ we obtained in \cref{cor:zigzag better} using the completion construction explained in the last item of \cref{the:spd enriched css}. 
\end{remone}

\begin{remone} \label{rem:no segal cats}
	Unlike $\sP(\cD)$-enriched Segal spaces, we do not construct any model structure for $\sP(\cD)$-enriched Segal categories. Indeed, following \cite[Section 8]{bergnerrezk2020comparisonii}, constructing such a model structure for a general diagram category $\cD$ might not be possible or at least quite challenging.
\end{remone}

\subsection{The Classical Grothendieck Construction} \label{subsec:grothendieck fib}
One important step in our work is to generalize the Grothendieck construction to $\sP(\cD)$-enriched categories. Hence, we will review it here quickly. The ideas are in no way new and a more detailed approach can be found in many places, such as \cite[I.5]{maclanemoerdijk1994topos}, or \cite[A1.1.7, B1.3.1]{johnstone2002elephanti}. Our review here is motivated by \cite[Section 1]{rasekh2017left}

\begin{defone} \label{def:int}
	Let $\C$ be a category. Define the {\it Grothendieck construction}
	$$\int_\C : \Fun(\C,\set) \to \cat_{/\C}$$ 
	as the functor that takes $F: \C \to \set$ to the category $\int_\C F \to \C$ with 
	\begin{itemize}
		\item Objects: Pairs $(c,x)$ where $c$ is an object in $\C$ and $x \in F(c)$.
		\item Morphisms: A morphisms $(c,x) \to (d,y)$ is a choice of morphism $f: c \to d$ in $\C$ such that $F(f)(x) = y$.
	\end{itemize}
	It comes with an evident projection map $\pi_F: \int_\C F \to \C$. 
\end{defone}

This functor has a left adjoint and a right adjoint that we want to define in detail. The key observation is the following example.

\begin{exone} \label{ex:int hom under}
	Let $\Hom_\C(c,-): \C \to \set$ be a corepresentable functor. Then $\int_\C \Hom(c,-) = \C_{c/}$.
\end{exone}

\begin{defone}
	Let $\C$ be a category. We define the functor  
	$$\C_{/-}: \C \to \cat_{/\C}$$
	that takes an object to the over-category $\C_{/c}$ and a morphism $f:c \to d$ to the post-composition $f_!: \C_{/c} \to \C_{/d}$.
	
	Similarly, define the functor 
	$$\C_{-/}: \C^{op} \to \cat_{\C/}$$
	that takes an object to the under-category $\C_{c/}$ and a morphism $f: c \to d$ to the precomposition $f^*: \C_{d/} \to \C_{c/}$.
\end{defone}

For a given category over $\C$, $p: \D \to \C$ define 
$$\T_\C(p: \D \to \C): \C \to \set$$
as the composition 
$$\C \xrightarrow{ \ \C_{/-} \ } \cat_{/\C} \xrightarrow{ \ - \times_\C \D \ } \cat \xrightarrow{ \ \pi_0 \ } \set$$
in other words we have $\T_\C(c) = \pi_0(\C_{/c} \times_\C \D)$.
Similarly, define 
$$\H_\C(p: \D \to \C): \C \to \set$$ 
as the composition
$$\C \xrightarrow{ \ (\C_{-/})^{op} \ } (\cat_{/\C})^{op} \xrightarrow{ \ \Hom_{/\C}(-,\D) \ } \set$$
meaning we have $\H_\C(c) = \Hom_{/\C}(\C_{c/}, \D)$.
We now have the following result.

\begin{propone} \label{prop:integral adjunctions}
	We have the following diagram of adjunctions
	\begin{center}
		\begin{tikzcd}[row sep=0.5in, column sep=0.9in]
			\Fun(\C,\set) \arrow[r, "\int_\C" description] & \cat_{/\C} \arrow[l, bend left = 30, "\H_\C", "\bot"'] \arrow[l, bend right=30, "\T_\C"', "\bot"]  
		\end{tikzcd},
	\end{center}
where $\int_\C$ is fully faithful.
\end{propone}

For a proof see \cite[Proposition 1.10, Lemma 1.11]{rasekh2017left}.

Notice $\int_\C$ is in fact not essentially surjective, as the following result in \cite[Lemma 1.14]{rasekh2017left} illustrates. Recall that a functor $p: \D \to \C$ is {\it conservative} if it reflects isomorphisms. 

\begin{lemone} \label{lemma:Grothendieck fibration conservative}
	Let $F: \C \to \set$ be a functor. Then $\pi_F:\int_\C F \to \C$ is conservative.
\end{lemone}

Restricting our attention to the essential image of $\int_\C$ leads us to {\it discrete Grothendieck opfibrations}.

\begin{defone} \label{def:discrete Groth opfib}
	A functor $P: \D \to \C$ is a {\it discrete Grothendieck opfibration} over $\C$ if it is in the essential image of $\int_\C$, 
	meaning there exists a functor 
	$F: \C \to \set$ and isomorphism
	$\int_\C F \xrightarrow{ \ \cong \ } \D$ 
	over $\C$. 
\end{defone}

Fortunately, there is also an internal characterization of discrete Grothendieck opfibrations.

\begin{lemone} \label{lemma:discrete Groth fib lifting} 
	A functor $P: \D \to \C$  is a discrete Grothendieck opfibration over $\C$ if and only if for 
	any map $f: C \to C'$ in $\C$ and object $D$ in $\mathscr{D}$ such that $P(D) = C$,
	there exists a {\it unique} lift $\hat{f}:D \to D'$ such that $P(\hat{f}) = f$.
\end{lemone}

The fact that $\int_\C$ is fully faithful (\cref{prop:integral adjunctions}) implies there is an equivalence of categories between the functor category $\Fun(\C,\set)$ and Grothendieck opfibrations over $\C$. The main goal of \cref{sec:grothendieck} is to generalize this equivalence to a homotopical version in the $\sP(\D)$-enriched setting. 

The equivalence between Grothendieck opfibrations and functors implies the existence of a {\it universal Grothendieck opfibration}, which is simply the under category projection $\pi_*:\set_{*/} \to \set$. Concretely, we have the following observation:

\begin{propone} \label{prop:universal fib}
	The functor
	$$(-)^*\pi_*:\Fun(\C,\set) \to \cat_{/\C}$$
	given by pulling back $\pi_*$ is precisely the functor $\int_\C$ and so pulling back $\pi_*$ gives us an equivalence between functors and discrete Grothendieck opfibrations over $\C$.
\end{propone} 

Given that we can understand $\int_\cD$ via pullback, we would like a version of \cref{prop:integral adjunctions} that holds for all categories at once. Notice that discrete Grothendieck opfibrations are stable under pullback as they are characterized via a lifting condition (\cref{lemma:discrete Groth fib lifting}). Hence, we obtain a {\it pseudo-functor} 
\begin{equation} \label{eq:pseudofunctor}
 \opGroth_{/-}: \cat \to \widehat{\cat},	
\end{equation}
that takes a category $\C$ to the {\it large} category of discrete Grothendieck opfibrations over $\C$. 

\begin{defone} \label{def:opgroth}
	Let $\opGroth$ be the category with objects discrete Grothendieck opfibrations and morphisms pullbacks and notice it comes with a projection functor $\opGroth \to \cat$ which takes each fibration to its target. 
\end{defone}

Concretely, the category and its projection can obtained by applying the pseudo-categorical Grothendieck construction \cite[Theorem B1.3.6]{johnstone2002elephanti} to the pseudo-functor \ref{eq:pseudofunctor}. We now have the following generalization of  \cref{prop:integral adjunctions}.

\begin{propone} \label{prop:grothendieck fiber}
	The pullback functor 
	$$(-)^*\pi_*: \cat_{/\set} \to \opGroth$$
	is an equivalence over $\cat$. 
\end{propone}

\begin{proof} 
	By \cite[Theorem B1.3.6]{johnstone2002elephanti} it suffices to prove it is a fiber-wise equivalence, however, this follows immediately from  \cref{prop:integral adjunctions} and \cref{prop:universal fib}.
\end{proof}

\begin{remone} \label{rem:unique pb}
	Notice, \cref{prop:grothendieck fiber} in particular implies that for every discrete Grothendieck opfibration $\D \to \C$, there exists a pullback square of the form 
	\begin{center}
		\pbsq{\D}{\set_*}{\C}{\set}{}{}{\pi_*}{},
	\end{center}
	which is unique up to equivalence. 
\end{remone}

Finally there is also a contravariant version of this definition, which similarly relates to functors $\Fun(\C^{op},\set)$.

 \begin{defone}
	$P: \D \to \C$ is called a {\it discrete Grothendieck fibration} for 
	any map $f: C \to C'$ in $\C$ and object $D'$ in $\mathscr{D}$ such that $P(D') = C'$
	there exists a {\it unique} lift $\hat{f}:D \to D'$ such that $P(\hat{f}) = f$.
\end{defone}

 \subsection{A Reminder on the Covariant Model Structure} \label{subsec:lfib}
 Our main goal is to study fibrations of $\cD$-simplicial spaces. The key is to generalize the construction from the case of fibrations of simplicial spaces, {\it left fibrations}. Left fibrations of simplicial spaces have been studied extensively. Here we will review some key results that we will use throughout and for more details we refer the reader to \cite{rasekh2017left}.
 
 A Reedy fibration of simplicial spaces $p: L \to X$ ($q: R \to X$) is called a {\it left fibration} ({\it right fibration}) if the following is a homotopy pullback square (using \cref{not:brackets})
 \begin{equation} \label{eq:leftright}
 	\pbsq{L_n}{X_n}{L_0}{X_0}{\ordered{0}^*}{p_n}{p_0}{\ordered{0}^*}, \
 	\pbsq{R_n}{X_n}{R_0}{X_0}{\ordered{n}^*}{q_n}{q_0}{\ordered{n}^*}.
 \end{equation}
 There is unique left proper combinatorial simplicial model structure on the over category $\ss_{/X}$, called the {\it covariant model structure} ({\it contravariant model structure}) \cite[Theorem 3.12]{rasekh2017left}. The study of this model structure and its important properties is the main subject of \cite{rasekh2017left}.  Here we will only state the relevant properties of the covariant model structure:
 
 \begin{enumerate}
 	\item \label{leftitem:recognition} {\bf Recognition Principle for Covariant Equivalences} \cite[Theorem 4.39]{rasekh2017left}: A map $f$ is a covariant equivalence if and only if for every map $x:F[0] \to X$, if the diagonal of the induced map 
 	$$Y \underset{X}{\times} R_x \to Z \underset{X}{\times} R_x$$
 	is a Kan equivalence. Here $R_x$ is the right fibrant replacement of the map $x$ over $X$.
 	\item \label{leftitem:CSS invariant} {\bf Invariance of Left Fibrations} \cite[Theorem 5.1]{rasekh2017left}: Let $f: X \to Y$ be a map of simplicial spaces. Then the adjunction
 	\begin{center}
 		\adjun{(\ss_{/X})^{Cov}}{(\ss_{/Y})^{Cov}}{f_!}{f^*}
 	\end{center}
 	is a Quillen adjunction, which is a Quillen equivalence whenever $f$ is a CSS equivalence.
 	\item \label{leftitem:CSS fib} {\bf Left Fibrations are CSS Fibrations} \cite[Theorem 5.11]{rasekh2017left}: The following is a Quillen adjunction
 	\begin{center}
 		\adjun{(\ss_{/X})^{CSS}}{(\ss_{/X})^{Cov}}{id}{id}
 	\end{center}
 	where the left hand side has the induced CSS model structure and the right hand side has the covariant model structure.
 	\item \label{leftitem:smooth maps} {\bf Exponentiability} \cite[Corollary 5.18]{rasekh2017left}:
 	Let $f:X \to Y$ be a CSS equivalence and $p:L \to Y$ a left fibration over $Y$. Then the map $f^*L \to L$ is also a CSS equivalence.
 \end{enumerate}
 
 There are several results regarding left fibrations that we need, but were not covered in \cite{rasekh2017left}, and hence we will be discussed in greater detail.
 
 \begin{defone} \label{def:left map}
 	Let $p:L \to X$ be a map of simplicial spaces. We say $p$ is a {\it left morphism} if for all $n$ the square 
 	\begin{center} 
     \pbsq{L_n}{X_n}{L_0}{X_0}{\ordered{0}^*}{p_n}{p_0}{\ordered{0}^*}
 	\end{center} 
 	is a homotopy pullback square. We similarly define {\it right morphisms}.
 \end{defone}

 We have the following immediate lemma.
 
 \begin{lemone} \label{lemma:left map}
 	Let $p:L \to X$ be a left map. Then the Reedy fibrant replacement $\hat{L} \to X$ is a left fibration.
 \end{lemone}
 
 There is one key example of a left map.
 
 \begin{exone}\label{ex:left map}
 	Let $\C$ be a simplicially enriched category and $c$ an object. Define $N\C_{c/}$ as the simplicial space with 
 	$$(N\C_{c/})_n = \coprod_{c_0,..., c_n} \Map(c,c_0) \times ... \times \Map(c_{n-1},c_n)$$
 	and with simplicial operators given via identity and composition. Then the evident projection map $N\C_{c/} \to N\C$ satisfies the condition of a  left map, in fact the map $(N\C_{c/})_n \to (N\C_{c/})_0 \times_{N\C_0} N\C_n$ is a bijection of spaces. However, it usually is not a Reedy fibration and so is an example of a left morphism that is not a left fibration.
 \end{exone}

 Although the map $N\C_{/c} \to N\C$ is not a right fibration, but only a right morphism we want to prove that in a certain way it does behave like right fibrations.

 \begin{lemone} \label{lemma:very technical lemma}
 	Let $\C$ be a simplicially enriched category and $c$ an object. Then the pullback functor 
 	$$p^*:(\ss_{/N\C})^{cov} \to (\ss_{/N\C})^{cov}$$
 	is a left Quillen functor between covariant model structures. Here $p$ is the map of simplicial spaces $p:N\C_{/c} \to N\C$.
 \end{lemone}
 
 \begin{proof}
 	Evidently pullback preserves monomorphisms and so $p^*$ preserves cofibrations, which are just the level-wise monomorphisms. 
 	Hence, we only need to show that $p^*$ preserves trivial cofibrations. Let $A \to B$ be a trivial Reedy cofibration over $N\C$. We want to show $p^*A \to p^*B$ is a trivial Reedy cofibration over $N\C$. It suffices to prove that $p^*A[k] \to p^*B[k]$ is a Kan equivalence, as we already established it is a cofibration and Reedy equivalences are determined level-wise.
 	
 	Notice $N\C[0]= \Obj_\C$, the set of objects. Hence, we can write $A[0] = \coprod_{c_0 \in \Obj_\C} A[0]_{c_0}$, where $A[0]_{c_0}$ denotes the subspace of $A[0]$ that maps to $c_0$. More generally, we have a projection map of spaces $A[k] \to (\Obj_\C)^{k+1}$. So we can write $A[k] = \coprod_{c_0,...,c_k} A[k]_{c_0,...,c_k}$, where $A[k]_{c_0,...,c_k}$ is the subspace of $A[k]$ that maps to $(c_0,...,c_k)$ in $(\Obj_\C)^{k+1}$. We can use a similar argument for $B$ to conclude that $B[k] = \coprod_{c_0,...,c_k} B[k]_{c_0,...,c_k}$ and the fact that $A[k] \to B[k]$ is an equivalence of spaces over $N\C[0]$ implies that $A[k]_{c_0,..., c_k} \to B[k]_{c_0,...,c_k}$ is an equivalence of spaces for all tuples  $(c_0,...,c_k)$ in $(\Obj_\C)^{k+1}$.
 	
 	Now, by direct computation we have 
 	$$p^*A[k] = \coprod_{c_0,...,c_k}(\Map(c_0,c_1) \times ... \times \Map(c_{k-1},c_k) \times \Map(c_k,c)) \times_{\ds\coprod_{c_0,...,c_k}\Map(c_0,c_1) \times ... \times \Map(c_{k-1},c_k)} A[k] = $$ 
 	$$\coprod_{c_0,...,c_k} \left( (\Map(c_0,c_1) \times ... \times \Map(c_{k-1},c_k) \times \Map(c_k,c)) \times_{\Map(c_0,c_1) \times ... \times  \Map(c_{k-1},c_k)} A[k]_{c_0,...,c_k} \right) \cong $$ $$\coprod_{c_0,...,c_k} \Map(c_k,c) \times A[k]_{c_0,...,c_k} $$
 	Similarly we have 
 	$$p^*B[k] \cong \coprod_{c_0,...,c_k} \Map(c_k,c) \times B[k]_{c_0,...,c_k}$$
 	Now the map 
 	$$\coprod_{c_0,...,c_k} \Map(c_k,c) \times A[k]_{c_0,...,c_k} \to \coprod_{c_0,...,c_k} \Map(c_k,c) \times B[k]_{c_0,...,c_k}$$
 	is in fact a weak equivalence of spaces, which implies that $p^*$ preserves trivial Reedy cofibrations, proving that $p^*$ is left Quillen with respect to the Reedy model structure.
 	
 	We want to prove that $p^*$ is in fact left Quillen with respect to the covariant model structure. For that we need to prove that for every map $F[0] \to F[n] \to N\C$, the map $p^*F[0] \to p^*F[n]$ is a covariant equivalence over $N\C$. By induction it suffices to consider the case $\{0\}: F[0] \to F[1]$. So fix a map $F[1] \to N\C$, which corresponds to a morphisms $c_0 \to c_1$. By direct computation $p^*F[0] = \Map_\C(c_0,c)$ as a constant simplicial space. Let  $\widehat{p^*F[1]}$ be the Reedy fibrant replacement of $p^*F[1]$ and notice $\widehat{p^*F[1]} \to F[1]$ is a right fibration as $N\C_{/c} \to N\C$ is a right map (\cref{ex:left map}) and by the contravariant analogue to \cref{lemma:left map}. We now have the following diagram of simplicial spaces
 	\begin{center}
 		\begin{tikzcd}
 			p^*F[0]	\arrow[r] \arrow[d, "\simeq"]  & p^*F[1] \arrow[d, "\simeq"] \\
 			0^*\widehat{p^*F[1]} \arrow[r, "\simeq"] & \widehat{p^*F[1]}
 		\end{tikzcd}.
 	\end{center}  
 	The vertical morphisms are Reedy equivalences by construction and the bottom map is a covariant equivalence over $F[1]$ by \cite[Theorem 4.28]{rasekh2017left}. So, by $2$-out-of-$3$, the top map is a covariant equivalence over $F[1]$ and so by \cite[Theorem 3.15]{rasekh2017left} also a covariant equivalence over $N\C$ finishing the proof.
 \end{proof}
 
 \begin{remone}
 	Notice we could not use \cite[Theorem 4.28]{rasekh2017left} here directly as the map $N\C_{/c} \to N\C$ is not a Reedy fibration (unless $\C$ is a simplicially enriched category with discrete mapping spaces) and so in particular not a right fibration (\cref{ex:left map}). 
 \end{remone}

 We need a further result relating over-categories and mapping spaces.
 
 \begin{lemone} \label{lemma:mapping vs overcat} 
 	Let $\C$ be a simplicially enriched category and $c,d$ two objects. Then $N\C_{c/} \times_{N\C} N\C_{/d}$ is a homotopically constant simplicial space and equivalent to the constant simplicial spaces $\Map(c,d)$.
 \end{lemone} 

 \begin{proof}
 	By \cite[Thoerem 3.49]{rasekh2017left} the map $F[0] \to N\C_{c/}$ is a covariant equivalence over $N\C$ and so, by \cref{lemma:very technical lemma}, the induced map 
 	$F[0] \times_{N\C} N\C_{/d} \to N\C_{c/} \times_{N\C} N\C_{/d}$ is also a covariant equivalence and so in particular also a diagonal equivalence \cite[Theorem 3.17]{rasekh2017left}. By direct computation for every map ${d}:F[0] \to N\C$, we have the equality of constant simplicial spaces
 	\begin{equation} \label{eq:s}
 		 s:F[0] \times_{N\C} N\C_{c/} = \Map_\C(c,d),
 	\end{equation} 
    proving we have a diagonal equivalence $\Map(c,d) \to N\C_{c/} \times_{N\C} N\C_{/d}$, meaning the simplicial space is homotopically constant as well. 
 \end{proof} 

\begin{remone} \label{rem:ps}
 Notice the space $(N\C_{c/} \times_{N\C} N\C_{/d})[k]$ has points composable morphisms $c \to c_0 \to ... \to c_k \to d$, which can be composed to a morphism $c \to d$. As the composition respects the simplicial identities (given by the identity and composition) this combines into a morphism of spaces $\fDiag(N\C_{c/} \times_{N\C} N\C_{/d}) \to \Map(c,d)$, which, by \cref{def:diag}, corresponds to a map of simplicial spaces 
 $$p:N\C_{c/} \times_{N\C} N\C_{/d} \to \Map(c,d),$$ 
 such that $ps: \Map(c,d) \to \Map(c,d)$ is the identity map, where $s$ is defined in \ref{eq:s}, meaning $p$ is an equivalence as well. 
\end{remone}

We will end this section with a review of the {\it twisted arrow construction}. In \cite[Section 2]{barwickglasmannardin2018dualizingfibrations} the authors describe how to naturally construct a left fibration of quasi-categories $\Tw(S) \to S^{op} \times S$, known as the {\it twisted arrow construction}. It has the property that the fiber over a point $(x,y)$ in $S^{op} \times S$ is the mapping space $\map_S(x,y)$ of the quasi-category. Explicitly it is given by 

$$\Tw(S)_n = \Hom_{\sset}(\Delta[l]^{op} \ast \Delta[l],S) = \Hom_{\sset}(\Delta[2l+1],S) \cong S_{2l+1}.$$

In the later sections we need the corresponding construction for Segal spaces. In order to motivate things let us give the construction for the analogous complete Segal space. For a given complete Segal space $W$ let $\Tw(W)$ be the simplicial space given as $\Tw(W)[n]_{l} = \Tw(W[\bullet]_{l})_n \cong W[2n+1]_{l}$, meaning we are applying the twisted arrow construction to complete Segal space level-wise. This construction agrees with the quasi-categorical one. Indeed, let $i_1^*: \ss \to \sset$ be the right Quillen functor of the Quillen equivalence $(p_1^*,i_1^*)$ between quasi-categories and complete Segal spaces defined in \cite{joyaltierney2007qcatvssegal}. Then $\Tw(i_1^*(W)) = i_1^*(\Tw(W))$.

\begin{remone}\label{rem:twisted} 
 Notice, by \cite[Theorem B.12]{rasekh2017left}, $\Tw(W) \to W^{op} \times W$ is a left fibration as the map $i_1^*(\Tw(W)) \to i_1^*W^{op} \times i_1^*W$ is the twisted arrow construction of the quasi-category $i_1^*W$. Moreover, the fiber over a point $(x,y)$ is the mapping space $\map_W(x,y)$.
\end{remone}

With this analysis at hand we can now define the twisted arrow construction for Segal spaces.

\begin{defone} \label{def:twisted}
	For a Segal space $W$ define the twisted arrow Segal space $\Tw(W) \to W^{op} \times W$ as $\Tw(W)[n]_l= W[2n+1]_l$.
\end{defone}

We have the following basic result with regard to the twisted arrow construction for Segal spaces.

\begin{propone}\label{prop:twisted segal}
	Let $W$ be a Segal space. Then the twisted arrow $\Tw(W)$ is also a Segal space and $\Tw(W) \to W^{op} \times W$ is a left fibration. Moreover, the fiber of $\Tw(W)$ over a point $(x,y)$ is precisely $\map_W(x,y)$. Finally the construction agrees with the one for quasi-categories if $W$ is a complete Segal space. 
\end{propone}

\begin{proof}
	We will show $\Tw(W) \to W^{op} \times W$ is a left fibration. Then \cite[Theorem 5.11]{rasekh2017left} implies $\Tw(W)$ is a Segal space. Let $W \to \hat{W}$ be the completion of the Segal space $W$ via a Dwyer-Kan equivalence \cite[Section 14]{rezk2001css}. Then by construction we have a homotopy pullback square of simplicial spaces 
	\begin{center}
		\pbsq{\Tw(W)}{\Tw(\hat{W})}{W^{op} \times W}{\hat{W}^{op}\times \hat{W}}{}{}{}{},
	\end{center}
	which implies that $\Tw(W) \to W^{op} \times W$ is a left fibration, by \cref{rem:twisted}. Moreover, \cref{rem:twisted} and the fact that $W \to \hat{W}$ is a Dwyer-Kan equivalence implies that the fiber of $\Tw(W)$ over a point $(x,y)$ in $W^{op} \times W$ is $\map_W(x,y)$.
	Finally, if $W$ is complete, then the twisted arrow construction coincides with the one for quasi-categories with $i_1^*$.  
\end{proof}

 \section{From Left Fibrations to \texorpdfstring{$\cD$}{D}-left Fibrations} \label{sec:left fib}
 In this section we define $\cD$-left fibrations (\cref{def:nleft}) and give several alternative characterizations. We then move on to show that $\cD$-left fibrations are fibrant objects in a model structure (\cref{the:ncov model}) and study some basic properties of this model structure. The key idea is to think of a $\cD$-simplicial space as a $\cD$-space (the values) multiplied with a simplicial space (the functoriality). Hence we need to impose the appropriate condition on an injective fibration to make this intuition precise. We end with an analysis of $\cD$-left fibrations over level-wise Segal spaces and in particular the Yoneda lemma (\cref{the:levelwise Yoneda}).
  
  \subsection{Introducing $\cD$-Left Fibrations} \label{subsec:cdlfib}
  Let us start with a definition and its various alternatives.
  
  \begin{defone} \label{def:nleft}
  	Let $p: Y \to X$ be an injective fibration of $\cD$-simplicial spaces. Then $p$ is a {\it $\cD$-left fibration} if for all $d$ in $\cD$, the map $\Und_dY \to \Und_dX$ is a left fibration of simplicial spaces.
  \end{defone}
   
   We can generalize the definition of a $\cD$-left fibration similar to \cref{def:left map}.
   
   \begin{defone} \label{def:cdleft map}
   	Let $p:L \to X$ be a map of $\cD$-simplicial spaces. We say $p$ is a {\it $\cD$-left morphism} if for all object $d$ in $\cD$, $\Und_dL \to \Und_dX$ is a left morphism. 
   \end{defone}
 
   \begin{lemone} \label{lemma:cdleft map}
   	Let $p:L \to X$ be a $\cD$-left map. Then the injective fibrant replacement $\hat{L} \to X$ is a $\cD$-left fibration.
   \end{lemone}
   
   Our first task is to give several alternative characterizations of $\cD$-left fibrations generalizing results in \cite{rasekh2017left}. In order to be able to make appropriate inductive arguments we introduce the following terminology.
   
   \begin{defone}\label{def:levelwise prop}
   	Let $P$ be a property about $\cD$-simplicial spaces. We say the property $P$ {\it is level-wise} if an injectively fibrant $\cD$-simplicial space satisfies $P$ if and only if for all $d$ in $\cD$ $\Und_dX$ satisfies property $P$.
   \end{defone}

  Combining \cref{def:nleft} and \cref{def:levelwise prop} we immediately have the following useful result.
  
  \begin{lemone}\label{lemma:levelwise}
  	Let $p: Y \to X$ be an injective fibration. Then being a $\cD$-left fibration is a level-wise property.
  \end{lemone}

  We can now combine results from \cite{rasekh2017left} with \cref{lemma:levelwise} for many interesting results.
  
    \begin{lemone} \label{lemma:left fib triv Kan fib}
   	Let $p: Y \to X$ be an injective fibration of $\cD$-simplicial spaces. The following are equivalent:
   	\begin{enumerate}
   		\item $p$ is a $\cD$-left fibration.
   		\item For each $(d,k)$ in $\cD\times\DD$, the map 
   		$$(p[d,k] , \ordered{0}^*): Y[d,k] \to X[d,k] \underset{X[d,0]}{\times} Y[d,0]$$
   		is a Kan equivalence.
   		\item For each $n \geq 0$, the map 
   		$$(p[d,k] , \ordered{0}^*): Y[d,k] \to X[d,k] \underset{X[d,0]}{\times} Y[d,0]$$
   		is a trivial Kan fibration.
   	\end{enumerate}
   \end{lemone}

 \begin{proof}
 	Follows from combining \cite[Lemma 3.5]{rasekh2017left} and the fact that being a $\cD$-left fibration is level-wise (\cref{lemma:levelwise}). 
 \end{proof}

	\begin{lemone}
		Let $p: Y \to X$ be an injective fibration of $\cD$-simplicial spaces. The following two are equivalent:
		\begin{enumerate}
			\item The commutative square
			\begin{center}
				\pbsq{Y\edge{d}{k}}{Y\edge{d}{0}}{X\edge{d}{k}}{X\edge{d}{0}}{\ordered{0}^*}{p\edge{d}{k}}{p\edge{d}{0}}{\ordered{0}^*}
			\end{center}
			is a homotopy pullback square for all $(d,k)$ in $\cD \times\DD$, meaning $p$ is a $\cD$-left fibration.
			\item The commutative square 
			\begin{center}
				\pbsq{Y\edge{d}{k}}{Y\edge{d}{k-1}}{X\edge{d}{k}}{X\edge{d}{k-1}}{\ordered{0 \comma ... \comma k-1}^*}{p\edge{d}{k}}{p\edge{d}{k-1}}{\ordered{0,...,k_n-1}^*}
			\end{center}
			is a homotopy pullback square for all $(d,k)$ in $\cD \times\DD$.  
		\end{enumerate}
	\end{lemone}

 \begin{proof}
	Again we use the fact that being a $\cD$-left fibration is level-wise (\cref{lemma:levelwise}) and \cite[Lemma 3.6]{rasekh2017left}.
 \end{proof}

   \begin{propone} \label{prop:left fib lifting}
  	Let $p: Y \to X$ be an injective fibration of $\cD$-simplicial spaces. Then the following are equivalent:
  	\begin{enumerate}
  		\item $p$ is a $\cD$-left fibration.
  		\item For every object $(d,k)$ in $\cD\times\DD$, the map $p$ satisfies the right lifting property with respect to maps
  		$$(\ordered{0}: F\edge{d}{0} \to F\edge{d}{k}) \square (\partial \Delta[l] \to \Delta[l]).$$
  	\end{enumerate}
  \end{propone}
 
  \begin{proof}
  	By definition $p: Y \to X$ is a $\cD$-left fibration if and only if for all objects $d$ in $\cD$ the map of simplicial spaces $\Und_dp: \Und_dY \to \Und_dX$ is a left fibration. By \cite[Proposition 3.7]{rasekh2017left}, this is equivalent to 
  	$$(F[0] \to F[n]) \square (\partial \Delta[l] \to \Delta[l]) \pitchfork \Und_dp.$$
  	By the adjunction \cref{def:und}/\cref{rem:ss analogue}, this is equivalent to 
    is equivalent to 
    $$\Inc_d((F[0] \to F[n]) \square (\partial \Delta[l] \to \Delta[l])) \pitchfork p,$$
    which by direct computation gives us
  	$$ (F[d,0] \to F[d,n]) \square (\partial \Delta[l] \to \Delta[l]) \pitchfork p,$$
  	finishing the proof. 
 \end{proof}

 \begin{lemone}\label{lemma:cdleft comp}
	Let $f:Y \to X$ and $g:Z \to Y$ be two injective fibrations.
	\begin{enumerate}
		\item If $f$ and $g$ are $\cD$-left fibrations then $fg$ is also a $\cD$-left fibration.
		\item If $f$ and $fg$ are $\cD$-left fibrations then $g$ is also a $\cD$-left fibration.
	\end{enumerate}
\end{lemone}
 
  \begin{proof}
 	We can apply \cite[Lemma 3.8]{rasekh2017left} to the level-wise property of $\cD$-left fibrations (\cref{lemma:levelwise}).
 \end{proof}

\begin{lemone} \label{lemma:pullback nlfib}
	The pullback of a $\cD$-left fibration is a $\cD$-left fibration.
\end{lemone}

 \begin{proof}
	Follows from combining  \cite[Lemma 3.9]{rasekh2017left} and the level-wise property of $\cD$-left fibrations (\cref{lemma:levelwise}).
\end{proof}

 \begin{lemone} \label{lemma:cdleft local}
 		Let $p: L \to X$ be an injective fibration. Then the following are equivalent:
 		\begin{enumerate}
 			\item $p$ is an $\cD$-left fibration.
 			\item For every map $F\edge{d}{k} \times \Delta[l] \to X$ the pullback $L \times_X (F\edge{d}{k} \times \Delta[l]) \to F\edge{d}{k} \times \Delta[l]$ is a $\cD$-left fibration.
 			\item For every map $F\edge{d}{k} \to X$ the pullback $L \times_X F\edge{d}{k} \to F\edge{d}{k}$ is a $\cD$-left fibration.
 		\end{enumerate}
 \end{lemone}

\begin{proof}
	We simply have to adapt the proof of \cite[Lemma 3.10]{rasekh2017left} as we cannot use the level-wise argument immediately. It is immediate that $(1)$ implies $(3)$. Now we want to prove that condition $(2)$ implies that $p$ is an $\cD$-left fibration. By \cref{prop:left fib lifting} it suffices to check that it satisfies the right lifting property with respect to maps $ F\vecknedge{k_1}{k_n} \times \partial \Delta[l] \coprod_{F\vecknedgetwo{k_1}{k_{n-1}}{0} \times \partial \Delta[l]} F\vecknedgetwo{k_1}{k_{n-1}}{0} \times \Delta[l] \to F\vecknedge{k_1}{k_n} \times \Delta[l]$. We now have the following diagram  
 \begin{center}
	\begin{tikzcd}[row sep=0.5in, column sep=0.5in]
		\ds F\edge{d}{k} \times \partial \Delta[l] \coprod_{F\edge{d}{0} \times \partial \Delta[l]} F\edge{d}{0} \times \Delta[l] \arrow[d] \arrow[r]  & 
		f^*L \arrow[r] \arrow[d] & L \arrow[d, "p"] \\ 
		F\edge{d}{k} \times \Delta[l] \arrow[r, "\id"] \arrow[ur, dashed] & F\edge{d}{k} \times \Delta[l] \arrow[r, "f"] & X
	\end{tikzcd}
	.
\end{center}
The left hand square has a lift by assumption and so $p$ is a $\cD$-left fibration. Finally, we want to prove that $(3)$ implies $(2)$, but this condition is level-wise and so follows from \cite[Lemma 3.10]{rasekh2017left} combined with \cref{lemma:levelwise}.
\end{proof}
  
\subsection{The $\cD$-Covariant Model Structure}\label{subsec:cdcov}
   We are now in the right position to define a model structure, generalizing \cite[Theorem 3.12]{rasekh2017left}.
  
   \begin{theone} \label{the:ncov model}
  	Let $X$ be an $\cD$-simplicial space. 
  	There is a unique combinatorial left proper model structure enriched in $\sP(\cD)^{inj}$ with the injective model structure on the over-category $\sP(\cD\times\DD)_{/X}$, called the $\cD$-covariant model structure and denoted by $(\sP(\cD\times\DD)_{/X})^{\cD-cov}$, which satisfies the following conditions:
  	\begin{enumerate}
  		\item The fibrant objects are the $\cD$-left fibrations over $X$.
  		\item Cofibrations are monomorphisms.
  		\item A map $f: A \to B$ is a weak equivalence if it satisfies one of the following equivalent conditions:
  		\begin{itemize}
  			\item For all $d$ in $\cD$ the induced map 
  			$$\Und_df: \Und_dA \to \Und_dB$$
  			is a covariant equivalence of simplicial spaces over $\Und_dX$.
  			\item For all $\cD$-left fibrations $L \to X$ the induced map 
  				$$\Map_{/X}(B,L) \to \Map_{/X}(A,L)$$ 
  			is a Kan equivalence of spaces. 
  		\end{itemize}
  		\item A weak equivalence ($\cD$-covariant fibration) between fibrant objects is a level-wise equivalence (injective fibration). 
  	\end{enumerate}
  \end{theone}
  
  \begin{proof}
  	We can use the theory of Bousfield localizations with respect to the injective model structure on $\sP(\cD\times\DD)_{/X}$ and localize with respect to 
  	$$\mathcal{L} = \{ F[d,0] \to F[d,k]  \to X: d \in \cD , k \geq 0 \}.$$
  	We want to prove the resulting model structure satisfies all the conditions stated above. By the definition of Bousfield localizations the resulting model structure is simplicial, combinatorial, left proper with cofibrations monomorphisms. Hence we need to verify the remaining properties. 
    
    Let us first confirm the fibrations. Let $p:L \to X$ be an injective fibration of $\cD$-simplicial spaces. Then, by \cref{lemma:left fib triv Kan fib}, $p$ is a $\cD$-left fibration if for all objects $d$ in $\cD$ and $n \geq 0$ the map of spaces is a Kan equivalence
    $$L[d,n] \to X[d,n] \times_{X[d,0]} L[d,0].$$
    We can extend this map to a commutative diagram 
    \begin{center}
    	\begin{tikzcd}
    		L[d,n] \arrow[rr] \arrow[dr, twoheadrightarrow] & &  X[d,n] \times_{X[d,0]} L[d,0] \arrow[dl, twoheadrightarrow] \\
    		 & X[d,n] & 
    	\end{tikzcd},
    \end{center}
	where two diagonal maps are Kan fibrations. By a classical result for Kan fibrations (\cite[Section 38]{rezk2017qcats}), the top map is an equivalence if and only if it is a fiber-wise equivalence, meaning the map  
	$$L[d,n] \times_{X[d,n]} \Delta[0] \to \Delta[0] \times_{X[d,n]} X[d,n] \times_{X[d,0]} L[d,0] \cong \Delta[0] \times_{X[d,0]} L[d,0]$$
	is a Kan equivalence for all points in $X[d,n]$. By the Yoneda lemma this is equivalent to the map of spaces
	$$\Map_{/X}(F[d,n],L) \to \Map_{/X}(F[d,0],L)$$
	being a Kan equivalence, which is the definition of fibrant objects in the $\cD$-covariant model structure.
    
    We move on to characterize weak equivalences. Let $Y \to X$ be an arbitrary map of $\cD$-simplicial spaces. Recall that we obtain the fibrant replacement in the localization model structure by applying the small object argument with respect to maps of the following form:
    	\begin{itemize}
    		\item $(A \to B) \square (\Lambda[l]_i \to \Delta[l]) \to X$,
    		\item $(F[d,0] \to F[d,n]) \square (\partial \Delta[l] \to \Delta[l]) \to X$,
    	\end{itemize}
    where $A \to B$ is a injective cofibration. This means we obtain a fibrant replacement of $Y \to X$ in the $\cD$-covariant model structure via a chain of morphisms 
    $$Y = Y_0 \to Y_1 \to ... \to \hat{Y}$$
    over $X$, where each morphism is a pushout along the two classes of morphisms of the form above. As both of these classes of morphisms are level-wise covariant equivalences, this means the map $Y \to \hat{Y}$ is also a level-wise covariant equivalence over $X$. 
    
    We will now use that to characterize arbitrary equivalences. Let $f:Y \to Z$ be an arbitrary map over $X$ and let $\hat{f}: \hat{Y} \to \hat{Z}$ be the fibrant replacement. By definition, $f$ is a $\cD$-covariant equivalence if and only if $\hat{f}$ is an injective equivalence, which is equivalent to $\Und_d\hat{f}: \Und_d\hat{Y} \to \Und_d\hat{Z}$ being a Reedy equivalence of simplicial spaces for every object $d$ in $\cD$. For every object $d$ in $\cD$ we have the following commutative diagram
    \begin{center}
    	\begin{tikzcd}
    		\Und_dY \arrow[r, "f"] \arrow[d, "\simeq"] &  \Und_dZ \arrow[d, "\simeq"] \arrow[ddr, bend left=20] \\
    		\Und_d\hat{Y} \arrow[r, "\hat{f}"] \arrow[drr, bend right=20] & \Und_d\hat{Z} \arrow[dr] \\
    		& & \Und_dX
    	\end{tikzcd}.
    \end{center}
    By the argument above, the maps $\Und_dY \to \Und_d\hat{Y}$ and $\Und_dZ \to \Und_d\hat{Z}$ are covariant equivalences, hence $\Und_d\hat{Y} \to \Und_d\hat{Z}$ is a Reedy equivalence if and only if $\Und_dY \to \Und_dZ$ is a covariant equivalence over $\Und_dX$, which proves that a map $Y \to Z$ is a $\cD$-covariant equivalence if and only if it is a level-wise covariant equivalence over $X$. 
     
     Finally, we want to show that the covariant model structure is enriched over $\sP(\cD)^{inj}$ via the inclusion $\VEmb: \sP(\cD) \to \sP(\cD\times\DD)$. Let $i:A\to B$ be a cofibration of $\cD$-spaces and $j:C \to D$ over $X$ be a cofibration of $\cD$-simplicial spaces. By \cref{def:enriched model cat}, we need to prove that $i \square j$ is a cofibration over $X$, which is trivial if either $i$ or $j$ is trivial. The cofibration part is evident as as all cofibrations are monomorphisms, thus we only need to check the second part. Assume that either $i$ or $j$ are trivial cofibrations. Then $\Und_di$ or $\Und_dj$ are trivial cofibrations for all objects $d$ in $\cD$. This implies that $\Und_d(i \square j) = \Und_dj \square \Und_dj$ is a trivial cofibration for all objects $d$ in $\cD$, which implies that $i \square j$ is a trivial cofibrations in the $\cD$-covariant model structure.
  \end{proof}
  
  \begin{remone}
  	There is a certain similarity between the way we constructed the $\cD$-covariant model structure and the level-wise localized model structures (\cref{prop:levelwise localized model structure}), which might suggest we could have just constructed the $\cD$-covariant model structure as a certain level-wise localized model structure. Unfortunately this approach does not work. Indeed, although we have a level-wise covariant model structure, each level is over a different base ($\Und_dX$) and so the different levels do not have the same model structure. 
  \end{remone}

   \begin{lemone} \label{lemma:fib between lfib}
  	Let $p: L \to X$ and $q: L' \to X$ be two $\cD$-left fibrations. A map $f: L \to L'$ over $X$ is a fibration in the $\cD$-covariant 
  	model structure if and only if it is an $\cD$-left fibration. 
  \end{lemone}
   
   \begin{proof}
	Being a $\cD$-left fibration is a level-wise condition (\cref{lemma:levelwise}) and so it follows from \cite[Lemma 3.14]{rasekh2017left}.
   \end{proof}
   
    \begin{theone} \label{the:diag local ncov}
   	The following is a Quillen adjunction
   	\begin{center}
   		\adjun{(\sP(\cD\times\DD)_{/X})^{\cD-cov}}{(\sP(\cD\times\DD)_{/X})^{\cD-diag}}{id}{id}
   	\end{center}
   	which is a Quillen equivalence if for all $d$ in $\cD$ the simplicial space $\Und_dX$ is homotopically constant. Here the left side has the $\cD$-covariant model structure and the right side has the $\cD$-diagonal model structure (\cref{the:diagmodel}) on the over category. This implies that the $\cD$-diagonal model structure over $X$ is a localization of the $\cD$-covariant model structure over $X$.
   \end{theone}
   
   \begin{proof}
   	It suffices to prove the left adjoint preserves cofibrations and trivial cofibrations. However, by \cref{the:ncov model}, both are determined level-wise (\cref{def:levelwise prop}) and so the result follows from \cite[Theorem 3.17]{rasekh2017left}. Now let us assume  $\Und_dX$ is homotopically constant for all $d$ in $\cD$. We want to prove that the left adjoint reflects equivalences and the derived counit map is an equivalence.
   	
   	By \cref{the:ncov model}, every object is cofibrant and so the derived counit map is simply the counit map. Hence, both conditions (left adjoint reflecting equivalences and the counit being an equivalence) are  determined level-wise (\cref{def:levelwise prop}) and so the result follows from  \cite[Theorem 3.17]{rasekh2017left} combined with the fact that $\Und_dX$ is homotopically constant.
   \end{proof}

   \begin{lemone} \label{lemma:exp cdleft}
  	Let $p:L \to X$ be an $\cD$-left fibration and $i:A \to B$, $j: C \to D$ monomorphisms.
  	\begin{enumerate}
  		\item If $i$ or $j$ are $\cD$-covariant equivalences then $i \square j$ is a trivial $\cD$-covariant cofibration. 
  		\item $\exp{i}{p}$ is an $\cD$-left fibration.
  	\end{enumerate}
  \end{lemone}
  
  \begin{proof}
  	As coproducts, products, monomorphisms and weak equivalences are determined level-wise (\cref{def:levelwise prop}) the first statement follows from \cite[Lemma 3.25]{rasekh2017left} applied level-wise. The second one then follows from the first by \cref{prop:joyal tierney lifting}.
  \end{proof}

  \begin{theone} \label{the:retract}
  	Let $i:A \to B$ be a monomorphism of $\cD$-simplicial spaces over $X$. Then $i$ is a trivial cofibration in the $\cD$-covariant model structure on $\sP(\cD\times\DD)_{/X}$ if there exists a retraction $r: B \to A$ (not necessarily over $X$) and a $H: B \times D[1] \to B$ such that 
  	$H(-,0) = ir$ and $H(-,1) = \id_B$.
  \end{theone}
   
   \begin{proof}
   	$i:A \to B$ being a trivial cofibration is determined level-wise (\cref{def:levelwise prop}) and so the result follows from \cite[Theorem 3.27]{rasekh2017left}.
   \end{proof}
   
  We can appropriately relate the $\cD$-covariant model structure and the level-wise complete Segal space model structure (\cref{prop:levelcss}).

   \begin{theone} \label{the:dcov css inv}
   	Let $f: X \to Y$ be map of $\cD$-simplicial spaces. Then the adjunction
   	\begin{center}
   		\adjun{(\sP(\cD\times\DD)_{/X})^{\cD-cov}}{(\sP(\cD\times\DD)_{/Y})^{\cD-cov}}{f_!}{f^*}
   	\end{center}
   	is a $\sP(\cD)$-enriched Quillen adjunction, which is a Quillen equivalence whenever $f$ is a level-wise CSS equivalence (\cref{prop:levelcss}). Here both sides have the $\cD$-covariant model structure.
   \end{theone} 

    \begin{proof}
    	By \cref{lemma:enriched adjunction via tensors} the adjunction is evidently $\sP(\cD)$-enriched as post-composition preserves tensors. It is a Quillen adjunction as $\cD$-left fibrations are closed under pullbacks (\cref{lemma:pullback nlfib}) and both sides have the same cofibrations (inclusions). So, let us assume $f$ is a level-wise CSS equivalence. We want to prove that $(f_! \dashv f^*)$ is a Quillen equivalence. We will prove that the derived counit is an equivalence and that the left adjoint reflects equivalences. 
    	
    	By \cite[Theorem 5.1]{rasekh2017left} the derived counit map is a level-wise covariant equivalence and so an equivalence in the $\cD$-covariant model structure (\cref{the:ncov model}). On the other hand assume $A \to B$ a map over $X$ such that it is an $\cD$-covariant equivalence over $Y$. Fix an element $d$ in $\cD$. Then, again, by \cref{the:ncov model}, $\Und_dA \to \Und_dB$ is a covariant equivalence over $\Und_dY$ and so, by \cite[Theorem 5.1]{rasekh2017left}, is a covariant equivalence over $\Und_dX$, which, again by \cref{the:ncov model}, implies that $A \to B$ is an $\cD$-covariant equivalence over $X$. 
    \end{proof}
   
    \begin{theone} \label{the:cov loc css}
   	Let $X$ be a $\cD$-simplicial space. Then the following adjunction
   	\begin{center}
   		\adjun{(\sP(\cD\times\DD)_{/X})^{\cD-CSS}}{(\sP(\cD\times\D))_{/X})^{\cD-cov}}{id}{id}
   	\end{center}
   	is a Quillen adjunction.
   	Here the left hand side has the level-wise CSS model structure (\cref{prop:levelcss}) on the over-category and the right hand side has the $\cD$-covariant model structure.
   	\par 
   	This implies that the $\cD$-covariant model structure over $X$ is a localization of the induced level-wise CSS model structure over $X$.
   \end{theone}
    
    \begin{proof}
    	Evidently the left adjoint preserves cofibrations as they are simply the inclusions. Hence it suffices to show that the right adjoint preserves fibrant objects, meaning we have to show every $\cD$-left fibration is a level-wise complete Segal space fibration. As both conditions are level-wise (\cref{def:levelwise prop}) it follows from \cite[Theorem 5.11]{rasekh2017left}.
    \end{proof}

   \begin{remone} \label{rem:nrfib}
     Analogous to left fibrations, which can be dualized and give us {\it right fibrations}, $\cD$-left fibrations can be dualized to {\it $\cD$-right fibrations}. Concretely, a $\cD$-right fibration is an injective fibration $p: Y \to X$ such that $\Und_d:\Und_dY \to \Und_dX$ is a right fibration for all objects $d$ in $\cD$. We can use this definition to repeat all the results we have stated up until now regarding $\cD$-left fibrations for $\cD$-right fibrations. In particular, $\cD$-right fibrations are also the fibrant objects in a model structure, the {\it $\cD$-contravariant model structure}.
   \end{remone}
   
   We only state the following lemma explicitly connecting $\cD$-left and $\cD$-right fibrations. 
   
   \begin{lemone} \label{lemma:opposite fib}
   	A map of $\cD$-simplicial spaces $p: L \to X$ is a $\cD$-left fibration if and only if $p^{op}: L^{op} \to X^{op}$ is a $\cD$-right fibration, $(-)^{op}$ is given in \ref{eq:op}.
   \end{lemone} 

   \begin{proof}
   	It is immediate $p$ is an injective fibration if and only if $p^{op}$ is an injective fibration. Now, for all objects $d$ in $\cD$, $(<0>^*)^{op}:L[d,n] \to L[d,0]$ is equal to $<n>^*:L^{op}[d,n] \to L^{op}[d,0]$, which implies that in the following two commutative squares
   	\begin{center}
   		\begin{tikzcd}[row sep=0.3in, column sep=0.3in]
   			L[d,n] \arrow[r, "<0>^*"] \arrow[d, "p \leb d \comma n \reb"] & L[d,0] \arrow[d, "p \leb d \comma 0 \reb"] & & L^{op}[d,n] \arrow[r, "<n>^*"] \arrow[d, "p^{op} \leb d \comma n\reb"] & L^{op}[d,0] \arrow[d, "p^{op} \leb d \comma 0\reb"] \\ 
   			X[d,n] \arrow[r, "<0>^*"] & X[d,0] & & X^{op}[d,n] \arrow[r, "<n>^*"] & X^{op}[d,0] 
   		\end{tikzcd}
   	\end{center} 
   	one is a homotopy pullback square if and only if the other one is a homotopy pullback square and hence we are done.
   \end{proof}

    \begin{theone} \label{the:pullback rfib}
   	Let $p:R \to X$ be a $\cD$-right or $\cD$-left fibration. Then the adjunction
   	\begin{center}
   		\adjun{(\sP(\cD\times\DD)_{/X})^{lev_{CSS}}}{(\sP(\cD\times\DD)_{/R})^{lev_{CSS}}}{p^*}{p_*}
   	\end{center}
   	is a Quillen adjunction. Here both sides have the level-wise CSS model structure (\cref{prop:levelcss}) on the over-category.
   \end{theone}
   
    \begin{proof}
    	It suffices to prove the left adjoint preserves cofibrations and trivial cofibrations. Both of those are determined level-wise (\cref{def:levelwise prop}) and so the statement follows from \cite[Theorem 5.15]{rasekh2017left}.
    \end{proof}

   \begin{theone} \label{the:nlfib pb nrfib}   
   	Let $p: R \to X$ be a $\cD$-right fibration. Then the adjunction
   	\begin{center}
   		\adjun{(\sP(\cD\times\DD)_{/X})^{\cD-cov}}{(\sP(\cD\times\DD)_{/X})^{\cD-cov}}{p_!p^*}{p_*p^*}
   	\end{center}
   	is a Quillen adjunction where both sides have the $\cD$-covariant model structure.
   \end{theone}
   
   \begin{proof}
   	It suffices to prove the left adjoint preserves cofibrations and trivial cofibrations. However, being a (trivial) cofibration in the $\cD$-covariant model structure is a level-wise property (\cref{def:levelwise prop}). Hence, this follows from \cite[Theorem 4.28]{rasekh2017left}.
   \end{proof}
 
  We can use these results to give various characterizations of $\cD$-covariant equivalences between $\cD$-left fibrations. 
  
  \begin{theone}  \label{the:ncov equiv nlfib}
 	Let $L \to X$ and $L' \to X$ be $\cD$-left fibrations and $f:L \to L'$ a map over $X$. 
 	The following are equivalent:
 	\begin{enumerate}
 		\item $f$ is a $\cD$-covariant equivalence.
 		\item $f$ is an injective equivalence.
 		\item $f$ is a $\cD$-value equivalence (\cref{the:valmodel}).
 		\item $f$ is a fiberwise injective equivalence ($\Fib_xf: \Fib_xL\to \Fib_xL'$ is an injective equivalence for all maps $x:F[d,0] \to X$).
 		\item $f$ is a fiberwise $\cD$-value equivalence ($\Fib_xf: \Fib_xL\to \Fib_xL'$ is a $\cD$-value equivalence for all maps $x:F[d,0] \to X$).
 		\item $f$ is a fiberwise $\cD$-diagonal equivalence ($\Fib_xf: \Fib_xL\to \Fib_xL'$ is a $\cD$-diagonal equivalence for all maps $x:F[d,0] \to X$).
 	\end{enumerate}
 \end{theone}
 
 \begin{proof}
  The six conditions are all analogues of the six conditions in \cite[Theorem 4.34]{rasekh2017left}, hence it suffices to confirm that they are all determined level-wise (\cref{def:levelwise prop}). The fact that $\cD$-covariant equivalences are level-wise follows from \cref{the:ncov model}, that injective equivalences are level-wise from \cref{prop:injectivemod}, that the $\cD$-value model structure equivalences are level-wise from \cref{the:valmodel} and finally that the $\cD$-diagonal model structure equivalences are level-wise from \cref{the:diagmodel}.
 \end{proof}
 
 \begin{theone} \label{the:recognition principle} 
 	Let $g: Y \to Z$ be a map over $X$ and $R_x \to X$ a choice of $\cD$-contravariant fibrant replacement of the map $\{x\}:F\edge{d}{0} \to X$.
 	Then $g:Y \to Z$ over $X$ is a $\cD$-covariant equivalence if and only if for every $\{x\} : F\edge{d}{0} \to X$
 	$$ R_x \underset{X}{\times} Y \to R_x \underset{X}{\times} Z$$
 	is a $\cD$-diagonal equivalence. Here $R_x \to X$ is the fibrant replacement of $x$ in the $\cD$-contravariant model structure.
 \end{theone}

  \begin{proof}  
	Let $\hat{g}: \hat{Y} \to \hat{Z}$ be a $\cD$-left fibrant replacement. By \cref{the:ncov model}, $g:Y \to Z$ is a $\cD$-covariant equivalence if and only if $\hat{g}:\hat{Y} \to \hat{Z}$ is an injective equivalence. We now have the following diagram:
	\begin{center}
		\begin{tikzcd}[row sep=0.5in, column sep=0.5in]
			Y \underset{X}{\times} R_x \arrow[r, "g \times id"] \arrow[d, "\simeq", "i \times id"'] & 
			Z \underset{X}{\times} R_x \arrow[d, "\simeq"', "j \times id"]\\
			\hat{Y} \underset{X}{\times} R_x \arrow[r, "\hat{g} \times id"] & \hat{Z} \underset{X}{\times} R_x \\
			\hat{Y} \underset{X}{\times} F\leb d \comma 0 \reb \arrow[r, "\hat{g} \times id"] \arrow[u, "\simeq"'] & \hat{Z} \underset{X}{\times} F\leb d \comma 0 \reb \arrow[u, "\simeq"]
		\end{tikzcd}
		.
	\end{center} 
	By \cref{the:nlfib pb nrfib}, the two top vertical maps are $\cD$-covariant equivalences and by the dual of \cref{the:nlfib pb nrfib} (explained in \cref{rem:nrfib}) the two bottom vertical maps are $\cD$-contravariant equivalences.
	Hence, by \cref{the:diag local ncov} and \cref{rem:nrfib}, all vertical maps are $\cD$-diagonal equivalences and so the top horizontal map 
	is a $\cD$-diagonal equivalence if and only if the bottom horizontal map is one. 
	But the bottom map is a $\cD$-diagonal equivalence for every $x: F\edge{d}{0} \to X$ if and only if $\hat{Y} \to \hat{Z}$ is an injective equivalence (by \cref{the:ncov equiv nlfib}) and hence we are done.
\end{proof}
 
 \begin{remone}
 	Notice, the proof of \cref{the:recognition principle} cannot be immediately reduced to a level-wise proof and so we cannot simply cite the proof in \cite[Theorem 4.39]{rasekh2017left}, as the property of being a fibrant replacement is not necessarily a level-wise property.
 \end{remone}

 In certain situations we can simplify the conditions in \cref{the:recognition principle}.
 
 \begin{corone} \label{cor:weakly constant recognition principle}
 	Let $X$ be a $\cD$-simplicial space with weakly constant objects. Then a map of $\cD$-simplicial spaces $g: Y \to Z$ over $X$ is a $\cD$-covariant equivalence if and only if for every map $\{x\}:D[0] \to X$, 
 	$$ R_x \underset{X}{\times} Y \to R_x \underset{X}{\times} Z$$
 	is a $\cD$-diagonal equivalence.
 \end{corone} 

 \begin{proof} 
  The condition given here is a special case of  \cref{the:recognition principle} and so we only need to prove it suffices when $X$ has weakly constant objects. Let $X \to \hat{X} \leftarrow \tilde{X}$ be the zigzag given in \cref{cor:zigzag better}. The fact that both maps are complete Segal space equivalences implies that any map $F[d,0] \to X$ is equivalent to a map $F[d,0] \to \tilde{X}$ (by \cref{the:dcov css inv}). Now, by assumption $X$ has weakly constant objects, which means, by \cref{cor:zigzag better}, that $\tilde{X}$ has discrete objects and the map factors as $F[d,0] \to D[0] \xrightarrow{ \{ x \} } \tilde{X}$, where $x$ is the image of the map from $D[0]$.
  
We can post-compose and, by \cref{the:ncov model}, $F[d,0] \cong F[d,0] \times D[0] \to F[d,0] \times \hat{R}_x$ is the fibrant replacement of $F[d,0] \to D[0] \to \tilde{X} \to \hat{X}$. By \cref{the:dcov css inv}, this means $F[d,0] \times R_x = F[d,0] \times i^*\hat{R}_x$ is the fibrant replacement of $F[d,0]$ in the $\cD$-contravariant model structure over $X$. 
  
  Now by direct computation, for a given map $Y \to Z$ the map 
  $$(F[d,0] \times R_x) \times_X Y \to (F[d,0] \times R_x) \times_X Z$$
  is isomorphic to 
  $$F[d,0] \times (R_x \times_X Y) \to F[d,0] \times (R_x \times_X Z),$$
  which is a $\cD$-diagonal equivalence if and only if $R_x \times_X Y \to R_x \times_X Z$ is a $\cD$-diagonal equivalence. 
 \end{proof} 

 \subsection{Fibrations over Segal Spaces and the Yoneda Lemma}\label{subsec:yoneda}
 In this section we focus on $\cD$-left fibrations over level-wise Segal spaces (\cref{def:levelsegal}). 
 
 \begin{defone} \label{def:under segal}
 	Let $W$ be a level-wise Segal space and $\{x\}: D[0] \to W$ be a map and define $W_{x/}$ via the pullback
 	\begin{center}
 		\pbsq{W_{x/}}{W^{D[1]}}{D[0]}{W}{}{}{<0>^*}{x}
 	\end{center}
  	which we call the {\it level-wise under-Segal space}.
 \end{defone}
 The name suggests the resulting $\cD$-simplicial space is in fact a level-wise Segal space.
 
  \begin{lemone} 
 	Let $W$ be a level-wise Segal and $x:D[0] \to W$, then $W_{x/}$ is also a level-wise Segal space.
 \end{lemone}

 \begin{proof}
 	By direct computation we have 
 	$$\Und_{d}W_{x/} = (\Und_{d}W)^{F[1]} \times_{\Und_{d}W} F[0] = (\Und_dW)_{x/}$$
 	which is a Segal space as $\Und_d(W)$ is a Segal space (\cref{prop:levelsegal}) combined with \cite[Lemma 3.43]{rasekh2017left}.
 \end{proof}
  
  \begin{theone}  \label{the:undercat left}
  Let $W$ be a level-wise Segal space and $x:D[0] \to W$. Then the projection map $W_{x/} \to W$ is a $\cD$-left fibration.
  \end{theone}

  \begin{proof}
  	The injective model structure is Cartesian and so the map $W^{D[1]} \to W \times W$ is an injective fibration and so the pullback $W_{x/} \to W$ is also an injective fibration. Hence, it suffices to prove that $\Und_{d}W_{x/} \to \Und_{d}W$ is a left fibration for all objects $d$ in $\cD$. By direct computation this map is simply equal to 
  	$$(\Und_dW)_{x/} \to \Und_d(W),$$
  	which is a left fibration by \cite[Theorem 3.44]{rasekh2017left}.
  \end{proof}
 
  \begin{theone}  \label{the:levelwise Yoneda}
 	Let $W$ be a level-wise Segal space and $x$ an object. Then the map $\{id_x\}: D[0] \to W_{x/}$ is a $\cD$-covariant equivalence over $W$.
 \end{theone}
   
   \begin{proof}
   	The property of being a $\cD$-covariant equivalence is level-wise (\cref{def:levelwise prop}), which means it suffices to prove that $F[0] \to (\Und_{d}W)_{x/}$ is a covariant equivalence over $\Und_{d}(W)$ for all objects $d$ in $\cD$. This fact follows immediately from \cite[Theorem 3.49]{rasekh2017left}.
   \end{proof}
 
 \begin{remone}
 	Notice we could have proven \cref{the:levelwise Yoneda} directly using \cref{the:retract} and a diagram similar to the one given in the proof of \cite[Theorem 3.49]{rasekh2017left}.
 \end{remone}
  
  There is an analogous contravariant version.
  
   \begin{defone} \label{def:over segal}
  	Let $W$ be a level-wise Segal space and $\{x\}: D[0] \to W$ be a map and define $W_{/x}$ via the pullback
  	\begin{center}
  		\pbsq{W_{/x}}{W^{D[1]}}{D[0]}{W}{}{}{<1>^*}{x}
  	\end{center}
  	which we call the {\it level-wise over-Segal space}.
  \end{defone}

 Similar to \cref{the:undercat left} $W_{/x} \to W$ is a right fibration and similar to \cref{the:levelwise Yoneda} $D[0] \to W_{/x}$ is a $\cD$-contravariant equivalence over $W$. 
  
   The construction of the over-Segal space only depends on a morphism $D[0] \to W$ and hence suggest that we should focus on attention on $\cD$-simplicial spaces with weakly constant objects (\cref{def:weakly constant objects}). Let $\cD$ have a terminal object $t$ and let $W$ be an $\sP(\cD)$-enriched Segal space. Then a map $D[0] \to W$ is precisely a choice of object in $W$ (\cref{def:enriched segal}) and \cref{cor:weakly constant recognition principle} takes the following form.
  
  \begin{corone} \label{cor:enriched seg recognition principle}
  	Let $\cD$ be a small category with terminal object.
  	Let $W$ be an $\sP(\cD)$-enriched Segal space. Then a morphism of $\cD$-simplicial spaces $Y \to Z$ over $W$ is a $\cD$-covariant equivalence if and only if for every object $x$ in $W$, the map 
  	$$Y \times_W W_{/x} \to Z \times_W W_{/x}$$
  	is a $\cD$-diagonal equivalence.
  \end{corone}
  
  We can generalize the construction from slice categories (\cref{def:under segal}) to {\it twisted arrow categories}. For a level-wise Segal space $W$ define $\Tw(W) \to W^{op} \to W$ level-wise as the twisted arrow construction given in \cref{def:twisted}. Here $W^{op}$ is the opposite $\cD$-simplicial space described in \ref{eq:op}.
 
  \begin{propone}\label{prop:twisted}
  	Let $W$ be a level-wise Segal space and let $\Tw(W) \to W^{op} \times W$ be the level-wise twisted arrow construction. Then the map is a $\cD$-left fibration, with fiber over two objects $x,y$ the mapping $\cD$-space $\map_W(x,y)$ (\cref{def:map}).
  \end{propone}

 \begin{proof}
 	For every object $d$ in $\cD$, we have $\Und_d\Tw(W) = \Tw(\Und_dW)$ and so result from the level-wise characterization of $\cD$-left fibrations and mapping $\cD$-spaces.
 \end{proof}

We want to give one last example of a fibration in the level-wise Segal setting, however, this requires us to expand the playing field. 

\begin{exone} \label{ex:target fibration}
Let $W$ be a level-wise Segal space. Then we can take $W$ as a constant $\cD$-bisimplicial space i.e. as a constant simplicial object in $\cD$-simplicial spaces. In this setting the map $W^{op} \times W \to W$ is a $\cD\times\DD$-left fibration, which follows from applying \cref{the:diag local ncov} to the $\cD\times\DD$-covariant model structure. 

On the other side, we can also take $\Tw(W) \to W^{op} \times W$ as a map of constant simplicial objects in $\cD$-simplicial spaces, and, by \cref{prop:twisted}, this is also a $\cD\times\DD$-left fibration. Hence, by \cref{lemma:cdleft comp}, the composition $\Tw(W) \to W$ is a $\cD\times\DD$-left fibration. For a given object $x$, the fiber is the level-wise under-Segal space $W_{x/}$, hence we call it the {\it target fibration}. 
\end{exone} 

\begin{exone} \label{ex:source fibration}
	Using the same arguments as in \cref{ex:target fibration} we can deduce that $\Tw(W) \to W^{op} \times W \to W^{op}$ is a $\cD\times\DD$-left fibration and so, by \cref{lemma:opposite fib}, $\Tw(W)^{op} \to W$ is a $\cD\times\DD$-right fibration, which has fiber over a point $x$ the level-wise over Segal space $W_{/x}$ and is hence called the {\it source fibration}.
\end{exone} 
 
\section{Grothendieck Construction for \texorpdfstring{$\cD$}{D}-Simplicial Spaces} \label{sec:grothendieck}
 In this section we want to present the Grothendieck construction in the $\sP(\cD)$-enriched setting, relating functors and fibrations. We will proceed in three steps each generalizing the previous step.
 \begin{enumerate}
 	\item A strict Grothendieck construction for categories, relating $\P(\cD)$-enriched functors $\C \to \P(\cD)$ with $\cD$-Grothendieck opfibrations $\D \to \C$, where $\C$ is a $\P(\cD)$-enriched category (\cref{subsec:strict groth}).
 	\item A homotopical Grothendieck construction for categories, relating $\sP(\cD)$-enriched functors $\C\to \sP(\cD)$ with $\cD$-left fibrations $L \to N_\cD\C$, where $\C$ is a $\sP(\cD)$-enriched category (\cref{subsec:enriched groth}).
 	\item A homotopical Grothendieck construction for $\cD$-spaces with weakly constant objects, relating $\sP(\cD)$-enriched functor $\C_X \to \sP(\cD)$ with $\cD$-left fibrations $L \to X$, where $X$ is a $\cD$-simplicial spaces with weakly constant objects and $\C_X$ is its associated $\sP(\cD)$-enriched category (\cref{subsec:grothendieck general}).
 \end{enumerate}

\subsection{Grothendieck Construction for \texorpdfstring{$\P(\cD)$}{P(D)}-Enriched Categories}\label{subsec:strict groth}
 In this subsection we generalize the classical Grothendieck construction from (set-enriched) categories (\cref{subsec:grothendieck fib}) to $\P(\cD)$-enriched categories and we then use to construct the desired Quillen equivalence. 
 
 \begin{remone} 
 	There is an enriched Grothendieck construction due to Beardsley and Wong \cite{beardsleywong2019grothendieck}, which holds for more general enrichments, however, only over free enriched categories, whereas we need the Grothendieck construction over arbitrary $\P(\cD)$-enriched categories. 
 \end{remone}
 
 Recall that the category of $\cD$-simplicial sets is enriched over itself (\cref{lemma:snset enriched}) and so for a category enriched over $\cD$-simplicial sets $\C$ we can define the category of enriched functors $\Fun(\C,\P(\cD))$ (\cref{ex:psh}). We would like to study this category in terms of unenriched functors. 

\begin{defone} \label{def:dcat}
	A {\it $\cD$-category} is a functor $\cD^{op}\to \cat$. We denote the category of $\cD$-categories by $\Fun(\cD^{op},\cat)$ and notice it is enriched and tensored over $\P(\cD)$, with tensor given by applying $\Fun(\cD^{op},-)$ to the tensor of $\set \times \cat \to \cat$.
\end{defone}

\begin{lemone} \label{lemma:undcat}
	There exists a fully faithful (set-enriched) functor 
	$$\Und: \cat_{\P(\cD)} \to \Fun(\cD^{op},\cat)$$
	which takes a $\P(\cD)$-enriched category $\C$ to the $\cD$-category $\Und(\C): \cD^{op} \to \cat$, which takes an element $d$ to the category $\Und_d\C$ with $\Obj_{\Und_d(\C)} = \Obj_\C$ and morphisms given by 
	$$\Hom_{\Und_d(\C)}(x,y) = \uHom_{\C}(x,y)[d].$$
	Its essential image given by functors $Q:\cD^{op} \to \cat$ such that for all morphisms $f:d \to d'$ in $\cD$ the induced map $P(f):P(d') \to P(d)$ is the identity on objects.
\end{lemone}

\begin{proof}
 Let $\ev_\cD: \cD^{op} \to \Fun(\P(\cD),\set)$ be the functor that takes an object $d$ to the representable functor $\Hom_{\P(\cD)}(F[d],-):\P(\cD) \to \set$, which corresponds to evaluating a presheaf $P: \cD^{op} \to \set$ at $d$, giving us $P(d)$. Notice representable functors are product preserving and so we can restrict the codomain to $\ev_\cD: \cD^{op} \to \Fun^\times(\P(\cD),\set)$. We can now post-compose this functor with $\Ch$ (\cref{lemma:ch}) to define a functor 
 $$\cD^{op} \to \Fun^\times(\P(\cD),\set) \to \Fun(\cat_{\P(\cD)},\cat).$$
 Using the adjunction between product and functor categories the data of this functor is equivalent to a functor 
 $$\Und: \cat_{\P(\cD)} \to \Fun(\cD^{op},\cat).$$
 We now want to prove it is fully faithful and characterize the essential image. To observe it is fully faithful note we have the following diagram
 \begin{center}
 	\begin{tikzcd}
 		\cat_{\P(\cD)} \arrow[rr, "\Und"] \arrow[dr, "N_\cD"'] & &  \Fun(\cD^{op},\cat) \arrow[dl, "\Fun(\cD^{op} \comma N)"] \\
 		& \P(\cD\times\DD) &
	\end{tikzcd}.
 \end{center}
 The left hand functor is fully faithful by \cref{prop:nervesnset ff} and the right hand functor is fully faithful as the nerve is fully faithful.
 The fact that the two diagonal maps are fully faithful implies that $\Und$ is fully faithful as well. Finally the characterization of the essential image follows from the construction of $\Und$.
\end{proof}
 
 \begin{exone}
 	By \cref{lemma:snset enriched}, the category $\P(\cD)$ is $\P(\cD)$-enriched and the category $\Und_d\P(\cD)$ has objects $\cD$-sets and set of morphisms 
 	$$\Hom_{\Und_d\P(\cD)}(X,Y) = \uHom_{\P(\cD)}(X,Y)[d]= \Hom_{\P(\cD)}(X \times F[d],Y).$$
 \end{exone}

  This lemma gives us the following important identification.
 
 \begin{corone} \label{cor:iso functor und}
 	Let $\C$ be a $\P(\cD)$-enriched category.
 	There is an isomorphism of categories
 	$$\Fun(\C,\P(\cD)) \to \Fun(\Und(\C),\Und(\P(\cD))).$$
 \end{corone}

 The name $\Und$ should remind us of the definition of the underlying simplicial set (\cref{def:und}). In fact for a given $d$ in $\cD$ we have the following commutative diagram 
 \begin{center}
 	\begin{tikzcd}[row sep=0.5in, column sep=0.6in]
 		\cat_{\P(\cD)} \arrow[r, "\Und_d"] \arrow[d, "N_\cD"'] & \cat \arrow[d, "N"] \\
 		\P(\cD\times\DD) \arrow[r, "\Und_d"] & \sset 
 	\end{tikzcd}.
 \end{center} 

\begin{defone} \label{def:cd groth opfib}
	Let $\C$ be a small $\P(\cD)$-enriched category. 
	A {\it $\cD$-Grothendieck opfibration} over $\C$ is a map of $\cD$-categories $\D\to \Und\C$ such that for all $d$ in $\cD$, the functor $\Und_d\D \to \Und_d\C$ is a discrete Grothendieck opfibration of categories (\cref{def:discrete Groth opfib}).
\end{defone}

\begin{defone}
	Let $\C$ be a small $\P(\cD)$-enriched category.
	Let $\opGroth_{\cD/\C}$ be the full subcategory of $\Fun(\cD^{op},\cat)_{/\Und\C}$ consisting of $\cD$-Grothendieck opfibrations. 
\end{defone}

We want to construct an adjunction that induces an equivalence between $\P(\cD)$-enriched functors of $\P(\cD)$-enriched categories $\C\to \P(\cD)$ and $\cD$-Grothendieck opfibrations over $\C$. 

\begin{defone} \label{def:intCD}
	Let $F: \C \to \P(\cD)$ be a $\P(\cD)$-enriched functor. For each object $d$ in $\cD$, define the set-enriched functor $\Hom_{\Und_d\P(\cD)} (D[0],\Und_dF): \Und_d\C \to \set$ as the composition 
	$$\Und_d\C \xrightarrow{\Und_dF} \Und_d\P(\cD) \xrightarrow{\Hom_{\Und_d\P(\cD)}(D[0],-)} \set $$
	Now, define the $\cD$-Grothendieck opfibration $\ds\int_{\cD/\C} F \to \Und\C$ as follows. It takes an object $d$ in $\cD$ to the Grothendieck opfibration
	$$\Und_d\int_{\cD/\C} F = \int_{\Und_d\C}\Hom_{\Und_d\P(\cD)} (D[0],\Und_dF),$$
	where $\int$ is defined in \cref{def:int}. A morphism $f: d \to d'$ is taken to the functor 
	$$\int_{\Und_{d'}\C}\Hom_{\Und_{d'}\P(\cD)} (D[0],\Und_{d'}F) \to \int_{\Und_d\C}\Hom_{\Und_d\P(\cD)} (D[0],\Und_dF) $$
	which takes an object $(c \in \Obj_\C,x\in F(c)[d'])$ in $\Und_{d'}\int_\C \Hom_{\Und_{d'}\P(\cD)} (D[0],\Und_{d'}F)$ to $(c\in\Obj_\C,F(c)[f](x) \in F(c)[d])$ and is defined analogously on morphisms.
\end{defone} 

Notice the construction of $\Und_d\int_{\cD/\C} F$ gives us a functor 
$$\Fun(\C,\P(\cD)) \to \opGroth_{\cD/\C}.$$
Indeed, by \cref{cor:iso functor und}, we have an isomorphism $\Und: \Fun(\C,\P(\cD)) \to \Fun(\Und\C,\Und(\P(\cD)))$ and so it suffices to construct a functor $\Fun(\Und\C,\Und(\P(\cD))) \to (\opGroth_\cD)_{/\C}$. We have already defined it for objects (\cref{def:intCD}). Now for a morphism $\alpha: F \to G$ in $\Fun(\Und\C,\Und(\P(\cD)))$, we define $\ds\int_{\cD/\C} \alpha: \int_{\cD/\C} F \to \int_{\cD/\C} G$ level-wise for an object $d$ in $\cD$ as 
\begin{center}
	\begin{tikzcd}[column sep=0.6in]
		\ds\int_{\Und_d\C} \Hom_{\Und_d\P(\cD)} (D\leb0\reb,\Und_dF) \arrow[dr] \arrow[rr, "\int_{\Und_d \C} \Hom_{\Und_d\P(\cD)} (D\leb0\reb\comma\Und_d\alpha)"] & & \ds\int_{\Und_d\C} \Hom_{\Und_d\P(\cD)} (D\leb0\reb,\Und_dG) \arrow[dl] \\
		& \Und_d \C &  
	\end{tikzcd}
\end{center} 
Here we note that $\Und_d\alpha: \Und_dF \to \Und_dG$ is a natural transformation and so $\Hom_{\Und_d\P(\cD)} (D\leb0\reb\comma\Und_d\alpha)$ is the natural transformation defined as the composition 
\begin{center}
	\begin{tikzcd}[column sep=1.3in]
		\Und_d\C \arrow[r, bend left=25, ""{name=U, below}, "\Und_dF"] \arrow[r, bend right=25, ""{name=A, above}, "\Und_dG"'] & \Und_d\P(\cD)  \arrow[r, "\Hom_{\Und_d\P(\cD)}(D\leb0 \reb \comma -)"] & \set 
		\arrow[from=U, to=A, Rightarrow, "\Und_d\alpha" description]
	\end{tikzcd}.
\end{center}

Let us compute some important examples.

\begin{exone} \label{ex:intcd rep}
	Let $\C$ be a small $\P(\cD)$-enriched category and $c$ an object in $\C$. Then we can define the corepresentable functor $\uHom(c,-): \C \to \P(\cD)$, which takes an object $c'$ to the enriched hom object $\uHom(c,c')$. Using \cref{ex:int hom under} we can now compute $\int_{\cD/\C}\uHom(c,-) \to \C$ as follows. For a given object $d$ in $\cD$ we have
	$$\Und_d \int_{\cD/\C}\uHom_\C(c,-) = \int_{\Und_d\C} \Hom_{\Und_d\P(\cD)} (D[0],\Und_d\uHom_\C(c,-)) = \int_{\Und_d\C} \Hom_{\Und_d\C}(c,-) = (\Und_d\C)_{c/}.$$
\end{exone}

\begin{exone} \label{ex:intcd constant}
	Let $\C$ be a small $\P(\cD)$-enriched category and $G$ an object in $\P(\cD)$. Then the constant functor $\{G\}: \C \to \P(\cD)$ is an object in $\Fun(\C,\P(\cD))$ and we have $\int_{\cD/\C} \{G\} = \C \times \{G\} \xrightarrow{\pi_1} \C$.
\end{exone}

We can use this last example to get an enrichment.

\begin{lemone} \label{lemma:intcd enriched}
	The functor $\int_{\cD/\C}: \uFun(\C,\P(\cD)) \to \Fun(\cD^{op},\cat)_{/\Und\C}$ is a $\P(\cD)$-enriched functor.
	Here $\uFun(\C,\P(\cD))$ has the $\P(\cD)$-enrichment given in \cref{ex:psh} and $\Fun(\cD^{op},\cat)_{/\Und\C}$ the $\P(\cD)$-enrichment given in \cref{def:dcat}.
\end{lemone}

\begin{proof}
	As explained in \cref{ex:psh} and \cref{def:dcat} both enrichments are given via tensor, which is the level-wise product, which, as shown in \cref{ex:intcd constant}, commutes with the functor $\int_{\cD/\C}$. Hence, the results follows from \cref{lemma:pcd enriched functor}.
\end{proof}

Before we generalize our construction to spaces, we also need a right adjoint. We can construct one using the previous example, which motivates the functorial over-category construction.

\begin{defone} \label{def:undc}
 Let $\C$ be a small $\P(\cD)$-enriched category. Define the $\P(\cD)$-enriched functor $\Und\C_{/-}: \C \to \Fun(\cD^{op},\cat)_{/\Und\C}$ as the composition of the enriched Yoneda embedding (\cref{eq:enriched yon}) $\C \to \uFun(\C^{op},\P(\cD))$ with $\int_{\cD/\C}: \Fun(\C^{op},\P(\cD)) \to \Fun(\cD^{op},\cat)_{/\Und\C}$. We similarly define $\Und\C_{-/}:\C^{op} \to \Fun(\cD^{op},\cat)_{/\Und\C}$.
\end{defone} 

We can use this functor to finally define the right adjoint. 

\begin{defone} \label{def:hdc}
	Define the $\P(\cD)$-enriched functor $\H_{\cD/\C}:\Fun(\cD^{op},\cat)_{/\Und\C} \to \Fun(\C,\P(\cD))$ as the adjoint (using \cref{the:enriched prod hom adj}) to the following composition map
	\begin{center}
		\begin{tikzcd}
			\C \times  \Fun(\cD^{op},\cat)_{/\Und\C} \arrow[dr, "(\Und\C_{-/})^{op}\times \id"'] \arrow[rr] & &  \P(\cD)\\
			& (\Fun(\cD^{op},\cat)_{/\Und\C})^{op} \times \Fun(\cD^{op},\cat)_{/\Und\C} \arrow[ur, "\uHom_{/\Fun(\cD^{op}\comma\cat)_{/\Und\C}}(-\comma-)"']
		\end{tikzcd},
	\end{center}
	where the first map is the opposite of $\Und\C_{-/}$ (\cref{def:undc}) and the second is the enriched mapping object (\cref{rem:enriched mapping}).  
	In other words $\H_{\cD/\C}(p)(c) = \uHom_{/\C}(\C_{c/},\D)$. 
\end{defone}

Notice the similarity of $\H_{\cD/\C}$ with $\H_\C$ (\cref{prop:integral adjunctions}), which justifies the notation we have chosen. We now want to prove that $\int_{\cD/\C} \vdash \H_{\cD/\C}$ gives us an adjunction, similar to $\int_{\C}$ and $\H_\C$.

  \begin{theone} \label{the:groth cd ff}
 	Let $\C$ be a small $\P(\cD)$-enriched category. Then we have a $\P(\cD)$-adjunction 
 	\begin{center}
 		\begin{tikzcd}[row sep=0.5in, column sep=0.9in]
 			\uFun(\C,\P(\cD)) \arrow[r, "\int_{\cD/\C}", shift left=1.8, "\bot"'] & \Fun(\cD^{op},\cat)_{/\Und\C} \arrow[l, shift left = 1.8, "\H_{\cD/\C}"]
 		\end{tikzcd}
 	\end{center}
 	where the left adjoint $\int_{\cD/\C}$ preserves (set-enriched) limits and is fully faithful with essential image $\opGroth_{\cD/\C}$. 
 \end{theone}

 \begin{proof}
 	First, observe that $\int_{\cD/\C}$ evidently preserves (set-enriched) limits and colimits in $\Fun(\cD^{op},\cat)_{/\C}$, as they are evaluated point-wise, where $\int_{\cD/\C}$ simply coincides with the Grothendieck construction for categories, which we already know preserves limits and colimits (\cref{prop:integral adjunctions}). Moreover, by \cref{lemma:intcd enriched}, $\int_{\cD/\C}$ is $\P(\cD)$-enriched and in particular preserves tensor and so, by \cref{the:enriched colimit}, preserves $\P(\cD)$-enriched colimits. Knowing that $\int_{\cD/\C}$ preserves colimits means that in order to get an unenriched adjunction it suffices to prove that for every object $c$ and object $p: \D \to \C$ in $\Fun(\cD^{op},\cat)_{/\C}$ there is a natural isomorphism 
 	$$ \uHom(\uHom(c,-),\H_{\cD/\C}(p)) \cong \uHom_{/\C}(\int_{\cD/\C}\uHom(c,-),p) = \uHom_{/\C}(\C_{c/},p)= \H_{\cD/\C}(p)(c),$$
 	which follows from the Yoneda lemma (\cref{lemma:enriched yoneda}). Notice, by \cref{lemma:enriched adjunction via tensors} and the fact the left adjoint commutes with the tensor (\cref{ex:intcd constant}), the adjunction is in fact a $\P(\cD)$-enriched adjunction. Finally as the essential image of $\int_{\Und_d\C}$ are precisely the Grothendieck opfibrations (\cref{lemma:discrete Groth fib lifting} ), the essential image of $\int_{\cD/\C}$ are precisely the $\cD$-Grothendieck opfibrations (\cref{def:cd groth opfib}).
 \end{proof}
  
 \begin{remone}
 	Notice, we could have deduced the existence of an un-enriched right adjoint of $\int_{\cD/\C}$ from the fact that both categories are locally presentable, however, the benefit of this approach is an explicit description of the right adjoint.
 \end{remone}

We will end this subsection with an analysis of the {\it universal $\cD$-Grothendieck opfibration}, generalizing \cref{prop:grothendieck fiber}. First of all, following \cref{prop:universal fib}, the universal $\cD$-Grothendieck opfibration is given via the projection map $\pi_*: \P(\cD)_{D[0]/} \to \Und\P(\cD)$ and we again have the following result.

\begin{propone} \label{prop:universal fib cd}
	Let $\C$ be a small $\P(\cD)$-enriched category. The functor
	$$\Und(-)^*\pi_*:\uFun(\C,\P(\cD)) \to \opGroth_{\cD/\C}$$
	given by pulling back $\pi_*$ is precisely the functor $\int_{\cD/\C}$ and so gives us an equivalence between $\P(\cD)$-enriched functors and $\cD$-Grothendieck opfibrations over $\C$.
\end{propone} 

 Notice, again $\cD$-Grothendieck fibrations are stable under pullback (\cref{lemma:pullback nlfib}) and applying \ref{eq:pseudofunctor} level-wise we obtain the pseudo-functor 
$$
\opGroth_{\cD/-}: \cat_{\P(\cD)} \to  \Fun(\cD^{op},\widehat{\cat}),
$$
and again applying \cite[Theorem B1.3.6]{johnstone2002elephanti} leads to the following definition analogous to \cref{def:opgroth}.

\begin{defone}
	Let $\opGroth_\cD$ be the category with objects $\cD$-Grothendieck opfibrations and morphisms pullback squares and notice it comes with a target projection functor to $\cat_{\P(\cD)}$. 
\end{defone}

We now have the following result generalizing \cref{prop:grothendieck fiber}

\begin{propone} \label{prop:grothendieck fiber cd}
	The pullback functor 
	$$\Und(-)^*\pi_*: (\cat_{\P(\cD)})_{/\P(\cD)} \to \opGroth_\cD$$
	is an equivalence of $\P(\cD)$-enriched functor over $\cat_{\P(\cD)}$. 
\end{propone}

\begin{proof} 
	First of all for a $\cD$-set $X$ and $\cD$-Grothendieck opfibration $\D \to \C$, the composition $X \times \D \to \C$ is again a $\cD$-Grothendieck opfibration. Hence the tensor on $\Fun(\cD^{op},\Fun([1],\cat))$ (as described in \cref{def:dcat}) restricts to a tensor on $\opGroth_\cD$, making it into a $\P(\cD)$-enriched category. Moreover, pulling back along $\pi_*$ respects products with $\cD$-sets and so the functor $\Und(-)^*\pi_*$ is $\P(\cD)$-enriched.
	
	Using the same argument as in \cref{prop:grothendieck fiber} the equivalence follows from \cref{the:groth cd ff} and \cref{prop:universal fib cd}.
\end{proof}

\begin{remone} \label{rem:universal opfib}
 \cref{prop:grothendieck fiber cd} implies that for $\cD$-Grothendieck opfibration $\D \to \Und\C$, there is a functor of $\P(\cD)$-enriched categories $F:\C \to \P(\cD)$ and a pullback square of $\cD$-categories of the form
 \begin{center}
 	\pbsq{\D}{\P(\cD)_{D[0]/}}{\Und\C}{\Und\P(\cD)}{}{}{\pi_*}{\Und F}
 \end{center}
unique up to equivalence.
\end{remone} 
 
\subsection{Grothendieck Construction for \texorpdfstring{$\sP(\cD)$}{sP(D)}-Enriched Categories}\label{subsec:enriched groth}
 We are now in a position to generalize the construction from $\P(\cD)$-enriched categories to $\sP(\cD)$-enriched categories, which is the goal of this subsection. Notice, there is an isomorphism of categories $\sP(\cD) \cong \P(\cD\times\DD)$ and so we could simply use that to generalize all our previous constructions from $\P(\cD)$-enriched categories to $\sP(\cD)$-enriched categories. However, the additional simplicial direction should help us obtain homotopical properties and in particular construct model structures and Quillen equivalences, which requires us to adjust certain constructions.
 
 Following \cref{lemma:undcat} define the functor 
 	$$\Und: \cat_{\sP(\cD)} \to \Fun(\cD^{op} \times \DD^{op},\cat).$$
  For the next definition we use the fact that we have an (unenriched) isomorphism of categories $\sP(\cD) \cong \P(\cD\times\DD)$.
 
 \begin{defone}
 	Let $\C$ be a small $\sP(\cD)$-enriched category.
 	Define $\sint_{\cD/\C}:\Fun(\C,\sP(\cD)) \to \sP(\cD\times\DD)_{/N_\cD\C}$ as the composition of $\int_{\cD\times\DD/\C}$ and the nerve functor.
 \end{defone}
 
  \begin{remone}  \label{rem:sint computation}
 	By direct computation and \cref{def:intCD}, the $\cD$-simplicial space $\sint_{\cD/\C} G$ is level-wise equal to 
 	$$(\sint_{\cD/\C} G)[d,k]_l = \coprod_{c_0 \to \cdots \to c_k \in \Und_{d,l}\C} \Hom_{\Und_{d,l}\sP(\cD)}(D[0],G(c_0))$$
 	with the map $(\sint_{\cD/\C} G)[d,k]_l \to \Und\C[d,k]_l$ being the evident projection to $c_0 \to ... \to c_k$.
 	
 	Alternatively, we can only fix $k$ and observe that we have an equality of $\cD$-spaces 
 	$$(\sint_{\cD/\C} G)[-,k] = \coprod_{c_0,...,c_k \in \Obj_\C} G(c_0) \times \uMap(c_0,c_1) \times ... \times \uMap(c_{k-1},c_k)$$
 \end{remone}

 The computation has the following valuable implication regarding $\sint_{\cD/\C}$.
 
 \begin{lemone}\label{lemma:sint cdleft}
 	The map $\sint_{\cD/\C} F \to N_\cD\C$ is a $\cD$-left morphism (\cref{def:cdleft map}) and so in particular the injective fibrant replacement is a $\cD$-left fibration (\cref{lemma:cdleft map}).
 \end{lemone} 
  
 Let us do a sample computation similar to \cref{ex:intcd rep} 
 
\begin{exone}\label{ex:sintcd rep}
	Let $\C$ be a small $\sP(\cD)$-enriched category and $c$ an object in $\C$. Then we can define the corepresentable functor $\uHom_\C(c,-): \C \to \sP(\cD)$, which takes an object $c'$ to the enriched hom object $\uHom(c,c')$. Using an argument similar to \cref{ex:intcd rep} we have 
	$$ \sint_{\cD/\C}\uHom_\C(c,-) = N_\cD\C_{c/}.$$
\end{exone}

\begin{exone} \label{ex:sintcd constant}
	Let $\C$ be a small $\sP(\cD)$-enriched category and $G$ an object in $\sP(\cD)$. Similar to \cref{ex:intcd constant} we have $\sint_{\cD/\C} \{G\} = N_\cD\C \times \{G\} \xrightarrow{\pi_1} \C$.
\end{exone}

 We also have the following useful observation.
 
 \begin{lemone} \label{lemma:sint colimits}
 	Let $\C$ be a $\sP(\cD)$-enriched category. 
 	The functor $\sint_{\cD/\C}:\uFun(\C,\sP(\cD)) \to \sP(\cD\times\DD)_{/N_\cD\C}$ is $\sP(\cD)$-enriched and commutes with $\sP(\cD)$-colimits and set-enriched limits. 
 \end{lemone}
 
 \begin{proof} 
    By \cref{lemma:pcd enriched functor}, to observe that $\sint_{\cD/\C}$ is $\sP(\cD)$-enriched it suffices to observe that for every $F$ in $\sP(\cD)$ we have $\sint_{\cD/\C} \{F\} \times G = F \times \sint_{\cD/\C} G$, which simply follows from \cref{ex:sintcd constant}. As we have shown preservation of tensors, by \cref{the:enriched colimit}, in order to show $\sint_{\cD/\C}$ is $\sP(\cD)$-colimit preserving it suffices to prove that the underlying functor preserves colimits.  

 	Now, recall that (co)limits in $\sP(\cD\times\DD)_{/N_\cD\C}$ are evaluated level-wise. Fix $(d,k,l)$ in $\cD \times \DD \times \DD$. Then the functor 
 	$$\sint_{\cD/\C}(-)[d,k]_l: \Fun(\C,\sP(\cD)) \to \set$$ 
 	takes a functor $G$ to the set $\coprod_{c_0 \to \cdots \to c_k \in \Und_{d,l}\C} \Hom_{\Und_{d,l}\sP(\cD)}(D[0],G(c_0))$ (\cref{rem:sint computation}), which is a coproduct of a  composition of an evaluation functor (at $c_0$) and a corepresentable functor, all of which are colimit preserving.
 \end{proof}

 We want to show that $\sint_{\cD/\C}$ is a left Quillen functor, which requires a right adjoint.
 
 \begin{notone} \label{not:undcs}
 	In \cref{def:undc} we introduced the $\cD$-category $\Und\C_{/-}$ and similarly $\Und\C_{-/}$ for a given $\P(\cD)$-enriched category $\C$. Using the analogous construction for $\sP(\cD)$-enriched category $\C$ and applying the nerve gives us $\cD$-simplicial spaces $N\Und\C_{-/}$ ($N\Und\C_{/-}$). Recall, in \cref{def:nervesnset} we defined the $\cD$-nerve $N_\cD: \cat_{\sP(\cD)} \to \sP(\cD\times\DD)$, which in particular satisfies $N_\cD=N\Und$ (similar to what we observed in \cref{lemma:undcat}). Hence, we will simply denote the $\cD$-simplicial space $N\Und\C_{/-}$ by $N_\cD\C_{/-}$ and $N\Und\C_{-/}$ by $N_\cD\C_{-/}$.
 \end{notone}

\begin{remone}\label{rem:nDCdash} 
 Notice, similar to \cref{def:undc}, we get $\sP(\cD)$-enriched functors $N_\cD\C_{-/}: \C^{op} \to \sP(\cD\times\DD)_{/N_\cD\C}$ and  $N_\cD\C_{/-}: \C \to \sP(\cD\times\DD)_{/N_\cD\C}$
\end{remone}

We can now define $\sH_{\cD/\C}$ analogous to $\H_{\cD/\C}$ (\cref{def:hdc}).
 
 \begin{defone}
 	Define the $\sP(\cD)$-enriched functor $\sH_{\cD/\C}: \sP(\cD\times\DD)_{/N_\cD\C} \to \Fun(\C,\sP(\cD))$ as the adjoint (using \cref{the:enriched prod hom adj}) to the functor 
 	\begin{center}
 		\begin{tikzcd}
 			\C \times \sP(\cD\times\DD)_{/N_\cD\C} \arrow[dr, "(N_\cD\C_{-/})^{op} \times \id "'] \arrow[rr] & &  \sP(\cD)\\
 			& (\sP(\cD\times\DD)_{/N_\cD\C})^{op} \times  \sP(\cD\times\DD)_{/N_\cD\C} \arrow[ur, "\uMap_{/\sP(\cD\times\DD)_{/N_\cD\C}}(-\comma-)"']
		\end{tikzcd},
 	\end{center}
 where the first map is the opposite of $\Und\C_{-/}$ (\cref{rem:nDCdash}) and the second is the enriched mapping object (\cref{rem:enriched mapping}).
 	Concretely, for a given map $Y\to N_\cD\C$ we have 
 	\begin{equation} \label{eq:sHDc value}
 		\sH_{\cD/\C}(Y \to N_\cD\C)(c)=\uMap_{/N_\cD\C}(N_\cD\C_{c/},Y).
 	\end{equation}
 \end{defone}

We are finally in a position to prove that the functors $(\sint_{\cD/\C},\sH_{\cD/\C})$ give us a $\sP(\cD)$-enriched adjunction between $\Fun(\C,\sP(\cD))$ and $\sP(\cD\times\DD)_{/N_\cD\C}$.

\begin{lemone} \label{lemma:sintsH adj}
	Let $\C$ be a small $\sP(\cD)$-enriched category. There is a $\sP(\cD)$-enriched adjunction
	\begin{center}
		\begin{tikzcd}[row sep=0.5in, column sep=0.9in]
			\uFun(\C,\sP(\cD)) \arrow[r, shift left = 1.8, "\sint_{\cD/\C}"] & 
			\sP(\cD\times\DD)_{/N\C} \arrow[l, shift left=1.8, "\sH_{\cD/\C}", "\bot"']
		\end{tikzcd}
	\end{center}
\end{lemone}

\begin{proof}
	Notice we proved in \cref{lemma:sint colimits} that $\sint_{\cD/\C}$ is $\sP(\cD)$-enriched colimit preserving functor. Hence, the result follows from the fact that $\uFun(\C,\sP(\cD))$ is the free $\sP(\cD)$-enriched cocompletion (\cref{the:enriched Kan ext}), meaning $\sint_{\cD/\C}$ is the left Kan extension along the enriched Yoneda embedding and the right adjoint must take a map $Y \to N_\cD\C$ to a functor with value $\uMap_{/N_\cD\C}(N_\cD\C_{c/},Y)$, which by \ref{eq:sHDc value} is precisely $\sH_{\cD/\C}$. 
\end{proof}

We now move on to the next step: Constructing a second adjunction between $\Fun(\C,\sP(\cD))$ and $\sP(\cD\times\DD)_{/N_\cD\C}$, where the left (and right) adjoint goes the other way. 

\begin{defone}\label{def:sbTcd}
	Define the $\sP(\cD)$-enriched functor $\sbT_{\cD/\C}: \sP(\cD\times\DD)_{/N_\cD\C} \to \Fun(\C,\sP(\cD))$ as the adjoint (using \cref{the:enriched prod hom adj}) to the functor 
	$$ \sP(\cD\times\DD)_{/N_\cD\C} \times \C \xrightarrow{ \id \times N_\cD\C_{/-}}  \sP(\cD\times\DD)_{/N_\cD\C} \times  \sP(\cD\times\DD)_{/N_\cD\C} \xrightarrow{ - \times_{N_\cD\C} - } \sP(\cD\times\DD) \xrightarrow{\fDiag} \sP(\cD).$$
	Here the first functor is $\sP(\cD)$-enriched by \cref{rem:nDCdash}, the second by \cref{lemma:prod enriched} and the last because $\fDiag$ preserves tensors and \cref{lemma:pcd enriched functor}. In other words $\sbT_{\cD/\C}(p:Y \to N_\cD\C)(c) = \fDiag(Y \times_{N_\cD\C} N_\cD\C_{/-})$.
\end{defone}
\begin{remone}
	The definition of $\sbT_{\cD/\C}$ should be seen as the correct homotopical generalization of $\T_\C$ (\cref{prop:integral adjunctions}), as the direct generalization of $\T_\C$ to $\sP(\cD)$-enriched categories would not have the desired homotopical properties. 
\end{remone}
The description suggests an alternative characterization. Notice, for a given $\P(\cD)$-enriched category we have the following commutative diagram
\begin{center}
	\begin{tikzcd}
		\sP(\cD\times\DD)_{/N_\cD\C} \times  \sP(\cD\times\DD)_{/N_\cD\C} \arrow[d, " - \times_{N_\cD\C} - "] \arrow[r, "\fDiag \times \fDiag"] & \sP(\cD)_{/N_\cD\C} \times  \sP(\cD)_{/N_\cD\C}\arrow[d, " - \times_{N_\cD\C} - "] \\
		\sP(\cD\times\DD) \arrow[r, "\fDiag"] & \sP(\cD)
	\end{tikzcd},
\end{center}
which gives us the following alternative characterization of $\sbT_{\cD/\C}$.

\begin{lemone}
	The functor $\sP(\cD\times\DD)_{/N_\cD\C} \to \Fun(\C,\sP(\cD))$ defined as the adjoint 
	$$ \sP(\cD\times\DD)_{/N_\cD\C} \times \C \xrightarrow{ \id \times N_\cD\C_{/-}}  \sP(\cD\times\DD)_{/N_\cD\C} \times  \sP(\cD\times\DD)_{/N_\cD\C} \xrightarrow{ \fDiag \times \fDiag } $$
	$$\sP(\cD)_{/N_\cD\C} \times  \sP(\cD)_{/N_\cD\C}\xrightarrow{- \times_{N_\cD\C} -} \sP(\cD).$$
	is equal to $\sbT_{\cD/\C}$ and so in particular we have $\sbT_{\cD/\C}(p:Y \to N_\cD\C)(c) = \fDiag(Y) \times_{\fDiag(N_\cD\C)} \fDiag(N_\cD\C_{/-})$.
\end{lemone}

Notice $\sbT_{\cD/\C}$ behaves as expected.

\begin{exone}\label{ex:sbT constant}
	Let $A$ be a $\cD$-space and  $\VEmb(A) \to D[0] \to N_\cD\C$ a map of $\cD$-simplicial spaces. Then $\sbT_{\cD/\C}(\VEmb(A)) = \{A\} \times \uMap(c,-)$, where $c$ is the image of the point in $N_\cD\C$. Indeed, for a given object $d$ we have 
	$$\sbT_{\cD/\C}(\VEmb(A) \to D[0] \to N_\cD\C)(d) = \fDiag(\VEmb(A)) \times_{\fDiag(N_\cD\C)} \fDiag(N_\cD\C_{/d}) = $$
	$$A \times (D[0] \times_{\fDiag(N_\cD\C)} \fDiag(N_\cD\C_{/d})) = A \times \uMap(c,d)$$
\end{exone}

Let us look at another example.

\begin{exone} \label{ex:sbT of rep}
 Fix a map $p:F[d,n] \times \Delta[l] \to N_\cD\C$ and notice this map corresponds to a chain of $n+1$ objects $c_0 \to c_1 \to  ... \to  c_n$ in the category $\Und_d\C_l$. We want to understand the functor $\sbT_{\cD/\C}(F[d,n] \times \Delta[l] \to N_\cD\C): \C \to \sP(\cD)$. Fix an object $d'$ in $\cD$ and notice we have the equality of simplicial spaces $\Und_{d'}F[d,n] = \coprod_{\Hom_\cD(d',d)} F[n]$ and so in particular the map $\Und_{d'}F[d,n] \times \Delta[l] \to \Und_{d'}N_\cD\C$ corresponds to a map of simplicial spaces $\coprod_{\Hom_\cD(d',d)} F[n] \times \Delta[l] \to \Und_{d'}N_\cD\C$. We can now compute
 $$\Und_{d'}\sbT_{\cD/\C}(p)= \Und_{d'}\fDiag((F[d,n] \times \Delta[l]) \times_{N_\cD\C}N_\cD\C_{/-})= $$
 $$\fDiag(\coprod_{\Hom_\cD(d',d)} F[n] \times \Delta[l] \times_{N\Und_{d'}\C}N\Und_{d'}\C_{/-}) = \coprod_{\Hom_\cD(d',d)} \fDiag(F[n] \times \Delta[l] \times_{\Und_{d'}N\C} \Und_{d'}N\C_{/-})$$
 Notice again the map $F[n] \times \Delta[l] \to \Und_dN_\cD\C$ corresponds to a chain of morphisms $c_0 \to ... \to c_{n+1}$ in the category $\Und_d\C_l$. The map $<0>:F[0] \to F[n] \times \Delta[l]$ is a covariant equivalence over $N\C$ by definition and so, by \cref{lemma:very technical lemma}, for every object $c$ the map of simplicial spaces $F[0]\times_{\Und_{d'}N\C} \Und_{d'}N\C_{/c} \to F[n] \times \Delta[l] \times_{\Und_{d'}N\C} \Und_{d'}N\C_{/c}$ is a covariant equivalence and so, by \cite[Theorem 3.17]{rasekh2017left}, gives us the diagonal equivalence
 $$\coprod_{\Hom_\cD(d',d)} \Map_{\Und_{d'}\C}(c_0,-) \xrightarrow{ \ \simeq \ } \coprod_{\Hom_\cD(d',d)} \fDiag(F(n) \times \Delta[l] \times_{\Und_{d'}N\C} \Und_{d'}N\C_{/-}).$$
  Using the same argument as in \cref{rem:ps}, this equivalence has an inverse that takes chains of morphisms $c_0 \to c_1 \to ... \to c$ to the composition $\Map_{\Und_{d'}\C}(c_0,c)$ giving us a diagonal equivalence
 $$\coprod_{\Hom_\cD(d',d)} \fDiag(F(n) \times \Delta[l] \times_{\Und_{d'}N\C} \Und_{d'}N\C_{/-}) \overset{\simeq}{\twoheadrightarrow} \coprod_{\Hom_\cD(d',d)} \Map_{\Und_{d'}\C}(c_0,-).$$
 However, the space $\Map_{\Und_{d'}\C}(c_0,-)$ inside the coproduct does not depend on the choice of morphism in $\Hom_\cD(d',d)$. Hence we can separate them and get
 $$\coprod_{\Hom_\cD(d',d)} \Map_{\Und_{d'}\C}(c_0,-) = \Hom_{\cD}(d',d) \times \Map_{\Und_{d'}\C}(c_0,-) =F[d] \times \uMap_\C(c_0,-)[d],$$
 which proves that we have a level-wise equivalences of $\cD$-space valued functors 
 $$F[d] \times \uMap_\C(c_0,-) \xrightarrow{\simeq} \sbT_{\cD/\C}(p) \xrightarrow{\simeq} F[d] \times \uMap_\C(c_0,-)$$
 composing to the identity.
\end{exone}

 The goal is to show that $\sbT_{\cD/\C}$ is part of an adjunction.
 
  \begin{lemone}\label{lemma:sbTsbI adj}
 	Let $\C$ be a small $\sP(\cD)$-enriched category. We have a $\sP(\cD)$-enriched adjunction
 	\begin{center}
 		\begin{tikzcd}[row sep=0.5in, column sep=0.9in]
 			\sP(\cD\times\DD)_{/N\C}  \arrow[r, shift left=1.8, "\sbT_{\cD/\C}"] &
 			\uFun(\C,\sP(\cD)) \arrow[l, shift left=1.8, "\sbI_{\cD/\C}", "\bot"']
 		\end{tikzcd}
 	\end{center}
 \end{lemone}
 
 \begin{proof}
 	We first show it is an unenriched adjunction and then prove the adjunction is in fact enriched. The functor $\sbT_{\cD/\C}: \sP(\cD\times\DD)_{/N_\cD\C} \to \Fun(\C,\sP(\cD))$ is colimit preserving. Indeed, colimits in $\Fun(\C,\sP(\cD))$ are evaluated level-wise and for a given object $c$ the functor $\sbT_{\cD/\C}(-)(c): \sP(\cD\times\DD)_{/N_\cD\C} \to \sP(\cD)$ is given via the composition of the pullback functor $- \times_{N_\cD\C} N_\cD\C_{/c}$ and $\fDiag$, both of which are colimit preserving.
 	
 	This implies that $\sbT_{\cD/\C}$ is given as the left Kan extension of the functor $(\cD \times \DD \times \DD)_{/N\C} \to \sP(\cD\times\DD)_{/N\C}$, which takes a map $F[d,k] \times \Delta[l] \to N\C$ to $\sbT_{\cD/\C}(F[d,k] \times \Delta[l] \to N\C)$. It then follows formally that the left Kan extension has a right adjoint $\sbI_{\cD/\C}: \Fun(\C,\sP(\cD)) \to \sP(\cD\times\DD)_{/N_\cD\C}$.
 	Now we observe that this adjunction is in fact $\sP(\cD)$-enriched. However, this follows from the fact that $\sbT_{\cD/\C}$ commutes with tensors (\cref{ex:sbT constant}) and \cref{lemma:enriched adjunction via tensors}.
 \end{proof} 

  Let us better understand the right adjoint $\sbI_{\cD/\C}$.
  Let $G: \C \to \sP(\cD)$ be a $\sP(\cD)$-enriched functor such that the functor 
  $$\Map_\C(-,G): \Fun(\C,\sP(\cD)) \to \s$$
  preserves level-wise equivalence (which holds in particular if $G$ is projectively fibrant, see \cref{prop:projinj model}).
  and fix a map $p:F[d,n] \times \Delta[l] \to \C$ that corresponds to the chain of morphisms $c_0 \to ... \to c_n$ in $\Und_d\C_l$. Then the space of fibers of $\sbI_{\cD/\C}G$ over $p$ is equivalent to the space $G(c_0)[d]$. Indeed, we have 
  \begin{align*}
  	\Map_{/N_\cD\C}(F[d,n]\times\Delta[l],\sbI_{\cD/\C}G) & \cong \Map_{\Fun(\C,\sP(\cD))}(\sbT_{\cD/\C}(p),G) & \text{\cref{lemma:sbTsbI adj}}\\
  	& \simeq \Map_{\Fun(\C,\sP(\cD))}(F[d] \times \uMap_\C(c_0,-),G) & \text{\cref{ex:sbT of rep}}\\ 
  	& \cong \Map_{\sP(\cD)}(F[d],G(c_0))  & \text{\cref{lemma:enriched yoneda}}\\
  	& \cong G(c_0)[d] 
  \end{align*}

 Having constructed two adjunctions $(\sint_{\cD/\C},\sH_{\cD/\C})$ and $(\sbT_{\cD/\C},\sbI_{\cD/\C})$ we move on to show they are Quillen adjunctions. This requires us understanding the relevant model structures first!
 
 \begin{propone} \label{prop:projinj model}
 	Let $\C$ be a small $\sP(\cD)$-enriched category. Then the category of $\sP(\cD)$-enriched functors $\Fun(\C,\sP(\cD))$ (\cref{def:ufuncd}) has two combinatorial proper $\sP(\cD)$-enriched model structures, the {\it projective model structure} and the {\it injective model structure} given as follows.
 	\begin{itemize}
 		\item A natural transformation $\alpha: F \to G$ in both model structures is a weak equivalence if for all objects $c$ the map of $\cD$-spaces $\alpha_c: F(c) \to G(c)$ is an injective equivalence. 
 		\item A natural transformation $\alpha: F \to G$ is a projective fibration if for all objects $c$ the map of $\cD$-spaces $\alpha_c: F(c) \to G(c)$ is an injective fibration. 
 		\item A natural transformation $\alpha:F \to G$ is an injective cofibration if for all objects $c$ the map of $\cD$-spaces $\alpha_c: F(c) \to G(c)$ is an injective cofibration. 
 	\end{itemize}
 Moreover, the generating (trivial) cofibrations of the projective model structure are of the form $\id \times \{\alpha\}:\uHom(c,-) \times \{F\} \to \uHom(c,-) \times \{G\}$, where $\alpha:F \to G$ is (trivial) cofibration in $\sP(\cD)$. Finally, the adjunction
 \begin{center}
 	\adjun{\uFun(\C,\sP(\cD))^{proj}}{\uFun(\C,\sP(\cD))^{inj}}{\id}{\id}
 \end{center}
 gives us a $\sP(\cD)$-enriched Quillen equivalence. Here the left hand side has the projective model structure and the right hand side the injective model structure.
 \end{propone}

 \begin{proof}
 	The fact that this model structure exists and is combinatorial is a direct application of \cite[Theorem 4.4]{moser2019enrichedproj} with $\mathcal{V} =  \mathcal{A} = \sP(\cD)$ with the injective model structure enriched over itself and $D= \C$. Indeed, it is evident that $- \times X:\sP(\cD) \to \sP(\cD)$ preserves (trivial) cofibrations. The fact that the model structure is proper follows from the fact that the injective model structure on $\sP(\cD)$ is proper and that each projective (injective) (co)fibration is a level-wise cofibration combined with the fact that pullbacks and pushouts are evaluated level-wise. The fact that it is enriched over $\sP(\cD)$ follow from the fact that the injective model structure on $\sP(\cD)$ is enriched over itself and \cite[Theorem 5.4]{moser2019enrichedproj}.
 	
 	Finally, every (trivial) cofibration in the projective model structure is evidently a (trivial) cofibration in the injective model structure and so $\id: \Fun(\C,\sP(\cD))^{proj} \to \Fun(\C,\sP(\cD))^{inj}$ is a Quillen left adjoint. Moreover, both sides have the same weak equivalences, hence the Quillen adjunction is a Quillen equivalence.
 \end{proof}

  Here are some comments regarding this model structure.
  
  \begin{lemone} \label{lemma:some tool}
  	A map $\alpha: F \to G$ is a projective (or injective) equivalence if and only if for all $d$ in $\cD$ and all objects $c$ the map of spaces $\alpha_c[d]: F(c)[d] \to G(c)[d]$ is a Kan equivalence. 
  \end{lemone}
  
  \begin{proof}
  	By definition $\alpha$ is a projective equivalence if and only if $\alpha_c: F(c) \to G(c)$ is an injective equivalence in $\sP(\cD)$ which, by \cref{prop:injectivemod}, is equivalent to $\alpha_c[d]: F(c)[d] \to G(c)[d]$ being a Kan equivalence. 
  \end{proof}

 \begin{corone}
  The projective and injective model structures on $\Fun(\C,\sP(\cD))$ are simplicial. 
 \end{corone}
  
  \begin{proof}
  	Notice the simplicial enrichment is given by applying the projection map $\sP(\cD) \to \s$, which preserves cofibrations and trivial cofibrations. The result now follows from \cref{lemma:some tool}.
  \end{proof}
 
 We can now study the model categorical aspects of the functor $\sint_{\cD/\C}$.
 
  \begin{lemone} \label{lemma:sint left quillen}
 	Let $\C$ be a small $\sP(\cD)$-enriched category. Then the functor $\sint_{\cD/\C}: \Fun(\C,\sP(\cD)) \to \sP(\cD\times\DD)_{/\C}$ takes projective (trivial) cofibrations to $\cD$-covariant (trivial) cofibrations. Moreover $\sint_{\cD/\C}$ reflects equivalences.
  \end{lemone}

\begin{proof}
	It suffices to prove that $\sint_{\cD/\C}$ takes generating (trivial) projective cofibrations to (trivial) cofibrations in the $\cD$-covariant model structure. By \cref{prop:projinj model}, the generating (trivial) cofibrations are given by $\uHom(c,-) \times A \to \uHom(c,-) \times B$ where $A \to B$ is a (trivial) cofibration of $\cD$-spaces. By \cref{ex:sintcd rep} and \cref{ex:sintcd constant}, $\sint_{\cD/\C}$ takes this map to $N\C_{c/} \times A \to N\C_{c/} \times B$, which is a (trivial) cofibration in the $\cD$-covariant model structure.

	Now, let $\alpha: F \to G$ be natural transformation such that $\sint_{\cD/\C}\alpha: \sint_{\cD/\C}F \to \sint_{\cD/\C}G$ is a $\cD$-covariant equivalence. We want to prove $\alpha$ is a level-wise equivalence. Notice for every object $d$ in $\cD$ and $n$ in $\DD$, we have a bijection $\sint_{\cD/\C}F[d,n] \cong \sint_{\cD/\C}F[d,0] \times_{N_\cD\C[d,0]}N_\cD\C[d,n]$ (\cref{rem:sint computation}), which means the injectively fibrant replacement is already a $\cD$-left fibration.  	 
	
	Taking the injectively fibrant replacement we get the following diagram where the second row is the fibrant replacement in the injective model structure
	\begin{center}
		\begin{tikzcd}
			\sint_{\cD/\C}F \arrow[r, "\sint_{\cD/\C}\alpha"] \arrow[d, "\simeq"] & \sint_{\cD/\C}G \arrow[d, "\simeq"] \\
			R\sint_{\cD/\C}F \arrow[r, "R\sint_{\cD/\C}\alpha"] & R\sint_{\cD/\C}G
		\end{tikzcd}.
	\end{center}
	By definition $\sint_{\cD/\C}\alpha$ is a covariant equivalence if and only if $R\sint_{\cD/\C}\alpha$ is an injective equivalence of $\cD$-simplicial spaces, which, by \cref{the:ncov equiv nlfib}, is equivalent to $\Val(R\sint_{\cD/\C}\alpha)$ being an injective equivalence of $\cD$-spaces. By \cref{rem:sint computation}
	$$\Val\sint_{\cD/\C}\alpha= \coprod_{c \in \Obj_\C}\alpha_c: \coprod_{\Obj_\C}F(c) \to \coprod_{\Obj_\C}G(c)$$
	and so  $\Val(R\sint_{\cD/\C}\alpha)$ is a injective equivalence of $\cD$-spaces over $\Obj_\C$ if and only if $\alpha_c: F(c) \to G(c)$ is an injective equivalence of $\cD$-spaces, which by definition means $\alpha$ is an equivalence in the projective model structure (\cref{prop:projinj model}).
\end{proof}

\begin{lemone} \label{lemma:sbt left quillen}
	Let $\C$ be a small $\sP(\cD)$-enriched category.
	Then the functor $\sbT_{\cD/\C}: \sP(\cD\times\DD)_{/N\C} \to \Fun(\C,\sP(\cD))$ takes (trivial) $\cD$-covariant cofibrations to (trivial) injective cofibrations.
\end{lemone}

 \begin{proof}
  Let $f:A \to B$ be a map over $N\C$. Then, by \cref{the:ncov model}, $f$ is a (trivial) cofibration if and only if for every $d$, the map of simplicial spaces $\Und_d f: \Und_d A \to \Und_d B$ is a (trivial) cofibration in the covariant model structure. Now, by \cref{def:sbTcd}, we have
  $$\Und_d\sbT_{\cD/\C}(p:Y \to N_\cD\C)(c) = \fDiag(Y[d] \times_{N\Und_d\C} N\Und_d\C_{/c}),$$
  which preserves (trivial) cofibrations in the covariant model structure as it is left Quillen (\cref{lemma:very technical lemma}). Finally, by \cref{prop:projinj model}, being a (trivial) cofibration in the injective model structure on $\Fun(\C,\sP(\cD))$ is determined level-wise.
 \end{proof}

 We now use these adjunctions to prove that we have the desired Quillen equivalences giving the appropriate generalization of \cite[Theorem 4.18]{rasekh2017left}. 
 
  \begin{theone} \label{the:simp grothendieck}
  	Let $\C$ be a small $\sP(\cD)$-enriched category.
  	The two $\sP(\cD)$-enriched adjunctions  
  	\begin{center}
  		\begin{tikzcd}[row sep=0.5in, column sep=0.9in]
  			\uFun(\C,\sP(\cD))^{proj} \arrow[r, shift left = 1.8, "\sint_{\cD/\C}"] & 
  			(\sP(\cD\times\DD)_{/N_\cD\C})^{\cD-cov} \arrow[l, shift left=1.8, "\sH_{\cD/\C}", "\bot"'] \arrow[r, shift left=1.8, "\sbT_{\cD/\C}"] &
  			\uFun(\C,\sP(\cD))^{inj} \arrow[l, shift left=1.8, "\sbI_{\cD/\C}", "\bot"']
  		\end{tikzcd}
  	\end{center}
  	are Quillen equivalences. Here the left hand $\Fun(\C,\sP(\cD))$ has the projective model structure, $\sP(\cD\times\DD)_{/N\C}$ has the $\cD$-covariant model structure over $N\C$ and the right hand $\Fun(\C,\sP(\cD))$ has the injective model structure. 
  \end{theone}
  
  \begin{proof}
  	First we observe that $\sint_{\cD/\C}$ is left Quillen by \cref{lemma:sint left quillen} and $\sbT_{\cD/\C}$ is left Quillen by \cref{lemma:sbt left quillen} showing both adjunctions are Quillen adjunctions. 
  	
  	Next we show that the composition functor $\sbT_{\cD/\C}\circ \sint_{\cD/\C}$ is naturally equivalent to the identity functor $\Fun(\C,\sP(\cD)) \to \Fun(\C,\sP(\cD))$. This will then prove that the composition is a Quillen equivalence. As both functors $\sbT_{\cD/\C}$, $\sint_{\cD/\C}$ are left adjoints they commute with colimits and it suffices to show the desired equivalence for functors of the form $\{A\} \times \uMap(c,-)$, where $A$ is an $\cD$-simplicial space and $c$ is an object in $\C$. 
  	
  	By \cref{ex:sintcd rep}, $\sint_{\cD/\C} \{A\} \times \uMap(c,-) = \{A\} \times N_\cD\C_{c/}$, which by \cref{the:levelwise Yoneda} is equivalent in $\cD$-covariant model structure to the map $A \to N_\cD\C$ with constant value $c$. By \cref{ex:sbT constant}, $\sbT_{\cD/\C}(A \to N_\cD\C) \simeq \{A\} \times \uMap(c,-)$ giving us the desired natural equivalence.
  	
  	 So, by $2$-out-of-$3$ for Quillen equivalence, it suffices to prove the first is a Quillen equivalence. First of all, by \cref{lemma:sint left quillen}, $\sint_{\cD/\C}$ reflects equivalences. Hence, it suffices to check the derived counit map is an equivalence, which means we have to prove that for a given $\cD$-left fibration $L \to N\C$ the map $\sint_{\cD/\C}\sH_{\cD/\C} L \to L$ is a $\cD$-covariant equivalence over $N_\cD\C$. 
  	
  	We will in fact observe that the map is an injective equivalence. For an arbitrary $(d,k)$ in $\cD \times \DD$ we have the diagram 
  	\begin{center}
  		\begin{tikzcd}
  		 \sint_{\cD/\C}\sH_{\cD/\C} L[d,k] \arrow[d, "\cong"] \arrow[r] & L[d,k] \arrow[d, "\simeq"] \\
  		 \sint_{\cD/\C}\sH_{\cD/\C} L[d,0] \times_{N_\cD\C[d,0]} N_\cD\C[d,k] \arrow[r] & L[d,0]  \times_{N_\cD\C[d,0]} N_\cD\C[d,k]
  		\end{tikzcd}
  	\end{center}
  	The left hand vertical map is a weak equivalence by assumption on $L$ and the right hand vertical map is a bijection by construction (\cref{rem:sint computation}). Hence the top map is an equivalence of spaces if and only if the bottom map is an equivalence. However, by \cref{def:nervesnset}, the space $N_\cD\C[d,0]$ is discrete and so it suffices to prove that $\sint_{\cD/\C}\sH_{\cD/\C} L[d,0] \to L[d,0]$ is an equivalence. 
  	Now, every map $F[d,0] \to N_\cD\C$ is necessarily constant with value an object of $\C$ and so it suffices to fix a map $F[d,0] \to N_\cD\C$ and prove we a natural equivalence 
  	$$\Map_{/N_\cD\C}(F[d,0],\sint_{\cD/\C}\sH_{\cD/\C}L) \simeq \Map_{N_\cD/\C}(F[d,0],L).$$
  	We have the following chain of isomorphisms and equivalences:
  	\begin{align*}
  		\Map_{/N_\cD\C}(F[d,0],\sint_{\cD/\C}\sH_{\cD/\C}L) & \cong \sH_{\cD/\C}L(c)[d] & \text{\cref{rem:sint computation}} \\
  		& \cong \uNat(\uHom(c,-) \times F[d],\sH_{\cD/\C}L) & \text{\cref{lemma:enriched yoneda}} \\
  		& \cong \uMap_{/N_\cD\C}(\sint_{\cD/\C}(\uHom(c,-) \times \{F[d]\}),L) & \text{\cref{lemma:sintsH adj}} \\
  		& \cong \uMap_{/N_\cD\C}(N_\cD\C_{c/} \times F[d,0],L) & \text{\cref{ex:sintcd rep}} \\
  		& \simeq \uMap_{/N_\cD\C}(F[d,0],L) & \text{\cref{the:levelwise Yoneda}}
  	\end{align*}
  	giving us the desired equivalence and finishing the proof.
  \end{proof}
 
 Notice the two functor categories in the statement of \cref{the:simp grothendieck} have two different (even if Quillen equivalent) model structures and we would like to have a statement which involves only one model structure.
 
 \begin{theone} \label{the:grothendieck double proj}
 	Let $\C$ be a small $\sP(\cD)$-enriched category. Assume that $\sbI_{\cD/\C}$ takes projective fibrations to injective fibration. Then, the two $\sP(\cD)$-enriched adjunctions  
 	\begin{center}
 		\begin{tikzcd}[row sep=0.5in, column sep=0.9in]
 			\uFun(\C,\sP(\cD))^{proj} \arrow[r, shift left = 1.8, "\sint_{\cD/\C}"] & 
 			(\sP(\cD\times\DD)_{/N_\cD\C})^{\cD-cov} \arrow[l, shift left=1.8, "\sH_{\cD/\C}", "\bot"'] \arrow[r, shift left=1.8, "\sbT_{\cD/\C}"] &
 			\uFun(\C,\sP(\cD))^{proj} \arrow[l, shift left=1.8, "\sbI_{\cD/\C}", "\bot"']
 		\end{tikzcd}
 	\end{center}
 	are Quillen equivalences. Here the two $\Fun(\C,\sP(\cD))$ have the projective model structure and $\sP(\cD\times\DD)_{/N_\cD\C}$ has the $\cD$-covariant model structure. 
 \end{theone} 

 \begin{proof}
 	We have the following diagram of $\sP(\cD)$-enriched adjunctions
 	\begin{center}
 		\begin{tikzcd}[column sep=0.9in]
 			(\sP(\cD\times\DD)_{/N_\cD\C})^{\cD-cov}\arrow[r, shift left=1.8, "\sbT_{\cD/\C}"] &
 			\uFun(\C,\sP(\cD))^{proj} \arrow[l, shift left=1.8, "\sbI_{\cD/\C}", "\bot"'] \arrow[r, shift left=1.8, "\id"]&
 			\uFun(\C,\sP(\cD))^{inj} \arrow[l, shift left=1.8, "\id", "\bot"']
 		\end{tikzcd}
 	\end{center}
 	where the middle $\Fun(\C,\sP(\cD))$ has the projective model structure and the right hand $\Fun(\C,\sP(\cD))$ has the injective model structure. 
 	The composition of the two adjunctions is a Quillen equivalence by \cref{prop:projinj model} and the right hand is a Quillen equivalence by \cref{the:simp grothendieck}. Hence, by $2$-out-of-$3$ it suffices to prove the left hand adjunction is a Quillen adjunction. 
 	
 	First we show $\sbI_{\cD/\C}$ takes projective fibrations to $\cD$-left fibrations. By assumption $\sbI_{\cD/\C}$ takes projective fibrations to injective fibrations of $\cD$-simplicial spaces, so it suffices to show the map satisfies the locality condition of $\cD$-left fibration. By adjunction this is equivalent to $\sbT_{\cD/\C}(F[d,0] \to F[d,n])$ being a projective equivalence in $\uFun(\C,\sP(\cD))^{proj}$. By \cref{prop:projinj model} projective and injective equivalences coincide and so this follows from \cref{lemma:sbt left quillen}.
 	
 	We have hence proven that $\sbI_{\cD/\C}$ takes projective fibrations between projectively fibrant functors to $\cD$-left fibrations between $\cD$-left fibrations, which are precisely the fibrations between fibrant objects in the $\cD$-covariant model structure (\cref{lemma:fib between lfib}). Moreover, it takes trivial projective fibrations to $\cD$-left fibrations, which are fiber-wise weak equivalences. By \cref{the:ncov equiv nlfib}, this means they are weak equivalences in the $\cD$-covariant model structure and so are trivial fibrations in the $\cD$-covariant model structure. Hence, by \cite[Proposition 7.15]{joyaltierney2007qcatvssegal}, it is a right Quillen functor. 
 \end{proof}
 
 We will end this section with an analysis of the {\it universal $\cD$-left fibration}. Following \cref{prop:universal fib cd}, the universal $\cD$-left fibration is given by the projection functor $N_\cD\sP(\cD)_{D[0]/} \to N_\cD\sP(\cD)$ (as defined in \cref{not:undcs}). We now have the following analogous result.
 
 \begin{propone} \label{prop:universal fib lfib}
 	Let $\C$ be a small $\sP(\cD)$-enriched category. The functor
 	$$N_\cD(-)^*\pi_*:\uFun(\C,\sP(\cD)) \to \sP(\cD\times\DD)_{/N_\cD\C}$$
 	given by pulling back $\pi_*$ along $N_\cD(-)$ is precisely the functor $\sint_{\cD/\C}$.
 \end{propone} 

\begin{proof}
 The pullback functor and $\sint_{\cD/\C}$ are both colimit preserving so it suffices to compare them on representables $\uMap(c,-) \times \{A\}$ where $c$ is an object in $\C$ and $A$ is a $\cD$-space. By \cref{ex:sintcd rep} and \cref{ex:sintcd constant}, $\sint_{\cD/\C}\uMap(c,-) \times \{A\} = N_\cD\C_{c/} \times A$. The same holds for $N_\cD(-)^*\pi_*$ by applying \cref{prop:universal fib cd}.
\end{proof}

The next desired step would be to generalize it to a global equivalence not relying on $\C$, similar to \cref{prop:grothendieck fiber cd}. However, that would require us to have a fiber-wise approach to model categories which quite challenging and not very developed. Instead we will postpone this discussion to \cref{subsec:grothendieck general}, when we have translated the equivalence into the language of quasi-categories. 

\subsection{Grothendieck Construction for \texorpdfstring{$\cD$}{D}-Simplicial Spaces with weakly Constant Objects}\label{subsec:grothendieck general}
 Up until now our Grothendieck construction has focused on functors over nerves of $\sP(\cD)$-enriched categories. In this subsection we want to generalize the result to more general $\cD$-simplicial spaces. For that we need the strictification results from \cite{bergner2007threemodels}. Bergner proves the existence of the following diagram of simplicial spaces 
 $$X \xleftarrow{c_X} CX \xrightarrow{u_X} RFCX$$
 natural in the simplicial space $X$. Here the first map is a cofibrant replacement in a model structure for Segal categories given in  \cite[Theorem 7.1]{bergner2007threemodels} and is a weak equivalence in the complete Segal space model structure by \cite[Theorem 6.3]{bergner2007threemodels}. The second map is the derived unit map of the Quillen equivalence given in \cite[Theorem 8.6]{bergner2007threemodels} and is (again by \cite[Theorem 6.3]{bergner2007threemodels}) an equivalence in the model structure for complete Segal spaces and the object $RFCX$ is in fact the nerve (in the sense of \cref{def:nervesnset}) of a simplicially enriched category. 
 
 Our goal is to prove that we can use these results by Bergner to construct a zigzag of level-wise complete Segal space equivalences between a large class of $\cD$-simplicial spaces and nerves of $\sP(\cD)$-enriched categories. Before we proceed, let us simplify the notation and make it more consistent with with the notation we have been using up until now. 
 
 \begin{notone} \label{not:segal cat rep}
 	Let $\cC: \ss \to \ss$ denote the cofibrant replacement functor in model structure constructed by Bergner in \cite[Theorem 7.1]{bergner2007threemodels}. Moreover, we use the notation $N\C_{-}: \ss \to \ss$ for the functor that first applies the cofibrant replacement functor $\cC$ and then the adjunction $RFC$ defined in  \cite[Theorem 8.6]{bergner2007threemodels}. Notice, these functors come with a natural zigzag of complete Segal space equivalences $X \leftarrow \cC X \rightarrow N\C_X$
 \end{notone}
 
 We now proceed to generalize these construction to $\cD$-simplicial spaces. 
 
 \begin{defone}
 	Let $\cC_\cD: \sP(\cD\times\DD) \to \sP(\cD\times\DD)$ be the functor given by applying $\cC$ level-wise. More precisely, using the isomorphism $\Fun(\cD^{op},\ss) \cong \sP(\cD\times\DD)$ we can define it as the post-composition functor 
 	$$\cC_\cD = \Fun(\cD^{op},\cC): \Fun(\cD^{op},\ss) \to \Fun(\cD^{op},\ss).$$ 
 	Similarly, define $\cF_\cD: \sP(\cD\times\DD) \to \sP(\cD\times\DD)$ as 
 	$$\cF_\cD = \Fun(\cD^{op},N\C_{(-)}): \Fun(\cD^{op},\ss) \to \Fun(\cD^{op},\ss).$$
 \end{defone}
 
 Starting with an arbitrary $\cD$-simplicial space, $\cF_{\cD}X$ is level-wise the nerve of a simplicially enriched category, in the sense that for every object $d$ in $\cD$ $\Und_d\cF_{\cD}X$ is the nerve of a simplicially enriched category. This does not generally imply that the whole $\cD$-simplicial space is the nerve of a category enriched over $\sP(\cD)$. For that we need the {\it underlying $\sP(\cD)$-enriched category}.

\begin{defone} \label{def:cx}
	Let $\cD$ be a small category with terminal object $t$.
	Let $N_\cD\C_{(-)}: \sP(\cD\times\DD) \to \sP(\cD\times\DD)$, be the functor that takes a $\cD$-simplicial space $X$ to $\bH\cF(X)$, where $\bH$ was defined in \cref{def:bH}.
\end{defone} 

 We now have the following result regarding these definitions using \cref{lemma:more technical lemma}.
 
 \begin{lemone} \label{lemma:zigzag strictification}
 	Let $\cD$ be a small category with terminal object.
 	Let $X$ be a $\cD$-simplicial space with weakly constant objects. Then there exists a zigzag natural in $X$
 	$$ X \overset{c_X}{\longleftarrow} \cC_\cD X \overset{u_X}{\longrightarrow} \cF_\cD X \overset{r_X}{\longleftarrow} N_\cD\C_{X},$$
 	where $\C_X$ is an $\sP(\cD)$-enriched category and such that all morphisms in the zigzag are level-wise complete Segal space equivalences.
 \end{lemone} 

\begin{remone}
	As already mentioned in \cref{rem:no segal cats}, we are not claiming (and it is not necessary for our purposes) that the object $N_\cD\C_X$ we constructed here is a fibrant object in a model structure on $\sP(\cD\times\DD)$. Indeed, according to \cite[Beginning of Section 8]{bergnerrezk2020comparisonii} constructing model structures for Segal categories on $\cD$-simplicial spaces does not seem to be possible for arbitrary $\cD$.
\end{remone}

We can use this result to state the following valuable result.

\begin{theone} \label{the:strictification dleft fib}
 Let $\cD$ be a small category with terminal object.
 Let $X$ be a $\cD$-simplicial space with weakly constant objects. Then we have the following diagram of $\sP(\cD)$-enriched Quillen equivalences.
 \begin{center}
 		\begin{tikzcd}[row sep=0.3in, column sep=0.9in]
 		\Fun(\C_X,\sP(\cD))^{proj} \arrow[r, shift left = 1.8, "\sint_{\cD/\C_X}"] & 
 		(\sP(\cD\times\DD)_{/N_\cD\C_X})^{\cD-cov} \arrow[l, shift left=1.8, "\sH_{\cD/\C_X}", "\bot"'] \arrow[r, shift left=1.8, "\sbT_{\cD/\C_X}"] \arrow[d, shift left =1.8, "(r_X)_!"] & \Fun(\C_X,\sP(\cD))^{inj} \arrow[l, shift left=1.8, "\sbI_{\cD/\C_X}", "\bot"'] \\
 		 &(\sP(\cD\times\DD)_{/\cF_\cD X})^{\cD-cov} \arrow[u, shift left=1.8, "(r_X)^*", "\rbot"'] \arrow[d, shift right=1.8, "(u_X)^*"'] & \\
 		 &(\sP(\cD\times\DD)_{/\cC_\cD X})^{\cD-cov} \arrow[u, shift right=1.8, "(u_X)_!"', "\rbot"] \arrow[d, shift left=1.8, "(c_X)_!", "\rbot"'] & \\
 		 &(\sP(\cD\times\DD)_{/X})^{\cD-cov} \arrow[u, shift left=1.8, "(c_X)^*"] & 
 	\end{tikzcd}
 \end{center}
\end{theone}

\begin{proof}
 The horizontal Quillen equivalences follow from \cref{the:simp grothendieck} and the vertical Quillen equivalences from \cref{the:dcov css inv} combined with the fact that $r_X,c_X,u_X$ are level-wise complete Segal space equivalences.	
\end{proof}
 
This result has the following direct corollary that is worth stating explicitly and can be thought of as a {\it strictification of $\cD$-left fibrations}.

\begin{corone} \label{cor:left fib strict}
	Let $\cD$ be a small category with terminal object.
	Let $L \to X$ be a $\cD$-left fibration over a $\cD$-simplicial space $X$ with weakly constant objects. Then there exists an (up to equivalence) unique $\sP(\cD)$-enriched functor $G:\C_X \to \sP(\cD)$ and a diagram 
	\begin{center}
		\begin{tikzcd}
			L \arrow[d, twoheadrightarrow] & (c_X)^*L \arrow[d, twoheadrightarrow] \arrow[l, "\simeq_{CSS}"'] \arrow[r, "\simeq_{CSS}"] \arrow[dl, phantom, "\urcorner", very near start] \arrow[dr, phantom, "\ulcorner", very near start] &  \widehat{(u_X)_!(c_X)^*L} \arrow[d, twoheadrightarrow] & \sbI_{\cD/\C_X} G \arrow[d, twoheadrightarrow] \arrow[l, "\simeq_{CSS}"'] \arrow[dl, phantom, "\urcorner", very near start]& \sint_{\cD/\C_X} G \arrow[dl, bend left=20] \arrow[l, "\simeq_{inj}"'] \\
			X & \cC_\cD X \arrow[l, "\simeq_{CSS}"'] \arrow[r, "\simeq_{CSS}"] & \cF_\cD X& N_\cD\C_X \arrow[l, "\simeq_{CSS}"']
		\end{tikzcd}
	\end{center}
	satisfying the following conditions:
	\begin{itemize}
		\item The squares are homotopy pullback squares.
		\item The horizontal morphism $\sint_{\cD/\C_X} F \to \sbI_{\cD/\C_X}F$ is an injective equivalence.
		\item The remaining three horizontal morphisms are all level-wise complete Segal space equivalences.
		\item The vertical morphism $\sint_{\cD/\C_X} F \to N_\cD\C_X$ is a $\cD$-left morphism.
		\item The remaining four vertical morphisms are $\cD$-left fibrations. 
	\end{itemize} 
\end{corone} 

We can use this result to better understand the relation between values of and the fibers. As is often the case we first need a technical lemma regarding simplicial spaces. 

\begin{lemone} \label{lemma:technical equivalent objects}
	Let $X$ be an arbitrary simplicial space. Let $x,y: F[0] \to X$ be two maps. Let $X \hookrightarrow \tilde{X} \hookrightarrow \hat{X}$ be a choice of fibrant replacement in the Segal space model structure ($\tilde{X}$) followed by a complete Segal space fibrant replacement $\hat{X}$. Then the following statements are equivalent.
	\begin{enumerate}
		\item The objects $x,y$ in $\tilde{X}$ are equivalent in the Segal space.
		\item The points $x,y$ in $\hat{X}$ are in the same path component of $\hat{X}_0$.
		\item The objects $x,y$ in the homotopy category $\Ho\tilde{X}$ are isomorphic.
	\end{enumerate}
\end{lemone}

\begin{proof}
	The equivalence $(1) \Leftrightarrow (3)$ follows from the definition of homotopy equivalences in Segal spaces \cite[Subsection 5.5]{rezk2001css} and the equivalence $(1) \Leftrightarrow (2)$ from the implication of the completeness condition \cite[Proposition 6.4 (4)]{rezk2001css}.
\end{proof}

\begin{defone} \label{def:eq objects simp space}
	Let $X$ be a simplicial space. Define an equivalence relation on $X_{00}$ as follows: $x \sim y$ if and only if it satisfies one of the three equivalent conditions in \cref{lemma:technical equivalent objects}.
\end{defone}

\begin{lemone} \label{lemma:surjective path component}
	Let $X \simeq Y$ be two simplicial spaces that are equivalent in the complete Segal space model structure. Then we have a bijection $X_{00}/_\sim \cong Y_{00}/_\sim$, where $\sim$ is the equivalence relation defined in \cref{def:eq objects simp space}.
\end{lemone}

\begin{proof}
	Let $X \to \hat{X}$ be a choice of fibrant replacement in the complete Segal space model structure, which, by \cite[Theorem 7.2]{rezk2001css}, is obtained by applying the small object argument to maps of the form
	\begin{itemize}
		\item $(\partial F[n] \to F[n]) \square (\Lambda[l]_i \to \Delta[l])$,
		\item $(G[n] \to F[n]) \square (\partial \Delta[l] \to \Delta[l])$,
		\item $(F[0] \to E[1]) \square (\partial \Delta[l] \to \Delta[l])$.
	\end{itemize}
	Then the map $X_0 \to \hat{X}$ is surjective on path-components as all three classes of maps in the list above have this property. This means the map $X_{00}/_\sim \to \pi_0(\hat{X}_0)$ is surjective. On the other hand, by \cref{lemma:technical equivalent objects}, $x \sim y$ in $X_{00}$ if and only if $[x] = [y]$ in $\pi_0(\hat{X}_0)$, which implies that the map $X_{00}/_\sim \to \pi_0(\hat{X}_0)$ is injective as well. Hence, we have a bijection $X_{00}/_\sim \cong \pi_0(\hat{X}_0)$. 
	
	Finally, recall that two equivalent complete Segal spaces are necessarily Reedy equivalent, so the result follows from the following chain of bijections
	$$X_{00}/_\sim \cong \pi_0(\hat{X}_0) \cong \pi_0(\hat{Y}_0) \cong  Y_{00}/_\sim,$$
	where $\hat{X}$ and $\hat{Y}$ are the fibrant replacement of $X$  and $Y$ in the complete Segal space model structure. 
\end{proof}

The lemma implies that there is a bijection between equivalence classes of objects of $X$ and $N_\cD\C_X$. Let us denote an object in $X$ and its corresponding object in $N_\cD\C_X$ both by $x$. Then we have the following result.

\begin{corone} \label{cor:value preservation}
	Let $\cD$ be a small category with terminal object and let $X$ have weakly constant objects. Let $L_x \to X$ be the $\cD$-covariant fibrant replacement of the map $\{x\}: F[t,0] \to X$. Then corresponding functor given via \cref{cor:left fib strict} is the corepresentable functor $\uMap(x,-): \C_X\to \sP(\cD)$. In particular for every $\cD$-left fibration $L$ with corresponding functor $F$ we have an equivalence of $\cD$-spaces $\Fib_xL \simeq F(x)$. 
\end{corone}

\begin{proof}
	By assumption the map $F[t,0] \to L_x$ is a $\cD$-covariant equivalence over $X$, which, by \cref{the:ncov model}, is equivalent to $F[0] \to \Und_dL_x$ being a covariant equivalence over $\Und_dX$ for all objects $d$ in $\cD$. By \cite[Theorem 3.55]{rasekh2017left}, this implies that $\Und_dL_x$ has an initial object for all object $d$ in $\cD$. Now, by \cref{cor:left fib strict}, we have a complete Segal space equivalence $\Und_dL_x \simeq \Und_d \sint_{\cD/\C}G$ for all objects $d$ in $\cD$, which means that $\Und_d \sbI_{\cD/\C}G$ also has a terminal object. Again, by \cite[Theorem 3.55]{rasekh2017left}, this means it is a representable left fibration and in particular we have an injective equivalence $\sbI_{\cD/\C}G\simeq (N_\cD\C_X)_{x/}$, which, by \cref{ex:sintcd rep}, implies that $G \simeq \uMap(x,-)$.
	
	Now, the $\sP(\cD)$-enriched Quillen equivalences in \cref{the:simp grothendieck} imply that for given left fibrations $L,L'$ over $X$ with corresponding functors $G,G'$, we have an equivalence of $\cD$-spaces
	\begin{equation} \label{eq:lfib ff}
     	 \uMap_{/X}(L,L') \simeq \uNat(G,G').
	\end{equation}
	So, fix an object $x$ in $X$ and let $\{x\}:F[t,0] \to X$ be the constant map. Then for every $\cD$-left fibration $L \to X$ we have equivalences $\cD$-simplicial spaces
	$$\Fib_xL \cong \uMap_{/X}(F[t,0],L) \simeq \uMap_{/X}(L_x,L) \simeq \uMap_{/N_\cD\C_X}(\uMap(x,-),G) \cong G(x),$$
	where the middle equivalence follows from \ref{eq:lfib ff} and the second bijection from the enriched Yoneda lemma (\cref{lemma:enriched yoneda}),
	giving us the desired equivalence of $\cD$-spaces. 
\end{proof}

\begin{exone} \label{ex:twisted}
	Let $W$ be a level-wise Segal space with weakly constant objects. Then $W^{op} \times W$ is also a level-wise Segal space with weakly constant objects (where $W^{op}$ was defined in \ref{eq:op}). By \cref{cor:left fib strict}, the $\cD$-left fibration $\Tw(W) \to W^{op} \times W$ (\cref{prop:twisted}) corresponds to an $\sP(\cD)$-functor $(\C_W)^{op} \times \C_W \to \sP(\cD)$. Here we used the fact that the category corresponding to the $\cD$-simplicial space $W^{op}$ is $(\C_W)^{op}$, by \cref{lemma:op}. Notice, as described in \cref{prop:twisted segal}, the fiber of $\Tw(W)$ over a point $x,y$ in $W^{op} \times W$ is the mapping space $\map_W(x,y)$ (\cref{def:map}). Hence, using the same argument used in \cref{cor:value preservation}, the functor is the mapping object functor 
	$$\Map_{\C_W}(-,-): (\C_W)^{op} \times \C_W \to \sP(\cD).$$
\end{exone}

\begin{exone} \label{ex:source and target fib}
	We can use the same arguments as in \cref{ex:twisted} and apply them to \cref{ex:source fibration} and \cref{ex:target fibration}. Hence, we deduce that for every level-wise Segal space $W$ with weakly constant objects there are two functors
	$$W_{-/}:\C_W \to \sP(\cD\times\DD), $$
	$$W_{/-}:(\C_W)^{op} \to \sP(\cD\times\DD),$$
	where the value of $W_{-/}$ at an object $x$ is equivalent to the level-wise under-Segal space $W_{x/}$ and $W_{/-}$ the level-wise over Segal space $W_{/x}$.
\end{exone} 

  Unfortunately the Quillen equivalences cannot be composed to a direct Quillen equivalence between the $\cD$-covariant model structure on $\sP(\cD\times\DD)_{/X}$ and the projective model structure on $\Fun(\C,\sP(\cD))$. 
 
 We end this section by observing that we can construct an equivalence of {\it $\sP(\cD)$-enriched quasi-categories}, as defined in \cite{gepnerhaugseng2015enrichedinftycat}, out of the $\sP(\cD)$-enriched Quillen equivalence of $\sP(\cD)$-model structures, which do in fact give us direct equivalences of quasi-categories, rather than just a zigzag. Concretely, in \cite{haugseng2015rectenrichedinftycat} Haugseng proves an equivalence between enriched model categories and their Quillen equivalences and enriched quasi-categories, which we can apply to our enriched Quillen equivalences. Before we can proceed we need the following notational convention from \cite{haugseng2015rectenrichedinftycat}.
 
 \begin{notone}\label{not:underlying qcat}
 	For a given simplicial model category $\cM$, denote the underlying quasi-category by $\cM[\cW^{-1}]$. We use the same notation for the underlying enriched quasi-category of an enriched model category. Moreover for a given Cartesian closed quasi-category $\C$, denote the quasi-category of quasi-categories by $\QCat$ and the quasi-category of $\C$-enriched quasi-categories by $\QCat_\C$.
 \end{notone} 
 
 Following \cref{not:underlying qcat}, let $\sP(\cD)[\cW^{-1}]$ denote the underlying quasi-category of the Cartesian model injective structure on $\sP(\cD)$ and notice it is a Cartesian closed quasi-category. For a strictly $\sP(\cD)$-enriched category $\C$, we again use $\C$ to denote the corresponding $\sP(\cD)$-enriched quasi-category. Now, again by \cref{not:underlying qcat}, the underlying quasi-category of the functor category is denoted $\Fun(\C,\sP(\cD))[\cW^{-1}]$. Notice, by \cite[Theorem 1.1]{haugseng2015rectenrichedinftycat}, there is an equivalence of $\sP(\cD)[\cW^{-1}]$-enriched quasi-categories 
 \begin{equation} \label{eq:functor cat}
 	\Fun(\C,\sP(\cD))[\cW^{-1}] \simeq \Fun_{\QCat_{\sP(\cD)[\cW^{-1}]}}(\C,\sP(\cD)[\cW^{-1}]),
 \end{equation}
  meaning there is an equivalence between the underlying quasi-category of the enriched projective model structure and the enriched functor quasi-category.
  
  We will make one exception to \cref{not:underlying qcat}. For an arbitrary $\cD$-simplicial space the bifibrant objects in the $\cD$-covariant model structure on $\sP(\cD\times\DD)_{/X}$ are precisely the $\cD$-left fibrations over $X$. Hence, the underlying $\sP(\cD)$-enriched quasi-category has objects $\cD$-left fibrations and we denote it by $\LFib_{\cD/X}$. Now applying \cite[Theorem 1.1]{haugseng2015rectenrichedinftycat} to \cref{the:strictification dleft fib} gives us the following result.
  
  \begin{corone} \label{cor:equiv qcat}
  	Let $\cD$ be a small category with terminal object.
  	Let $X$ be a $\cD$-simplicial space with weakly constant objects.
   	There is an equivalence of $\sP(\cD)$-enriched quasi-categories 
   		\begin{center}
   			\begin{tikzcd}[row sep=0.3in, column sep=0.6in]
   				\Fun_{\QCat_{\sP(\cD)[\cW^{-1}]}}(\C_X,\sP(\cD)[\cW^{-1}]) \arrow[r, shift left = 1.8, "\sint_{\cD/\C_X}"] & 
   				\LFib_{\cD/N_\cD\C_X} \arrow[l, shift left=1.8, "\sH_{\cD/\C_X}", "\bot"'] \arrow[r, shift left=1.8, "\sbT_{\cD/\C_X}"] \arrow[d, shift left =1.8, "(r_X)_!"] & \Fun_{\QCat_{\sP(\cD)[\cW^{-1}]}}(\C_X,\sP(\cD)[\cW^{-1}]) \arrow[l, shift left=1.8, "\sbI_{\cD/\C_X}", "\bot"'] \\
   				&\LFib_{\cD/\cF_\cD X} \arrow[u, shift left=1.8, "(r_X)^*", "\rbot"'] \arrow[d, shift right=1.8, "(u_X)^*"'] & \\
   				&\LFib_{\cD/\cC_\cD X} \arrow[u, shift right=1.8, "(u_X)_!"', "\rbot"] \arrow[d, shift left=1.8, "(c_X)_!", "\rbot"'] & \\
   				&\LFib_{\cD/X} \arrow[u, shift left=1.8, "(c_X)^*"] & 
   			\end{tikzcd}
   	\end{center}
  \end{corone} 
  Unlike in the in the context of model categories in the context of quasi-categories in any adjoint equivalence a left adjoint is also a right adjoint \cite[Proposition 2.1.12]{riehlverity2018elements}. Hence, we can in fact compose the various equivalences to get the following final result. 
  
  \begin{corone} \label{cor:equiv lfib qcat}
   Let $\cD$ be a small category with terminal object.
   Let $X$ be a $\cD$-simplicial space with weakly constant objects.
   There is an equivalence of $\sP(\cD)$-enriched quasi-categories
   	\begin{center}
   	\begin{tikzcd}[row sep=0.5in, column sep=0.9in]
   		\Fun_{\QCat_{\sP(\cD)[\cW^{-1}]}}(\C_X,\sP(\cD)[\cW^{-1}]) \arrow[r, shift left = 1.8] & 
   		\LFib_{\cD/X} \arrow[l, shift left=1.8, "\simeq"'] 
   	\end{tikzcd}.
   \end{center}
  \end{corone} 
   
  Having done an analysis of the quasi-category we move on to generalize to a global equivalence. This will require some further quasi-categorical techniques and notation. Let $\CSS_{\sP(\cD)}$ be the underlying quasi-category of the $\sP(\cD)$-enriched complete Segal space model structure (\cref{the:spd enriched css}). The object $N_\cD\sP(\cD)$ is by construction an $\sP(\cD)$-enriched Segal space, but might not satisfy the completeness condition. Hence, let 
  \begin{equation}\label{eq:completion}
   \widehat{\pi_*}: \widehat{N_\cD\sP(\cD)_{D[0]/}} \to \widehat{N_\cD\sP(\cD)}  
  \end{equation} 
  be the completion of $\pi_*: N_\cD\sP(\cD)_{D[0]/} \to N_\cD\sP(\cD)$ in the $\sP(\cD)$-enriched complete Segal space model structure (\cref{the:spd enriched css}) and notice we have the following result.
  
  \begin{lemone} \label{lem:universal fib comp lfib}
 	Let $\C$ be a small $\sP(\cD)$-enriched category. The functor
 	$$N_\cD(-)^*\widehat{\pi}_*:\uFun(\C,\sP(\cD)[\cW^{-1}]) \to \LFib_{\cD/X}$$
 	given by pulling back $N_\cD$ along $\widehat{\pi}_*$ is equivalent to the functor $\sint_{\cD/\C}$.
 \end{lemone}
  
  \begin{proof}
  	The completion functor is a Dwyer-Kan equivalence (\cref{the:spd enriched css}) and so the argument follows from the same proof as in \cref{prop:universal fib lfib}.
  \end{proof}

  Let $\sOall_{\CSS_{\sP(\cD)}} \to \CSS_{\sP(\cD)}$ be the right fibration of quasi-categories representing the functor $((\CSS_{\sP(\cD)})_{/-})^{grpd}$, where $(-)^{grpd}$ is the underlying $\infty$-groupoid. See \cite[Subsection 3.3]{gepnerkock2017univalence} or \cite[Section 2]{rasekh2021univalence} for a more detailed discussion. Moreover, let $\LFib_\cD \hookrightarrow \sOall_{\CSS_{\sP(\cD)}}$ be the full sub quasi-category consisting of left fibrations, which is also a right fibration over $\CSS_{\sP(\cD)}$. By the Yoneda lemma \cite[Theorem 5.7.1]{riehlverity2018elements}, we have a map of quasi-categories $(\CSS_{\sP(\cD)})_{/\widehat{N_\cD\sP(\cD)}} \to \LFib$ over $\CSS_{\sP(\cD)}$, that takes a map $F:\C \to \widehat{N_\cD\sP(\cD)}$ to the pullback $F^*\widehat{\pi_*}$.
  
  \begin{corone}\label{cor:lfib classifier}
  	The map $(\CSS_{\sP(\cD)})_{/\widehat{N_\cD\sP(\cD)}} \to \LFib_\cD$ is an equivalence of quasi-categories.
  \end{corone} 

  \begin{proof}
  	By \cite[Lemma 2.2.3.6]{lurie2009htt} it suffices to prove that we have a fiber-wise equivalence. Hence, fix an $\sP(\cD)$-enriched complete Segal space $X$. By \cref{lemma:zigzag strictification}, $X$ is equivalent to the nerve of an $\sP(\cD)$-enriched category $N_\cD\C$ and equivalent objects have equivalent fibers, hence it suffices to show the fibers over $N_\cD\C$ are equivalences. Notice we have the equivalence $\sP(\cD)[\cW^{-1}] \simeq \widehat{N_\cD\sP(\cD)}$ (essentially by definition as both are the underlying quasi-category of $\cD$-spaces). Hence we have 
  	$$\Map(\C,\widehat{N_\cD\sP(\cD)}) \simeq \Map(\C,\sP(\cD)[\cW^{-1}]) \to \LFib_{\cD,N_\cD\C},$$
  	 where the second equivalence is given by the underlying $\infty$-groupoids of \cref{cor:equiv lfib qcat} and, by \cref{lem:universal fib comp lfib}, is given via pulling back $\widehat{\pi}_*$.
  \end{proof} 

\begin{remone} \label{rem:universal lfib}
	Building on \cref{rem:universal opfib}, \cref{cor:lfib classifier} implies that $\widehat{\pi_*}: \widehat{N_\cD\sP(\cD)_{D[0]/}} \to \widehat{N_\cD\sP(\cD)}$ is a {\it universal $\cD$-left fibration} and, in particular, for every $\cD$-left fibration $L \to N_\cD\C$, there is a pullback of $\cD$-simplicial spaces
	\begin{center}
		\pbsq{L}{\widehat{N_\cD\sP(\cD)_{D[0]/}}}{N_\cD\C}{\widehat{N_\cD\sP(\cD)}}{}{}{}{}
	\end{center}
	in a homotopically unique fashion.
\end{remone} 

Finally, we want to consider its implication with regard to {\it univalence}. In light of \cite[Proposition 2.4]{rasekh2021univalence} we have the following result.

\begin{corone} \label{cor:universal lfib univalent}
	The universal $\cD$-left fibration $\widehat{\pi_*}: \widehat{N_\cD\sP(\cD)}_{D[0]/} \to \widehat{N_\cD\sP(\cD)}$ is univalent.
\end{corone}

\begin{remone}
	We would expect \cref{cor:lfib classifier} to not just be an equivalence of quasi-categories, but rather enriched quasi-categories. In particular, we would expect $\widehat{\pi_*}: \widehat{N_\cD\sP(\cD)}_{D[0]/} \to \widehat{N_\cD\sP(\cD)}$ to satisfy a condition stronger than the simple univalence condition. For a discussion how such a condition could look like see \cite[Section 7]{rasekh2021univalence}.
\end{remone}
 
 \section{Localized \texorpdfstring{$\cD$}{D}-Left fibrations} \label{sec:localized}
 Up until now we studied fibrations that correspond to functors valued in general $\cD$-spaces. However, we want fibrations that correspond to functors that take value in localizations of $\cD$-spaces. The goal of this section is to generalize the results of the previous two sections by localizing the various definitions. In particular we present a localized form of most results in \cref{sec:left fib} in \cref{subsec:localized model} and localize the Grothendieck construction (\cref{sec:grothendieck}) in \cref{subsec:localized groth}. The last subsection (\cref{subsec:sloc features}) focuses on results regarding $S$-localized $\cD$-left fibrations that require the localized Grothendieck construction.
 
 \subsection{Localizations Model Structures}\label{subsec:localized model}
 In this subsection we introduce $S$-localized $\cD$-left fibrations and show it has a model structure that behaves very similar to \cref{the:ncov model}. In order to proceed it is helpful to establish relevant notation.
 
 \begin{notone} \label{not:localized}
  Let $S$ be a set cofibrations (or, equivalently, monomorphisms) of $\cD$-spaces.
 	\begin{itemize}
 		\item A $\cD$-space $X$ is called {\it local with respect to $S$} if for every every map 
 		$f: A \to B$ in $S$, 
 		$$\Map_{\sP(\cD)}(B,X) \to \Map_{\sP(\cD)}(A,X)$$
 		is a Kan equivalence.
 		\item A $\cD$-simplicial space $X$ is called {\it local with respect to $S$} if 
 		$\Val(X)$ is local with respect to $S$.
 		This is equivalent to 
 		$$\Map_{\sP(\cD\times\DD)}(\VEmb(B),X) \to \Map_{\sP(\cD\times\DD)}(\VEmb(A),X)$$
 		being a Kan equivalence for every $A \to B$ in $S$ (via the enriched adjunction $(\VEmb,\Val)$ defined in \cref{def:val}).
 		\item A map of $\cD$-simplicial spaces $p:Y \to X$ is called {\it local with respect to $S$} 
 		if the map 
 		$$
 		\Map_{\sP(\cD)}(\VEmb(B),Y) \to \Map_{\sP(\cD)}(\VEmb(A),Y) \times_{\Map_{\sP(\cD)}(\VEmb(A),X)} \Map_{\sP(\cD)}(\VEmb(B),X)
 		$$
 		is a weak equivalence for every map $f: A \to B$ in $S$. Note this is equivalent to the condition that
 		for every map $f: A \to B$ in $S$ and every map $\VEmb(B) \to X$, the induced map 
 		$$\Map_{/X}(\VEmb(B),Y) \to \Map_{/X}(\VEmb(A),Y)$$
 		is a Kan equivalence.
 	\end{itemize}
 \end{notone}

For the later parts we need some conditions on $S$. 

\begin{defone} \label{def:geometric realization}
Let $N(\cD_{/-}): \cD \to \s$ be the functor that takes an object $d$ to the nerve $N(\cD_{/d})$. This induces an adjunction
\end{defone} 
\begin{center}
	\adjun{\sP(\cD)}{\s}{|-|}{Sing}.
\end{center} 

\begin{defone} \label{def:contractible object}
	A morphism of $\cD$-space $X \to Y$ is a {\it geometric equivalence} if $|X| \to |Y|$ is an equivalence in $\s$. Moreover, an object $X$ is {\it geometrically contractible} if the map to the terminal object is a geometric equivalence.
\end{defone}

We now use geometric equivalences to add conditions on a set of cofibrations.

\begin{defone} \label{def:contractible realization}
	Let $S$ be a set of cofibrations of $\cD$-spaces. 
	\begin{itemize}
		\item $S$ has {\it contractible codomains} if for all $f: A \to B$ in $S$, $B$ is geometrically contractible.
		\item $S$ {\it consists of geometric equivalences} if for all $f: A \to B$ in $S$, $f$ is a  geometric equivalence.
		\item $S$ is {\it geometrically contractible} if it consists of geometric equivalences with contractible codomain, which is equivalent to all morphisms in $S$ being geometrically contractible.  
	\end{itemize}
\end{defone} 

The definitions are chosen with the following implications in mind, which follow very straightforwardly.

\begin{lemone} \label{lemma:geom contr}
	An object $B$ is geometrically contractible if and only if for every homotopically constant $\cD$-space $Y$, every map $B \to Y$ is injectively equivalent to a constant map $B \to D[0] \to Y$. 
\end{lemone}

\begin{lemone}\label{lemma:geom equiv}
	$S$ consists of geometric equivalences if and only if every homotopically constant $\cD$-space is $S$-local.
\end{lemone}

We also have the following implication we will use later on.

\begin{lemone} \label{lemma:coprod fibrant}
	Let $S$ be a set of cofibrations of $\cD$-spaces with contractible codomain. Let $X,Y$ be two $S$-local $\cD$-spaces. Then $X \coprod Y$ is also $S$-local.
\end{lemone} 

\begin{proof}
 The maps to the terminal object $X, Y \to D[0]$ gives us a map of coproducts $X \coprod Y \to D[0] \coprod D[0]$. Now let $A \to B$ be map in $S$. As $S$ has contractible codomains, there are precisely two morphisms $B \to D[0] \coprod D[0]$, which are the two that factor through the coproduct inclusions $i_0,i_1: D[0] \to D[0] \coprod D[0]$. This implies that every morphism $B \to X \coprod Y$ factors through the inclusions $X,Y \to X \coprod Y$, meaning we have $\Map(B,X \coprod Y) \simeq \Map(B,X) \coprod \Map(B,Y)$ and similarly $\Map(A,X \coprod Y) \simeq \Map(A,X) \coprod \Map(A,Y)$. The result now follows from the fact that Kan equivalences are closed under coproducts. 
\end{proof} 

As we want want to show that many equivalences from \cref{sec:left fib} and \cref{sec:grothendieck} still hold in the $S$-localized setting it is helpful to have a general model categorical result that helps us localize Quillen equivalences.

\begin{theone} \label{the:localizing equiv}
	Let $\C^{\cM}$, $\D^{\cN}$ be two simplicial left proper combinatorial model categories with a simplicial Quillen adjunction 
	\begin{center}
		\adjun{\C^{\cM}}{\D^{\cN}}{F}{G}.
	\end{center}
	Let $S$ be a set of cofibrations in $\C$. Let $\C^{\cM_S}$ be the Bousfield localization of the $\cM$-model structure with respect $S$ and let $\D^{\cN_{F(S)}}$ be the Bousfield localization with respect to $F(S)$. Then, the Quillen adjunction $(F,G)$ induces a Quillen adjunction of localized model structures
	\begin{center}
		\adjun{\C^{\cM_S}}{\D^{\cN_{F(S)}}}{F}{G}.
	\end{center}
	which is a Quillen equivalence if the original one was a Quillen equivalence.
\end{theone}

\begin{proof}
	The proof is analogous to \cite[Theorem 1.2]{rasekh2021cartfibcss} with the sole difference that we are not assuming every object in $\C^{\cM}$ is cofibrant and so we need an alternative argument regarding the derived counit map. Let $Y$ be a bifibrant object in $\D$. Let $p:CGY \to GY$ be a cofibrant replacement of $GY$ in the $\cM$-model structure (meaning $p$ is a trivial fibration in the $\cM$-model structure). By definition of Bousfield localizations, the $\cM_S$-model structure has the same trivial fibrations as $\cM$ and so $p$ is also a trivial fibration in $\cM_S$-model structure. Hence, the map $FCGY \to FGY \to Y$ is the derived counit map of localized model structures and we need to prove this map is an equivalence in the $\cN_{F(S)}$-model structure on $\D$.
	
	By assumption $(F,G)$ is a Quillen equivalence between the $\cM$ and $\cN$ model structure and so the derived counit map $FCGY \to FGY \to Y$ is already an equivalence in the $\cN$-model structure on $\D$, which by Bousfield localization means it is also an equivalence in the $\cN_{F(S)}$-model structure. This, combined with the explanations in the proof of \cite[Theorem 1.2]{rasekh2021cartfibcss}, finishes the proof.
\end{proof}

It is helpful to explicitly state the following variation of \cref{the:localizing equiv}.

\begin{corone}\label{cor:localizing equiv}
		Let $\C^{\cM}$, $\D^{\cN}$ be two simplicial left proper combinatorial model categories with a simplicial Quillen adjunction 
	\begin{center}
		\adjun{\C^{\cM}}{\D^{\cN}}{F}{G}.
	\end{center}
	Let $S_1$ be a set of cofibrations in $\C$ and let $S_2$ be a set of cofibrations in $\D$. Let $\C^{\cM_{S_1}}$ be the Bousfield localization of the $\cM$-model structure with respect $S$ and let $\D^{\cN_{S_2}}$ be the Bousfield localization with respect to $S_2$. Then, the Quillen adjunction $(F,G)$ induces a Quillen adjunction 
	\begin{center}
		\adjun{\C^{\cM_{S_1}}}{\D^{\cN_{S_2}}}{F}{G}
	\end{center}
	if and only if for every fibrant object $Y$ in the $S_2$-localized $\cN$-model structure, $G(Y)$ is $S_1$-local. Moreover, if $(F,G)$ is a Quillen equivalence between the $\cM$ and $\cN$ model structure, then $(F,G)$ is a Quillen equivalence between then $\cM_{S_1}$ and $\cN_{S_2}$ model structure if and only if for every $\cN$-fibrant object $Y$, $Y$ is $\cN_{S_2}$ local if and only if $G(Y)$ is $\cM_{S_1}$-fibrant.
\end{corone}

\begin{proof}
	The Quillen adjunction argument follows immediately from \cite[Corollary A.3.7.2]{lurie2009htt}. The Quillen equivalence follows from  \cref{the:localizing equiv} combined with the fact that the condition precisely states that the $S_2$-localized $\cN$-model structure coincides with the $F(S_1)$-localized $\cN$-model structure.
\end{proof}

Our main object of study is in this section is the notion of a {\it $S$-localized $\cD$-left fibration}.
\begin{defone} \label{def:cdleft s}
	Let $X$ be a $\cD$-simplicial space and $S$ a set of monomorphisms of $\cD$-spaces (not over $X$).
	Then a $\cD$-left fibration $p:L \to X$ is an {\it $S$-localized $\cD$-left fibration} if it is local with respect to $S$, which is equivalent to $\Val(p): \Val(L) \to \Val(X)$ being local with respect to $S$. 
\end{defone}

Let us start the study by localizing the injective model structure defined in \cref{prop:injectivemod}.

\begin{propone} \label{prop:injectivemod s}
	Let $S$ be a set of monomorphisms of $\cD$-spaces.
	There is a unique, combinatorial left proper simplicial model structure on $\sP(\cD)$, denoted by $\sP(\cD)^{inj_S}$ and called the 
	{\it $S$-localized injective model structure}, defined as follows. 
	\begin{itemize}
		\item[C] A map $Y \to Z$ is a cofibration if it is an inclusion.
		\item[F] An object $W$ is fibrant if it is injectively fibrant and local with respect to $S$.
		\item[W] A map $Y \to Z$ is a weak equivalence if for every fibrant object $W$ the map 
		$$\Map_{\sP(\cD)}(Z, W) \to \Map_{\sP(\cD)}(Y,W)$$
		is a Kan equivalence.
	\end{itemize}
\end{propone}

\begin{proof}
	This is a direct application of a left Bousfield localization \cite[Theorem 4.1.1]{hirschhorn2003modelcategories} to the injective model structure on $\cD$-spaces (\cref{prop:injectivemod}) with respect to the morphisms $S$.
\end{proof}

Notice, the $S$-localized model structure is not necessarily Cartesian prompting the following definition.

\begin{defone} \label{def:cartesian s}
	Let $S$ be a set of cofibrations of $\cD$-spaces. We say $S$ is Cartesian if the $S$-localized injective model structure on $\sP(\cD)$ is a Cartesian model structure.
\end{defone}

We move on to generalize the level-wise model structure given in \cref{prop:levelwise localized model structure}.

\begin{theone} \label{the:sone local level stwo}
	Let $S_1$ be a set of monomorphisms of $\cD$-spaces and let $S_2$ be a set of monomorphism of simplicial spaces.
	There is a unique simplicial combinatorial left proper model structure on $\sP(\cD\times\DD)$, denoted by $\sP(\cD\times\DD)^{(lev_{S_2})_{S_1}}$ and called the {\it $S_1$-localized level-wise $S_2$-localized model structure}, defined as follows. 
	\begin{enumerate}
		\item A map $Y \to Z$ is a cofibration if it is an inclusion.
		\item An injectively fibrant object $X$ is fibrant if for all objects $d$ in $\cD$, $\Und_dX$ is $S_2$-local and for all $k\geq 0$, $\Val_kX$ is $S_1$-local.
		\item A map $g:Y \to Z$ is a weak equivalence if for all fibrant objects $X$, the induced map 
		$$\Map(Z,X) \to \Map(Y,Z)$$
		is a Kan equivalence. 
		\item A map $g:Y \to Z$ between fibrant objects is a weak equivalence if and only if it is an injective equivalence.
		\item The model structure is Cartesian if and only if both $S_1$ and $S_2$ are Cartesian.
	\end{enumerate}
\end{theone}

\begin{proof}
	This is a direct application of a left Bousfield localization \cite[Theorem 4.1.1]{hirschhorn2003modelcategories} to the injective model structure on $\cD$-simplicial spaces (\cref{prop:injectivemod}) with respect to the morphisms 
	$$\{D[n] \times i: D[n] \times A \to D[n] \times B: i \in S_1, n \geq 0 \} \cup \{j \times F[d,0]: C \times F[d,0] \to D \times F[d,0]: j \in S_2, d \in \Obj_\cD \} .$$
\end{proof}

We now apply this result to the $\cD$-diagonal model structure defined in \cref{the:diagmodel}.

\begin{propone} \label{the:diagmodel s}
	Let $S$ be a set of monomorphisms of $\cD$-spaces. There is a unique simplicial combinatorial left proper model structure on $\sP(\cD\times\DD)$, denoted by $\sP(\cD\times\DD)^{\cD-diag_S}$ and called the {\it $\cD$-diagonal $S$-localized injective model structure}, defined as follows. 
	\begin{enumerate}
		\item A map $Y \to Z$ is a cofibration if it is an inclusion.
		\item A map $g:Y \to Z$ is a weak equivalence if the diagonal map 
		$$\fDiag(g): \fDiag(Y) \to \fDiag(Z)$$ is an $S$-localized injective equivalence.
		\item An object $W$ is fibrant if and only if it is fibrant in the $\cD$-diagonal model structure and local with respect to $S$. 
		\item In particular, the diagonal $S$-localized injective model structure is Cartesian if and only if $S$ is Cartesian (\cref{def:cartesian s}).
	\end{enumerate}
\end{propone}

\begin{proof}
	We can directly use \cref{the:sone local level stwo}, where $S_1=S$ and $S_2 = \{F[0] \to F[n]: n \geq 0\}$.
	The resulting model structure already satisfies most condition. We only need to confirm it is Cartesian if $S$ is Cartesian.
	Let $i,j$ be cofibrations of $\cD$-simplicial spaces such that $i$ is trivial. Then $i \square j$ is an equivalence if and only if $\fDiag(i \square j) = \fDiag(i) \square \fDiag(j)$ is an $S$-localized injective equivalence of $\cD$-spaces, which holds precisely if $S$ is Cartesian.
\end{proof}

\begin{propone} \label{prop:dcov vs diag s}
	 The following adjunction 
\begin{center}
	\adjun{\sP(\cD\times\DD)^{diag_S}}{\sP(\cD)^{inj_S}}{\fDiag = (\phiDiag)^*}{(\phiDiag)_*}
\end{center}
is a Quillen equivalence. Here the left hand side has the diagonal $S$-localized injective model structure and the 
right hand side has the localized injective model structure.
\end{propone} 

\begin{proof}
	By direct computation for a given $\cD$-space $A$, we have $\fDiag(\VEmb(A)) = A$. Moreover, we obtain the diagonal $S$-localized injective model structure on $\sP(\cD\times\DD)$ by localizing with respect to maps of the form $\VEmb(A) \to \VEmb(B)$ and we obtain the $S$-localized injective model structure on $\sP(\cD)$ by localizing with respect to $A \to B$ which is equal to $\fDiag(\VEmb(A)) \to \fDiag(\VEmb(B))$. Hence the result follows from \cref{the:localizing equiv}.
\end{proof}

Another application of \cref{the:sone local level stwo} is the $S$-localized level-wise complete Segal space model structure.

\begin{corone}\label{prop:levelcss s}
	There exist a unique simplicial left proper combinatorial model structure on $\P(\cD\times\DD)$, which we call the {\it $S$-localized level-wise (complete) Segal model structure} and denote $\sP(\cD\times\DD)^{(lev_{Seg})_S}$ ($\sP(\cD\times\DD)^{(lev_{CSS})_S})$, which has the following specifications.
	\begin{enumerate}
		\item Cofibrations are inclusions.
		\item An object $X$ is fibrant if it is injectively fibrant, for every $d$ in $\cD$ the simplicial space $\Und_dX$ is a (complete) Segal space and for every $k \geq 0$, the $\cD$-space $\Val_kX$ is $S$-local. 
		\item A map $g:Y \to Z$ is a weak equivalence if for all fibrant objects $X$, the induced map 
		$$\Map(Z,X) \to \Map(Y,Z)$$
		is a Kan equivalence. 
		\item A weak equivalence (fibration) between fibrant objects is an injective equivalence (injective fibration). 
		\item The model structure is Cartesian if $S$ is Cartesian.
	\end{enumerate}
\end{corone}

\begin{proof}
	Direct application of \cref{the:sone local level stwo} with $S_1 = S$ and $S_2$ the Segal and completeness maps as used to define the level-wise complete Segal space model structure (\cref{prop:levelcss}).
\end{proof}

We can now finally move on to our main model structure of interest, which localized the $\cD$-covariant model structure defined in \cref{the:ncov model}.

\begin{theone}  \label{the:cdcov s}
	Let $S$ be a set of monomorphisms of $\cD$-spaces. Let $X$ be a $\cD$-simplicial space. There is a unique simplicial combinatorial left proper model structure on $\sP(\cD\times\DD)_{/X}$, denoted by $(\sP(\cD\times\DD)_{/X})^{\cD-Cov_S}$ and called the 
	{\it $S$-localized $\cD$-covariant model structure}, defined as follows. 
	\begin{itemize}
		\item[C] A map $Y \to Z$ over $X$ is a cofibration if it is an inclusion.
		\item[F] An object $Y \to X$ is fibrant if it is a $\cD$-left fibration and local with respect to $S$.
		\item[W] A map $Y \to Z$ over $X$ is a weak equivalence if for every fibrant object $W \to X$ the map 
		$$\Map_{/X}(Z, W) \to \Map_{/X}(Y,W)$$
		is a Kan equivalence.
	\end{itemize}
	Moreover, if $S$ is Cartesian then the $S$-localized $\cD$-covariant model structure is enriched over the $S$-localized injective model structure on $\cD$-spaces (\cref{prop:injectivemod s}).
\end{theone}

\begin{proof}
	Notice the $\cD$-covariant model structure on the category $\sP(\cD\times\DD)_{/X}$ is still proper and cellular \cite[Proposition 12.1.6]{hirschhorn2003modelcategories} and so we can apply left Bousfield localization \cite[Theorem 4.1.1]{hirschhorn2003modelcategories} with respect to the set of morphisms	$\mathcal{L} = \{ \VEmb(A) \to \VEmb(B) \to X : A \to B \in S \}.$ The resulting model structure immediately has the expected cofibrations, fibrant objects and weak equivalences.
	
	Now, let us assume $S$ is Cartesian. We want to prove the $S$-localized $\cD$-covariant model structure is enriched over the $S$-localized injective model structure by verifying the fourth condition in \cref{def:enriched model cat}. Evidently the pushout product of cofibrations is still a cofibration as they are all monomorphisms. Notice for a given $\cD$-space $C$ and a $\cD$-simplicial space $Y \to X$, the enrichment is given by $C \times Y \to Y \to X$. 
	
	We need to prove that for every $\cD$-space, the functor $C \times - : (\sP(\cD\times\DD)_{/X})^{\cD-cov_S} \to (\sP(\cD\times\DD)_{/X})^{\cD-cov_S}$ is left Quillen. By \cref{the:ncov model} and \cref{def:enriched model cat} we already know it is a left Quillen functor of un-localized model structure, hence, by \cref{cor:localizing equiv}, we only need to prove for every $A \to B \to X$ where $A \to B$ in $S$, the map $C \times A \to C \times B \to B \to X$ is an $S$-localized $\cD$-covariant equivalence. However, this follows from the fact that $S$ is Cartesian.
	
	Finally, we need to prove for every $\cD$-simplicial space $Y \to X$, the map $- \times Y: \sP(\cD) \to \sP(\cD\times\DD)_{/X}$ is left Quillen. Again, by \cref{the:ncov model}, it is a left Quillen functor of un-localized model structures, and so, again by \cref{cor:localizing equiv}, it suffices to confirm for every map $A \to B$ in $S$ the map $A \times Y \to B \times Y \to Y \to X$ is $S$-localized $\cD$-covariant equivalence, which again follows from the fact that $S$ is Cartesian.
\end{proof}

We move on to the $S$-localized version of \cref{the:diag local ncov}.

  \begin{theone} \label{the:diag local ncov s}  
  	Let $S$ be a set of monomorphisms of $\cD$-spaces. The following is a Quillen adjunction
	\begin{center}
		\adjun{(\sP(\cD\times\DD)_{/X})^{\cD-cov_S}}{(\sP(\cD\times\DD)_{/X})^{\cD-diag_S}}{id}{id}
	\end{center}
	which is a Quillen equivalence if for all $d$ in $\cD$ the simplicial space $\Und_dX$ is homotopically constant. Here the left side has the diagonal $S$-localized model structure and the right side has the $S$-localized $\cD$-diagonal model structure (\cref{the:diagmodel s}) on the over category. 
\end{theone}

\begin{proof}
	By \cref{the:diag local ncov}, we already have a Quillen adjunction of unlocalized model structures
	\begin{center}
		\adjun{(\sP(\cD\times\DD)_{/X})^{\cD-cov}}{(\sP(\cD\times\DD)_{/X})^{\cD-diag}}{id}{id}.
	\end{center}
	Hence, by \cref{cor:localizing equiv}, it suffices to show that the left adjoint takes localizing maps in the $S$-localized $\cD$-covariant model structure (which by definition are of the form $\VEmb(A) \to \VEmb(B) \to X$) to weak equivalences in the diagonal $S$-localized injective model structure on the over-category. Notice we only need to show that $\VEmb(A) \to \VEmb(B)$ is an equivalence in the $S$-localized injective model structure as diagonal equivalences do not depend on the base. By definition, this means we have to prove $\fDiag(\VEmb(A)) \to \fDiag(\VEmb(B))$ is an $S$-localized injective equivalence. However, this holds by direct computation as this map is equal to $A \to B$.
\end{proof}

The $S$-localized $\cD$-covariant model structure is also invariant under level-wise complete Segal space equivalences generalizing \cref{the:dcov css inv}.

\begin{theone} \label{the:dcov invariant css s}
	Let $g: X \to Y$ be a map of $\cD$-simplicial spaces. Then the adjunction
	\begin{center}
		\adjun{(\sP(\cD)_{/X})^{\cD-Cov_S}}{(\sP(\cD)_{/Y})^{\cD-Cov_S}}{g_!}{g^*}
	\end{center}
	is an $\sP(\cD)$-enriched Quillen adjunction, which is a Quillen equivalence whenever $g$ is a level-wise CSS equivalence.
	Here both sides have the $S$-localized $\cD$-covariant model structure. 
\end{theone}

\begin{proof}
	Clearly it is a Quillen adjunction as fibrations are stable under pullback.
	
	Let us now assume that $g$ is a level-wise CSS equivalence. We want to prove that $(g_!,g^*)$ is a Quillen equivalence of $S$-localized $\cD$-covariant model structures. By \cref{the:dcov css inv} it is a Quillen equivalence of $\cD$-covariant model structures and we want to use \cref{cor:localizing equiv} to deduce that it remains a Quillen equivalence after localization, which means we have to confirm the relevant condition, which in our case translates to the following statement: Let $p:L \to Y$ be a $\cD$-left fibration. Then $p$ is $S$-localized if and only and only if $g^*p: g^*L \to X$ is an $S$-localized $\cD$-left fibration. However, this follows immediately from the fact that, by \cref{the:nlfib pb nrfib}, $g^*L \to L$ is a level-wise CSS equivalence, which means it is a level-wise covariant equivalence over $Y$ (\cref{the:cov loc css}). Hence, we can apply \cref{cor:localizing equiv} and the result follows.
\end{proof}	

\begin{remone} \label{rem:better result}
	With appropriate assumptions on $X$, $Y$ and $S$ we can further strengthen the result in \cref{the:dcov invariant css s} as we shall do in \cref{the:cdcov s invariant s equiv}.
\end{remone}

Unfortunately it is not generally true that the $S$-localized $\cD$-covariant model structure is a localization of the $S$-localized level-wise complete Segal space model structure and so \cref{the:cov loc css} does not generalize. However, it does hold in some cases, as indicated by this lemma.

\begin{lemone} \label{lemma:fib s over s loc}
	Let $S$ be a set of monomorphisms of $\cD$-spaces. Let $X$ be a $\cD$-simplicial space such that $\Val(X)$ is $S$-local and $p: L \to X$ be an injective fibration. $p$ is an $S$-localized $\cD$-left fibration if and only if $p$ is a $\cD$-left fibration and $\Val(L)$ is local with respect to $S$. Moreover, if $\Val_k(X)$ is $S$-local, then this is equivalent to $p$ being a $\cD$-left fibration and $\Val_k(L)$ is local with respect to $S$ for all $k \geq 0$. 
\end{lemone}

\begin{proof}
	Let $p$ be an $S$-localized Reedy $\cD$-fibration. 
	We have a commutative diagram 
	\begin{center}
		\comsq{\Map_{/X}(\VEmb(B),L) }{\Map_{/X}(\VEmb(A),L) }{\Map_{/\Val(X)}(B,\Val(L)) }{\Map_{/\Val(X)}(A,\Val(L))}{}{\cong}{\cong}{}.
	\end{center}
	The vertical maps are bijections using the enriched adjunction $(\VEmb,\Val)$.
	So the top map is an equivalence (which is the definition of an $S$-localized $\cD$-left fibration) if and only if the bottom map is an equivalence (which is equivalent to $\Val(L) \to \Val(X)$ being fibrant in the $S$-localized model structure on $\sP(\cD)_{/\Val(X)}$). As $\Val(X)$ is $S$-local this is equivalent to $\Val(L)$ being $S$-local.
	
	Now assume additionally that $\Val_k(X)$ is $S$-local for all $k \geq 0$ and let $p: L \to X$ be a $\cD$-left fibration. If $\Val_k(L)$ is $S$-local for all $k \geq 0$, then evidently $\Val(L)$ is $S$-local. For the other side, notice we have an injective equivalence of $\cD$-spaces 
	$$\Val_k(L) \simeq \Val(L) \times_{\Val(X)} \Val_k(X).$$
	The right hand side is local with respect to $S$, hence the right hand is local as well.
\end{proof}

With this lemma at hand we have the following result, generalizing \cref{the:cov loc css}.

\begin{theone} \label{the:cdcov loc css s}
		Let $S$ be a set of monomorphisms of $\cD$-spaces. Let $X$ be a $\cD$-simplicial space such that $\Val_k(X)$ is $S$-local. Then the following is a Quillen $\sP(\cD)$-adjunction
		\begin{center}
			\adjun{(\sP(\cD\times\DD)_{/X})^{(lev_{CSS})_S}}{(\sP(\cD\times\DD)_{/X})^{\cD-cov_S}}{id}{id}
		\end{center}
		where the left hand side has the $S$-localized level-wise CSS model structure and the right hand side has the $S$-localized $\cD$-covariant model structure.
\end{theone} 

\begin{proof}
	Evidently the left adjoint preserves cofibrations as they are just monomorphisms and the right adjoint preserves fibrant objects by \cref{lemma:fib s over s loc}. So, the result follows from \cite[Corollary A.3.7.2]{lurie2009htt}
\end{proof}

If the base is appropriately fibrant the $S$-localized $\cD$-covariant model structure is invariant under presentations.

\begin{notone} \label{not:slsr}
	Let $\cD_1,\cD_2$ be two small categories and $(L,R)$ an adjunction
	\begin{center}
		\adjun{\sP(\cD_1)}{\sP(\cD_2)}{L}{R}.
	\end{center}
    We use the notation 
    \begin{center}
    	\adjun{\sP(\cD_1\times\DD)}{\sP(\cD_2\times\DD)}{sL}{sR}
    \end{center}
	for the adjunction obtained by applying the functor $\Fun(\DD^{op},-)$ to the adjunction $(L,R)$. In particular, for a given $\cD_1$-simplicial space $X$ we have $\Val_k(sLX)=L\Val_k(X)$.
\end{notone}

\begin{theone} \label{the:cdcov s inv qe}
 Let $\cD_1,\cD_2$ be two small categories. Moreover, let $S_1, S_2$ be sets of monomorphisms in $\sP(\cD_1),\sP(\cD_2)$, respectively. Assume there exists a Quillen adjunction 
 \begin{center}
 	\begin{tikzcd}
 		\sP(\cD_1)^{inj_{S_1}} \arrow[r, "L", shift left=1.8, "\bot"'] & \sP(\cD_2)^{inj_{S_2}} \arrow[l, "R", shift left=1.8]
 	\end{tikzcd}.
 \end{center}	
 Then this induces a Quillen adjunction (using \cref{not:slsr})
 \begin{equation} \label{eq:adj three} 
  	\begin{tikzcd}
 	(\sP(\cD_1\times\DD)_{/sRX})^{\cD_1-cov_{S_1}} \arrow[r, "sL", shift left=1.8, "\bot"'] & (\sP(\cD_2\times\DD)_{/X})^{\cD_2-cov_{S_2}} \arrow[l, "sR", shift left=1.8]
 \end{tikzcd},
 \end{equation} 
 for every level-wise fibrant $\cD_2$-simplicial spaces $X$ which is simplicial if $(L,R)$ is simplicial. 
 
 Moreover, if $(L,R)$ is a (simplicial) Quillen equivalence, then the adjunction \ref{eq:adj three} is a (simplicial) Quillen equivalence, both over an $S_2$-local $\cD_2$-simplicial space $X$ and its image $sRX$ as well as over an $S_1$-local $\cD_1$-simplicial space and its image $sLX$.
\end{theone}

\begin{proof}
	We can apply the functor $\Fun(\DD^{op},-)$ to the Quillen adjunction $(L,R)$ and give it the injective model structure to induce a Quillen adjunction \cite[Remark A.2.8.6]{lurie2009htt}
	\begin{equation} \label{eq:adj}
		\begin{tikzcd}
			\sP(\cD_1\times\DD)^{lev_{S_1}} \arrow[r, "sL", shift left=1.8, "\bot"'] & \sP(\cD_2\times\DD)^{lev_{S_2}} \arrow[l, "sR", shift left=1.8]
		\end{tikzcd},
	\end{equation}	
	where both sides have the localized level-wise model structure (\cref{the:sone local level stwo}). Moreover, if $(L,R)$ is an equivalence or simplicial, then the same holds for $(sL,sR)$, as everything is determined level-wise. 
	
	Let $X$ be level-wise $S_2$-local $\cD_2$-simplicial space.	This adjunction gives us an induced Quillen adjunction on over-categories
	\begin{center}
		\begin{tikzcd}
			(\sP(\cD_1\times\DD)_{/sRX})^{lev_{S_1}} \arrow[r, "sL", shift left=1.8, "\bot"'] & (\sP(\cD_2\times\DD)_{/X})^{lev_{S_2}} \arrow[l, "sR", shift left=1.8]
		\end{tikzcd},
	\end{center}
	which is simplicial if $(sL,sR)$ is one. Now we localize the left hand side with respect to the maps $F[d,0] \to F[d,n] \to sRX$ and the right hand side with respect to the map $F[d,0] \to F[d,n] \to X$ and, by \cref{the:cdcov loc css s}, the resulting localization model structures are the $\cD_1$-covariant ($\cD_2$-covariant) model structures, giving us the adjunction
	\begin{equation} \label{eq:adj two}
		\begin{tikzcd}
			(\sP(\cD_1\times\DD)_{/sRX})^{\cD_1-cov_{S_1}} \arrow[r, "sL", shift left=1.8, "\bot"'] & (\sP(\cD_2\times\DD)_{/X})^{\cD_2-cov_{S_2}} \arrow[l, "sR", shift left=1.8]
		\end{tikzcd}.
	\end{equation}
	We want to prove this is a Quillen adjunction. By \cref{cor:localizing equiv}, it suffices to observe that $sL(F[d,0]) \to sL(F[d,n]) \to X$ is an $S_2$-localized $\cD_2$-covariant equivalence. We will in fact prove it is a $\cD_2$-covariant equivalence. By \cref{rem:D s F} we have the bijection $F[d,n] \cong F[d,0] \times D[n]$ and $sL$ is defined level-wise, hence $sL(F[d,n]) \cong sL(F[d,0] \times D[n]) = sL(F[d,0]) \times D[n]$ and the map $sL(F[d,0]) \times D[0] \to sL(F[d,0]) \times D[n] \to X$ is a $\cD_2$-covariant equivalence, as the $\cD_2$-covariant model structure is enriched over $\sP(\cD)^{inj}$ (\cref{the:ncov model}). Moreover, if $(L,R)$ is simplicial, then this Quillen adjunction is evidently simplicial as well.
	
	Now, let us assume $(L,R)$ is in fact a Quillen equivalence. We already established that the adjunction given in \ref{eq:adj} is a Quillen equivalence. Let $X$ be an arbitrary level-wise $S$-local $\cD_1$-space. Then, there exists a level-wise fibrant $\cD_2$-space $Y$, such that $sRY$ is injectively equivalent to $X$. Indeed, $Y$ is given as the fibrant replacement of $sLX$ in the level-wise $S_2$-localized model structure on $\cD_2$-spaces. 
	
	Using the adjunction above (\ref{eq:adj two}) we can hence deduce that 
	\begin{center}
		\begin{tikzcd}
			(\sP(\cD_1\times\DD)_{/X})^{\cD_1-cov_{S_1}} \arrow[r, "sL", shift left=1.8, "\bot"'] & (\sP(\cD_2\times\DD)_{/sLY})^{\cD_2-cov_{S_2}} \arrow[l, "sR", shift left=1.8]
		\end{tikzcd}
	\end{center}
	is a Quillen adjunction, which is simplicial if $(L,R)$ is. It remains to prove that the Quillen adjunction is a Quillen equivalence. Via the Quillen equivalence \ref{eq:adj} we already know the adjunction of over-categories is a Quillen equivalence between the level-wise $S_1$-localized model structure on $\cD_1$-simplicial spaces and the level-wise $S_2$-localized model structure on $\cD_2$-simplicial spaces. We only need to prove this remains the case after localizing both sides with respect to the maps $F[d,0] \to F[d,n]$ on both sides. We want to use \cref{cor:localizing equiv} and for that we need to verify that a map $Z \to sLY$ is an $S_2$-localized $\cD_2$-left fibration if and only if $sR(Z) \to X$ is an $S_1$-localized $\cD_1$-left fibration.
	
	This follows immediately from the fact that $R$ reflects equivalences between fibrant objects and so 
	$$\Val_n(Z) \to \Val_0(Z) \times_{\Val(X)_0} \Val_n(X)$$
	is an equivalence of $S_2$-local $\cD_2$-spaces if and only if 
	$$R(\Val_nZ)=\Val_n(sR(Z)) \to \Val_0(sRZ) \times_{\Val_0(sRX)} \Val_n(sRX) = R\Val_0(Z) \times_{R\Val(X)_0} R\Val_n(X)$$
	is an equivalence of $S_1$-local $\cD_1$-spaces. 
\end{proof}

For the remainder of this section we want to focus our studies of $S$-localized $\cD$-left fibrations over $\cD$-simplicial spaces with weakly constant objects, which will at the end of this subsection include a discussion over $\sP(\cD)$-enriched (complete) Segal spaces.

\begin{propone} \label{prop:fibrancies}
	Let $\cD$ be a small category with terminal object $t$.	Let $S$ be a set of cofibrations of $\cD$-spaces with contractible codomain (\cref{def:contractible realization}). Let $X$ be a $\cD$-simplicial space with weakly constant objects (\cref{def:weakly constant objects}). Moreover, let $p:L \to X$ be an injective fibration of $\cD$-simplicial spaces. Then the following are equivalent.
	\begin{enumerate}
		\item $p$ is an $S$-localized $\cD$-left fibration.
		\item For every map $\sigma: F[d,k] \times \Delta[l] \to X$, the pullback map $\sigma^*p: \sigma^*L \to F[d,k] \times \Delta[l]$ is an $S$-localized $\cD$-left fibration.
		\item For every map $\sigma: F[d,k] \to X$, the pullback map $\sigma^*p: \sigma^*L \to F[d,k]$ is an $S$-localized $\cD$-left fibration.
		\item $p$ is a $\cD$-left fibration and for every point $\{x\}: D[0] \to X$ the fiber $\Fib_xL$ is fibrant in the diagonal $S$-localized injective model structure.
		\item $p$ is a $\cD$-left fibration and for every point $\{x\}: D[0] \to X$ the fiber $\Val(\Fib_xL)$ is fibrant in the $S$-localized injective model structure.
		\item $p$ is a $\cD$-left fibration and for every map $\VEmb(B) \to D[0] \to X$, the induced map 
		$$\Map_{/X}(\VEmb(B),L) \to \Map_{/X}(\VEmb(A),L)$$
		is a Kan equivalence.
	\end{enumerate}
\end{propone}

\begin{proof}
	$(1 \Rightarrow 2)$ This follows immediately from the fact that $S$-localized $\cD$-left fibrations are stable under pullback.
	
	$(2 \Leftrightarrow 3)$ As $X$ is already an injective fibration we only need to confirm it is local with respect to certain morphisms. However, being local is invariant under injective equivalences and so a map is local over $F[d,k]\times \Delta[l]$ if and only if it is local over $F[d,k]$ as the projection $F[d,k] \times \Delta[l] \to F[d,k]$ is an injective equivalence.
	
	$(4 \Leftrightarrow 5)$ If $p$ is a $\cD$-left fibration, then the fiber $\Fib_xL$ is fibrant in the $\cD$-diagonal model structure (\cref{the:diag local ncov}) and so the result follows from the characterization of the fibrant objects in the diagonal $S$-localized model structure (\cref{the:diagmodel s}).
	
	$(3 \Rightarrow 4)$ Let $L \to X$ be an injective fibration such that $\sigma^*L \to F[d,k]$ is an $S$-localized $\cD$-left fibration. Then in particular $\sigma^*L \to F[d,k]$ is a $\cD$-left fibration (\cref{def:cdleft s}) which, by \cref{lemma:cdleft local}, means $L \to X$ is a $\cD$-left fibration. Moreover, the fiber is an $S$-localized $\cD$-left fibration as it is just a special case as $D[0] = F[t,0]$. 
	
	$(5 \Leftrightarrow 6)$ Let $L \to X$ be a $\cD$-left fibration. Fix a map $\VEmb(A) \to \VEmb(B) \to D[0] \to  X$. We have the following diagram 
	\begin{center}
		\begin{tikzcd}
			\Map_{/X}(\VEmb(B),L) \arrow[r] \arrow[d, "\cong"] & \Map_{/X}(\VEmb(A),L) \arrow[d, "\cong"] \\
			\Map_{/\Val(X)}(B,\Val(L)) \arrow[r] \arrow[d, "\cong"] &\Map_{/\Val(X)}(A,\Val(L)) \arrow[d, "\cong"] \\
			\Map(B,\Fib_x\Val(L)) \arrow[r] & \Map(A,\Fib_x\Val(L))
		\end{tikzcd}.
	\end{center}
	Here the top vertical morphisms are isomorphisms of spaces by the enriched adjunction $(\VEmb,\Val)$ (as further explained in \cref{not:localized}) and the bottom vertical morphisms are isomorphisms of spaces as the map $\VEmb(B) \to X$ factors through $x:D[0] \to X$. 
	
	This implies that the top map is a Kan equivalence, meaning $L \to X$ is $S$-local, if and only if the bottom map is a Kan equivalence for all $x: D[0] \to X$, meaning $\Fib_xL$ is $S$-local for all $x: D[0] \to X$.  
	
	$(5 \Leftrightarrow 1)$ Let $p: L \to X$ be a $\cD$-left fibration, such that $\Fib_xL$ is $S$-local for all $x:D[0] \to X$.
	Applying the Quillen equivalence (\cref{the:dcov css inv}) to the zigzag (\cref{cor:zigzag better}) implies that there is a diagram of the form 
	\begin{equation}\label{eq:pb square}
		\begin{tikzcd}
			L \arrow[r, "\simeq"] \arrow[d, twoheadrightarrow, "p"] \arrow[dr, phantom, "\ulcorner", very near start] & \hat{L} \arrow[d, twoheadrightarrow, "\hat{p}"] & \tilde{L} \arrow[l, "\simeq"'] \arrow[d,twoheadrightarrow, "\tilde{p}"] \\
			X \arrow[r, "\simeq"] & \hat{X} &  \tilde{X} \arrow[l, "\simeq"']
		\end{tikzcd},
	\end{equation}
	where the horizontal morphisms are level-wise complete Segal space equivalences, the vertical maps are $\cD$-left fibrations and the squares are homotopy pullback squares. Moreover, by the Quillen equivalence (\cref{the:dcov invariant css s}) $p: L \to X$ is an $S$-local $\cD$-left fibration if and only if $\hat{p}:\hat{L} \to \hat{X}$ is an $S$-local $\cD$-left fibration if and only if $\tilde{p}: \tilde{L} \to \tilde{X}$ is an $S$-local $\cD$-left fibration. 
	
	It follows from \cref{cor:zigzag better} that $\Val\tilde{X}$ is constant. Moreover, by assumption $S$ has contractible codomains and so every morphism $\VEmb(B) \to \tilde{X}$ is equal to a constant map $\VEmb(B) \to D[0] \xrightarrow{ \ \{x\} \ } \tilde{X}$ (\cref{lemma:geom contr}), which gives us an isomorphism 
	$$\Map_{/\tilde{X}}(\VEmb(B),\tilde{L}) \cong \Map(\VEmb(B),\Fib_x\tilde{L})$$ 
	and similarly $\Map_{/\tilde{X}}(\VEmb(A),\tilde{L}) \cong \Map(\VEmb(A),\Fib_x\tilde{L})$. This proves that $\tilde{p}: \tilde{L} \to \tilde{X}$ is $S$-local if and only if $\Fib_x\tilde{L}$ is $S$-local for all maps $\{x\}:D[0] \to \tilde{X}$.
	
	Fix a map $\{x\}: D[0] \to \tilde{X}$.  By construction $\tilde{L}$ and $\tilde{X} \times_{\hat{X}} \hat{L}$ are equivalent $\cD$-left fibrations over $\tilde{X}$ and so in particular, by \cref{the:ncov equiv nlfib}, gives us an injective equivalence of fibers 
	\begin{equation} \label{eq:equiv one}
	 \Fib_{x}\tilde{L}\simeq \Fib_{x}\hat{L}. 
	\end{equation}
	Next, the map $X[t,0] \to \hat{X}[t,0]$ is surjective on path components (as explained in the proof of \cref{lemma:surjective path component}) and so there exists a point $\{\hat{x}\}: D[0] \to X$ such that images of $x$ and $\hat{x}$ are equivalent in $\hat{X}[t,0]$. By \cref{the:diag local ncov}, this implies we have an equivalence of fibers of $\hat{L}$
	\begin{equation} \label{eq:equiv two}
		\Fib_{x}\hat{L} \simeq \Fib_{\hat{x}}\hat{L}.
	\end{equation}
	 Finally, applying the same argument used to deduce the equivalence of fibers given in \ref{eq:equiv one} on the left hand homotopy pullback square of $\cD$-left fibrations gives us the equivalence of fibers 
	 \begin{equation} \label{eq:equiv three}
	   \Fib_{\hat{x}}L \simeq \Fib_{\hat{x}}\hat{L}. 
 	\end{equation} 
	 Combining the equivalences \ref{eq:equiv one}, \ref{eq:equiv two} and \ref{eq:equiv three} implies that $\Fib_x\tilde{L}$ is $S$-local if and only if $\Fib_{\hat{x}}L$ is $S$-local, and hence we are done.
\end{proof}

\cref{prop:fibrancies} has the following valuable corollary.

\begin{corone} \label{cor:fibrancies}
	Let $\cD$ be a small category with terminal object $t$.	Let $S$ be a set of cofibrations of $\cD$-spaces with contractible codomain. Let $X$ be a $\cD$-simplicial space with weakly constant objects. Then the $S$-localized $\cD$-covariant model structure on $(\sP(\cD\times\DD)_{/X})^{\cD-cov_S}$ coincides with the localization of the $\cD$-covariant model structure with respect to the maps $\VEmb(A) \to \VEmb(B) \to D[0] \to X$.
\end{corone} 

\begin{proof}
	Let $S_1$ be the set of cofibrations $\VEmb(A) \to \VEmb(B) \to D[0] \to X$ and $S_2$ the set of cofibrations $\VEmb(A) \to \VEmb(B) \to X$. Evidently the identity functor gives us a Quillen equivalence from the $\cD$-covariant model structure over $X$ to itself. By \cref{cor:localizing equiv}, in order to prove that it remains a Quillen equivalence after localizing we need to prove that a $\cD$-left fibration is $S_1$-local if and only if it is $S_2$-local. However, this is precisely the statement of $(1) \Leftrightarrow (6)$ in \cref{prop:fibrancies}. The fact that the identity functor is a Quillen equivalence and both model structures have the same cofibrations implies that they also have the same weak equivalences and fibrations, proving the model structures coincide.
\end{proof}

We move on the analogous result for \cref{the:pullback rfib}.

\begin{theone} \label{the:pullback rfib s}
	Let $\cD$ be a small category with terminal object $t$.	Let $S$ be a Cartesian set of cofibrations of $\cD$-spaces with contractible codomain. Let $X$ be a $\cD$-simplicial space with weakly constant objects, such that $\Val_k(X)$ is local with respect to $S$.
	Let $p:R \to X$ be a $\cD$-right or $\cD$-left fibration. Then the adjunction
	\begin{center}
		\adjun{(\sP(\cD\times\DD)_{/X})^{lev_{CSS_S}}}{(\sP(\cD\times\DD)_{/R})^{lev_{CSS_S}}}{p^*}{p_*}
	\end{center}
	is a Quillen adjunction. Here both sides have the $S$-localized level-wise CSS model structure (\cref{prop:levelcss s}) on the over-category.
\end{theone}

\begin{proof}
	By \cref{the:pullback rfib} we already know it is a Quillen adjunction of un-localized model structures. Hence, by \cref{cor:localizing equiv}, it suffices to prove that for every $f:\VEmb(A) \to \VEmb(B)$ in $S$ over $X$ the image $p^*(\VEmb(A)) \to p^*(\VEmb(B))$ is an $S$-localized level-wise complete Segal space equivalence.
	
	As $X$ has weakly constant objects and $S$ has contractible codomains, by \cref{cor:fibrancies}, it suffices to consider maps of the form $\VEmb(A) \to \VEmb(B) \to D[0] \to X$ and so we have 
	$$p^*(\VEmb(A) \to \VEmb(B)) \cong \VEmb(A) \times (D[0] \times_X R) \to \VEmb(B) \times (D[0] \times_X R)$$
	Hence, we need to prove this map is an $S$-localized level-wise complete Segal space equivalence over $X$. However, this follows directly from the fact that $S$ is Cartesian.
\end{proof}	

We now have a result analogous to \cref{the:nlfib pb nrfib}.

 \begin{theone} \label{the:rfib pb cov s}
		Let $\cD$ be a small category with terminal object $t$.	Let $S$ be a Cartesian set of cofibrations of $\cD$-spaces with contractible codomain. Let $X$ be a $\cD$-simplicial space with weakly constant objects. Then for any $\cD$-right fibration $R \to X$, the adjunction 
	\begin{center}
		\adjun{(\sP(\cD\times\DDelta)_{/X})^{\cD-cov_S}}{(\sP(\cD\times\DDelta)_{/X})^{\cD-cov_S}}{p_!p^*}{p_*p^*}
	\end{center}
	is a $\sP(\cD)$-enriched Quillen adjunction. Here both sides have the $S$-localized $\cD$-covariant model structure.
\end{theone}

\begin{proof}
	Following the arguments of the previous proof, we can similarly use \cref{the:nlfib pb nrfib}, \cref{cor:localizing equiv} and \cref{cor:fibrancies} to reduce it to showing that the map 
	$$p_!p^*(\VEmb(A) \to \VEmb(B)) \cong \VEmb(A) \times (D[0] \times_X R) \to \VEmb(B) \times (D[0] \times_X R)$$
	is an $S$-localized $\cD$-covariant equivalence over $X$. As $R$ is a $\cD$-right fibration, the fiber $D[0] \times_X R$ is diagonally constant (\cref{the:diag local ncov}), hence it suffices to prove that this map is a diagonal $S$-localized injective equivalence. However, this follows directly from the fact that this model structure is Cartesian.
\end{proof}

The previous results are important as they hold over a fixed base. If we are willing to vary it we have the following result in the Cartesian setting.

\begin{propone}
	Let $S$ be a Cartesian set of cofibrations of $\cD$-spaces.
	Let $g: C \to D$ be a cofibration of $\cD$-simplicial spaces and $p: L \to X$ an $S$-localized $\cD$-left fibration. 
	Then $\exp{g}{p}$ is also an $S$-localized $\cD$-left fibration. 
\end{propone}

\begin{proof}
	By \cref{lemma:exp cdleft} $\exp{g}{p}$ is a $\cD$-left fibration and so it suffices to prove that it is local with respect to $S$.
	It suffices to observe that $\exp{f}{\exp{g}{p}}$ 
	is a trivial injective fibration for every $f$ in $S$. By direct computation we have 
	$$ \exp{f}{\exp{g}{p}} \cong \exp{f \square g}{p} \cong \exp{g}{\exp{f}{p}}.$$
	The result now follows from the fact that $\exp{f}{p}$ is a trivial injective fibration (as the $S$-local model structure is Cartesian) 
	and the injective model structure is Cartesian (\cref{prop:injectivemod}).
\end{proof}

\begin{corone} 
 Let $S$ be a Cartesian set of cofibrations of $\cD$-spaces. Let $L \to X$ be an $S$-localized $\cD$-left fibration. Then for any $\cD$-simplicial space $Y$, $L^Y \to X^Y$ is also an $S$-localized $\cD$-left fibration.
\end{corone}

We will end this section with a study of how $S$-localization effects $\sP(\cD)$-enriched (complete) Segal spaces. First of all we have the following result.
 
\begin{lemone} \label{lemma:fiber vs global}
 Let $\cD$ be small category with terminal object $t$. Let $S$ be a geometrically contractible set of cofibrations of $\cD$-spaces. Let $X$ be an $\sP(\cD)$-enriched Segal space (\cref{def:enriched segal}), such that $\Val(X)$ is homotopically constant. Then the following are equivalent:
 \begin{itemize}
 	\item $\Val_k(X)$ is $S$-local for all $k \geq 0$ (meaning it is fibrant in the $S$-localized level-wise Segal space model structure \cref{prop:levelcss s}) 
 	\item $\Val_1(X)$ is $S$-local
 	\item for all objects $x,y$, $\map_X(x,y)$ (\cref{def:map}) is $S$-localized. 	
 \end{itemize}
\end{lemone}

\begin{proof}
	First of all $X$ is a level-wise Segal space and so we have injective equivalences $\Val_k(X) \simeq \Val_1(X) \times_{\Val_0(X)} ... \times_{\Val_0(X)} \Val_1(X)$, which implies that $\Val_k(X)$ is $S$-local for all $k \geq 0$ if and only if $\Val_1(X), \Val_0(X)$ are $S$-local.
	
	Now, by assumption $\Val_0(X)$ is constant and so, by \cref{lemma:geom equiv}, always $S$-local, which reduces everything to $\Val_1(X)$ being $S$-local. 
 
	We now have the injective fibration $(s,t):\Val_1(X) \to \Val_0(X) \times \Val_0(X)$. Given that $\Val_0(X)$ is $S$-local, $\Val_1(X)$ is $S$-local if and only if for every $A \to B \to \Val(X)$ in $S$, the map 
	\begin{equation}\label{eq:mapping spaces}
	  \Map_{/\Val(X)\times\Val(X)}(B,\Val_1(X)) \to \Map_{/\Val(X)\times\Val(X)}(A,\Val_1(X))
	\end{equation} 
	is an equivalence. By \cref{lemma:geom equiv}, the map $B \to \Val(X)$ factors as $B \to D[0] \to \Val(X)$, hence, the equivalence in \ref{eq:mapping spaces} is equivalent to $\Val_1(X) \to \Val(X) \times \Val(X)$ being fiber-wise $S$-local. However, the fiber is by definition the mapping space (\cref{def:map}) proving that it is equivalent to the mapping space being $S$-local.
\end{proof} 
 
\begin{defone} \label{def:sloc spd enriched css} 
	Let $S$ be a geometrically contractible set of cofibrations of $\cD$-spaces.
	An $S$-localized $\sP(\cD)$-enriched Segal space $X$ is an $\sP(\cD)$-enriched Segal space with $\Val(X)$ homotopically constant that satisfies one of the following equivalent conditions.
	\begin{enumerate}
		\item For all objects $x,y$, $\map_X(x,y)$ is $S$-local.
		\item The $\cD$-space $\Val_1(X)$ is $S$-local.
		\item The $\cD$-spaces $\Val_k(X)$ are $S$-local for $k \geq 0$. 
	\end{enumerate}
\end{defone}

We can specialize this characterization to the complete setting.

\begin{defone}
	A $\cD$-simplicial space $X$ is an $S$-local $\sP(\cD)$-enriched complete Segal space if it is an $\sP(\cD)$-enriched complete Segal space and $\Val_1(X)$ is $S$-local.
\end{defone}

We can also use the condition to finally adjust the model structure defined in \cref{the:spd enriched css}.

\begin{theone} \label{the:spd enriched css s}
 Let $\cD$ be a small category with terminal object. Let $S$ be a geometrically contractible set of cofibrations of $\cD$-spaces. There exists a simplicial left proper combinatorial model structure on $\cD$-simplicial spaces, called the $S$-localized $\sP(\cD)$-enriched complete Segal space model structure and denoted $\sP(\cD)^{\cD-CSS_S}$ with the following specification:
 \begin{itemize}
 	\item Cofibrations are monomorphisms. 
 	\item Fibrant objects are $S$-local $\sP(\cD)$-enriched complete Segal spaces. 
 \end{itemize} 
\end{theone} 

\begin{proof}
	We obtain this model structure by localizing the $\sP(\cD)$-enriched complete Segal space model structure (\cref{the:spd enriched css}) with respect to to $F[d] \times S$.
\end{proof}

 Let us use $S$-localized $\sP(\cD)$-enriched Segal spaces to give an analysis of fibrant replacements. If $X$ is level-wise Segal then the $\cD$-covariant fibrant replacement of $x:D[0] \to X$ in the $\cD$-covariant model structure is given by $X_{x/}= X^{D[1]} \times_X D[0]$ (\cref{the:levelwise Yoneda}). In general this $\cD$-left fibration will not be $S$-localized. However, in certain situations they do coincide. 
 
\begin{propone} \label{prop:yoneda s}
	Let $\cD$ be a small category with terminal object $t$.	Let $S$ be a geometrically contractible set of cofibrations of $\cD$-spaces. Let $X$ be an $S$-localized $\sP(\cD)$-enriched Segal space. Then the $\cD$-covariant fibrant replacement of an object $x$ in $X$ is also the $S$-localized $\cD$-covariant fibrant replacement.
\end{propone}

\begin{proof}
	By \cref{def:under segal} the fiber of $X_{x/} \to X$ over a point $y$ is given by $\map_X(x,y)$ which by assumption is $S$-local hence $X_{x/}\to X$ is an $S$-localized $\cD$-left fibration by \cref{prop:fibrancies}. 
\end{proof}

Using the same argument, we get the analogous result for the twisted arrow construction, generalizing \cref{prop:twisted}.

\begin{corone} \label{cor:twisted loc}
		Let $\cD$ be a small category with terminal object $t$.	Let $S$ be a geometrically contractible set of cofibrations of $\cD$-spaces. Let $W$ be an $S$-localized $\sP(\cD)$-enriched Segal space. Then the twisted arrow construction $\Tw(W) \to W^{op} \times W$ is an $S$-localized $\cD$-left fibration.  
\end{corone} 

Finally we can also use this result to adjust the source and target fibrations.

\begin{corone} \label{cor:target loc} 
 Let $\cD$ be a small category with terminal object $t$. Let $S$ be a geometrically contractible set of cofibrations of $\cD$-spaces. Let $W$ be an $S$-localized $\sP(\cD)$-enriched Segal space. Then the source and target fibrations $\Tw(W) \to W$, $\Tw(W)^{op} \to W$ are $\cD\times\DD$-left (right) fibrations with fibers $S$-localized Segal spaces, which are complete if and only if $W$ is complete. 
\end{corone} 

\subsection{$S$-Localized Grothendieck Construction for \texorpdfstring{$\cD$}{D}-Simplicial Spaces}\label{subsec:localized groth}
We move on to study the Grothendieck construction for $S$-localized $\cD$-left fibrations. Concretely, we want to prove the $S$-localized version of \cref{the:simp grothendieck}. As a first step we need to adjust \cref{prop:projinj model} and construct appropriate model structures on functor categories. 

\begin{propone} \label{prop:projinj model s}
	Let $\C$ be a small $\sP(\cD)$-enriched category and $S$ a set of cofibrations of $\sP(\cD)$-spaces. Then the category $\Fun(\C,\sP(\cD))$ has a combinatorial simplicial left proper model structure, the {\it $S$-localized projective model structure}, given as follows.
	\begin{itemize}
		\item A natural transformation $\alpha: F \to G$ is a fibration if for all $c$ the map of $\cD$-spaces $F(c) \to G(c)$ is a fibration in the $S$-localized injective model structure. 
		\item A natural transformation $\alpha: F \to G$ is a weak equivalence if for all $c$, the maps of $\cD$-spaces $F(c) \to G(c)$ is a weak equivalence in the $S$-localized injective model structure.
		\item A natural transformation $\alpha: F \to G$ between fibrant objects is a fibration (weak equivalence) if and only if it is a projective fibration (equivalence). 
	\end{itemize}
	Moreover, if $S$ is Cartesian then the model structure are enriched over the $S$-localized injective model structure on $\sP(\cD)^{inj}$. 
\end{propone}

\begin{proof}
	Let us first assume that $S$ is Cartesian. Then for every functor $F:\C \to \sP(\cD)$, we can use the fact that the $S$-localized injective model structure on $\sP(\cD)$ has functorial fibrant replacements, to construct a functorial fibrant replacement $F \to \hat{F}$. Moreover, the terminal $\cD$-space $D[0]$, the unit of the Cartesian product, is certainly $S$-localized. Hence, by \cite[Theorem 6.5]{moser2019enrichedproj}, the projective model structure on $\Fun(\C,\sP(\cD))$ exists and is Cartesian.
	
	Now, let us assume $S$ is not Cartesian. Then \cite[Theorem 6.5]{moser2019enrichedproj} does not apply directly. However, the Cartesian condition is only used in the proof of \cite[Theorem 6.5]{moser2019enrichedproj} to satisfy the path object condition that is the required in \cite[Theorem 2.2.1]{hkrs2017induced}. Even if the $S$-localized injective model structure on $\cD$-spaces is not Cartesian, it is still simplicial (\cref{prop:injectivemod s}) and we can use the simplicial tensor to construct a path factorization 
	$$X \times \partial \Delta[1] \to X \times \Delta[1] \to X$$
	for every $\cD$-simplicial space $X$ in the $S$-localized injective model structure. With this one change the remaining steps of the proof of \cite[Theorem 6.5]{moser2019enrichedproj} remain the same and so the result follows.
\end{proof}

We constructed the $S$-localized projective model structure as a projective model structure on a $S$-localized injective model structure. However, we can also obtain it as an $S$-localized model structure on a projective model structure.

\begin{lemone}\label{lemma:projective s localization}
	Let $\C$ be a small $\sP(\cD)$-enriched category and $S$ a set of cofibrations of $\sP(\cD)$-spaces. A $\sP(\cD)$-enriched functor $G: \C \to \sP(\cD)$ is fibrant in the $S$-localized projective model structure if and only if for every object $c$, and map $f: A \to B$ in $S$, the induced map 
	$$\Nat(\uMap(c,-) \times \{B\},G) \to \Nat(\uMap(c,-) \times \{A\},G)$$
	is an equivalence of spaces.  
\end{lemone}

\begin{proof}
	By definition $G$ is fibrant if and only if $G(c)$ is an $S$-local $\cD$-space. By the enriched Yoneda lemma (\cref{lemma:enriched yoneda}) we have an isomorphism of $\cD$-spaces $\uNat(\uMap(c,-),G) \cong G(c)$. Moreover, by the tensor-hom adjunction we have a natural isomorphisms 
	$$\Map_{\sP(\cD)}(B,G(c))\cong \Map_{\sP(\cD)}(B,\uNat(\uMap(c,-),G)) \cong \Nat(\{B\} \times \uMap(c,-),G)$$
	Hence, $G(c)$ being $S$-local for all $c$ is equivalent to $G$ being local with respect to $\{A\} \times \uHom(c,-) \to \{B\} \times \uHom(c,-)$ for all $c$, giving us the desired result.
\end{proof}

We can use this alternative characterization to construct an $S$-localized injective model structure as well.

\begin{lemone} \label{lemma:sloc proj vs inj}
	Let $\C$ be small $\sP(\cD)$-enriched category and $S$ a set of cofibrations of $\sP(\cD)$-spaces. Then the $\sP(\cD)$-Quillen equivalence 
	\begin{center}
		\adjun{\Fun(\C,\sP(\cD))^{proj}}{\Fun(\C,\sP(\cD))^{inj}}{\id}{\id}
	\end{center}
	between the projective and injective model structures remain a Quillen equivalence if we localize both sides with respect to $\uHom(c,-) \times S$.
	Moreover, for a given morphism $\alpha: F \to G$ the following are equivalent.
	\begin{enumerate}
		\item $\alpha$ is an equivalence in the localization of the projective model structure with respect to $\uHom(c,-) \times S$.
		\item $\alpha$ is an equivalence in the localization of the projective model structure with respect to $\uHom(c,-) \times S$.
		\item For every object $c$ in $\C$, the morphism of $\cD$-spaces $F(c) \to G(c)$ is an $S$-localized injective equivalence. 
	\end{enumerate}
\end{lemone}

\begin{proof}
	The existence of the Quillen equivalence is a direct application of \cref{the:localizing equiv} to the Quillen equivalence given in \cref{prop:projinj model}. We move on to confirm the equivalence of the three conditions.
	
	$(1 \Leftrightarrow 3)$ Follows directly from \cref{lemma:projective s localization}.
	
	$(1 \Leftrightarrow 2)$ If $F$ and $G$ are projectively cofibrant, then this follows from the fact that left Quillen functors of Quillen equivalences reflect equivalences. For a general morphism $\alpha:F \to G$ we have the following diagram 
	\begin{center}
		\begin{tikzcd}
			\hat{F} \arrow[r, "\hat{\alpha}"] \arrow[d, twoheadrightarrow, "\simeq"]& \hat{G} \arrow[d, twoheadrightarrow, "\simeq"] \\
			F \arrow[r, "\alpha"] & G
		\end{tikzcd},
	\end{center}
	which satisfies the following conditions:
	\begin{itemize}
		\item $\hat{F},\hat{G}$ are projectively cofibrant.
		\item The vertical maps are trivial fibrations in the projective model structure, meaning for every object $c$, $\hat{F}(c) \to F(c)$ and $\hat{G}(c) \to G(c)$ are trivial fibrations in the injective model structure on $\cD$-spaces. 
	\end{itemize}  
	Given these conditions, $\alpha_c: F(c) \to G(c)$ is an $S$-localized injective equivalence if and only if $\hat{\alpha}_c$ is an $S$-localized injective equivalence for every object $c$ in $\C$, which combined with our previous analysis of the cofibrant case finishes the proof.
\end{proof}

\begin{propone} \label{prop:inj mod s}   
	Let $\C$ be a small $\sP(\cD)$-enriched category and $S$ a set of cofibrations of $\sP(\cD)$-spaces. Then the category $\Fun(\C,\sP(\cD))$ has a combinatorial simplicial proper $\sP(\cD)^{inj}$-enriched model structure, the {\it $S$-localized injective model structure}, given as follows.
	\begin{itemize}
		\item A natural transformation $\alpha: F \to G$ is a cofibration if for all $c$ the map of $\cD$-spaces $F(c) \to G(c)$ is a cofibration in the $S$-localized injective model structure. 
		\item A natural transformation $\alpha: F \to G$ is a weak equivalence if for all $c$, the maps of $\cD$-spaces $F(c) \to G(c)$ is a weak equivalence in the $S$-localized injective model structure.
		\item An object $F$ is fibrant if it is injectively fibrant and $F(c)$ is $S$-local for every object $c$ in $\C$.
		\item A natural transformation $\alpha: F \to G$ between fibrant objects is a fibration (weak equivalence) if and only if it is a projective fibration (equivalence). 
	\end{itemize}
	Moreover, if $S$ is Cartesian then the model structure are enriched over the $S$-localized injective model structure on $\sP(\cD)^{inj}$. 
\end{propone}

\begin{proof}
	Given that the $S$-localized injective model structure on $\sP(\cD)$ satisfies the condition of \cite[Theorem 4.4]{moser2019enrichedproj} we can deduce the existence of an injective model structure on $\Fun(\C,\sP(\cD)^{inj_S})$, where a morphism $\alpha: F \to G$ is a weak equivalence or a cofibration if and only if $\alpha_c: F(c) \to G(c)$ is a weak equivalence or cofibration in the $S$-localized injective model structure on $\cD$-spaces. The resulting model structure already has the expected cofibrations and weak equivalences, however, we cannot yet characterize the fibrant objects. This requires an additional step.
	
	Notice, we can construct a second model structure on $\Fun(\C,\sP(\cD)^{inj})$, by localizing it with respect to the set of cofibrations 
	$$\{\uMap(c,-) \times \{A \} \to \uMap(c,-) \times \{B\}: c \in \Obj_\C, A \to F \in S\}.$$
	The resulting localization model structure has the expected fibrant objects and morphisms between fibrant objects behave as we want. 
	
	In order to finish the proof, we need to show that these two model structure coincide. By construction they have the same cofibrations (level-wise monomorphisms). Hence, it order to show they coincide it suffices to prove they have the same weak equivalences. Let $\alpha:F \to G$ be an arbitrary morphism in $\Fun(\C,\sP(\cD))$. Then, by \cref{lemma:sloc proj vs inj}, it is an $S$-localized injective equivalence if and only if it is a level-wise $S$-localized equivalence and hence we are done.
\end{proof}

Before we move on to the main proof, we need the following useful characterization of the $S$-localized $\cD$-covariant model structure over $N_\cD\C$.

\begin{lemone} \label{lemma:grothendieck technical}
	Let $S$ be a set of inclusions of $\sP(\cD)$-spaces. Let $\C$ be a small $\sP(\cD)$-enriched category. A $\cD$-left fibration $L \to N_\cD\C$ is $S$-local if and only if for all $A \to B$ in $S$ and objects $c$ in $\C$ the map 
	$$\Map_{/N_\cD\C}(N_\cD\C_{c/} \times \VEmb(B),L) \to \Map_{/N_\cD\C}(N_\cD\C_{c/} \times \VEmb(A),L)$$
	is a Kan equivalence.
\end{lemone}

\begin{proof}
	Notice, $\Val(N_\cD\C)$ is equal to the constant diagram with value the set $\Obj_\C$. Hence, any map $\VEmb(B) \to \C$ is necessarily constant. Fix a map $\VEmb(B) \to N_\cD\C$ and assume it has constant value the object $c$. Then we have following diagram 
	\begin{center}
		\begin{tikzcd}
			\Map_{/N_\cD\C}(N_\cD\C_{c/} \times \VEmb(B),L) \arrow[r] \arrow[d, "\simeq"]  &  \Map_{/N_\cD\C}(N_\cD\C_{c/} \times \VEmb(A),L) \arrow[d, "\simeq"]  \\
			\Map_{/N_\cD\C}(\VEmb(B),L) \arrow[r]  &  \Map_{/N_\cD\C}(\VEmb(A),L)
		\end{tikzcd}
	\end{center}
	The vertical maps are Kan equivalences by \cref{the:levelwise Yoneda} and so the top map is a Kan equivalence if and only if the bottom map is a Kan equivalence giving us the desired result.
\end{proof}

We can finally prove the $S$-localized version of the Grothendieck construction.

\begin{theone} \label{the:grothendieck simp s}
	Let $\C$ be a small $\sP(\cD)$-enriched category. Then the $\sP(\cD)$-enriched adjunctions defined in \cref{the:simp grothendieck}
	\begin{center}
		\begin{tikzcd}[row sep=0.5in, column sep=0.9in]
			\uFun(\C,\sP(\cD)^{inj_S})^{proj} \arrow[r, shift left = 1.8, "\sint_{\cD/\C}"] & 
			(\sP(\cD\times\DD)_{/N_\cD\C})^{\cD-cov_S} \arrow[l, shift left=1.8, "\sH_{\cD/\C}", "\bot"'] \arrow[r, shift left=1.8, "\sbT_{\cD/\C}"] &
			\uFun(\C,\sP(\cD)^{inj_S})^{inj} \arrow[l, shift left=1.8, "\sbI_{\cD/\C}", "\bot"']
		\end{tikzcd}
	\end{center}
	are $\sP(\cD)$-enriched Quillen equivalences. Here the middle has the $S$-localized $\cD$-covariant model structure over $N_\cD\C$
	and the two sides have the projective model structure on the $S$-localized injective model structure.
\end{theone}

\begin{proof}
	We already know from \cref{the:simp grothendieck} that both adjunctions are Quillen equivalences between non-localized model structures and we want to use \cref{cor:localizing equiv} to prove they remain Quillen equivalences after localization, by confirming the condition stated there. 
	
	For the first adjunction we need to prove that a $\cD$-left fibration $L \to N_\cD\C$ is $S$ localized if and only it is local with respect to the map $\sint_{\cD/\C} (\uMap(c,-) \times \{A\} \to \uMap(c,-) \times \{B\})$. By \cref{ex:sintcd rep} and \cref{ex:sintcd constant} this map is equal to $N_\cD\C_{c/} \times A \to N_\cD\C_{c/} \times B$. Now the result follows from \cref{lemma:grothendieck technical}.
	
	For the second adjunction, we need to prove that an injectively fibrant $\sP(\cD)$-enriched functor $F: \C \to \sP(\cD)$ is fibrant in the $S$-localized injective model structure on $\uFun(\C,\sP(\cD))$ if and only if $F$ is local with respect  $\sbT_{\cD/\C}(\VEmb(A) \to \VEmb(B) \to N_\cD\C$). Every map $\VEmb(B) \to N_\cD\C$ is necessarily constant, so let us assume the value is the object $c$. By \cref{ex:sbT constant}, $\sbT_{\cD/\C}(\VEmb(A) \to \VEmb(B)) = \uMap(c, -) \times \{A\} \to \uMap(c,-) \times \{B\}$. The result now follows from the characterization of fibrant objects given in \cref{prop:inj mod s}.
\end{proof}

The Quillen equivalence gives us a possibility to understand $S$-localized $\cD$-covariant equivalences over $N_\cD\C$.

\begin{corone} \label{cor:reconition principle s}
	Let $\C$ be a small $\sP(\cD)$-enriched category.
	A map $Y \to Z$ over $N_\cD\C$ is an $S$-localized $\cD$-covariant equivalence if and only if for every object $c$ in $\C$ the map 
	$$Y \times_{N_\cD\C} N_\cD\C_{/c} \to Z \times_{N_\cD\C} N_\cD\C_{/c}$$
	is a diagonal $S$-localized equivalence.
\end{corone}

\begin{proof}
	By \cref{the:cdcov s} all objects in the $S$-localized $\cD$-covariant model structure are cofibrant, hence, by \cref{the:grothendieck simp s}, $\sbT_{\cD/\C}$ reflects weak equivalences. By \cref{prop:inj mod s}, the equivalences in the $S$-localized injective model structure are given point-wise and by \cref{def:sbTcd}, $\sbT_{\cD/\C}(Y)(c) = \fDiag(Y \times_{N_\cD\C} N_\cD\C_{/c})$ and so the result follows.
\end{proof}

We will later (\cref{the:recognition principle s}) generalize this to $S$-localized $\cD$-covariant equivalences over more general $\cD$-simplicial spaces. Continuing our previous steps, we now want a version of these Quillen equivalences where the two model structures on $\Fun(\C,\sP(\cD))$ coincide. 

\begin{theone} \label{the:grothendieck double proj s}
	Let $\C$ be a small $\sP(\cD)$-enriched category. Assume that $\sbI_{\cD/\C}$ takes projective fibrations to injective fibration. Then, the two $\sP(\cD)$-enriched adjunctions  
	\begin{center}
	\begin{tikzcd}[row sep=0.5in, column sep=0.9in]
		\uFun(\C,\sP(\cD)^{inj_S})^{proj} \arrow[r, shift left = 1.8, "\sint_{\cD/\C}"] & 
		(\sP(\cD\times\DD)_{/N\C})^{\cD-cov_S} \arrow[l, shift left=1.8, "\sH_{\cD/\C}", "\bot"'] \arrow[r, shift left=1.8, "\sbT_{\cD/\C}"] &
		\uFun(\C,\sP(\cD)^{inj_S})^{proj} \arrow[l, shift left=1.8, "\sbI_{\cD/\C}", "\bot"']
	\end{tikzcd}
\end{center}
	are Quillen equivalences. Here the two $\Fun(\C,\sP(\cD))$ have the projective model structure and $\sP(\cD\times\DD)_{/N\C}$ has the $\cD$-covariant model structure. Moreover, if $S$ is Cartesian, then these are $\P(\cD)^{inj_S}$-enriched Quillen equivalences. 
\end{theone} 

\begin{proof}
	Following \cref{the:grothendieck double proj} the assumption implies that $\sbI_{\cD/\C}$ is a right Quillen functor from the projective model structure to the $\cD$-covariant model structure. Moreover, by \cref{the:grothendieck simp s}, $\sbI_{\cD/\C}$ preserves $S$-local objects and so it is also a right Quillen functor from the projective model structure on the $S$-localized injective model structure to the $S$-localized $\cD$-covariant model structure. 
	
	In order to deduce it is a Quillen equivalence, we can again use $2$-out-of-$3$ to the diagram above using the results from \cref{the:grothendieck simp s}.
\end{proof}

We can now generalize \cref{the:strictification dleft fib}.

\begin{theone} \label{the:strictification dleft fib s}
	Let $\cD$ be a small category with terminal object.
	Let $X$ be a $\cD$-simplicial space with weakly constant objects. Then we have the following diagram of $\sP(\cD)$-enriched Quillen equivalences.
	\begin{center}
		\begin{tikzcd}[row sep=0.3in, column sep=0.9in]
			\qcatFun(\C_X,\sP(\cD)^{inj_S})^{proj} \arrow[r, shift left = 1.8, "\sint_{\cD/\C_X}"] & 
			(\sP(\cD\times\DD)_{/N_\cD\C_X})^{\cD-cov_S} \arrow[l, shift left=1.8, "\sH_{\cD/\C_X}", "\bot"'] \arrow[r, shift left=1.8, "\sbT_{\cD/\C_X}"] \arrow[d, shift left =1.8, "(r_X)_!"] & \qcatFun(\C_X,\sP(\cD)^{inj_S})^{inj} \arrow[l, shift left=1.8, "\sbI_{\cD/\C_X}", "\bot"'] \\
			&(\sP(\cD\times\DD)_{/\cF_\cD X})^{\cD-cov_S} \arrow[u, shift left=1.8, "(r_X)^*", "\rbot"'] \arrow[d, shift right=1.8, "(u_X)^*"'] & \\
			&(\sP(\cD\times\DD)_{/\cC_\cD X})^{\cD-cov_S} \arrow[u, shift right=1.8, "(u_X)_!"', "\rbot"] \arrow[d, shift left=1.8, "(c_X)_!", "\rbot"'] & \\
			&(\sP(\cD\times\DD)_{/X})^{\cD-cov_S}\arrow[u, shift left=1.8, "(c_X)^*"] & 
		\end{tikzcd}
	\end{center}
\end{theone}

\begin{proof}
	The horizontal Quillen equivalences follow from \cref{the:grothendieck simp s} and the vertical Quillen equivalences from \cref{the:dcov invariant css s} combined with the fact that $r_X,c_X,u_X$ are level-wise complete Segal space equivalences (\cref{lemma:zigzag strictification}).	
\end{proof}

We move on to find the related result regarding quasi-categories. 

\begin{notone} \label{not:underlying qcat s}
 Generalizing \cref{not:underlying qcat}, we denote the underlying quasi-category of the $S$-localized injective model structure on $\sP(\cD)$ by $\sP(\cD)[\cW^{-1}_S]$ and similarly use $\LFib^S_{\cD/X}$ denote the quasi-category of $S$-localized $\cD$-left fibrations over $X$.
\end{notone}

\begin{remone} \label{rem:proj not fun}
 Notice, if $S$ is not Cartesian, then $\sP(\cD)[\cW^{-1}_S]$ is not a Cartesian closed quasi-category and so in particular also not enriched over itself. This implies that the underlying quasi-category of the projective model structure $\Fun(\C,\sP(\cD))[\cW^{-1}_S]$ is not equivalent to a functor quasi-category of some enriched quasi-categories. On the other hand, if $S$ is Cartesian, by \cref{the:cdcov s}, $\sP(\cD)[\cW^{-1}_S]$ is Cartesian closed and we have an equivalence with a functor quasi-category 
 $$
 	\Fun(\C,\sP(\cD))[\cW^{-1}_S] \simeq \Fun_{\QCat_{\sP(\cD)[\cW^{-1}_S]}}(\C,\sP(\cD)[\cW^{-1}_S]),
 $$
 analogous to \ref{eq:functor cat}.
\end{remone} 

 \begin{corone} \label{cor:equiv qcat s}
	Let $\cD$ be a small category with terminal object.	Let $X$ be a $\cD$-simplicial space with weakly constant objects.
	Then we have an equivalence of quasi-categories, 
	\begin{center}
		\begin{tikzcd}[row sep=0.3in, column sep=0.9in]
			\Fun(\C_X,\sP(\cD))[\cW^{-1}_S] \arrow[r, shift left = 1.8, "\sint_{\cD/\C_X}"] & 
			\LFib^S_{\cD/N_\cD\C_X} \arrow[l, shift left=1.8, "\sH_{\cD/\C_X}", "\bot"'] \arrow[r, shift left=1.8, "\sbT_{\cD/\C_X}"] \arrow[d, shift left =1.8, "(r_X)_!"] & \Fun(\C_X,\sP(\cD))[\cW^{-1}_S] \arrow[l, shift left=1.8, "\sbI_{\cD/\C_X}", "\bot"'] \\
			&\LFib^S_{\cD/\cF_\cD X} \arrow[u, shift left=1.8, "(r_X)^*", "\rbot"'] \arrow[d, shift right=1.8, "(u_X)^*"'] & \\
			&\LFib^S_{\cD/\cC_\cD X} \arrow[u, shift right=1.8, "(u_X)_!"', "\rbot"] \arrow[d, shift left=1.8, "(c_X)_!", "\rbot"'] & \\
			&\LFib^S_{\cD/X} \arrow[u, shift left=1.8, "(c_X)^*"] & 
		\end{tikzcd},
	\end{center}
 which is an equivalence of $\sP(\cD)^S$-enriched quasi-categories and equivalent to the functor quasi-category $\Fun_{\QCat_{\sP(\cD)[\cW^{-1}_S]}}(\C,\sP(\cD)[\cW^{-1}_S])$ if $S$ is Cartesian.
\end{corone} 

Again we can apply \cite[Proposition 2.1.12]{riehlverity2018elements} to get the following corollary.

 \begin{corone} 
	Let $\cD$ be a small category with terminal object.	Let $X$ be a $\cD$-simplicial space with weakly constant objects.
	There is an equivalence of quasi-categories
	\begin{center}
		\begin{tikzcd}[row sep=0.5in, column sep=0.9in]
			\Fun(\C_X,\sP(\cD))[\cW_S^{-1}] \arrow[r, shift left = 1.8] & 
			\LFib^S_{\cD/X} \arrow[l, shift left=1.8, "\bot"'] 
		\end{tikzcd}.
	\end{center}
 Moreover, if $S$ is Cartesian then there is an equivalence $\sP(\cD)[\cW^{-1}_S]$-enriched quasi-categories 
 \begin{center}
 	\begin{tikzcd}[row sep=0.5in, column sep=0.9in]
 		\Fun_{\QCat_{\sP(\cD)[\cW^{-1}_S]}}(\C_X,\sP(\cD)[\cW^{-1}_S]) \arrow[r, shift left = 1.8] & 
 		\LFib^S_{\cD/X} \arrow[l, shift left=1.8, "\bot"'] 
 	\end{tikzcd}.
 \end{center}
\end{corone}

We end this section with an analysis of the {\it universal $S$-localized $\cD$-left fibration}, generalizing \cref{rem:universal lfib}. Following \ref{eq:completion}, let $\widehat{\pi_*}: \widehat{N_\cD\sP(\cD)^S_{D[0]/}} \to \widehat{N_\cD\sP(\cD)^S}$ be the completion of the map $\pi_*:N_\cD\sP(\cD)^S_{D[0]/} \to N_\cD\sP(\cD)^S$. Let $\CSS_{\sP(\cD)^S}$ be the underlying quasi-category of the $S$-localized $\sP(\cD)$-enriched complete Segal space model structure (\cref{the:spd enriched css s}). Let $\LFib_\cD^S \to \CSS_{\sP(\cD)^S}$ be the sub-right fibration of $S$-localized $\cD$-left fibrations in  $\sOall_{\CSS_{\sP(\cD)^S}}$ given analogous to \cref{cor:lfib classifier}. Again, we have a map of quasi-categories $(\CSS_{\sP(\cD)^S})_{/\widehat{N_\cD\sP(\cD)^S}} \to \LFib^S_\cD$ over $\CSS_{\sP(\cD)^S}$, that takes a map $F:\C \to \widehat{N_\cD\sP(\cD)^S}$ to the pullback $F^*\widehat{\pi_*}$. We now have the following analogous result.

 \begin{corone}\label{cor:lfib classifier s}
 	We have the following diagram of right fibrations over the inclusion $\CSS_{\sP(\cD)^S} \hookrightarrow \CSS_{\sP(\cD)}$
 	\begin{center}
 		\begin{tikzcd}
 			(\CSS_{\sP(\cD)^S})_{/\widehat{N_\cD\sP(\cD)^S}} \arrow[r, "\simeq"] \arrow[d, hookrightarrow] &  \LFib^S_\cD \arrow[d, hookrightarrow] \\
 			(\CSS_{\sP(\cD)})_{/\widehat{N_\cD\sP(\cD)}} \arrow[r, "\simeq"] & \LFib_\cD
 		\end{tikzcd}
 	\end{center}
 	where the horizontal maps are equivalences and the vertical maps are inclusions.  
\end{corone} 

\begin{remone}\label{rem:universal lfib s}
	\cref{cor:lfib classifier s} implies that $\widehat{\pi_*}: \widehat{N_\cD\sP(\cD)^S_{D[0]/}} \to \widehat{N_\cD\sP(\cD)^S}$ is a {\it universal $S$-localized $\cD$-left fibration}. More interestingly, building on \cref{rem:universal lfib}, the inclusions of right fibrations in \cref{cor:lfib classifier s} imply that for every $\cD$-left fibration $p:L \to N_\cD\C$, there is a (homotopically unique) pullback of $\cD$-simplicial spaces
	\begin{center}
		\begin{tikzcd}[row sep=0.5in, column sep=0.5in]
			L \arrow[r, dashed] \arrow[d,"p"] \arrow[dr, phantom, "\ulcorner", very near start] \arrow[rr, bend right=20]& \widehat{N_\cD\sP(\cD)}^S_{D[0]/} \arrow[r] \arrow[d] \arrow[dr, phantom, "\ulcorner", very near start] & \widehat{N_\cD\sP(\cD)}_{D[0]/}  \arrow[d] \\
			N_\cD\C \arrow[r, dashed] \arrow[rr, bend right=20]&  \widehat{N_\cD\sP(\cD)}^S \arrow[r, hookrightarrow] & \widehat{N_\cD\sP(\cD)}
		\end{tikzcd}.
	\end{center}
	that factors through $N_\cD\sP(\cD)^S$ if and only if $p$ is $S$-localized.
\end{remone} 

Again, \cite[Proposition 2.4]{rasekh2021univalence} implies the following result regarding univalence.

\begin{corone} \label{cor:universal lfib univalent s}
	The universal $\cD$-left fibration $\widehat{\pi_*}: \widehat{N_\cD\sP(\cD)}_{D[0]/} \to \widehat{N_\cD\sP(\cD)}$ is univalent.
\end{corone}

Notice that we could have deduced \cref{cor:universal lfib univalent s} also by using \cref{cor:universal lfib univalent} with \cite[Proposition 2.5]{rasekh2021univalence}.

\subsection{Further Properties of the $S$-Localized \texorpdfstring{$\cD$}{D}-Left Fibrations}\label{subsec:sloc features}
We move on to understand the $S$-localized $\cD$-covariant model structure, in particular its equivalences.

 \begin{theone} \label{the:cdcov s equiv lfib}
	Let $L$ and $M$ be two $S$-localized $\cD$-left fibrations over $X$.
	Let $g:L \to M$ be a map over $X$. Then the following are equivalent.
	\begin{enumerate}
		\item $g: L \to M$ is an $S$-localized $\cD$-covariant equivalence.
		\item $g: L \to M$ is an injective equivalence of $\cD$-simplicial spaces.
		\item $\Val(g): \Val(L) \to \Val(M)$ is an injective equivalence of $\cD$-spaces.
		\item For every $\{ x \}: F[d,0] \to X$, the map  $\Fib_x \Val(L) \to \Fib_x \Val(M)$ is an injective equivalence of $\cD$-simplicial spaces.
		\item For every $\{x \}: F[d,0] \to X$, the map $\Fib_x(L) \to \Fib_x (M)$ is a $\cD$-diagonal equivalence of $\cD$-simplicial spaces.
	\end{enumerate}
\end{theone}

\begin{proof}
	The first two conditions are equivalent as $L \to X$, $M \to X$ are fibrant objects in the $S$-localized $\cD$-covariant model structure and \cref{the:cdcov s}. The other condition are all equivalent to the second condition by \cref{the:ncov equiv nlfib} and the fact that $L \to X, M \to X$ are $\cD$-left fibrations. 
\end{proof}

\begin{propone} \label{prop:cdcov s equiv lfib}
    Let $\cD$ have a terminal object $t$, let $S$ be a set of cofibrations and assume $X$ has weakly constant objects. 
	Let $p: L \to X$, $q: M \to X$ be $\cD$-left fibrations (not necessarily localized) and	let $f:L \to M$ be a map over $X$. 
	Then the following are equivalent:
	\begin{enumerate}
		\item $f$ is an $S$-localized $\cD$-covariant equivalence.
		\item For every object $\{x\}: F[t,0] \to X$, the induced map on fibers 
		$$\Fib_x L \to \Fib_x M$$
		is a diagonal $S$-localized injective equivalence of $\cD$-simplicial spaces.
		\item For every object $\{x\}: F[t,0] \to X$, the induced map on fibers 
		$$\Val(\Fib_x L) \to \Val(\Fib_x M)$$
		is an $S$-localized injective equivalence of $\cD$-spaces.
	\end{enumerate}
\end{propone}

\begin{proof} 
	{\it (1) $\Leftrightarrow$ (2)} 
	The map $L \to M$ is level-wise CSS equivalent (via a zigzag of equivalences by \cref{cor:left fib strict}) to a map of the form $\sint_{\cD/\C_X} G \to \sint_{\cD/\C_X} H$, where $G:\C \to \sP(\cD)$ is an $\sP(\cD)$-enriched functor with $G(c) \simeq \Fib_cL$ (\cref{cor:value preservation}) and the same for $H$.  Hence, by the Quillen equivalence in \cref{the:strictification dleft fib}, $L \to M$ is an equivalence in the $S$-localized $\cD$-covariant model structure if and only if $G \to H$ is a level-wise equivalence in the $S$-localized injective model structure, which is precisely the statement that $\Fib_cL \to \Fib_cM$ is a diagonal $S$-localized injective equivalence.

	{\it (2) $\Leftrightarrow$ (3)}
	This follows from the fact that $L \to X$ is a $\cD$-left fibration and so 
	$\VEmb\Val\Fib_x(L) \to \Fib_x(L)$ is an injective equivalence (\cref{the:diag local ncov}). 
\end{proof}
	
\begin{theone} \label{the:recognition principle s}
	Let $\cD$ have a terminal object $t$, let $S$ be a set of cofibrations and assume $X$ has a weakly constant objects. 
	A map $g:Y \to Z$ of $\cD$-simplicial spaces over $X$ is an equivalence in the $S$-localized $\cD$-covariant model structure if and only if for each map $\{x\}: F[t,0] \to X$, the induced map 
	$$ Y \underset{X}{\times} R_x \to Z \underset{X}{\times} R_x$$
	is an equivalence in the diagonal $S$-localized injective model structure.
	Here $R_x$ is a choice of $\cD$-right fibrant replacement of the map $\{x\}$.
\end{theone}

\begin{proof}
	Let $\hat{g}: \hat{Y} \to \hat{Z}$ be a fibrant replacement of $g$ in the $\cD$-covariant model structure (note: not localized). 
	Moreover, let $\{x \}: F[t,0] \to X$ be a vertex in $X$.
	This gives us the following zigzag of maps:
	\begin{center}
		\begin{tikzcd}[row sep=0.25in, column sep=0.25in]
			\hat{Y} \underset{X}{\times} F\leb t\comma0\reb \arrow[d, "\simeq"'] \arrow[r] & \hat{Z} \underset{X}{\times} F\leb t\comma0\reb \arrow[d, "\simeq"] \\
			\hat{Y} \underset{X}{\times} R_x \arrow[r] &\hat{Z} \underset{X}{\times} R_x \\
			Y \underset{X}{\times} R_x \arrow[r] \arrow[u, "\simeq"] & Z \underset{X}{\times} R_x \arrow[u, "\simeq"']
		\end{tikzcd}
		.
	\end{center}
	According to \cref{the:nlfib pb nrfib} the top vertical maps are $\cD$-contravariant equivalences and the bottom vertical maps are $\cD$-covariant equivalences. By \cref{the:diag local ncov} both of these are $\cD$-diagonal equivalences, which are always	diagonal localized injective equivalences (\cref{the:diagmodel s}). Thus the top map is a diagonal localized injective equivalence if and only if the bottom map is one, but, by \cref{prop:cdcov s equiv lfib}, this is equivalent to $Y \to Z$ being a localized $\cD$-covariant equivalence over $X$.
\end{proof}

We can use this result to prove the following invariance property, giving a result similar to \cref{the:cdcov s inv qe} with different assumptions.

\begin{theone} \label{the:cdcov s inv qe two}
	Let $\cD_1,\cD_2$ be two small categories with terminal objects $t_1, t_2$. Moreover, let $S_1$ ($S_2$) be a set of monomorphisms in $\sP(\cD_1)$ ($\sP(\cD_2)$) with contractible codomain. Assume there exists a Quillen adjunction 
	\begin{center}
		\begin{tikzcd}
			\sP(\cD_1)^{inj_{S_1}} \arrow[r, "L", shift left=1.8, "\bot"'] & \sP(\cD_2)^{inj_{S_2}} \arrow[l, "R", shift left=1.8]
		\end{tikzcd}.
	\end{center}	
	Let $X$ be a $\cD_1$-simplicial space with weakly constant objects.	Then this induces a Quillen adjunction (using \cref{not:slsr}).
	\begin{center} 
		\begin{tikzcd}
			(\sP(\cD_1\times\DD)_{/X})^{\cD_1-cov_{S_1}} \arrow[r, "sL", shift left=1.8, "\bot"'] & (\sP(\cD_2\times\DD)_{/sLX})^{\cD_2-cov_{S_2}} \arrow[l, "sR", shift left=1.8]
		\end{tikzcd},
	\end{center} 
	for all $\sP(\cD_1)$-simplicial spaces $X$ which is simplicial or a Quillen equivalence if $(L,R)$ is one.
\end{theone}

\begin{proof}
	First let us assume $(L,R)$ is a Quillen adjunction. Using again the fact that $F[d,n] \cong F[d,0] \times D[n]$ (\cref{rem:D s F}) and the same argument as in the proof of \cref{the:cdcov s inv qe}  (and in particular right below \ref{eq:adj two}) it follows that 
	\begin{center} 
		\begin{tikzcd}
			(\sP(\cD_1\times\DD)_{/X})^{\cD_1-cov} \arrow[r, "sL", shift left=1.8, "\bot"'] & (\sP(\cD_2\times\DD)_{/sLX})^{\cD_2-cov} \arrow[l, "sR", shift left=1.8]
		\end{tikzcd}
	\end{center} 
	 is a Quillen adjunction from the $\cD_1$-covariant model structure to the $\cD_2$-covariant model structure. We want to prove it remains a Quillen adjunction of localized model structures. By \cref{cor:fibrancies} in order to obtain localized covariant model structures it suffices to localize with respect to maps of the form 
	 $$\{ \VEmb(A) \to \VEmb(B) \to D[0] \to X: A \to B \in S_1\},$$
	 $$\{ \VEmb(A) \to \VEmb(B) \to D[0] \to sLX: A \to B \in S_2\}.$$
	 By \cref{cor:localizing equiv} it suffices to prove that $sL(\VEmb(A) \to \VEmb(B) \to D[0] \to X) = \VEmb(L(A)) \to \VEmb(L(B)) \to D[0] \to sLX$ is an $S_2$-localized $\cD_2$-covariant equivalence. However, by assumption $(L,R)$ is a Quillen adjunction between $S_1$-localized and the $S_2$-localized injective model structures and so the result again follows from \cref{prop:fibrancies}. This proves the adjunction is a Quillen adjunction, which is simplicial if $(L,R)$ is. 
	 
	 Now, we assume that $(L,R)$ is a Quillen equivalence. Using the zigzag of Quillen equivalences given in \cref{the:strictification dleft fib s} we can assume that $X = N_{\cD_1}\C_1$, where $\C_1$ is an $\sP(\cD_1)$-enriched category. In that case $sL(X)= N_{\cD_2}\C_2$, where $\C_2$ has the same objects as $\C_1 $ and for two objects $c,d$ in $\C_1$, $\uMap_{\C_2}(c,d) \simeq L\uMap_{\C_1}(c,d)$ (using the fact that the left Quillen equivalence $L$ commutes with products up to equivalence).
	 
	 We want to prove the Quillen adjunction 
	 \begin{center} 
	 	\begin{tikzcd}
	 		(\sP(\cD_1\times\DD)_{/N_{\cD_1}\C_1})^{\cD_1-cov} \arrow[r, "sL", shift left=1.8, "\bot"'] & (\sP(\cD_2\times\DD)_{/N_{\cD_2}\C_2})^{\cD_2-cov} \arrow[l, "sR", shift left=1.8]
	 	\end{tikzcd},
	 \end{center} 
	 is a Quillen equivalence. We will prove that the right adjoint reflects equivalences between fibrant objects and the derived unit map is an equivalence of $S_1$-localized $\cD_1$-left fibrations.
	 
	 Let $Y \to N_{\cD_1}\C_1$ be an $S_1$-localized $\cD_1$-left fibration. Let  $\widehat{sL(Y)} \to N_{\cD_2}\C_2$ be the fibrant replacement of $sL(Y) \to N_{\cD_2}\C_2$ in the $S_2$-localized $\cD_2$-covariant model structure and notice by \cref{prop:cdcov s equiv lfib} for every object $c$ in $\C_2$, the fiber $\Fib_csL(Y) \to \Fib_c\widehat{sL(Y)}$ is a fibrant replacement in the $S_2$-localized injective model structure on $\cD_2$-spaces. Applying $sR$ we get the derived unit map $sL(Y) \to sR\widehat{sL(Y)}$, which over an object $c$ restricts to the map  $\Fib_cY \to \Fib_csR\widehat{sL(Y)}$. As $sL,sR$ are defined level-wise and the fiber of a $\cD_1$-left fibration is constant (\cref{the:diag local ncov}) we have the injective equivalence $\Fib_csR\widehat{sL(Y)} \simeq R\widehat{L\Fib_cY}$, meaning the derived unit map $sL(Y) \to sR\widehat{sL(Y)}$ restricts over the object $c$ to the derived unit map of the $cD$-space $\Fib_cY$. By the assumption that $(L,R)$ is a Quillen equivalence, this map is an $S_1$-localized equivalence of $\cD_1$-spaces, which, by \cref{the:cdcov s equiv lfib}, implies that the derived unit map of $L$ is an $S_1$-localized $\cD_1$-covariant equivalence.
	 
	 Now, let $Y \to Z$ be a map of $S_2$-localized $\cD_2$-left fibrations over $N_{\cD_2}\C_2$, such that $sR(Y) \to sR(Z)$ is an $S_1$-localized $\cD_2$-equivalence over $N_{\cD_1}\C_1$. By \cref{the:cdcov s equiv lfib}, the assumption implies that $\Fib_csR(Y) \to \Fib_csR(Z)$ is an injective equivalence of $\cD_1$-spaces (as both $\Fib_csR(Y)$, $\Fib_csR(Z)$ are fibrant in the $S_1$-localized injective model structure). However, this map is injectively equivalent to $R(\Fib_c(Y)) \to R(\Fib_c(Z))$ and $R$ reflects equivalences, which means that $\Fib_c(Y) \to \Fib_(Z)$ is an $S_2$-localized injective equivalence, which implies that $Y \to Z$ is an $S_2$-localized $\cD_2$-covariant equivalence (\cref{the:cdcov s equiv lfib}).
 \end{proof}

We want to generalize \cref{the:dcov invariant css s} to $S$-localized equivalences, however, this requires a technical detour. First we need the following definition.

\begin{defone} \label{def:sloc dk}
	Let $\cD$ have a terminal object $t$. A functor $F: \C \to \D$ of $\sP(\cD)$-enriched categories is an $S$-localized Dwyer-Kan equivalence if it satisfies the following two conditions.
	\begin{itemize}
		\item {\bf Fully faithful:} For all objects $c,c'$ in $\C$, the map of $\cD$-spaces 
		$$\uMap(c,c') \to \uMap(Fc,Fc')$$
		is an $S$-localized injective equivalence.
		\item {\bf Essentially Surjective:} The map of simplicially enriched categories $\Und_tF: \Und_t\C \to \Und_t\D$ induces a bijection on equivalence classes of objects $\pi_0(\Obj_\C)\to \pi_0(\Obj_\D)$.
	\end{itemize}
\end{defone}

We have following evident lemma with regard to $S$-localized Dwyer-Kan equivalences and Dwyer-Kan equivalences (\cref{def:loc dk}).

\begin{lemone}
	Let $\cD$ be a category with terminal object and $S$ a set of cofibrations of $\cD$-spaces. Every Dwyer-Kan equivalence of $\sP(\cD)$-enriched categories $F: \C\to \D$ is an $S$-localized Dwyer-Kan equivalence. Moreover, the opposite holds if $\C$ and $\D$ are enriched over $S$-local $\cD$-spaces.  
\end{lemone}

We want to be able to relate $S$-localized Dwyer-Kan equivalences between general categories with Dwyer-Kan equivalences. For that we first need the following condition on mapping spaces. 

\begin{defone}\label{def:cartesian spaces}
	Let $\cD$ be a category with terminal objects and let $S$ be a set of cofibrations of $\cD$-spaces. Fix a functorial fibrant replacement $R: \sP(\cD) \to \sP(\cD)$ in the $S$-localized injective model structure. An $\sP(\cD)$-enriched Segal category (\cref{def:enriched segal}) has {\it Cartesian mapping spaces} if for all objects $c,c'$ in $\C$ and $\cD$-space $A$ the map $A \times \map(c,c') \to A \times R\map(c,c')$ is an $S$-localized injective equivalence. Moreover, an $\sP(\cD)$-enriched category has Cartesian mapping spaces if $N_\cD\C$ has Cartesian mapping spaces.
\end{defone}

\begin{notone} \label{not:segcatcart}
	We denote the full sub-category of $\cD$-simplicial spaces consisting of $\sP(\cD)$-enriched Segal categories with Cartesian mapping spaces by $\cat^{cart}_{\sP(\cD)}$.
\end{notone}

\begin{lemone} \label{lemma:sR}
	Let $\cD$ be a category with terminal objects and let $S$ be a geometrically contractible set of cofibrations of $\cD$-spaces. Let $R: \sP(\cD) \to \sP(\cD)$ be a functorial fibrant replacement in the $S$-localized injective model structure and define $sR: \sP(\cD\times\DD) \to \sP(\cD\times\DD)$ as applying $R$ level-wise (\cref{not:slsr}). Then $sR$ restricts to a functor (using \cref{not:segcatcart})
	$$sR: \Seg\cat^{cart}_{\sP(\cD)} \to \Seg\cat^{cart}_{\sP(\cD)}$$
	that takes a Segal category to one with the same objects.
\end{lemone}

\begin{proof}
	Fix an $\sP(\cD)$-enriched Segal category with Cartesian mapping spaces $X$. As $S$ consists of geometric equivalence, $\Val(sR(X)) = R(\Val(X)) = \Val(X)= \Obj_X$, as $\Val(X)$ is discrete and hence fibrant (\cref{lemma:geom equiv}). Next, by definition of a Segal category we have the injective equivalence $\Val_k(X) \simeq \coprod_{x_0,...,x_k} \map(x_0,x_1) \times ... \times \map(x_{k-1},x_k)$ and so we get the injective equivalence
	\begin{equation}\label{eq:first}
	\begin{tabular}{ccc} 
	$R\Val_k(X)$ & $\ds\simeq$ & $\ds R(\coprod_{x_0,...,x_k} \map(x_0,x_1) \times ... \times \map(x_{k-1},x_k))$ \\
	&$\ds\simeq$ & $\ds \coprod_{x_0,...,x_k} R(\map(x_0,x_1) \times ... \times \map(x_{k-1},x_k))$,
	\end{tabular}
	\end{equation}
	where in the last step we used the fact that fibrant replacements preserve coproducts (\cref{lemma:coprod fibrant}).
	Now, for given mapping spaces $\map_X(c,c')$, $\map_X(d,d')$ we can use the fact that $X$ has Cartesian mapping spaces (\cref{def:cartesian spaces}) twice to get $S$-localized injective equivalences 
 	$$\map_X(c,c')\times\map_X(d,d') \to R\map_X(c,c')\times\map_X(d,d') \to R\map_X(c,c')\times R\map_X(d,d'),$$
 	which gives us an injective equivalence 
    $$ R\uMap_X(c,c')\times R\uMap_X(d,d') \simeq R(\uMap_X(c,c')\times \uMap_X(d,d')) $$
 	as both sides are fibrant in the $S$-localized injective model structure. Combining this equivalence with \ref{eq:first} gives us the desired injective equivalences 
 	\begin{align*}
 		R\Val_k(X) & \simeq \coprod_{x_0,...,x_k} R(\map(x_0,x_1) \times ... \times \map(x_{k-1},x_k)) \\
 		& \simeq \coprod_{x_0,...,x_k} R\map(x_0,x_1) \times ... \times R\map(x_{k-1},x_k) \\
 		& \simeq R\Val_1(X) \times_{\Val(X)}... \times_{\Val(X)} R\Val_1(X)
 	\end{align*}
 	proving that $sRX$ is a Segal category with the same set of objects. 
 \end{proof}

\begin{lemone} \label{lemma:strict}
 There exists a functor 
 $$\Str: \Seg\cat \to \cat_{\Delta}$$
 from Segal categories to simplicially enriched categories, which takes a Segal category $X$ to a simplicially enriched category $\Str(X)$ with the following properties:
 \begin{itemize}
 	\item $\Str(X)$ has the same objects as $X$.
 	\item There is a Dwyer-Kan equivalence $\Str(X) \to \C_X$ (\cref{not:segal cat rep}).
 	\item If $X = N\C$, the nerve of a simplicially enriched category, then we have the diagram $\Str(N\C) \to \C_{N\C} \to \C$, such that the composition map is a Dwyer-Kan equivalence that is the identity on objects. 
 \end{itemize}
\end{lemone} 

\begin{proof}
	Let $X$ be a Segal category and let $X[0]$ be the discrete simplicial space given by the set $X[0]$ and notice there is an evident inclusion map $X[0] \to X$. Notice, the functor $N\C_{(-)}$ defined in \cref{not:segal cat rep} takes $X[0] = \coprod_{\Obj_X} F[0]$ to the discrete category $\coprod_{\Obj_X} [0]$ (as it is a left adjoint and preserves coproducts), which we can simply denote as the discrete category $\Obj_X$. 
	
	Now applying $N\C_{(-)}$ to the map $X[0] \to X$ gives us the functor $\Obj_X \to \C_X$. Let $\Str(X)$ be the full subcategory of $\C_X$ with objects in $\Obj_X$ given via the functor $\Obj_X \to \C_X$. Notice, this inclusion via the full subcategory is in fact essentially surjective and hence a Dwyer-Kan equivalence. Indeed, by \cite[Theorem 7.1]{bergner2007threemodels}, we have a zigzag of Dwyer-Kan equivalences of Segal categories 
	$$X \leftarrow \cC(X) \to N\C_X$$
	and  notice the Segal category $N\C_X$ and the simplicially enriched category $\C_X$ have the same objects. 
	
	Now, if $X= N\C$, then by \cite[Theorem 8.6]{bergner2007threemodels} we have a derived counit map $\C_{N\C} \to \C$ and so, combining with the previous step, we get the desired map $\Str(\C) \to \C_{N\C} \to \C$, which is, by construction of $\Str(\C)$, the identity on objects.
\end{proof}
 
\begin{lemone} \label{lemma:bet}
	Let $\cD$ be a category with terminal objects and let $S$ be a geometrically contractible set of cofibrations of $\cD$-spaces. There exists a functor $R: \cat_{\sP(\cD)} \to \cat_{\sP(\cD)}$, such that for all $\sP(\cD)$-enriched categories $\C$ with Cartesian mapping spaces, $R\C$ has the same objects as $\C$ and $\uMap_{R\C}(c,c')$ is the fibrant replacement of $\uMap_\C(c,c')$ in the $S$-localized injective model structure. Moreover, there is a zigzag of $\sP(\cD)$-enriched categories 
	$$\C \overset{\simeq}{\leftarrow} \Bet(\C) \rightarrow R\C$$
	where both maps are identity on objects and the left hand map is a Dwyer-Kan equivalence.
\end{lemone}

\begin{proof}
	Let $R$ be defined as the composition of the following three functors
	$$\cat_{\sP(\cD)} \xrightarrow{  \ N_\cD  \ } \sP(\cD\times\DD) \xrightarrow{ \ sR \ } \sP(\cD\times\DD) \xrightarrow{ \ s\Str \ } \cat_{\sP(\cD)},$$
	 where $sR$ is defined in \cref{lemma:sR} and $s\Str$ applies the functor $\Str$ (\cref{lemma:strict}) level-wise (\cref{not:slsr}). By \cref{lemma:sR}, we have a map of Segal categories $N_\cD\C \to sRN_\cD\C$, which is the identity on objects and so applying $s\Str$, we get 
	 $s\Str(N_\cD\C) \to R\C$. Using \cref{lemma:strict}, we can extend this diagram to the desired zigzag
	 $$\C \overset{\simeq}{\leftarrow} s\Str(N_\cD\C) \to R\C$$
	 where the left hand arrow is a Dwyer-Kan equivalence, both maps are identity on objects and right hand map is given as fibrant replacement and so the desired result follows with $\Bet(\C) = s\Str(N_\cD\C)$.
\end{proof} 
	 
We would like to say that a functor of $\sP(\cD)$-enriched categories $\C \to \D$ is an $S$-localized Dwyer-Kan equivalence (\cref{def:sloc dk}) if and only if $R\C\to R\D$ is a Dwyer-Kan equivalence. However, for this conclusion we need one more assumptions.

\begin{defone} \label{def:stable objects}
	Let $\C$ be an $\sP(\cD)$-enriched category. We say $\C$ has {\it $S$-stable objects} if the map $\Bet(\C) \to R\C$ constructed in \cref{lemma:bet} induces a bijection of equivalence classes of objects $\pi_0(\C) \cong \pi_0(\Bet(\C)) \cong \pi_0(R\C)$.
\end{defone}

\begin{lemone} \label{lemma:zigzag equivalences}
	Let $\cD$ be a category with terminal objects and let $S$ be a geometrically contractible set of cofibrations of $\cD$-spaces.
	Then a functor $F: \C\to \D$ of $\sP(\cD)$-enriched categories with Cartesian mapping spaces is an $S$-localized Dwyer-Kan equivalence if and only if $RF: R\C \to R\D$ is a Dwyer-Kan equivalence.
\end{lemone}

\begin{proof}
	The natural zigzag in \cref{lemma:bet} gives us a diagram 
	\begin{center}
		\begin{tikzcd}
		 \C \arrow[r] & \D \\
		 \Bet(\C) \arrow[r] \arrow[u, "\simeq"', "d_\C"] \arrow[d, "\simeq", "s_\C"'] & \Bet(\D) \arrow[u, "\simeq", "d_\D"'] \arrow[d, "\simeq"', "s_\D"] \\
		 R\C \arrow[r] & R\D 
		\end{tikzcd}.
	\end{center} 
	The top vertical maps are Dwyer-Kan equivalences and so give us an injective equivalence of mapping spaces and and a bijection of equivalence classes of objects. The bottom vertical maps are $S$-localized equivalences of mapping spaces (by \cref{lemma:bet}) and a bijection of equivalence classes of objects, by assumption (\cref{def:stable objects}). Hence, the result follows by $2$-out-of-$3$.
\end{proof}

We now have the necessary background results and need one final definition.

\begin{defone} \label{def:X cart and stable}
	Let $X$ be a $\cD$-simplicial space with weakly constant objects. Then $X$ has {\it Cartesian mapping spaces} with {\it $S$-stable objects} if the corresponding category $\C_X$ (\cref{lemma:zigzag strictification}) has Cartesian mapping spaces (\cref{def:cartesian spaces}) and has $S$-stable objects (\cref{def:stable objects}).
\end{defone}

\begin{defone}\label{def:s loc dk}
	A map of $\cD$-simplicial spaces with weakly constant objects is an $S$-localized Dwyer-Kan equivalence if the map $\C_X \to \C_Y$ given via \cref{lemma:zigzag strictification} is an $S$-localized Dwyer-Kan equivalence of $\sP(\cD)$-enriched categories (\cref{def:sloc dk}). 
\end{defone}

\begin{remone} 
 Notice, \cref{def:X cart and stable} implies that having Cartesian mapping spaces, having $S$-stable objects and being an $S$-localized Dwyer-Kan equivalence is invariant under level-wise complete Segal space equivalences of $\cD$-simplicial spaces with weakly constant objects.  
\end{remone}

Under the appropriate conditions we can relate $S$-localized Dwyer-Kan equivalences to $S$-localized $\sP(\cD)$-enriched complete Segal space equivalences (\cref{the:spd enriched css s}).

\begin{lemone} \label{lemma:s loc dk is spd enriched css}
 Let $\cD$ be a small category with terminal object $t$. Let $X,Y$ be $\cD$-simplicial spaces with weakly constant objects, Cartesian mapping spaces and $S$-stable objects. Then the following are equivalent:
 \begin{enumerate}
 	\item $f:X \to Y$ is an $S$-localized Dwyer-Kan equivalence.
 	\item $R\C_f: R\C_X \to R\C_Y$ constructed in \cref{lemma:bet} is a Dwyer-Kan equivalence.
 	\item $f: X \to Y$ is an $S$-localized $\sP(\cD)$-enriched complete Segal space equivalence.
 \end{enumerate}
\end{lemone} 

\begin{proof}
	The equivalence of the first two items is the content of \ref{def:s loc dk} and \cref{lemma:zigzag equivalences}. Now notice $f$ is an $S$-localized $\sP(\cD)$-enriched complete Segal space equivalence if and only if $N_\cD R\C_f$ is one. Indeed, all maps involved in the zigzag between $f$ and $N_\cD R\C_f$ are either level-wise complete Segal space equivalences or $S$-localized equivalences, both are which are $S$-localized $\sP(\cD)$-enriched complete Segal space equivalences. 
	
	Now, notice that the fibrant replacement of $N_\cD R\C_f$ in the $\sP(\cD)$-enriched complete Segal space model structure $\widehat{N_\cD R\C_f}$ is already the fibrant replacement in the $S$-localized $\sP(\cD)$-enriched complete Segal space model structure (by \cref{def:sloc spd enriched css} and \cref{lemma:fiber vs global}). By definition $N_\cD R\C_f$ is an $S$-localized $\sP(\cD)$-enriched complete Segal space equivalence if and only if $\widehat{N_\cD R\C_f}$ is an injective equivalence. However, by \cref{lemma:dk equiv segal enriched}, $N_\cD R\C_f$ and $\widehat{N_\cD R\C_f}$ are also Dwyer-Kan equivalent and so this is equivalent to $R\C_f$ being a Dwyer-Kan equivalence.
\end{proof}

We are finally in a position to improve upon \cref{the:dcov invariant css s} as promised in \cref{rem:better result}.

\begin{theone} \label{the:cdcov s invariant s equiv}
	Let $\cD$ have a terminal object $t$ and $X,Y$ weakly constant objects. Let $S$ be geometrically contractible set of cofibration of $\cD$-spaces. Let $g: X \to Y$ be a map of $\cD$-simplicial spaces with Cartesian mapping $\cD$-spaces and $S$-stable objects. Then the $\sP(\cD)$-enriched adjunction
	$$ \adjun{(\sP(\cD\times\DD)_{/X})^{\cD-Cov_S}}{(\sP(\cD\times\DD)_{/Y})^{\cD-Cov_S}}{g_!}{g^*} $$
	is a Quillen adjunction, which is a Quillen equivalence whenever $g$ is an $S$-localized Dwyer-Kan equivalence (\cref{def:s loc dk}) or $S$-localized $\sP(\cD)$-enriched complete Segal space equivalence (\cref{the:spd enriched css s}). Here both sides have the $S$-localized $\cD$-covariant model structure.
\end{theone}

\begin{proof}
	It is evidently a Quillen adjunction as $S$-localized $\cD$-left fibrations are stable under pullback and $g_!$ preserves cofibrations and \cite[Corollary A.3.7.2]{lurie2009htt}. Now, let us assume $g$ is an $S$-localized Dwyer-Kan equivalence, which, by \cref{lemma:s loc dk is spd enriched css}, is equivalent to being an $S$-localized $\sP(\cD)$-enriched complete Segal space equivalence.  We want to prove $(g_!,g^*)$ is a Quillen equivalence. By \cref{lemma:zigzag strictification}, $g$ is level-wise complete Segal space equivalent to a map of nerves $N_\cD G: N_\cD \C \to N_\cD\D$ and so, by \cref{the:dcov invariant css s} and $2$-out-of-$3$ for Quillen equivalences, $(g_!,g^*)$ is a Quillen equivalence if and only if $(N_\cD G_!,N_\cD G^*)$ is a Quillen equivalence. 
	
	Now, using \cref{lemma:zigzag equivalences} we get the following diagram of Quillen adjunctions
	\begin{center}
		\begin{tikzcd}[row sep=0.5in, column sep=0.7in]
		(\sP(\cD\times\DD)_{/N_\cD\C})^{\cD-Cov_S} \arrow[r, "N_\cD G_!", "\bot"', shift left=1.8] \arrow[d, shift right=1.8, "(d_\C)^*"'] & (\sP(\cD\times\DD)_{/N_\cD\D})^{\cD-Cov_S} \arrow[l, "N_\cD G^*", shift left=1.8] \arrow[d, shift right=1.8, "(d_\D)^*"'] \\
		(\sP(\cD\times\DD)_{/N_\cD\Bet(\C)})^{\cD-Cov_S} \arrow[r, "N_\cD \Bet G_!", "\bot"', shift left=1.8] \arrow[u, shift right=1.8, "(d_\C)_!"' ,"\rotatebot"] \arrow[d, shift left=1.8, "(s_\C)_!" ,"\rotatebot"'] & (\sP(\cD\times\DD)_{/N_\cD\Bet(\D)})^{\cD-Cov_S} \arrow[l, "N_\cD \Bet G^*", shift left=1.8] \arrow[u, shift right=1.8, "(d_\D)_!"' ,"\rotatebot"] \arrow[d, shift left=1.8, "(s_\D)_!" ,"\rotatebot"'] \\
		(\sP(\cD\times\DD)_{/N_\cD R\C})^{\cD-Cov_S} \arrow[r, "N_\cD R G_!", "\bot"', shift left=1.8]  \arrow[u, shift left=1.8, "(s_\C)^*"] & (\sP(\cD\times\DD)_{/N_\cD R\D})^{\cD-Cov_S}  \arrow[l, "N_\cD RG^*", shift left=1.8] \arrow[u, shift left=1.8, "(s_\D)^*"]
		\end{tikzcd}
	\end{center}
	By \cref{lemma:zigzag equivalences} $R\C \to R\D$ is a Dwyer-Kan equivalence and so the bottom adjunction is a Quillen equivalence by \cref{lemma:dk equiv segal enriched} and \cref{the:dcov invariant css s}. Moreover, applying the same result to \cref{lemma:bet} implies that the top vertical Quillen adjunctions are Quillen equivalences. So, by $2$-out-of-$3$, in order to prove that the top adjunction is a Quillen equivalence it suffices to show that the two bottom vertical adjunctions are Quillen equivalences and as they are defined similarly, it suffices to prove that $((s_\C)_!, (s_\C)^*)$ is a Quillen equivalence. 
	 	
 	In order to simplify notation, we will henceforth denote the functor $s_\C: \Bet(\C) \to R\C$ by $I: \C \to \hat{\C}$ and notice the functor $I$ is the identity on objects and for any two objects, $c,c'$ the map of $\cD$-spaces $\uMap_{\C}(c,c') \to \uMap_{\hat{\C}}(c,c')$ is the fibrant replacement in the $S$-localized injective model structure on $\cD$-spaces.
    
    Let $\widetilde{N_\cD\hat{\C}_{c/}} \to N_\cD \hat{\C}$ be the injectively fibrant replacement of the over-category $N_\cD\hat{\C}_{c/} \to N_\cD\hat{\C}$ and notice it is the fibrant replacement of $\{c\}: D[0] \to N_\cD\hat{\C}$ in the $\cD$-covariant model structure (\cref{lemma:cdleft map}). Moreover, the fiber of the derived unit map 
    $$\widetilde{N_\cD\C_{c/}} \to N_\cD I^*(\widetilde{N_\cD\hat{\C}_{c/}}) = N_\cD\C \times_{N_\cD \hat{\C}} \widetilde{N_\cD\hat{\C}_{c/}} $$
    over an object $c$ is injectively equivalent to $\uMap_{\C}(c,c')  \to \uMap_{\hat{\C}}(c,c')$, which, by \cref{lemma:bet}, is the fibrant replacement in the $S$-localized injective model structure. Hence, the derived unit map is an $S$-localized $\cD$-covariant equivalence over 
    $N_\cD\C$ (\cref{prop:cdcov s equiv lfib}). This, in particular, implies that we have the $S$-localized $\cD$-contravariant equivalence 
    \begin{equation}\label{eq:sloc equiv}
     N_\cD\hat{\C}_{/c} \times_{\hat{\C}} \C \simeq N_\cD \C_{/c}.
    \end{equation}     
    We will use this observation to prove the left adjoint reflects equivalences. Let $Y \to Z$ be an arbitrary map over $N_\cD\C$ such that $Y \to Z$ is an $S$-localized $\cD$-covariant equivalence over $N_\cD\hat{\C}$. By \cref{cor:reconition principle s}, this is equivalent to $Y \times_{N_\cD\hat{\C}} N_\cD\hat{\C}_{/c} \to Z \times_{N_\cD\hat{\C}} N_\cD\hat{\C}_{/c}$ being a diagonal $S$-localized injective equivalence. Now we also have injective equivalences 
    $$Y \times_{N_\cD\hat{\C}} N_\cD\hat{\C}_{/c} \cong Y \times_{N_\cD\C} N_\cD\C \times_{N_\cD\hat{\C}} N_\cD\hat{\C}_{/c} \simeq Y \times_{N_\cD\C} N_\cD\C_{/c}$$
    where the first follows from the pasting property of pullbacks and the second follows from \ref{eq:sloc equiv}. This, again by \cref{cor:reconition principle s}, implies that $Y \to Z$ is an $S$-localized $\cD$-covariant equivalence over $N_\cD\C$.
    
    Hence, in order to prove a Quillen equivalence, we only need to show the counit map $N_\cD I_! N_\cD I^* L \to L$ is an $S$-localized $\cD$-covariant equivalence over $N_\cD\hat{\C}$ for all $S$-local $\cD$-left fibrations $L \to N_\cD \hat{C}$. We will prove that $\Val_k(N_\cD I_! N_\cD I^* L \to L)$ is an $S$-localized injective equivalence of $\cD$-spaces and the fact that $\Val_k(N_\cD\hat{\C})$ is $S$-local (by \cref{lemma:bet}) along with \cref{the:cdcov loc css s} implies the desired result.
    
    As $L \to N_\cD \hat{C}$ is a $\cD$-left fibration, we have an injective equivalence 
    \begin{equation} \label{eq:first one}
    	\begin{tabular}{ccl}
    		$\ds\Val_k(L)$ & $\ds\simeq$ & $\ds\Val_0(L) \times_{\Val_0(N_\cD\hat{\C})} \Val_k(N_\cD\hat{\C})$ \\
    		&  $\cong$ &  $\ds\coprod_{c_0,..., c_k} \Fib_{c_0}L \times \uMap_{\hat{\C}}(c_0,c_1) \times ... \times \uMap_{\hat{\C}}(c_{k-1},c_k)$
    	\end{tabular}.
    \end{equation}
    Similarly we have
    \begin{equation}\label{eq:second one}  
     \Val_k(N_\cD I_! N_\cD I^* L) = \coprod_{c_0,..., c_k} \Fib_{c_0}L \times \uMap_{\C}(c_0,c_1) \times ... \times \uMap_{\C}(c_{k-1},c_k). 	
    \end{equation} 
 The fact that the map $\uMap_{\C}(c_0,c_1) \to \uMap_{\hat{\C}}(c_0,c_1)$ is the fibrant replacement (by \cref{lemma:bet}) and that $\C$ has Cartesian mapping objects (\cref{def:cartesian spaces}) immediately implies the map from \ref{eq:second one}  to \ref{eq:first one} is an $S$-localized equivalence of $\cD$-spaces. 
\end{proof}  

We now have the following result mirroring \cref{the:cdcov s invariant s equiv}.

\begin{corone} \label{cor:proj inv spd equiv}
	Let $\cD$ have a terminal object $t$. Let $S$ be geometrically contractible set of cofibration of $\cD$-spaces. Let $G: \C \to \D$ be an $\sP(\cD)$-enriched functor of $\sP(\cD)$-enriched categories, such that the categories have Cartesian mapping $\cD$-spaces (\cref{def:cartesian spaces}) with $S$-stable objects (\cref{def:stable objects}). Then the adjunction
	\begin{center}
		\adjun{\uFun(\C,\sP(\cD))^{proj}}{\uFun(\C,\sP(\cD))^{proj}}{G_!}{G^*}
	\end{center}
	is a Quillen adjunction, which is a Quillen equivalence whenever $G$ is an $S$-localized Dwyer-Kan equivalence (\cref{def:sloc dk}).
	Here both sides have the $S$-localized projective model structure.
\end{corone}

\begin{proof}
	The fact that it is a Quillen adjunction follows directly from the fact that the restriction functor $G^*$ preserves weak equivalences and fibrations. Now, if $G$ is an $S$-localized Dwyer-Kan equivalence, then we have the following diagram of Quillen adjunctions
	 	\begin{center}
	 	\begin{tikzcd}[row sep=0.8in, column sep=0.9in]
	 		\uFun(\C,\sP(\cD))^{proj}  \arrow[r, shift left=1.8, "G_!"] \arrow[d, shift left =1.8, "\sint_{\cD/\C}"] & \uFun(\D,\sP(\cD))^{inj} \arrow[d, shift left =1.8, "\sbI_{\cD/\D}"]\arrow[l, shift left=1.8, "G^*", "\bot"']\\
	 		(\sP(\cD\times\DD)_{/N_\cD\C})^{\cD-cov_S} \arrow[r, shift left=1.8, "N_\cD G_!"] \arrow[u, shift left=1.8, "\sH_{\cD/\C}", "\rbot"']& (\sP(\cD\times\DD)_{/N_\cD\D})^{\cD-cov_S} \arrow[l, shift left=1.8, "N_\cD G^*", "\bot"'] \arrow[u, shift left=1.8, "\sbT_{\cD/\D}", "\rbot"']
	 	\end{tikzcd}.
	 \end{center}
	 The vertical adjunctions are Quillen equivalences by \cref{the:grothendieck simp s} and the bottom is a Quillen equivalence by \cref{the:cdcov s invariant s equiv} and so, by $2$-out-of-$3$ the top is a Quillen equivalence. The result now follows from again applying $2$-out-of-$3$ for Quillen equivalences to the top Quillen equivalence combined with \cref{prop:inj mod s} and \cref{lemma:sloc proj vs inj}.
\end{proof} 

\begin{remone}
	An alternative proof of \cref{cor:proj inv spd equiv} with somewhat different assumptions is given \cite[Proposition A.3.3.8]{lurie2009htt}. It is interesting to contrast the assumptions and notice the centrality of Cartesian mapping spaces and $S$-stable objects. Indeed, the condition that objects are $S$-stable is very similar to the {\it invertibility hypothesis}, also known as condition $(*)$, as given in \cite[Definition A.3.2.12]{lurie2009htt}.
	
	The proof of \cite[Proposition A.3.3.8]{lurie2009htt} also uses an assumption very analogous so the existence of Cartesian mapping spaces as illustrated in the diagram in \cite[Bottom of Page 872]{lurie2009htt}, which is part of the proof of \cite[Proposition A.3.3.8]{lurie2009htt}. On a more concrete side, we will give a counter-example to \cref{the:cdcov s invariant s equiv} (and hence \cref{cor:proj inv spd equiv}) illustrating the necessity of Cartesian mapping spaces (or a similar condition) in \cref{ex:noncartesian mapping spaces not invariant}. 
\end{remone}

In \cref{the:cdcov s invariant s equiv} we proved that under the appropriate assumptions the $S$-localized $\cD$-covariant model structure is invariant under $S$-localized Dwyer-Kan equivalences. How does the model structure behave with respect to $S$-localized level-wise complete Segal space equivalences? First we have the following lemma relating these two notions.

\begin{lemone}\label{lemma:sloc dk implies s lev css}
  Let $\cD$ be a small category with terminal object $t$. Let $X,Y$ be $\cD$-simplicial spaces with weakly constant objects, Cartesian mapping spaces and $S$-stable objects. Let $f: X \to Y$ be an $S$-localized Dwyer-Kan equivalence, then $f$ is an $S$-localized level-wise complete Segal space equivalence.
\end{lemone} 

\begin{proof}
	By \cref{lemma:zigzag strictification} $X \to Y$ is level-wise complete Segal space equivalent to a map of nerve $N_\cD\C_X \to N_\cD\C_Y$ and \cref{lemma:zigzag equivalences} $N_\cD\C_X \to N_\cD\C_Y$ is level-wise complete Segal space equivalent to $N_\cD\Bet\C_X\to N_\cD\Bet\C_Y$ and is $S$-local equivalent to the map $N_\cD R\C_X \to N_\cD R\C_Y$. As all of these maps are $S$-localized level-wise complete Segal space equivalences (\cref{the:spd enriched css s}), $X \to Y$ is an $S$-localized level-wise complete Segal space equivalence if and only if $N_\cD R\C_X \to N_\cD R\C_Y$ is one. 
	
	Now, by \cref{lemma:zigzag equivalences}, if $X \to Y$ is an $S$-localized Dwyer-Kan equivalence, then $R\C_X \to R\C_Y$ is a Dwyer-Kan equivalence, or, equivalently, a level-wise complete Segal space equivalence (\cref{lemma:dk equiv segal enriched}), which implies it is an $S$-localized level-wise complete Segal space equivalence (\cref{the:spd enriched css s}).
\end{proof}

We would expect that under appropriate conditions on the objects $X,Y$ $S$-localized level-wise complete Segal space equivalences are in fact equivalent to $S$-localized Dwyer-Kan equivalences, meaning \cref{the:cdcov s invariant s equiv} would then imply the $S$-localized $\cD$-covariant model structure is invariant under $S$-localized level-wise complete Segal space equivalences, which would provide a more precise analogue to \cref{the:dcov invariant css s}.

This leaves us with the question whether the $S$-localized $\cD$-covariant model structure could be invariant under more general $S$-localized level-wise complete Segal space equivalences that go beyond $S$-localized Dwyer-Kan equivalences. The next result suggests that this should be very unlikely.

  \begin{propone}  \label{prop:right proper}
  	Let $S$ be a set of cofibrations of $\cD$-spaces, such that for any $S$-localized level-wise CSS equivalence $g: X \to Y$ the adjunction 
  	\begin{center}
  		\adjun{(\sP(\cD)_{/X})^{\cD-Cov_S}}{(\sP(\cD)_{/Y})^{\cD-Cov_S}}{g_!}{g^*}
  	\end{center}
    is a Quillen equivalence, where both sides have the $S$-localized $\cD$-covariant model structure. Then the $S$-localized injective model structure on $\sP(\cD)$ is right proper.
\end{propone}

\begin{proof}
	Fix an $S$-localized injective equivalence of $\cD$-spaces $g: X \to Y$ and notice by direct computation $\fDiag(\VEmb(X)) = X$ and similarly for $Y$. Then we have the following diagram 
	\begin{center}
		\begin{tikzcd}[row sep=0.7in, column sep=0.7in]
			\sP(\cD\times\DD)_{/\VEmb(X)}  \arrow[r, shift left=1.8, "\VEmb(g)_!"] \arrow[d, shift left =1.8, "\Diag^*"] & \sP(\cD\times\DD)_{/\VEmb(Y)} \arrow[d, shift left =1.8, "\Diag"]\arrow[l, shift left=1.8, "\VEmb(g)^*", "\bot"']\\
			\sP(\cD)_{/X} \arrow[r, shift left=1.8, "g_!"] \arrow[u, shift left=1.8, "\Diag_*", "\rbot"']& \sP(\cD)_{/Y} \arrow[l, shift left=1.8, "g^*", "\bot"'] \arrow[u, shift left=1.8, "\Diag_*", "\rbot"']
		\end{tikzcd}
	\end{center}
	The vertical adjunctions are Quillen equivalences by \cref{prop:dcov vs diag s} and the top adjunction is a Quillen equivalence by assumption, which proves that the bottom is a Quillen equivalence for every map $g: X \to Y$. However, this is equivalent to the $S$-localized injective model structure being right proper \cite[Proposition 2.5]{rezk2002rightproper}.
\end{proof}

\begin{remone} \label{rem:right proper}
	The right properness assumption is very strong and is not satisfied by most examples of interest. Hence we needed a result such as \cref{the:cdcov s invariant s equiv} that still holds for relevant examples (\cref{rem:localizations are right proper}).
\end{remone}

 We can use \cref{the:cdcov s invariant s equiv} to improve \cref{the:pullback rfib s}.
 
  \begin{theone} \label{the:pullback rfib s nonfib base}
 	Let $\cD$ be a small category with terminal object $t$.	Let $S$ be a Cartesian geometrically contractible set of cofibration of $\cD$-spaces. Let $X$ be a $\cD$-simplicial space with weakly constant and $S$-stable objects.
 	Let $p:R \to X$ be a $\cD$-right or $\cD$-left fibration. Then the adjunction
 	\begin{center}
 		\adjun{(\sP(\cD\times\DD)_{/X})^{(lev_{CSS})_S}}{(\sP(\cD\times\DD)_{/R})^{(lev_{CSS})_S}}{p^*}{p_*}
 	\end{center}
 	is a Quillen adjunction. Here both sides have the $S$-localized level-wise CSS model structure (\cref{prop:levelcss s}) on the over-category.
 \end{theone}
 
 \begin{proof}
 	Evidently $p^*$ preserves cofibrations and so we only need to prove it preserves trivial cofibrations. For that we need an additional construction. Let $g:X \to \hat{X}$ be a fibrant replacement of $X$ in the $S$-localized $\sP(\cD)$-enriched complete Segal space model structure. Then, by \cref{the:cdcov s invariant s equiv}, there exists an $S$-localized $\cD$-right fibration $\hat{R} \to \hat{X}$ such that there is an injective equivalence $R \simeq \hat{R} \times_{\hat{X}} X$. 
 	
 	Now, let $A \to B$ be a trivial cofibration over $X$. We want to prove that $p^*(A \to B)$ is a trivial cofibration over $R$. By direct computation
 	$$p^*(A \to B) = p^*A \to p^*B = R \times_X A \to R\times_XB \simeq \hat{R} \times_{\hat{X}} A \to \hat{R} \times_{\hat{X}} B.$$
 	The result now follows from the fact that this map is a trivial cofibration over $\hat{X}$, by \cref{the:pullback rfib s}, and the fact that $g_!$ is the left Quillen functor of the Quillen equivalence in \cref{the:cdcov s invariant s equiv} and hence reflects weak equivalences.
\end{proof} 

Finally, we can use similar ideas to get an adjunction of $\sP(\cD)$-enriched complete Segal spaces.

  \begin{corone} \label{cor:pullback rfib s nonfib base}
	Let $\cD$ be a small category with terminal object $t$.	Let $S$ be a Cartesian geometrically contractible set of cofibration of $\cD$-spaces. Let $X$ be a $\cD$-simplicial space with weakly constant and $S$-stable objects.
	Let $p:R \to X$ be a $\cD$-right or $\cD$-left fibration. Then the adjunction
	\begin{center}
		\adjun{(\sP(\cD\times\DD)_{/X})^{\cD-CSS_S}}{(\sP(\cD\times\DD)_{/R})^{\cD-CSS_S}}{p^*}{p_*}
	\end{center}
	is a Quillen adjunction. Here both sides have the $S$-localized $\sP(\cD)$-enriched CSS model structure (\cref{the:spd enriched css s}) on the over-category.
\end{corone}

\begin{proof}
	By \cref{the:spd enriched css} and \cref{the:spd enriched css s}, the $S$-localized $\sP(\cD)$-enriched complete Segal space model structure can be obtained by localizing the $S$-localized level-wise Segal space model structure with respect to the maps $D[0] \to \Disc(E[1]) \to X$ and $F[t,0] \to F[d,0] \to X$. The first map is already covered by \cref{the:pullback rfib s nonfib base} and so, by \cref{cor:localizing equiv}, it suffices to prove $p^*F[t,0] \to p^*F[d,0]$ is an $S$-localized $\sP(\cD)$-enriched complete Segal space equivalence. 
	
	As $X$ has weakly constant objects we can assume that $F[d,0] \to X$ factors through $F[d,0] \to D[0] \xrightarrow{ \  \{x\} \ } X$. Hence, we have $p^*(F[t,0] \to F[d,0]) = \Fib_xR \to \Fib_xR \times F[d,0]$, which by definition is an equivalence in the $\sP(\cD)$-enriched complete Segal space model structure (\cref{the:spd enriched css}) and so also in the $S$-localized $\sP(\cD)$-enriched complete Segal space model structure (\cref{the:spd enriched css s}). 
\end{proof}

\section{(Segal) Cartesian Fibrations of \texorpdfstring{$(\infty,n)$}{(oo,n)}-Categories} \label{sec:infn fib}
 In this section we apply the previous work to the particular case of interest: {\it (Segal) Cartesian fibrations of $(\infty,n)$-Categories}, for models of $(\infty,n)$-categories based on simplicial presheaves. We first review the relevant theories of $(\infty,n)$-categories (\cref{subsec:infn cat}) before moving on to the main results regarding $(\infty,n)$-categorical fibrations (\cref{subsec:infn fib}), ending with some low-dimensional examples (\cref{subsec:examples}).
 
\subsection{Theories of \texorpdfstring{$(\infty,n)$}{(oo,n)}-Categories} \label{subsec:infn cat} 
 Let us review the relevant models of $(\infty,n)$-categories as certain simplicial presheaves. 
 For that we first recall the $\Theta$ construction. For a category $\C$ let $\Theta\C$ be the category with objects of the form $[k](c_1,...,c_k)$, where $k \geq 0$ and $c_1, ..., c_k$ are objects in $\C$ . Moreover, the morphisms are given as 
 $$\Hom_{\Theta\C}([k](c_1,...,c_k),[l](d_1,...,d_l))= \coprod_{\delta: [k] \to [l]} \prod_{i=1}^k \prod_{j = \delta(i-1)+1}^{\delta(i)}\Hom_\C(c_i,d_j).$$
 The $\Theta$ construction is originally due to Joyal \cite{joyal1997disks}. For a nice review of the $\Theta$-construction see \cite[Subsection 3.2]{rezk2010thetanspaces}. 
 
 \begin{defone} 
   Define $\Theta_0$ as the terminal category and $\Theta_{k+1} = \Theta\Theta_k$. Notice $\Theta_1 = \DD$.
 \end{defone} 
 
 We want to study models of $(\infty,n)$-categories via $\Theta_n$-spaces with a model structure on the category $\sP(\Theta_n)$. Let $(-)[1]: \sP(\Theta_{n}) \to \sP(\Theta_{n-1})$ be the functor that takes a $\Theta_n$-space $X$ to the $\Theta_{n-1}$-space $X[1]$ with value $(X[1])(\theta) = X[1](\theta)$.
 Moreover, for a $\Theta_n$-space $X$, let $\Und(X)$, the {\it underlying simplicial space}, be the simplicial space defined as 
 \begin{equation} \label{eq:underlying css}
  \Und(X)[n] = X[n]([0],...,[0]).
 \end{equation} 
 For more details see \cite[Subsection 4.1]{rezk2010thetanspaces}, where it is denoted $T^*$. Finally, there is a functor $\Sigma:\sP(\Theta_{n-1}) \to \sP(\Theta_n)$, given by $\Sigma(A) = V[1](A)$, where $V$ is the {\it intertwining functor} \cite[Subsection 4.4]{rezk2010thetanspaces}.
  
 \begin{defone}
 Let $X$ be a $\Theta_n$-space. Let $(\infty,0)$-$\Theta$-spaces be Kan complexes and define an $(\infty,n)$-$\Theta$-space $X$ inductively as:
  \begin{enumerate}
  	\item {\bf Fibrancy:} $X$ is injectively fibrant.
  	\item {\bf Segal:} $X$ satisfies the Segal condition, meaning for every object $[n](c_1,..., c_n)$ in $\Theta_n$, the map 
  	$$X[n](c_1,...,c_n) \to X[1](c_1) \times_{X[0]} ... \times_{X[0]} X[1](c_n)$$
  	is an equivalence of spaces.
  	\item {\bf Completeness:} The underlying simplicial space $\Und X$ (\ref{eq:underlying css}), which is a Segal space by the previous condition, is complete.
  	\item {\bf Enrichment:} The $\Theta_{n-1}$-space $X[1]$ is an $(\infty,n-1)$-$\Theta$-space. 
  \end{enumerate}
 \end{defone}
 
 We now have the following theorem, with slightly different notation, due to Rezk \cite{rezk2010thetanspaces}.
 
 \begin{theone} \label{the:thetan model cat}
  Let $n \geq 0$. There exists a left proper combinatorial Cartesian model structure on $\sP(\Theta_n)$ with the following specifications:
  \begin{enumerate}
  	\item Cofibrations are monomorphisms.
  	\item Fibrant objects are $(\infty,n)$-$\Theta$-spaces. 
  \end{enumerate} 	
  Moreover, the model structure is obtained by localizing the injective model structure on $\Theta_n$-spaces with respect to the following inductively defined set of cofibrations:
  \begin{itemize}
  	\item If $n=0$, then the localizing set is empty.
  	\item The Segal morphisms $F[1](c_1) \coprod_{F[0]} ... F[1](c_n) \to F[n](c_1,...,c_n)$
  	\item The completeness map $F[0] \to \Disc(E[1])$
  	\item Maps of the form $\Sigma(f): \Sigma(A) \to \Sigma(B)$, where $f:A \to B$ is a map of $\Theta_{n-1}$-spaces that defined $(\infty,n-1)$-$\Theta$-spaces. 
  \end{itemize}
 \end{theone} 

 We now want to use this to define a model of $(\infty,n)$-categories on diagrams of the form $\Theta_k \times \DD^{n-k}$. In order to do that we need the following key result by Bergner and Rezk \cite[Corollary 7.2]{bergnerrezk2020comparisonii} with somewhat different terminology.
 
 \begin{theone} \label{the:cso in inftyn}
 	Let $\cD$ be a small category with terminal object and $S$ a set of cofibrations of $\cD$-spaces such that the $S$-localized injective model structure on $\cD$-spaces (\cref{prop:injectivemod s}) is a model for $(\infty,n-1)$-categories. Then the $S$-localized $\sP(\cD)$-enriched complete Segal space model structure (\cref{the:spd enriched css s}) is a model for $(\infty,n)$-categories. The model structure is obtained by localizing the injective model structure on $\cD$-simplicial spaces with respect to the following morphisms 
 	\begin{itemize}
 		\item $\VEmb(S) \times D[n]$, where $n\geq 0$.
 		\item $F[d,0] \times \Disc(G[n]) \to F[d,0] \times D[n]$, where $n \geq 2$ and $d\in \Obj_\cD$.
 		\item $D[0] \to \Disc(E[1])$
 		\item $D[0] \to F[d]$, where $d \in \Obj_\cD$
 	\end{itemize}
 \end{theone}

 Notice, by \cite[Corollary 7.3]{bergnerrezk2020comparisonii}, we can relate these model structures via simplicial Quillen equivalences.
 
 \begin{corone} \label{cor:inftyn}
 	Let $n\geq 0$. There is a chain of simplicial Quillen equivalences of simplicial model categories
 	\begin{center}
 		\begin{tikzcd}
 		 \sP(\DD^n) \arrow[r, shift left=1.8, "\bot"', "L_1"] & \sP(\DD^{n-1} \times \DD)  \arrow[r, shift left=1.8, "\bot"', "L_2"] \arrow[l, shift left=1.8, "R_1"] & \sP(\DD^{n-2} \times \Theta_2)  \arrow[r, shift left=1.8, "\bot"', "L_{n-1}"] \arrow[l, shift left=1.8, "R_2"]& ...  \arrow[r, shift left=1.8, "\bot"', "L_n"] \arrow[l, shift left=1.8, "R_{n-1}"]& \sP(\Theta_n) \arrow[l, shift left=1.8, "R_n"] 	
 		\end{tikzcd},
 	\end{center} 
 where the categories have the model structures given by \cref{the:cso in inftyn} or \cref{the:thetan model cat}. Moreover, only the right hand model structure is Cartesian. 
 \end{corone}

In order to use these results we need to establish some notation with regard to localizing morphisms.

\begin{notone}\label{not:infn notation}
	Fix the following notation for sets of cofibrations of $\cD$-spaces.
	\begin{enumerate}
		\item Let $S_{disc}$ be the set of cofibrations in $\sP(\Theta_k \times \DD^{n-k})$ necessary to get the discreteness condition (if $n=k$ then $S_{disc}$ is empty).
		\item Let $S_{Seg}$ be the set of cofibrations in $\sP(\Theta_k \times \DD^{n-k})$ necessary to get the Segal condition. Moreover, let $S_{Seg,disc} = S_{Seg} \cup S_{disc}$. 
		\item Let $S_{CSS}$ be the set of cofibrations in $\sP(\Theta_k \times \DD^{n-k})$ necessary to get the Segal and completeness condition. Moreover, let $S_{CSS,disc} = S_{CSS} \cup S_{disc}$. 
	\end{enumerate}
\end{notone}

We use the following terminology for the injectively fibrant $S$-local objects and the underlying quasi-category of $S$-local objects depending on the choice of $S$.

\begin{equation} \label{eq:inftyn categories}
	\begin{tabular}{c|c|c|c}
		Localizing Set & Local Object & Notation & Quasi-Cat \\ \hline 
		$S_{Seg}$ & $n$-Segal Space & $n$-Seg & $\Seg^n$\\ \hline 
		$S_{Seg,disc}$ &Segal $(\infty,n)$-Category & $Seg\cat_{(\infty,n)}$& $\segcat_{(\infty,n)}$ \\ \hline 
		$S_{CSS}$ & $n$-complete Segal space & $n$-CSS& $\CSS^n$\\ \hline 
		$S_{CSS,disc}$ & $(\infty,n)$-Category & $\cat_{(\infty,n)}$ & $\cat_{(\infty,n)}$
	\end{tabular}
\end{equation} 

In particular, a {\it theory of $(\infty,n)$-categories} is any choice of model structure in \cref{cor:inftyn} and an $(\infty,n)$-category is a fibrant object in one of these model structures. Let us review the relevant properties of the localizing sets.

\begin{remone}\label{rem:cartesian loc}
	For a given subset $S\subseteq S_{CSS,disc}$, the $S$-localized model structure is Cartesian if and only if $S \subseteq S_{CSS}$ \cite[Proposition 5.9]{bergnerrezk2020comparisonii}. This implies that the $n$-Segal space and the $n$-complete Segal space model structures are always Cartesian and the (Segal) $(\infty,n)$-category model structure is Cartesian if and only if $k=n$. Thus, by \cite[Theorem 1.1]{haugseng2015rectenrichedinftycat}, the underlying quasi-categories are enriched over themselves in the sense of \cite{gepnerhaugseng2015enrichedinftycat}. In particular, if $k=n$, then $\cat_{(\infty,n)}$ is enriched over itself meaning it is an $(\infty,n+1)$-category in the sense of \cite[Definition 6.1.5]{gepnerhaugseng2015enrichedinftycat}.
\end{remone}

\begin{remone} \label{rem:localizations are right proper}
	Notice, these localizations are mostly not right proper. Indeed, even the Segal space model structure on simplicial spaces is not right proper \cite[Example 5.19]{rasekh2017left}. This fact, combined with \cref{prop:right proper} confirms what we mentioned in \cref{rem:right proper}, namely that we need some assumptions to get an invariance result such as \cref{the:cdcov s invariant s equiv}, which we apply to the context of $(\infty,n)$-categories in \cref{the:invariance property}.
\end{remone}

Having done the review, we can move on to prove several results with regard to these localizations that are relevant in the next subsection. First we tackle geometric contractibility.

 \begin{lemone} \label{lemma:contractible thetan}
	The set of generating trivial cofibrations given in \cref{the:thetan model cat} or \cref{the:cso in inftyn} are geometrically contractible.
\end{lemone} 

\begin{proof}
	We will make several inductive arguments. First of all we prove the generating cofibration in \cref{the:thetan model cat} are geometrically contractible. Notice if $f$ is a geometrically contractible map of $\Theta_{n-1}$-spaces, then $\Sigma f$ is also contractible. Indeed, $|\Sigma f| \simeq \Sigma |f|$ and the suspension of a contractible space is contractible.
	
	Now, for a given object $[n](c_1,...,c_n)$ in $\Theta_n$, $|F[n](c_1,...,c_n)|$, $|F[1](c_1)|$ and $|F[0]|$ are contractible spaces and so 
	$$ |F[1](c_1) \coprod_{F[0]} ... F[1](c_n)| \to |F[n](c_1,...,c_n)|$$
	is an equivalence of contractible spaces. Indeed, by definition of $|-|$ (\cref{def:geometric realization}) the image of every representable $\cD$-space is contractible. Moreover, $|\Disc(E[1])|$ is contractible, hence $|D[0]| \to |E[1]|$ is also a map of contractible spaces. 
	
	We will now move on to the trivial cofibrations given in \cref{the:cso in inftyn}, assuming $S$ is geometrically contractible. Notice, the product of geometrically contractible $\cD$-spaces are geometrically contractible, which means it suffices to check the individual objects $D[n]$, $\Disc(G[n])$, $\Disc(E[1])$ and $F[d,0]$ are geometrically contractible, which follows by direct observation.
\end{proof}

Next, we move on to $S$-stability. For that we need the following lemma.

\begin{lemone} \label{lemma:technical lemma segal spaces}
	Let $X,Y,Z$ be Segal spaces and $a: X \times Y \to Z$ a map of Segal spaces, where $Z$ has a distinguished object $z$. Then for a given object $x$ in $X$, there exists an object $y$ with $a(x,y)$ equivalent to $z$ in $Z$ if and only if there exists an object $y$ in the homotopy category $\Ho(Y)$, such that $a(x,y)$ is isomorphic to $z$.
\end{lemone} 

\begin{proof} 
 Follows directly from the fact that a map is an equivalence in a Segal space if and only if it is an isomorphism in the homotopy category \cite[Subsection 5.5]{rezk2001css}.
\end{proof} 

\begin{lemone}\label{lemma:thetan stable}
	Let $S \subseteq S_{CSS}$. Then every $\sP(\cD)$-enriched category $\C$ has $S$-stable objects (\cref{def:stable objects}).
\end{lemone}

\begin{proof}
	Following \cref{def:stable objects} we need to prove that for a given $\sP(\cD)$-enriched category $\C$, the functor $\Bet\C \to R\C$ constructed in \cref{lemma:bet} gives us a bijection on equivalence classes of objects. Fix two objects $c,c'$ in $\Bet\C$. By \cref{lemma:bet}, the map $\uMap_{\Bet\C}(c,c') \to \uMap_{R\C}(c,c')$ is the fibrant replacement in the $S$-localized injective model structure, which, up to injective equivalence is obtained by applying the small object argument to the $\cD$-space $\uMap_{\Bet\C}(c,c')$ with respect to the following classes of cofibrations of $\cD$-spaces: injective equivalences, Segal morphisms and completeness morphism. 
	
	Hence, it suffices to show that all three classes of maps individually are $S$-stable. Injective equivalences are evident. Segal morphisms do not change the space $\Map_{\Bet\C}(c,c')$ at all, as the map $G[n]_0 \to F[n]_0$ is the identity for $n \geq 2$. This leaves us with completeness maps. Hence, assume $\uMap_{\Bet\C}(c,c')$ is already local with respect to the Segal maps and denote the completion by $\widehat{\uMap_{\Bet\C}(c,c')}$. 
	
	We want to prove that a morphism is a weak equivalence in the mapping $\cD$-simplicial space $\uMap_{\Bet\C}(c,c')$ if and only if it is an equivalence in the mapping $\cD$-simplicial space $\widehat{\uMap_{\Bet\C}(c,c')}$. By definition of completions, the space $\widehat{\uMap_{\Bet\C}(c,c')}[t]$ are the space of weak equivalences weak equivalences in the Segal space $\Und\uMap_{\Bet\C}(c,c')$, hence it suffices to prove an object is an equivalence in the Segal space $\Und\uMap_{\Bet\C}(c,c')$ if and only if it is an equivalence in the completed complete Segal space $\Und\widehat{\uMap_{\Bet\C}(c,c')}$. 
	
	Let $\Und\C$ be the Segal space enriched category with the same objects as $\C$ and $\uMap_{\Und\C}(c,c') = \Und\uMap_{\C}(c,c')$ and notice the equivalences in $\C$ and $\Und\C$ coincide. Now, let $\Ho\Und\C$ be the $2$-category with the same objects as $\C$ and $\uHom_{\Ho\Und\C}(c,c') = \Ho\Und\uMap_{\C}(c,c')$, where the right hand $\Ho$ is the homotopy category of the Segal space $\Und\uMap_{\C}(c,c')$ \cite[Subsection 5.5]{rezk2001css}. Then, by \cref{lemma:technical lemma segal spaces}, an object in the Segal space is an equivalence if and only if it is one in the homotopy category. However, the completion functor from a Segal space to the complete Segal space is a Dwyer-Kan equivalence \cite[Theorem 7.7]{rezk2001css} and so induces equivalent homotopy categories. Hence, we are done.
\end{proof} 

Finally, we need the following result with regard to Cartesian localizations.

\begin{lemone} \label{lemma:cart css comp}
	Let $\cD = \Theta_k \times \DD^{n-k}$ and let $X$ be a $\cD$-space that is local with respect to $S_{Seg,disc}$. Let $A$ be a $\cD$-space and $X \to \hat{X}$ be the fibrant replacement in the model structure for $(\infty,n)$-categories. Then $A \times X \to A \times \hat{X}$ is an equivalence in the model structure for $(\infty,n)$-categories.
\end{lemone}

\begin{proof}
	By assumption $X$ is already local with respect to $S_{disc}$ and so the fibrant replacement $\hat{X}$ with respect to $S_{CSS}$ is already local with respect to $S_{CSS,disc}$. However, the $n$-complete Segal space model structure is Cartesian (\cref{rem:cartesian loc}) and so the map $A \times X \to A \times \hat{X}$ is an equivalence in the $n$-complete Segal space model structure and so, in particular, a equivalence in the model structure for $(\infty,n)$-categories.
\end{proof} 

\subsection{Fibrations of \texorpdfstring{$(\infty,n)$}{(oo,n)}-Categories} \label{subsec:infn fib}
In this section we will finally study fibrations of $(\infty,n)$-categories. Fix natural numbers $0 \leq k \leq n < +\infty$ and let $\cD$ be the category $\cD=\Theta_k \times \DD^{n-k}$. Notice this category has a terminal object $([0], [0], ..., [0])$, where there are $n-k+1$ copies of $[0]$, and we will hence implicitly use the fact that $\cD$ has a terminal object whenever necessary. 	

Throughout this section $S$ is a set of cofibrations of $\cD$-spaces, such that $S \subseteq S_{CSS,disc}$ (\cref{not:infn notation}). Finally, by \cref{lemma:contractible thetan}, all cofibrations in $S_{CSS,disc}$ are geometrically contractible. We will now consider fibrations in the context of $(\infty,n)$-categories. Recall an $S$-localized $\cD$-left fibration, as defined in (\cref{def:cdleft s}), comes with a model structure (\cref{the:cdcov s}), which in our context takes the following form.

\begin{theone}\label{the:cocart model}
 Let $X$ be a $\cD$-simplicial space. There is a left proper combinatorial simplicial model structure on $\cD$-simplicial spaces over $X$, the $S$-localized $\cD$-covariant model structure, that satisfies the following conditions:
 \begin{itemize}
 	\item The cofibrations are monomorphisms.
 	\item The fibrant objects are $S$-localized $\cD$-left fibrations.
 	\item A map $A \to B$ is an equivalence if and only if for all fibrant objects $L \to X$ the induced map 
 	$$\Map_{/X}(B,L) \to \Map_{/X}(A,L)$$ 
 	is a Kan equivalence of spaces. 
    \item A weak equivalence (fibration) between fibrant objects is a level-wise equivalence (injective fibration). 
    \item It is enriched over $S$-localized injective model structure if  $S\subseteq S_{CSS}$ and in particular if $n=k$.
 \end{itemize}
\end{theone} 

\begin{proof}
 The existence and properties of the $S$-localized $\cD$-covariant model structure follows from \cref{the:cdcov s}. Notice the enrichment follows from the fact that $S$ is Cartesian if and only if $S \subseteq S_{CSS}$, by \cref{rem:cartesian loc}.
\end{proof} 

For historical reasons several notions of $\cD$-left fibrations deserve their own name.

\begin{notone}\label{not:left fib}
 We will use the following terminology with regard to the model structure and its fibrant objects, depending on the choice of localizations.
 \begin{center}
	\begin{tabular}{c|c|c|c|c} 
		$S$ & {\bf Fibration} & {\bf Model Structure} & {\bf Notation} & {\bf QCat}\\ \hline 
		$\emptyset$ & $(\infty,n)$-left & $(\infty,n)$-covariant& ${cov_{(\infty,n)}}$ & $\LFib_{(\infty,n)/X}$\\
		$S_{Seg}$ & $(\infty,n)$-Segal left & $(\infty,n)$-Segal covariant& ${Segcov_{(\infty,n)}}$ & $\SegLFib_{(\infty,n)/X}$\\
		$S_{Seg, disc}$ & Segal $(\infty,n)$-coCartesian & Segal $(\infty,n)$-coCartesian& ${SegcoCart_{(\infty,n)}}$ & $\SegcoCart_{(\infty,n)/X}$\\
		$S_{CSS}$ & $(\infty,n)$-complete Segal left & $(\infty,n)$-complete Segal covariant& ${CScov_{(\infty,n)}}$ & $\CSLFib_{(\infty,n)/X}$\\
		$S_{CSS, disc}$ & $(\infty,n)$-coCartesian & $(\infty,n)$-coCartesian & ${coCart_{(\infty,n)}}$ & $\coCart_{(\infty,n)/X}$
	\end{tabular}
\end{center}
Notice, using \cref{the:cocart model}, we have the following enrichments: The $(\infty,n)$-covariant model structure and the $(\infty,n)$-(complete) Segal covariant model structures are always enriched and the $(\infty,n)$-(Segal) coCartesian model structures are enriched if $k=n$. 
\end{notone}

Notice, the $(\infty,n)$-coCartesian model structure is in fact independent of our choice of $(\infty,n)$-categories. 

\begin{theone}
 The chain of simplicial Quillen equivalences in \cref{cor:inftyn} induces a chain of Quillen equivalences of $(\infty,n)$-coCartesian model structures
 \begin{center}
 	\begin{tikzcd}
 		\sP(\DD^n)_{/X_0} \arrow[r, shift left=1.8, "\bot"', "L_1"] & \sP(\DD^{n-1} \times \DD)_{/X_1}  \arrow[r, shift left=1.8, "\bot"', "L_2"] \arrow[l, shift left=1.8, "R_1"] & \sP(\DD^{n-2} \times \Theta_2)_{/X_2}  \arrow[r, shift left=1.8, "\bot"', "L_{n-1}"] \arrow[l, shift left=1.8, "R_2"]& ...  \arrow[r, shift left=1.8, "\bot"', "L_n"] \arrow[l, shift left=1.8, "R_{n-1}"]& \sP(\Theta_n)_{/X_n} \arrow[l, shift left=1.8, "R_n"] 	
 	\end{tikzcd},
 \end{center} 
 where for each $ 0 \leq i \leq n-1$ we have $LX_i \simeq X_{i+1}$ (or equivalently $X_i \simeq R_iX_{i+1}$). Notice only the right hand model structure is enriched over the model structure for $(\infty,n)$-categories.
\end{theone} 

\begin{proof}
	Direct application of \cref{the:cdcov s inv qe two} using \cref{lemma:contractible thetan}.
\end{proof} 

\begin{remone}
	There is a contravariant approach to $\cD$-left fibrations via $\cD$-right fibrations. Accordingly, there is also a contravariant version to all the results we state in this section. Instead of repeating of them, we will give the following translation table and only give more details when necessary.
\begin{center}
	\begin{tabular}{|c|c|c|}
		\hline 
		& {\bf Covariant} & {\bf Contravariant} \\ \hline 
		{\bf Fibration:} & $(\infty,n)$-left  & $(\infty,n)$-right\\ 
		{\bf Model:} & $(\infty,n)$-covariant & $(\infty,n)$-contravariant \\ 
		{\bf Notation:} & $(\sP(\cD)_{/X})^{cov_{(\infty,n)}}$ &  $(\sP(\cD)_{/X})^{contra_{(\infty,n)}}$ \\ \hline 
		{\bf Fibration:}& $(\infty,n)$-(complete) Segal left & Segal $(\infty,n)$-(complete) Segal right \\
		{\bf Model:} & $(\infty,n)$-(complete) Segal covariant &  $(\infty,n)$-(complete) Segal contravariant \\ 
		{\bf Notation:}& $(\sP(\cD)_{/X})^{Seg(CS)cov_{(\infty,n)}}$ & $(\sP(\cD)_{/X})^{Seg(CS)contra_{(\infty,n)}}$ \\ \hline 
		{\bf Fibration:}& (Segal) $(\infty,n)$-coCartesian & (Segal) $(\infty,n)$-Cartesian \\
		{\bf Model:} & (Segal) $(\infty,n)$-coCartesian & (Segal) $(\infty,n)$-Cartesian \\ 
		{\bf Notation:}& $(\sP(\cD)_{/X})^{(Seg)coCart_{(\infty,n)}}$  & $(\sP(\cD)_{/X})^{(Seg)Cart_{(\infty,n)}}$ \\ \hline 
	\end{tabular}
\end{center}
\end{remone}

We proceed towards studying the properties of the fibrant objects and weak equivalences in the model structures defined in \cref{the:cocart model}. First of all we can understand the fibrations by analyzing fibers, which follows from \cref{prop:fibrancies}.

\begin{propone} \label{prop:fibrancies nfib}
 Let $X$ be a $\cD$-simplicial space with weakly constant objects and let $S \subseteq S_{CSS,disc}$. Then an $(\infty,n)$-left fibrations $L \to X$ is an $S$-localized $(\infty,n)$-left fibration if and only it is fiber-wise $S$-local.  
\end{propone} 

This in particular has the following implications.

\begin{corone}  \label{cor:fibrancies nfib}
 Let $\C$ be an $(\infty,n+1)$-category and $\sL \to \C$ be an $(\infty,n)$-left fibration. Then $\sL$ is a (Segal) $(\infty,n)$-coCartesian fibration if and only if for every object $c$ in $\C$, $\Fib_c\sL$ is a (Segal) $(\infty,n)$-category. 
\end{corone}

Next, the fibrations over $(\infty,n+1)$-categories are not necessarily $(\infty,n+1)$-categories, but are quite close, as follows directly from \cref{the:cdcov loc css s}.

\begin{propone} \label{prop:doubecat}
	Let $S \subseteq S_{CSS,disc}$.
	Let $X$ be $\sP(\cD)$-enriched complete Segal space such that $\Val_1(X)$ is $S$-local and let $p: L \to X$ be an $S$-localized $\cD$-covariant fibration. Then $L$ is a complete Segal object in $S$-local $\cD$-spaces.  
\end{propone}

\begin{remone}
	\cref{prop:doubecat} implies that the domain $L$ satisfies all conditions of an $S$-localized $\sP(\cD)$-enriched complete Segal space except for the condition that $\Val(L)$ is homotopically constant.
\end{remone}

This has the following direct application.

\begin{corone} \label{cor:double seg}
	Let $\C$ be a Segal $(\infty,n+1)$-category and $p: \sL \to \C$ a Segal $(\infty,n)$-coCartesian fibration. Then $\sL$ is an $(n+1)$-Segal space.
\end{corone} 

\begin{corone} \label{cor:double css}
	Let $\C$ be an $(\infty,n+1)$-category and $p: \sL \to \C$ an $(\infty,n)$-coCartesian fibration. Then $\sL$ is an $(n+1)$-complete Segal space.
\end{corone} 

We move on to the Yoneda lemma.

\begin{theone} \label{the:yoneda}
	Let $W$ be an $\sP(\cD)$-enriched Segal space. Then $W_{x/} = D[0] \times_W W^{D[1]}$ (\cref{def:under segal}) is the $(\infty,n)$-covariant fibrant replacement of the object $x$. Moreover, if $W$ is a (Segal) $(\infty,n+1)$-category, then $W_{x/} \to X$ is in fact (Segal) $(\infty,n)$-coCartesian.
\end{theone}

\begin{proof}
	Direct result of \cref{the:undercat left}, \cref{the:levelwise Yoneda} and \cref{prop:yoneda s}. 
\end{proof}

This in particular, gives us the more usual form of the Yoneda lemma for $(\infty,n+1)$-categories. 

\begin{corone}\label{cor:yoneda} 
	Let $\C$ be a Segal $(\infty,n+1)$-category and $\sL \to \C$ a Segal $(\infty,n)$-coCartesian fibration and $c$ an object in $\C$. Then we have an equivalence of Segal $(\infty,n)$-categories 
	$$\Fib_c\sL \simeq \uMap_{/\C}(\C_{c/},\sL)$$
\end{corone}

Given the notational differences, it is instructive to state the contravariant case separately, with the proof being analogous

\begin{theone}\label{cor:yoneda contra}
	Let $W$ be an $\sP(\cD)$-enriched Segal space. Then $W_{/x} = W^{D[1]} \times_W D[0]$ (\cref{def:over segal}) is the $(\infty,n)$-contravariant fibrant replacement of $x$. In particular, if $\C$ is a (Segal) $(\infty,n+1)$-category and $c$ an object, then $\C_{c/}$ is in fact (Segal) $(\infty,n)$-Cartesian. Hence, for every Segal $(\infty,n)$-Cartesian fibration $\R \to \C$ we have an equivalence of Segal $(\infty,n)$-categories
	$$\Fib_c\R \simeq \uMap_{/\C}(\C_{/c},\R)$$
\end{theone}

We move on to weak equivalences between $(\infty,n)$-left fibrations.

\begin{theone} \label{the:fib equiv}
	Let $X$ be a $\cD$-simplicial space with weakly constant objects and let $S \subseteq S_{CSS,disc}$. 
	Let $f:L \to L'$ be a map of $(\infty,n)$-left fibration over $X$. Then $f$ is an $S$-localized $\cD$-covariant equivalence if and only if for every object $x$ in $X$, the map on fibers $\Fib_xL \to \Fib_xL'$ is an $S$-localized equivalence of $\cD$-spaces. In particular, $f$ is
	\begin{itemize}
		\item a Segal $(\infty,n)$-coCartesian equivalence if and only if it is a fiber-wise Segal $(\infty,n)$-category equivalence. 
		\item an $(\infty,n)$-coCartesian equivalence if and only if it is a fiber-wise $(\infty,n)$-category equivalence.	
	\end{itemize}
\end{theone} 

\begin{proof}
	Direct application of \cref{prop:cdcov s equiv lfib} combined with the characterization of localizations in \cref{not:left fib}.
\end{proof}

This in particular has the following corollary, which can be seen as a generalization of the result \cite[Theorem 7.7]{rezk2001css} or \cite[Proposition 8.17]{bergnerrezk2020comparisonii} from Segal spaces to Segal fibrations. 

\begin{remone} \label{rem:dk infn}
Recall that a map of Segal $(\infty,n)$-categories is a {\it Dwyer-Kan equivalence} if it is essentially surjective and for any two objects the induced map on mapping spaces is a Dwyer-Kan equivalence of Segal $(\infty,n-1)$-categories. This corresponds to the $S_{CSS,disc}$-localized Dwyer-Kan equivalences (\cref{def:sloc dk}).
\end{remone}

\begin{corone} \label{cor:fib dk}
	Let $\C$ be a Segal $(\infty,n+1)$-category. Then a map of Segal $(\infty,n)$-coCartesian fibrations $f:\sL \to \sL'$ is an $(\infty,n)$-coCartesian equivalence if and only if it is a fiber-wise Dwyer-Kan equivalence of Segal $(\infty,n)$-categories. 
\end{corone}

\begin{proof}
	By \cref{the:fib equiv} $f$ is an $(\infty,n)$-coCartesian equivalence if and only if it is a fiber-wise an $(\infty,n)$-categorical equivalence, which by \cite[Proposition 8.17]{bergnerrezk2020comparisonii} is equivalent to being a Dwyer-Kan equivalence of Segal $(\infty,n)$-categories.
\end{proof} 

We can generalize this result to maps that are not fibrations.

\begin{notone} \label{not:diag grpd}
	In \cref{def:diag} we define $\fDiag: \sP(\cD\times\DD) \to \sP(\cD)$ as the left adjoint to the inclusion functor $\Diag_*: \sP(\cD) \to \sP(\cD\times\DD)$. In the standard categorical literature this left adjoint is often known as {\it groupoidification} and denoted $(-)^{grpd}$ or $(-)^\simeq$. We will hence use the notation $(-)^{grpd}$ in this section to match our notation in this regard. 
\end{notone}

Now \cref{the:recognition principle s} immediately gives us the following result.

\begin{theone} \label{the:recognition principle inftyn}
	Let $X$ be a $\cD$-simplicial space with weakly constant objects and let $S \subseteq S_{CSS,disc}$.
	Let $f:Y \to Z$ be a map of $\cD$-simplicial spaces over $X$. For every object $x$ in $X$ let $R_x \to X$ be fibrant replacement in the $(\infty,n)$-contravariant model structure. Then $f$ is an $S$-localized $(\infty,n)$-covariant equivalence if and only if $(Y \times_X R_x)^{grpd} \to (Z \times_X R_x)^{grpd}$ is an $S$-localized equivalence of $\cD$-spaces for all objects $x$ in $X$. 
	
	In particular,
	\begin{itemize}
		\item $f$ is a Segal $(\infty,n)$-coCartesian equivalence if and only if $(Y \times_X R_x)^{grpd} \to (Z \times_X R_x)^{grpd}$ is a Segal $(\infty,n)$-category equivalence. 
		\item $f$ is an $(\infty,n)$-coCartesian equivalence if and only if $(Y \times_X R_x)^{grpd} \to (Z \times_X R_x)^{grpd}$ is an $(\infty,n)$-category equivalence.	
	\end{itemize}
\end{theone} 

Combining this result with \cref{cor:yoneda contra} gives us the following concrete application.

\begin{corone} \label{cor:recognition principle inftyn}
 	Let $W$ be be an $\sP(\cD)$-enriched Segal space and let $S \subseteq S_{CSS,disc}$. Then a morphism $f: Y \to Z$ over $W$ is an $S$-localized $(\infty,n)$-covariant equivalence if and only if $(Y \times_W W_{/x})^{grpd} \to (Z \times_W W_{/x})^{grpd}$ is an $S$-localized equivalence of $\cD$-spaces for all objects $x$. 
 
 In particular, if $\C$ is a Segal $(\infty,n)$-category and $f: \D \to \E$ over $\C$, then 
 \begin{itemize}
 	\item $f$ is a Segal $(\infty,n)$-coCartesian equivalence if and only if $(\D \times_\C \C_{/c})^{grpd} \to (\E \times_\C \C_{/c})^{grpd}$ is a Segal $(\infty,n)$-category equivalence for all objects $c$ in $\C$.
 	\item $f$ is an $(\infty,n)$-coCartesian equivalence if and only if $(\D \times_\C \C_{/c})^{grpd} \to (\E \times_\C \C_{/c})^{grpd}$ is an $(\infty,n)$-category equivalence for all objects $c$ in $\C$.	
 \end{itemize}
\end{corone} 
We move on to study the {\it $(\infty,n)$-twisted arrow} construction. 

\begin{propone}\label{prop:twisted infn}
	Let $W$ be an $\sP(\cD)$-enriched Segal space, then $\Tw(W) \to W^{op} \times W$ (as defined in \cref{prop:twisted}) is an $(\infty,n)$-left fibration. Moreover, let $S \subseteq S_{CSS,disc}$. If for all objects $x,y$ in $W$ the $\cD$-space $\map_W(x,y)$ is $S$-localized, then $\Tw(W) \to W^{op} \times W$ is an $S$-localized $\cD$-left fibration.  
\end{propone} 

\begin{proof}
	The first sentence follows from \cref{prop:twisted}, the second from \cref{cor:twisted loc}.
\end{proof}

In the particular case of (Segal) $(\infty,n+1)$-category gives us the following result.

\begin{corone}  \label{cor:twisted cocart}
	If $\C$ is a Segal $(\infty,n+1)$-category, then $\Tw(\C) \to \C^{op} \times \C$ is a Segal $(\infty,n)$-coCartesian fibration, which is an $(\infty,n)$-coCartesian fibration if $\C$ is an $(\infty,n+1)$-category.
\end{corone} 

Finally, we can use the $(\infty,n)$-twisted arrow construction to study the source and target fibrations. 

\begin{propone} \label{prop:target infn}
	Let $W$ be an $\sP(\cD)$-enriched Segal space, then the target fibration $\Tw(W) \to W$ (\cref{ex:target fibration}) is an $(\infty,n+1)$-left fibration and the source fibration $\Tw(W)^{op} \to W$ (\cref{ex:target fibration}) is an $(\infty,n+1)$-right fibration. Moreover, if $S \subseteq S_{CSS,disc}$ and for all objects $x,y$ the $\cD$-space $\map_W(x,y)$ is $S$-localized, then the target fibration $\Tw(W) \to W$ (source fibration $\Tw(W)^{op} \to W$) is an $S$-localized $(\infty,n+1)$-left fibrations ($S$-localized $(\infty,n+1)$-right fibration). 
\end{propone} 

\begin{proof}
  The first sentence follows from \cref{ex:target fibration} and \cref{ex:source fibration}. The second from \cref{cor:target loc}.
\end{proof}

This, in particular, gives us the following concrete result.

\begin{corone} \label{cor:target infn}
	If $\C$ is a Segal $(\infty,n)$-category, then the target fibration $\Tw(\C) \to \C$ is a Segal $(\infty,n+1)$-coCartesian fibration which is an $(\infty,n+1)$-coCartesian fibration if $\C$ is an $(\infty,n)$-category. Similarly, the source fibration $\Tw(\C)^{op} \to \C$ is a Segal $(\infty,n+1)$-Cartesian fibration, which is a $(\infty,n+1)$-Cartesian fibration if $\C$ is an $(\infty,n)$-category.
\end{corone}

We now move on to invariance properties.

\begin{theone} \label{the:invariance property}
 Let $g: X \to Y$ be a map of $\cD$-simplicial spaces. Then the adjunction 
 \begin{center}
 	\adjun{(\sP(\cD\times\DD)_{/X})^{(cov_{(\infty,n)})_S}}{(\sP(\cD\times\DD)_{/Y})^{(cov_{(\infty,n)})_S}}{g_!}{g^*}
 \end{center}
 is a Quillen adjunction, where both sides have $S$-localized $(\infty,n)$-covariant model structures. Moreover, the adjunction is a Quillen equivalence under the following conditions:
 \begin{itemize}
 	\item If $g$ is a level-wise complete Segal space equivalence.
 	\item If $g$ is a Dwyer-Kan equivalence and $X,Y$ have weakly constant objects.
 	\item If $g$ is an $S$-localized Dwyer-Kan equivalence and $X$ and $Y$ have weakly constant objects and Cartesian mapping spaces (\cref{def:cartesian spaces}).
 	\item If $g$ is an $S$-localized $\sP(\cD)$-enriched complete Segal space equivalence and $X,Y$ have weakly constant objects and Cartesian mapping spaces.
 \end{itemize}
\end{theone}
 
\begin{proof}
	The fact that $(g_!,g^*)$ is an adjunction is proven in  \cref{the:dcov invariant css s}. Now, if $g$ is a level-wise complete Segal space equivalence, then $(g_!,g^*)$ is a Quillen equivalence by \cref{the:dcov invariant css s}. If $g$ is a Dwyer-Kan equivalence between $\cD$-simplicial spaces with weakly constant objects, then it is a level-wise complete Segal space equivalence by \cref{lemma:dk equiv segal enriched} and so the result follows. 
	
	Finally, if $g$ is an $S$-localized Dwyer-Kan equivalence or $S$-localized $\sP(\cD)$-enriched complete Segal space equivalence between $\cD$-simplicial spaces with weakly constant objects and Cartesian mapping spaces, then the result follows from \cref{the:cdcov s invariant s equiv}. Notice, here we used the fact that Cartesian mapping spaces are already $S$-stable, by \cref{lemma:thetan stable}. This covers the last two items and hence we are done.
\end{proof}

\begin{remone}
	Most assumptions on the localizing set $S$ that we used in \cref{sec:localized} are already satisfied by the localizing maps we are using in this section (\cref{not:infn notation}). \cref{the:invariance property} is in fact an exception where the Cartesian mapping space assumption is necessary and does not always hold. For a concrete counter example using $3$-fold complete Segal spaces see \cref{ex:noncartesian mapping spaces not invariant}.
\end{remone}

In the case of Segal $(\infty,n)$-categories we can simplify the statement.

\begin{corone} \label{cor:invariance property}
	Let $g: \C \to \D$ be Dwyer-Kan equivalence of Segal $(\infty,n+1)$-categories. Then the adjunction 
	 \begin{center}
		\adjun{(\sP(\cD\times\DD)_{/\C})^{(Seg)coCart_{(\infty,n)}}}{(\sP(\cD\times\DD)_{/\D})^{(Seg)coCart_{(\infty,n)}}}{g_!}{g^*}
	\end{center}
	is a Quillen equivalence if $g$ is a Dwyer-Kan equivalence of $(\infty,n+1)$-categories (\cref{rem:dk infn}). 
	Here both sides have the (Segal) $(\infty,n)$-coCartesian model structure.
\end{corone}

\begin{proof}
	By definition $\C,\D$ have weakly constant objects and by \cref{lemma:cart css comp} $\C$ and $\D$ have Cartesian mapping spaces. So, $g$ satisfies the third condition of \cref{the:invariance property}.
\end{proof}

Having discussed invariance we move on to various exponentiability properties. First we can use \cref{the:rfib pb cov s} to get the following result with regard to exponentiability of $S$-localized $(\infty,n)$-left fibrations.

\begin{theone} \label{the:rfib pb cov infn}
	Let $S \subseteq S_{CSS}$ and let $X$ be a $\cD$-simplicial space with weakly constant objects. Then for any $(\infty,n)$-right fibration $R \to X$, the adjunction 
	\begin{center}
		\adjun{(\sP(\cD\times\DDelta)_{/X})^{(cov_{(\infty,n)})_S}}{(\sP(\cD\times\DDelta)_{/X})^{(cov_{(\infty,n)})_S}}{p_!p^*}{p_*p^*}
	\end{center}
	is an $\sP(\cD)$-enriched Quillen adjunction. Here both sides have the $S$-localized $(\infty,n)$-covariant model structure.
\end{theone}

This gives us the following valuable case for $(\infty,n)$-categories.

\begin{corone}
	Let $\C$ be a Segal $(\infty,n+1)$-category. Then for any $(\infty,n)$-right fibration $p:\R \to \C$, the adjunction 
	\begin{center}
		\adjun{(\sP(\cD\times\DDelta)_{/\C})^{CScov_{(\infty,n)}}}{(\sP(\cD\times\DDelta)_{/\C})^{CScov_{(\infty,n)}}}{p_!p^*}{p_*p^*}
	\end{center}
	is a Quillen adjunction. Here both sides have the $(\infty,n)$-complete Segal covariant model structure.
\end{corone} 

It is worth explicitly pointing out the Cartesian case that should be the most useful one.

\begin{corone}
	Let $n=k$ and let $\C$ be a Segal $(\infty,n+1)$-category. Then for any $(\infty,n)$-right fibration $p:\R \to \C$, the adjunction 
	\begin{center}
		\adjun{(\sP(\cD\times\DDelta)_{/\C})^{(Seg)coCart_{(\infty,n)}}}{(\sP(\cD\times\DDelta)_{/\C})^{(Seg)coCart_{(\infty,n)}}}{p_!p^*}{p_*p^*}
	\end{center}
	is a Quillen adjunction. Here both sides have the (Segal) $(\infty,n)$-coCartesian model structure.
\end{corone}

We now move on to the exponentiability of fibrations of $(\infty,n)$-categories. Combining \cref{the:pullback rfib s nonfib base} and \cref{cor:pullback rfib s nonfib base} gives us the following result.

\begin{theone} \label{the:pullback rfib infn}
	Let $S \subseteq S_{CSS}$ and let $X$ be a $\cD$-simplicial space with weakly constant and $S$-stable objects.
	Let $p:R \to X$ be an $(\infty,n)$-right or $(\infty,n)$-left fibration. Then the adjunctions
	\begin{center}
		\adjun{(\sP(\cD\times\DD)_{/X})^{(lev_{CSS})_S}}{(\sP(\cD\times\DD)_{/R})^{(lev_{CSS})_S}}{p^*}{p_*},
	\end{center}
	\begin{center}
		\adjun{(\sP(\cD\times\DD)_{/X})^{\cD-CSS_S}}{(\sP(\cD\times\DD)_{/R})^{\cD-CSS_S}}{p^*}{p_*}
	\end{center}
	are Quillen adjunctions. Here, in the first adjunction both sides have the $S$-localized level-wise CSS model structure (\cref{prop:levelcss s}) on the over-category and in the second adjunction both sides have the $S$-localized $\sP(\cD)$-enriched complete Segal space model structure (\cref{the:spd enriched css s}).
\end{theone}

Restricting to Segal $(\infty,n+1)$-categories and applying \cref{lemma:thetan stable} to this theorem gives us the following important corollary.

\begin{corone} 
	Let $\C$ be a Segal $(\infty,n+1)$-category. Let $p:\R \to \C$ be an $(\infty,n)$-right or $(\infty,n)$-left fibration. Then the adjunctions
	\begin{center}
		\adjun{(\sP(\cD\times\DD)_{/\C})^{(n+1)-Seg}}{(\sP(\cD\times\DD)_{/\R})^{(n+1)-Seg}}{p^*}{p_*}
	\end{center}
	\begin{center}
		\adjun{(\sP(\cD\times\DD)_{/\C})^{(n+1)-CSS}}{(\sP(\cD\times\DD)_{/\R})^{(n+1)-CSS}}{p^*}{p_*}
	\end{center}
	are Quillen adjunctions. Here the top has the $(n+1)$-Segal space model structure and the bottom the $(n+1)$-complete Segal space model structure on the over-category.
\end{corone}

Again, it is valuable to make the most useful case, $n=k$, more explicit.

 \begin{corone} \label{cor:pullback rfib s nonfib base nfib}
 	Let $n=k$ and $\C$ be an $(\infty,n+1)$-category and $p:\R \to \C$ an $(\infty,n)$-right or $(\infty,n)$-left fibration. Then the adjunctions
	\begin{center}
		\adjun{(\sP(\cD\times\DD)_{/\C})^{\Seg\cat_{(\infty,n+1)}}}{(\sP(\cD\times\DD)_{/\R})^{\Seg\cat_{(\infty,n+1)}}}{p^*}{p_*}
		\adjun{(\sP(\cD\times\DD)_{/\C})^{\cat_{(\infty,n+1)}}}{(\sP(\cD\times\DD)_{/\R})^{\cat_{(\infty,n+1)}}}{p^*}{p_*}
	\end{center}
	are Quillen adjunctions. Here the top has the Segal $(\infty,n+1)$-category model structure and the bottom the $(\infty,n+1)$-category model structure.
\end{corone}

\begin{remone}
	This result can be seen as a generalization of the fact that in the $(\infty,1)$-categorical setting Cartesian fibrations are {\it exponentiable}, which has been observed in \cite{lurie2009htt}, but also \cite{rasekh2021cartfibcss}, and is in fact an important aspect of the work in \cite{ayalafrancis2020fibrations}.
\end{remone}

We move on to the Grothendieck construction for $S$-localized $(\infty,n)$-left fibrations. For that we need the following lemma.

\begin{lemone} \label{lemma:grothendieck}
	The functor $\sbI_{\cD/\C}: \uFun(\C,\sP(\cD)) \to \sP(\cD\times\DD)_{/N_\cD\C}$ takes projective fibrations to injective fibrations. 
\end{lemone}
 
 \begin{proof}
 	By \cite[Corollary 4.5]{bergnerrezk2013elegantreedy} $\cD= \Theta_k \times \DD^{n-k}$ is an elegant Reedy category and so, by \cite[Proposition 3.15]{bergnerrezk2013elegantreedy}, the injective model structure on $\sP(\cD\times\DD)$ coincides with the Reedy model structure on $\sP(\cD\times\DD)$. Hence, it suffices to prove that for every projective natural transformation $\alpha: G \to H$, $\sbI_{\cD/\C}\alpha: \sbI_{\cD/\C}G \to \sbI_{\cD/\C}H$ is a Reedy fibration. By definition of the Reedy model structure we need to prove $\sbI_{\cD/\C}\alpha: \sbI_{\cD/\C}G \to \sbI_{\cD/\C}H$ has the lifting condition with respect to the maps 
 	$$\VEmb(A) \to \VEmb(B) \square \partial D[n] \to D[n],$$
 	where $A \to B$ is a trivial cofibration of $\cD$-spaces. 
 	By adjunction $(\sbT_{\cD/\C},\sbI_{\cD/\C})$ this is equivalent to 
 	\begin{equation} \label{eq:some map}
 	\sbT_{\cD/\C}(\VEmb(A) \to \VEmb(B) \square \partial D[n] \to D[n])
 	\end{equation}
 	being a trivial cofibration in the projective model structure on $\uFun(\C,\sP(\cD))$.
 	
 	By \cref{lemma:sbTsbI adj} $\sbT_{\cD/\C}$ commutes with tensor over $\sP(\cD)$-spaces and so the map \ref{eq:some map} is an equivalence if and only if 
 	$$\{A\} \to \{B\} \square \sbT_{\cD/\C}( \partial D[n] \to D[n])$$
 	is an equivalence, which means it suffices to prove that $\sbT_{\cD/\C}(\partial D[n] \to D[n])$ is a projective cofibration. For $n \geq 2$, $\sbT_{\cD/\C}(\partial D[n] \to D[n])$ is the identity. For the other two cases we have 
 	\begin{itemize}
 		\item $\sbT_{\cD/\C}(\emptyset \to D[0])= \{\emptyset\} \to \uMap(c_0,-)$, where $\{c_0\}: D[0] \to N_\cD\C$.
 		\item $\sbT_{\cD/\C}(\partial D[1] \to D[1])= \uMap(c_0,-) \coprod \uMap(c_1,-) \to D[1] \times_{\fDiag(N_\cD\C)} \fDiag(N_\cD\C_{/-})$, where $\{c_0 \to c_1\}: D[1] \to N_\cD\C$. 
 	\end{itemize}
  Both of which are evidently projective cofibrations. 
 \end{proof}

Applying  \cref{lemma:grothendieck} to \cref{the:grothendieck double proj s} gives us the Grothendieck construction for $S$-localized $(\infty,n)$-left fibrations.
 
\begin{theone} \label{the:grothendieck infn}
	Let $X$ be a $\cD$-simplicial space with weakly constant objects, which, in particular, includes (Segal) $(\infty,n+1)$-categories. Then we have the following diagram of $\sP(\cD)$-enriched Quillen equivalences.
	\begin{center}
		\begin{tikzcd}[row sep=0.3in, column sep=0.9in]
			\Fun(\C_X,\sP(\cD)^{inj_S})^{proj} \arrow[r, shift left = 1.8, "\sint_{\cD/\C_X}"] & 
			(\sP(\cD\times\DD)_{/N_\cD\C_X})^{(cov_{(\infty,n)})_S} \arrow[l, shift left=1.8, "\sH_{\cD/\C_X}", "\bot"'] \arrow[r, shift left=1.8, "\sbT_{\cD/\C_X}"] \arrow[d, shift left =1.8, "(r_X)_!"] & \Fun(\C_X,\sP(\cD)^{inj_S})^{proj} \arrow[l, shift left=1.8, "\sbI_{\cD/\C_X}", "\bot"'] \\
			&(\sP(\cD\times\DD)_{/\cF_\cD X})^{(cov_{(\infty,n)})_S} \arrow[u, shift left=1.8, "(r_X)^*", "\rbot"'] \arrow[d, shift right=1.8, "(u_X)^*"'] & \\
			&(\sP(\cD\times\DD)_{/\cC_\cD X})^{(cov_{(\infty,n)})_S} \arrow[u, shift right=1.8, "(u_X)_!"', "\rbot"] \arrow[d, shift left=1.8, "(c_X)_!", "\rbot"'] & \\
			&(\sP(\cD\times\DD)_{/X})^{(cov_{(\infty,n)})_S} \arrow[u, shift left=1.8, "(c_X)^*"] & 
		\end{tikzcd}, 
	\end{center}
	where the middle row has the $S$-localized $(\infty,n)$-covariant model structure the other ones have the projective model structure on the $S$-localized injective model structure. Moreover, the Quillen equivalence is an $\sP(\cD)$-enriched Quillen equivalence of $\sP(\cD)$-enriched model structures if $S \subseteq S_{CSS}$, which in particular always holds if $k=n$.
\end{theone}

In order to avoid having to draw the diagram in \cref{the:grothendieck infn} several times we state the special cases in a more succinct form.

\begin{corone}
	The diagram of Quillen equivalences in \cref{the:grothendieck infn} hold in particular with the following substitutions.
	\begin{center}
		\begin{tabular}{c|c|c}
			$S$ & $inj_S$ & $(cov_{(\infty,n)})_S$ \\ \hline 
			$\emptyset$ & $inj$ & $cov_{(\infty,n)}$ \\ 
			$S_{Seg,disc}$ & $Seg_{(\infty,n)}$ & $SegcoCart_{(\infty,n)}$ \\ 
			$S_{CSS,disc}$ & $CSS_{(\infty,n)}$ & $coCart_{(\infty,n)}$
		\end{tabular},
	\end{center}
	where the equivalence of the two last rows is one of enriched model structures if and only if $k=n$.
\end{corone}

We can again use this to get an equivalence of underlying quasi-categories. Motivated by \cref{not:underlying qcat s}, we will denote the underlying quasi-category of the $S$-localized $(\infty,n)$-covariant model structure over $X$ by $\LFib^S_{(\infty,n)/X}$. Now \cref{cor:equiv qcat s} gives us the following result.

\begin{corone} \label{cor:qcat equiv ncat}
	Let $X$ be a $\cD$-simplicial space with weakly constant objects, which in particular includes (Segal) $(\infty,n+1)$-categories. Then we have the following equivalence of quasi-categories
	\begin{center}
		\begin{tikzcd}[row sep=0.3in, column sep=0.9in]
			\Fun(\C_X,\sP(\cD))[\cW^{-1}_S] \arrow[r, shift left = 1.8, "\sint_{\cD/\C_X}"] & 
			\LFib^S_{(\infty,n)/N_\cD\C_X} \arrow[l, shift left=1.8, "\sH_{\cD/\C_X}", "\bot"'] \arrow[r, shift left=1.8, "\sbT_{\cD/\C_X}"] \arrow[d, shift left =1.8, "(r_X)_!"] & \Fun(\C_X,\sP(\cD))[\cW^{-1}_S] \arrow[l, shift left=1.8, "\sbI_{\cD/\C_X}", "\bot"'] \\
			&\LFib^S_{(\infty,n)/\cF_\cD X} \arrow[u, shift left=1.8, "(r_X)^*", "\rbot"'] \arrow[d, shift right=1.8, "(u_X)^*"'] & \\
			&\LFib^S_{(\infty,n)/\cC_\cD X} \arrow[u, shift right=1.8, "(u_X)_!"', "\rbot"] \arrow[d, shift left=1.8, "(c_X)_!", "\rbot"'] & \\
			&\LFib^S_{(\infty,n)/X} \arrow[u, shift left=1.8, "(c_X)^*"] & 
		\end{tikzcd}.
	\end{center}
	which is an equivalence of $\sP(\cD)[\cW^{-1}_S]$-enriched quasi-categories  and equivalent to the functor quasi-category $\Fun_{\QCat_{\sP(\cD)[\cW^{-1}_S]}}(\C,\sP(\cD)[\cW^{-1}_S])$ if $S \subseteq S_{CSS}$, which holds in particular if $n=k$.
\end{corone}

In particular, we get the following direct equivalences.

\begin{corone}  \label{cor:inftyn Grothendieck}
	Let $X$ be a $\cD$-simplicial space with weakly constant objects.
	There is an equivalence of quasi-categories
	\begin{center}
		\begin{tikzcd}[row sep=0.5in, column sep=0.9in]
			\Fun(\C_X,\sP(\cD))[\cW^{-1}_S] \arrow[r, shift left = 1.8] & 
			\LFib^S_{(\infty,n)/X} \arrow[l, shift left=1.8, "\bot"'] 
		\end{tikzcd}, 
	\end{center}
	Moreover, if $S \subseteq S_{CSS}$ (which holds in particular if $k=n$) then there is an equivalence $\sP(\cD)[\cW^{-1}_S]$-enriched quasi-categories 
	\begin{center}
		\begin{tikzcd}[row sep=0.5in, column sep=0.9in]
			\Fun_{\QCat_{\sP(\cD)[\cW^{-1}_S]}}(\C_X,\sP(\cD)[\cW^{-1}_S]) \arrow[r, shift left = 1.8] & 
			\LFib^S_{\cD/X} \arrow[l, shift left=1.8, "\bot"'] 
		\end{tikzcd}.
	\end{center}
\end{corone} 

\begin{corone}
	The diagram of Quillen equivalences in \cref{cor:qcat equiv ncat} and the direct equivalences \cref{cor:inftyn Grothendieck} hold in particular with the following substitutions.
	\begin{center}
		\begin{tabular}{c|c|c}
			$S$ & $\sP(\cD)[\cW^{-1}_S]$ & $\LFib^S_{(\infty,n)/X}$ \\ \hline 
			$\emptyset$ & $\sP(\cD)[\cW^{-1}]$ & $\LFib^S_{(\infty,n)/X}$ \\ 
			$S_{Seg}$ & $\Seg^n$ & $\SegLFib_{(\infty,n)}$ \\ 
			$S_{CSS}$ & $\CSS^n$ & $\CSLFib_{(\infty,n)}$
		\end{tabular}.
	\end{center}
\end{corone}

We will make two further equivalences that follow from \cref{cor:inftyn Grothendieck} more explicit. Let $k=n$ and $S=S_{Seg}$. Then, $\sP(\cD)[\cW^{-1}_S]$ is the quasi-category of Segal $(\infty,n)$-categories $\segcat_{(\infty,n)}$ (\ref{eq:inftyn categories}) and $\LFib^S_{(\infty,n)/X}$ is the quasi-category of Segal $(\infty,n)$-coCartesian fibrations, $\SegcoCart_{(\infty,n)/X}$, and \cref{cor:inftyn Grothendieck} gives us the Segal $(\infty,n)$-category enriched equivalence 
\begin{center}
	\begin{tikzcd}[row sep=0.5in, column sep=0.9in]
		\Fun_{\Seg\cat_{(\infty,n)}}(\C_X,\segcat_{(\infty,n)}) \arrow[r, shift left = 1.8] & 
		\SegcoCart_{(\infty,n)/X} \arrow[l, shift left=1.8, "\bot"'] 
	\end{tikzcd}.
\end{center} 
The next case is even more interesting. Again, let $k =n$ and $S= S_{CSS}$. Then $\sP(\cD)[\cW^{-1}_S]$ is the quasi-category of $(\infty,n)$-categories $\cat_{(\infty,n)}$ (\ref{eq:inftyn categories}) and $\LFib^S_{(\infty,n)/X}$ is the quasi-category of $(\infty,n)$-coCartesian fibrations, $\coCart_{(\infty,n)/X}$. Moreover, by \cite[Definition 6.1.5]{gepnerhaugseng2015enrichedinftycat}, $\cat_{(\infty,n)}$-enriched categories are precisely $(\infty,n+1)$-categories. 

Now, for a given $(\infty,n+1)$-category $X$, the corresponding $\sP(\cD)$-enriched category $\C_X$ is an $\sP(\cD)$-enriched Dwyer-Kan equivalence of $\sP(\cD)$-enriched complete Segal spaces (\cref{lemma:zigzag strictification}), which correspond to fully faithful and essentially surjective morphisms of $(\infty,n+1)$-categories (in the sense of \cite[Definition 5.3.1, Definition 5.3.3]{gepnerhaugseng2015enrichedinftycat}) and so an equivalence of $(\infty,n+1)$-categories \cite[Corollary 5.5.4]{gepnerhaugseng2015enrichedinftycat}, giving us an equivalence of $(\infty,n+1)$-categories
\begin{equation}\label{eq:strict}
\Fun_{\cat_{(\infty,n+1)}}(\C_X,\cat_{(\infty,n)}) \simeq \Fun_{\cat_{(\infty,n+1)}}(X,\cat_{(\infty,n)}).
\end{equation}
Now, combining the equivalence in \ref{eq:strict} with the result \cref{cor:inftyn Grothendieck} with a focus on $(\infty,n+1)$-categories, we get the following result.

\begin{corone} \label{cor:common grothendieck}
	Let $k=n$, $S=S_{CSS}$ and $\C$ be an $(\infty,n+1)$-category. Then the equivalence in \cref{cor:inftyn Grothendieck} induces an equivalence of $(\infty,n+1)$-categories 
	\begin{center}
		\begin{tikzcd}[row sep=0.5in, column sep=0.9in]
			\Fun_{\cat_{(\infty,n+1)}}(\C,\cat_{(\infty,n)}) \arrow[r, shift left = 1.8] & 
			\coCart_{(\infty,n)/\C} \arrow[l, shift left=1.8, "\simeq"'] 
		\end{tikzcd}
	\end{center} 
  and similarly an equivalence of $(\infty,n+1)$-categories 
 \begin{center}
  	\begin{tikzcd}[row sep=0.5in, column sep=0.9in]
  		\Fun_{\cat_{(\infty,n+1)}}(\C^{op},\cat_{(\infty,n)}) \arrow[r, shift left = 1.8] & 
  		\Cart_{(\infty,n)/\C} \arrow[l, shift left=1.8, "\simeq"'] 
  	\end{tikzcd}.
  \end{center} 
\end{corone}

This result is in some sense the most common form of the Grothendieck construction that we would expect in the setting of $(\infty,n)$-categories. 

\begin{exone} \label{ex:twisted infn}
	Let $\C$ be an $(\infty,n+1)$-category. Then, following \cref{ex:twisted}, the twisted arrow $(\infty,n)$-coCartesian fibration (\cref{cor:twisted cocart}) corresponds (under the correspondence given in \cref{cor:common grothendieck}) to the mapping object functor
	$$\uMap_\C(-,-):\C^{op} \times \C \to \cat_{(\infty,n)}.$$
	 Similarly, using \cref{ex:source and target fib}, the target $(\infty,n+1)$-coCartesian fibration and the source $(\infty,n+1)$-Cartesian fibration (\cref{cor:target infn}) correspond, again via \cref{cor:common grothendieck}, to {\it over-category functors} and {\it under-category functors}
	 $$ \C_{/-}: \C \to \cat_{(\infty,n+1)},$$
	 $$ \C_{-/}: \C^{op} \to \cat_{(\infty,n+1)},$$
	 respectively.
\end{exone}

We will end this section with an analysis of {\it universal $S$-localized $\cD$-left fibration} and, in particular, the {\it universal $(\infty,n)$-coCartesian fibration}. Denote by $\LFib^S_{(\infty,n)}$ the sub-right fibration of $\sOall_{/\widehat{N_\cD\sP(\cD)^S}}$ consisting of $S$-localized $(\infty,n)$-left fibrations. Notice, by \cref{cor:lfib classifier s}, the map of quasi-categories $(\CSS_{\sP(\cD)^S})_{/\widehat{N_\cD\sP(\cD)^S}} \to \LFib^S_{(\infty,n)}$ is an equivalence. Applying this to the case $S= S_{CSS,disc}$ and using notation in \ref{eq:inftyn categories} we get the following special case.

\begin{corone}\label{cor:lfib classifier ninf}
	The map $(-)^*\pi_*: (\cat_{(\infty,n+1)})_{/\cat_{(\infty,n)}} \to \coCart_{(\infty,n)}$ is an equivalence of $(\infty,n+1)$-categories.
\end{corone} 

\begin{remone} \label{rem:universal cocart}
	Applying \cref{rem:universal lfib s} to \cref{cor:lfib classifier ninf} it follows that there is a {\it universal $(\infty,n)$-coCartesian fibration} of the form $\pi_*: (\cat_{(\infty,n)})_{D[0]/} \to \cat_{(\infty,n)}$. Moreover, for every $(\infty,n)$-left fibration $p:L \to \C$, there is a (homotopically unique) pullback 
	\begin{center}
		\begin{tikzcd}[row sep=0.5in, column sep=0.5in]
			L \arrow[r, dashed] \arrow[d] \arrow[dr, phantom, "\ulcorner", very near start] \arrow[rr, bend right=20]& (\cat_{(\infty,n)})_{D[0]/} \arrow[r] \arrow[d, "\pi_*"] \arrow[dr, phantom, "\ulcorner", very near start] & \widehat{N_\cD\sP(\cD)}_{D[0]/}  \arrow[d, "\pi_*"] \\
			N_\cD\C \arrow[r, dashed] \arrow[rr, bend right=20]&   \cat_{(\infty,n)} \arrow[r, hookrightarrow] & \widehat{N_\cD\sP(\cD)}
		\end{tikzcd}.
	\end{center}
	that factors through $\cat_{(\infty,n)}$ if and only if $p$ is $(\infty,n)$-coCartesian.
\end{remone} 

Again, \cite[Proposition 2.4]{rasekh2021univalence} implies the following result regarding univalence.

\begin{corone} \label{cor:universal cocart univalent}
	The universal $(\infty,n)$-coCartesian fibration is univalent.
\end{corone}

\begin{remone} \label{rem:univalence connection to other work}
	The study of universal coCartesian fibrations and their relation to univalence has a very rich history. 
	\begin{enumerate}
		\item The underlying groupoid of the universal $(\infty,0)$-coCartesian fibration has been studied under names, such as {\it ``classifying space of spaces"}. It is in fact one of the first examples of univalent morphisms in higher category theory, due to the work by Kapulkin and Lumsdaine \cite{kapulkinlumsdaine2012kanunivalent}.
		\item The actual universal $(\infty,0)$-coCartesian fibration is commonly known as the {\it universal left fibration} \cite{cisinski2019highercategories,lurie2009htt,rasekh2021univalence}, again the univalence thereof was well-established.
		\item The universal $(\infty,1)$-coCartesian fibration was studied abstractly in \cite[Subsection 3.3.2]{lurie2009htt}. Moreover, it has been studied via internal $\infty$-categories (also known as {\it complete Segal objects}) in \cite{stenzel2020comprehension} and \cite{rasekh2017cartesian}.
	\end{enumerate}
\end{remone} 

Notice this notion of univalence is not necessarily appropriate (as discussed in \cite[Section 7]{rasekh2021univalence} for the universal $(\infty,1)$-Cartesian fibration) and we would expect a more refined definition of univalence more suitable for the $(\infty,n)$-categorical setting. See \cref{subsec:applications} for more details.
 
\subsection{Low-Dimensional Examples}\label{subsec:examples}
In this last subsection we consider several low-dimensional cases of the results in \cref{subsec:infn fib} more explicitly. We then use those examples to construct counter-examples, justifying the necessity of several conditions. 

Let us start with the lowest possible case. If $n=0$, then $\cD = \Theta_k \times \DD^{n-k}$ is necessarily the terminal category and $\sP(\cD)=\s$. In that case, the $\cD$-covariant model structure is simply the covariant model structure for simplicial spaces, studied extensively in \cite{rasekh2017left} and reviewed in \cref{subsec:lfib}. Notice, we can use this example in particular to observe that the Segal condition in \cref{cor:yoneda} and \cref{cor:yoneda contra} is necessary when deducing that $X^{D[1]} \times_X D[0]$ is the $(\infty,1)$-covariant fibrant replacement. Indeed a counter-example for the case $X= G[2]$ is given in \cite[Remark 3.48]{rasekh2017left}. As, according to \cref{not:infn notation}, the localizing set for $n=0$ is empty, we move on to the next level to get interesting examples.

Let us move on to the case $n=1$. In that case $\cD = \Theta_k \times \DD^{n-k}= \DD$ and so $\sP(\cD) = \ss$. Hence, the possible localizing sets (as described in \cref{not:infn notation}) are:
$$S_{Seg} = \{ G[n] \to F[n] : n \geq 2\}$$
$$S_{CSS} = S_{Seg} \cup \{ F[0] \to E[1]\}$$
$$S_{disc} = \emptyset$$
Let use the notation $\sss = \sP(\DD\times\DD)$ for the category of bisimplicial spaces. Then for every bisimplicial space $X$, we have two $S$-localized $(\infty,1)$-covariant model structure on bisimplicial spaces over $X$: The (Segal) $(\infty,1)$-coCartesian model structure. By \cref{the:cocart model} both model structures are left proper, combinatorial, simplicial and enriched over the (complete) Segal space model structures on simplicial spaces. 

Notice several results in \cref{sec:infn fib} rely on the assumption that $X$ has weakly constant objects. We want to show that this assumption is in fact necessary which we illustrate in the following example.

\begin{exone}
	Let $F[1]$ be the simplicial space embedded vertically in bisimplicial spaces. Then by \cref{the:diag local ncov s} and \cref{prop:dcov vs diag s} the adjunction 
	\begin{center}
		\adjun{(\sss_{/F[1]})^{coCart_{(\infty,1)}}}{(\ss_{/F[1]})^{CSS}}{\Diag^*}{\Diag_*}
	\end{center}
	is a Quillen equivalence. Let $G[2] \to F[1]$ be the map given as $G[2]= F[1] \coprod_{F[0]} F[1] \to F[0] \coprod_{F[0]}F[1] \cong F[1]$. Then we have pullback squares (using \cref{not:brackets})
	\begin{center}
		\pbsq{F[1]}{G[2]}{F[0]}{F[1]}{}{}{}{<0>}, 
		\pbsq{F[0]}{G[2]}{F[0]}{F[1]}{}{}{}{<1>}
	\end{center} 
	and so the map is a fiber-wise complete Segal space, however, is not a complete Segal space fibration as $F[1]$ is a complete Segal space, whereas $G[2]$ is not (and in fact not even a Segal space). This implies that $G[2] \to F[1]$ is a fiber-wise complete Segal space without being an $(\infty,1)$-coCartesian fibration justifying the weakly constant object assumption in \cref{prop:fibrancies nfib}.
	
	Next, notice the map $F[1] \coprod F[0] \to G[2]$ is a fiber-wise equivalence (in fact identity) over $F[1]$, but evidently not an equivalence and so gives a justification for the weakly constant object assumption in \cref{the:fib equiv}. Moreover, the two maps $<0>,<1>:F[0] \to F[1]$ are already $(\infty,1)$-coCartesian fibrations and so equal to their fibrant replacement, this example also applies to \cref{the:recognition principle inftyn}.
	
	Next, the adjunction 
	\begin{center}
		\adjun{(\sss_{/G[2]})^{coCart_{(\infty,1)}}}{(\sss_{/F[2]})^{coCart_{(\infty,1)}}}{p_!}{p^*}
	\end{center}
	induced by the equivalence $p:G[2] \to F[2]$ is not a Quillen equivalence as the derived unit map of $<02>: F[1] \to F[2]$ is $\partial F[1] \to F[1]$, which is not an equivalence justifying the weakly constant object requirement in \cref{the:invariance property}.
	
	We can use a similar argument for another counter-example. Letting $p$ be the $(\infty,1)$-right fibration $p = <02>: F[1] \to F[2]$ the adjunction 
	\begin{center}
		\adjun{(\sss_{/F[2]})^{coCart_{(\infty,1)}}}{(\sss_{/F[2]})^{coCart_{(\infty,1)}}}{p_!p^*}{p_*p^*}
	\end{center}
	is not a Quillen adjunction as it takes the equivalence $G[2] \to F[2]$ over $F[2]$ to the map $\partial F[1] \to F[1]$ over $F[2]$, which is certainly not an equivalence giving us a counter-example to \cref{the:rfib pb cov infn} without the weakly constant object assumption.
\end{exone} 

Having thoroughly covered the case $n=1$ we move on to the case $n=2$. In this case we either have $\cD = \Theta_2$ or $\cD=\DD^2$. If $\cD = \Theta_2$, then the resulting model structure for $(\infty,2)$-categories is in fact Cartesian (\cref{the:thetan model cat}) and hence this model structure behaves similarly to the complete Segal space model structure and we will not discuss it further. 

On the other hand, the $(\infty,2)$-categorical model structure on $\sP(\DD^2)$ is given via $2$-fold complete Segal spaces, which is in fact not Cartesian and we will use it to show that the assumption of Cartesian mapping spaces in \cref{the:cdcov s invariant s equiv}, and its $(\infty,n)$-categorical application in \cref{the:invariance property}, is necessary.

\begin{exone} \label{ex:noncartesian mapping spaces not invariant}
	Fix a bisimplicial space $A$ that is equivalent to the terminal object in the $2$-fold complete Segal space model structure and denote a choice of fibrant replacement by $A \to \hat{A}$. Let $\Sigma A$ be the category enriched over bisimplicial spaces with two objects $0,1$ and one non-trivial mapping object $\Map_{\Sigma A}(0,1) = A$. Define $\Sigma \hat{A}$ similarly and notice the equivalence $A \to \hat{A}$ gives us a localized equivalence of trisimplicial spaces $N_{\DD^2}(\Sigma A) \to N_{\DD^2}(\Sigma \hat{A})$.
	
	Let $B$ be a fibrant bisimplicial space in the $2$-fold complete Segal space model structure. Then the projection map $\hat{A} \times B \to\hat{A}$ is an $(\infty,2)$-coCartesian fibration over $\hat{A}$. If \cref{the:invariance property} was valid, then the derived counit map, which is the top map of the pullback diagram 
	\begin{center}
		\pbsq{N_{\DD^2}(\Sigma A) \times B}{N_{\DD^2}(\Sigma \hat{A}) \times B}{N_{\DD^2}(\Sigma A)}{N_{\DD^2}(\Sigma \hat{A})}{}{\pi_2}{\pi_2}{}
	\end{center}
	would be an $(\infty,2)$-coCartesian equivalence over $N_{\DD^2}\Sigma \hat{A}$, which would in particular imply that the diagonal is an equivalence in the $2$-fold complete Segal space model structure (\cref{the:diag local ncov s}).
	
	Now we have $N_{\DD^2}(\Sigma A)= D[1] \times A \coprod_{\partial D[1] \times A} \partial D[1] \simeq D[1] \times A \coprod_{\partial D[1] \times A} \partial D[1] \times \hat{A}$. Hence, 
	\begin{equation}\label{eq:pp one}
     \fDiag N_{\DD^2}(\Sigma A \times B) = \Delta[1] \times A \times B \coprod_{\partial D[1] \times A} \partial D[1] \times \hat{A}
    \end{equation} 
	and similarly 
	\begin{equation} \label{eq:pp two} 
	 \fDiag N_{\DD^2}(\Sigma \hat{A} \times B) = \Delta[1] \times \hat{A} \times B \coprod_{\partial D[1] \times \hat{A}} \partial D[1] \times \hat{A}= \Delta[1] \times \hat{A} \times B 
	\end{equation} 
	and so the map from \ref{eq:pp one} to \ref{eq:pp two} is simply the pushout product $(A \to \hat{A}) \square ( \emptyset \to B) \square (\partial \Delta[1] \to \Delta[1])$.  This pushout product is generally not a trivial cofibration. 
	
	Indeed, for a concrete counter-example let $A= F[0,1]$ and $B=F[1,0]$. By definition of the $2$-fold complete Segal space model structure (\cref{the:cso in inftyn}), the map $F[0,0] \to F[0,1]$ is an equivalence and $F[1,0]$ is fibrant, but the map $F[1,0] \to F[1,1]$, which is the pushout product of $F[0,0] \to F[0,1]$ with $\emptyset \to F[1,0]$ is not an equivalence.
\end{exone} 

\section{What needs to be done?}\label{sec:next}
 In this last section we discuss work that is left to be done. We use the notation and the results from the previous sections. 
 
 \begin{enumerate}
 	\item {\bf Comparison with other Models of Fibrations:} In the $(\infty,2)$-categorical setting there are several alternative studies of fibrations using marked (and scaled) simplicial sets \cite{lurie2009goodwillie,gagnaharpazlanari2021inftytwocartfib,garciastern2021twocat}. We do know that our approach to fibrations does compare appropriately when restricted to $(\infty,1)$-categories (this is proven in \cite[Appendix B]{rasekh2017left}), however, it remains to be seen whether this generalizes to a comparison between $(\infty,2)$-categorical fibrations given in \cref{subsec:infn fib} and scaled simplicial approach. This could be quite challenging as these fibrations use very different models of $(\infty,2)$-categories, which cannot be directly compared. However, there is the possibility of using recent results directly comparing $\Theta_2$-sets and $2$-complicial sets \cite{bergnerrovelliozornova2021comparisonthetatwo}.
 	
 	\item {\bf Fibrations for Lax Functors} As stated in \cref{the:grothendieck infn}, the notion of Cartesian fibration given here corresponds to {\it pseudo-functors} rather than {\it lax functors}. This in contrast to work in \cite{gagnaharpazlanari2021inftytwocartfib}, which consider fibrations with different levels of laxity. 
 	
	The notion $\cD$-left fibration given here (\cref{def:nleft}) is a direct generalization of left fibrations, which by its $(\infty,1)$-categorical nature is pseudo by definition. Hence, one valuable next step is to work out the appropriate generalization of level-wise left fibrations that is able to include the data of a lax functor. 
 
 	One interesting possibility is to use the $(\infty,n)$-categorical twisted arrow construction (\cref{cor:twisted cocart}/\cref{ex:twisted infn}) to study lax constructions, similar to what has been done in \cite{gepnerhaugsengnikolaus2017laxcolimits} in the context of lax limits of $(\infty,1)$-categories.
 	
 	\item {\bf Fibrations in $\Theta_n$-Spaces:}
 	Starting with a simplicial presheaf model of $(\infty,n)$-categories, there are generally two methods to generalize it to a model of $(\infty,n+1)$-categories: Either we apply the $\Theta$-construction to the diagram category, and applying this line of thinking to complete Segal spaces precisely gives us $\Theta_n$-spaces \cite{rezk2010thetanspaces}, or we take the product of the diagram with $\DD$, and applying this to complete Segal spaces gives us $n$-fold complete Segal spaces \cite{barwick2005nfoldsegalspaces}. 
 	
 	From this perspective the results of this paper can be understood as applying the product with $\DD$ to a given diagram $\cD$, leaving us with the question whether we can give an analogous construction on presheaf categories $\sP(\Theta\D)$. It is expected that such a fibration for $\Theta\cD$-spaces would require a very different construction, as for a given $\Theta\cD$-spaces $X$, the analogue to the value for $\cD$-simplicial spaces (\cref{def:val}) would just be a space $X[0]$ and hence could not be used to define a fibration with fibers in $\cD$-spaces. 
 	
 	\item {\bf Fibrations for Set-Valued Diagrams:} 
 	In \cite{ara2014highersegal} Ara uses the theory of {\it Cisinski model categories} \cite{cisinski2006cisinskimodelstructure} to construct a model structure on $\Theta_n$-sets, $\Fun(\Theta_n^{op},\set)$, and Quillen equivalences between $\Theta_n$-spaces and $\Theta_n$-sets, generalizing the Quillen equivalences between quasi-categories and complete Segal spaces due to Joyal and Tierney \cite{joyaltierney2007qcatvssegal}.
 	
 	We already know that we can adjust the Quillen equivalences by Joyal and Tierney to both prove an equivalence of $(\infty,0)$-coCartesian fibrations over quasi-categories and complete Segal spaces \cite[Appendix B]{rasekh2017left} and similarly an equivalence of $(\infty,1)$-coCartesian fibrations \cite{rasekh2021cartfibmarkedvscso}. 
 	
 	The results of this work suggest two natural next steps: First, can we use a similar approach to Ara (possibly again via results in \cite{cisinski2006cisinskimodelstructure}) to define fibrations for $\Theta_n$-set-enriched Segal spaces and prove it is equivalent to $\Theta_n$-space enriched Segal spaces studied here? Moreover, assuming we can develop a theory of fibrations of $\Theta_n$-spaces, as outlined in the previous item, can we construct an equivalent theory of fibrations of $\Theta_n$-sets, using similar methods?
 
 	\item {\bf Direct Grothendieck Construction:}
 	In \cref{subsec:enriched groth} we construct a zigzag of Quillen equivalences between functor categories and fibrations over $\cD$-simplicial spaces with weakly constant objects. This zigzag already enables us to gain deep insight into various properties of localized $\cD$-left fibrations and in particular $(\infty,n)$-coCartesian fibrations. 
 	
 	However, for many purposes it would be more advantageous to have a direct Quillen equivalence between functors and fibrations, similar to the {\it straightening construction} \cite[Theorem 3.2.0.1]{lurie2009htt}. The key difference is that the straightening construction takes two steps as once. For a given map of simplicial sets $X \to S$ it first constructs a simplicially enriched category $\mathfrak{C}[S]$, known as the {\it strictification} \cite[Definition 1.1.5.1]{lurie2009htt} and then a simplicial enriched functor from $\mathfrak{C}[S]$ valued in simplicial sets. Hence, this direct construction relies on the fact that there is a direct Quillen equivalence $\mathfrak{C}[-]$ from quasi-categories to Kan enriched categories \cite[Theorem 2.2.5.1]{lurie2009htt}. 
 	
 	In order to generalize this construction we would need strong strictification results that take $\cD$-simplicial spaces to  $\sP(\cD)$-enriched categories, which does not exist yet and might not be possible for a general category $\cD$. A possibly more feasible first step is to strictify (discrete) $\Theta_{k} \times \DD^{n-k+1}$-spaces to $\Theta_{k} \times \DD^{n-k+1}$-enriched categories in an efficient way that would then enable us to construct a direct Grothendieck construction. 
 	
 	One first advantage of a direct comparison is that we could better understand the relation between functors and universal $(\infty,n)$-coCartesian fibrations (\cref{rem:universal cocart}), rather than having to trace that relation through various equivalences (\cref{cor:qcat equiv ncat}). 

 	\item {\bf $(\infty,\infty)$-Cartesian Fibrations} We applied these results to simplicial presheaves on $\Theta_k \times \DD^{n-k}$, which model $(\infty,n)$-categories. However, the results in \cref{sec:localized} are general enough to apply to simplicial presheaves on $\Theta_\infty$ or other infinite versions of $\Theta_k \times \DD^{n-k}$. The reason we stayed with finite diagrams is that there is no agreed upon notion of {\it $(\infty,\infty)$-categories} using $\Theta_\infty$-spaces. Assuming that a theory of $(\infty,\infty)$-categories is developed using simplicial presheaves on some diagram, we expect the results here to immediately provide a theory of fibrations for $(\infty,\infty)$-categories.   

 	\item {\bf Flagged $(\infty,n)$-Cartesian Fibrations:} In \cite{ayalafrancis2018flagged} Ayala and Francis define {\it flagged $(\infty,n)$-categories} and prove these are equivalent to Segal $\Theta_n$-spaces. The notion of a flagged $(\infty,n)$-category can be readily generalized to a notion of {\it flagged $(\infty,n)$-Cartesian fibration}. Given the equivalence due to Ayala and Francis, we should expect a generalization to an equivalence between Segal $(\infty,n)$-Cartesian fibrations (\cref{not:left fib}) and flagged $(\infty,n)$-Cartesian fibrations. 
 
 	\item {\bf Target Fibrations and Limits:} In \cref{ex:twisted infn} we prove that for every $(\infty,n)$-category $\C$, there is a {\it target $(\infty,n)$-coCartesian fibration} with fiber over an object $c$ in $\C$ given by the over-category $\C_{/c}$. In the case of $(\infty,1)$-categories this coCartesian fibration is a Cartesian fibration if and only if $\C$ has finite limits \cite[Lemma 6.1.1.1]{lurie2009htt}. On the other hand the theory of limits of $(\infty,n)$-categories is still in its early stages \cite{gagnaharpazlanari2020inftytwolimits} and so this perspective suggests the following next step. 
 	
 	Assuming we have developed a theory of limits in a model of $(\infty,n)$-categories based on simplicial presheaves (such as $\Theta_n$-spaces), we want to prove that the target $(\infty,n)$-coCartesian fibration is a $(\infty,n)$-Cartesian fibration if and only if the $(\infty,n)$-category is finitely complete. Next, assuming we could prove a comparison theorem with fibrations in other models (as discussed above) we could then deduce the notion of being a finitely complete $(\infty,n)$-category is also model independent.
 
 	\item {\bf A model for the $(\infty,n)$-Category of $(\infty,n)$-Categories:}
 	One common way for constructing quasi-categories is to apply the {\it simplicial nerve} to a Kan enriched category, which, in particular, can be obtained as the sub-category of fibrant-cofibrant objects of a simplicial model category \cite[Proposition 1.1.5.10]{lurie2009htt}. Using this approach we can in particular obtain the quasi-category of spaces from the Kan model structure or the quasi-category of $(\infty,1)$-categories from the complete Segal space model structure. 
 	
 	However, the simplicial nerve construction relies on the {\it strictification functor} $\mathfrak{C}[-]$, which is computationally quite challenging (and a lot of effort has gone into a better understanding \cite{riehl2011simpcatofqcat,duggerspivak2011rigidification}). So, for some important examples, such as spaces and $(\infty,1)$-categories, it is beneficial to have more direct constructions. For example a key aspect in the construction of the {\it simplicial model of univalent foundations} due to Kapulkin and Lumsdaine is an explicit construction of the $(\infty,0)$-category of $(\infty,0)$-categories, which is necessary to prove the existence of {\it univalent universes} \cite{kapulkinlumsdaine2012kanunivalent}. 
 	
 	One established method for the explicit construction of such $(\infty,1)$-categories is via $(\infty,1)$-coCartesian fibrations, examples of which can be found in \cite{cisinski2019highercategories} for simplicial sets and \cite{kazhdanvarshvsky2014yoneda,rasekh2018model}  for simplicial spaces. Hence we expect that $(\infty,n)$-coCartesian fibrations can be used in a similar fashion to construct explicit models of $(\infty,n)$-categories by focusing on $(\infty,n)$-coCartesian fibrations over the appropriate diagram. 
 	
 	\item {\bf Internal $(\infty,n)$-Categories:} One of the benefits of $(\infty,1)$-Cartesian fibrations is that it enables us to effectively study {\it internal $\infty$-categories} via {\it complete Segal objects} \cite{rasekh2017cartesian}. This in particular allows us to internalize an important result by Rezk \cite[Proposition 7.6]{rezk2001css}, that characterizes equivalences of complete Segal objects via Dwyer-Kan equivalences \cite[Theorem 3.13]{rasekh2017cartesian}.
 	
 	Having a theory of $(\infty,n)$-Cartesian fibrations we should be able to study {\it internal $(\infty,n)$-categories} in a similar manner using {\it $\Theta_n$-objects}. This, in particular, could then play a crucial role in the study of univalence and $(\infty,n)$-topos theory (as already discussed in \cref{subsec:applications}).
 \end{enumerate}

 \bibliographystyle{alpha}
 \bibliography{main}
\end{document}